\documentclass[11pt]{article}

\usepackage[labelsep=endash]{caption}

\usepackage{graphicx}
\usepackage{latexsym,amssymb}
\usepackage{amsthm}
\usepackage{indentfirst}
\usepackage{amsmath}
\usepackage{color}
\usepackage{xcolor}

\newtheorem{theoremalph}{Theorem}
\newtheorem*{maintheorem}{Main Theorem}

\newtheorem{Theorem}{Theorem}[section]
\newtheorem*{Theorem A}{Theorem A}
\newtheorem*{Theorem A'}{Theorem A'}

\newtheorem*{Conj*}{Conjecture}
\newtheorem*{Conjecture}{Conjecture}
\newtheorem{Definition}[Theorem]{Definition}
\newtheorem{Proposition}[Theorem]{Proposition}
\newtheorem{Lemma}[Theorem]{Lemma}
\newtheorem{Sublemma}[Theorem]{Sublemma}

\newtheorem*{Remark}{Remark}

\newtheorem{Remark-numbered}[Theorem]{Remark}
\newtheorem{Remarks-numbered}[Theorem]{Remarks}
\newtheorem{Corollary}[Theorem]{Corollary}

\newtheorem*{Claim}{Claim}
\newtheorem{Claim-numbered}{Claim}

 \def\NN{{\mathbb N}} 

 \def\RR{{\mathbb R}} 
\def\TT{{\mathbb T}}

 \def\ZZ{{\mathbb Z}}

  \def\cG{{\cal G}}  
\def\cB{{\cal B}}  \def\cH{{\cal H}} \def\cN{{\cal N}} \def\cT{{\cal T}}
\def\cC{{\cal C}}   \def\cO{{\cal O}} \def\cU{{\cal U}}
   \def\cP{{\cal P}} 
\def\cE{{\cal E}}    \def\cW{{\cal W}}
\def\cF{{\cal F}}  \def\cL{{\cal L}}  \def\cX{{\cal X}}

\newcommand{\id}{\operatorname{Id}}

\newcommand{\sing}{{\operatorname{Sing}}}

\newcommand{\orb}{\operatorname{Orb}}

\def\Int{\operatorname{Interior}}
\def\dim{\operatorname{dim}}

\def\diam{\operatorname{Diam}}
\def\Sing{\operatorname{Sing}}
\def\orb{\operatorname{Orb}}
\def\supp{\operatorname{supp}}

\def\Interior{\operatorname{Interior}}
\def\lip{\operatorname{Lip}}

\makeatletter
\renewcommand{\l@section}{\@dottedtocline{2}{3.8em}{3.2em}}
\renewcommand{\l@subsection}{\@dottedtocline{3}{3.8em}{3.2em}}
\newcommand{\subsectionruninhead}{\@startsection{subsection}{2}{0mm}{-\baselineskip}{-0mm}{\bf\large}}
\newcommand{\subsubsectionruninhead}{\@startsection{subsubsection}{3}{0mm}{-\baselineskip}{-0mm}{\bf\normalsize}}
\makeatother
\begin{document}

\title{Homoclinic tangencies and singular hyperbolicity for three-dimensional vector fields}

\author{Sylvain Crovisier \and Dawei Yang
\footnote{S.C was partially supported by the ANR projects \emph{DynNonHyp} BLAN08-2 313375
and \emph{ISDEEC} ANR-16-CE40-0013. 
by the Balzan Research Project of J. Palis and by the ERC project 692925 \emph{NUHGD}.
D.Y. was partially supported by NSFC 11271152, ANR project \emph{DynNonHyp} BLAN08-2313375 and and A Project Funded by the Priority Academic Program Development of Jiangsu Higher Education Institutions(PAPD).}}

\date{\today}

\maketitle

\begin{abstract}
We prove that any vector field on a three-dimensional compact manifold
can be approximated in the $C^1$-topology by one which is singular hyperbolic
or by one which exhibits a homoclinic tangency associated to a regular hyperbolic
periodic orbit. This answers a conjecture by Palis~\cite{Pal00}.

During the proof we obtain several other results with independent interest:
a compactification of the rescaled sectional Poincar\'e flow and a generalization of Ma\~n\'e-Pujals-Sambarino theorem for three-dimensional $C^2$ vector fields with singularities.
\end{abstract}

{\small
\tableofcontents
}

\section{Introduction}
\subsection{Homoclinic tangencies and singular hyperbolicity}
A main problem in differentiable dynamics is to describe a class of systems as large as possible.
This approach started in the 60's with the theory of \emph{hyperbolic systems} introduced by Smale and Anosov, among others.
A flow $(\varphi_t)_{t\in \RR}$ on a manifold $M$, generated by a vector field $X$, is hyperbolic if its chain-recurrent set~(defined in \cite{Con}) is the finite union of invariant sets $\Lambda$
that are hyperbolic: each one is endowed with an invariant splitting into continuous sub-bundles
$$TM|{_\Lambda}=E^s\oplus (\RR X) \oplus E^u$$
such that $E^s$ (resp. $E^u$) is uniformly contracted by $D\varphi_T$ (resp. $D\varphi_{-T}$) for some $T>0$.
The dynamics of these systems has been deeply described.

The set of hyperbolic vector fields is open and dense in the space $\cX^r(M)$ of $C^r$-vector fields
when $M$ is an orientable surface and $r\geq 1$~\cite{Pe} or when $M$ is an arbitrary surface and $r=1$~\cite{Pu}.
Smale raised the problem of the abundance
of hyperbolicity for higher dimensional manifolds. Newhouse's work \cite{Ne1}
for surface diffeomorphisms implies that hyperbolicity is not dense in the spaces $\cX^r(M)$, $r\ge 2$, once the dimension of $M$ is larger or equal to three. Indeed a bifurcation called \emph{homoclinic tangency} leads to a robust phenomenon for $C^2$-vector fields
which is expected to be one of the main obstructions to hyperbolicity.
A vector field $X$ has a homoclinic tangency if there exist a hyperbolic non-singular periodic orbit
$\gamma$ and an intersection $x$ of the stable and unstable manifolds of $\gamma$ which is not transverse
(i.e. $T_xW^s(\gamma)+T_xW^u(\gamma)\neq T_xM$). It produces rich wild behaviors.

The flows with singularities may have completely different dynamics.
The class of three-dimensional vector fields contains a very early example of L. N. Lorenz \cite{Lo}, that he called ``butterfly attractors''.
Trying to understand this example, some robustly non-hyperbolic attractors (which are called ``geometrical Lorenz attractors'') were constructed by \cite{ABS,Gu,GW}.
The systems cannot be accumulated by vector fields with homoclinic tangencies. While they are less wild than systems
in the Newhouse domain, this defines a new class of dynamics, where the lack of hyperbolicity is related to the presence of a singularity.

Morales, Pacifico and Pujals \cite{MPP} have introduced the notion of \emph{singular hyperbolicity} to characterize these Lorenz-like dynamics.
A compact invariant set $\Lambda$ is called singular hyperbolic if either $X$ or $-X$ satisfies the following property.
There exists an invariant splitting into continuous sub-bundles
$$TM|{_\Lambda }=E^{s}\oplus E^{cu}$$
and a constant $T>0$ such that:
\begin{itemize}
\item[--] domination:
$\forall x\in\Lambda, u\in E^s(x)\setminus\{0\}, v\in E^{cu}(x)\setminus\{0\}, \; \frac{\|{D\varphi_{T}}.u\|}{\|u\|}\leq 1/2\frac{\|{D\varphi_{T}}.v\|}{\|v\|},$
\item[--] contraction: $\forall x\in\Lambda, \|{D\varphi_T}|_{E^s}(x)\|\leq 1/2,$
\item[--] sectional expansion: $\forall x\in\Lambda, \forall P \in \operatorname{Gr}_2(E^{cu}(x)),\; |\text{Jac}({D\varphi_{-T}}|_{P})|\leq 1/2.$
\end{itemize}
When $\Lambda\cap \sing(X)=\emptyset$ this notion coincides with the hyperbolicity.
A flow is \emph{singular hyperbolic} if its chain-recurrent set is a finite union of singular-hyperbolic sets.
This property defines an open subset in the space of $C^1$ vector fields.
Such flow has good topological and ergodic properties, see~\cite{AP}. Note that hyperbolicity implies singular hyperbolicity by this definition.

Palis ~\cite{palis,Pal00,Pal05,Pal08} formulated conjectures for typical dynamics of diffeomorphisms and vector fields.
He proposed that homoclinic bifurcations and Lorenz-like dynamics are enough to characterize the non-hyperbolicity.
For three-dimensional manifolds, this is more precise (see also~\cite[Conjecture 5.14]{BDV}):

\begin{Conj*}[Palis]
For any $r\geq 1$ and any three dimensional manifold $M$,
every vector field in $\cX^r(M)$ can be approximated by one which is hyperbolic, or by one which display a homoclinic
tangency, a singular hyperbolic attractor or a singular hyperbolic repeller.
\end{Conj*}

In higher topologies $r>1$, such a general statement is for now out of reach, but more techniques have been
developed in the $C^1$-topology~\cite{C-asterisque}. This allows us to prove the conjecture above for $r=1$. This has been announced in~\cite{CY}.

\begin{maintheorem}
On any three-dimensional compact manifold $M$,
any $C^1$ vector field can be approximated in $\cX^1(M)$ by singular hyperbolic vector
fields, or by ones with homoclinic tangencies.
\end{maintheorem}

An important step towards this result was 
the dichotomy between hyperbolicity and homoclinic
tangencies for surface diffeomorphisms by Pujals and Sambarino~\cite{PS1}.
Arroyo and Rodriguez-Hertz then obtained~\cite{ARH} a version of the theorem above for vector fields without singularities.
The main difficulty of the present paper is to address the existence of singularities.

\medskip
The Main Theorem allows to extend Smale's spectral theorem for the $C^1$-generic vector fields
far from homoclinic tangencies.
Recall that an invariant compact set $\Lambda$ is \emph{robustly transitive} for a vector field $X$ if there exists a neighborhood $U$
of $\Lambda$ and $\cU\subset \cX^1(M)$ of $X$ such that, for any
$Y\in \cU$, the maximal invariant set of $Y$ in $U$
is transitive (i.e. admits a dense forward orbit).

\begin{Corollary}\label{c.main}
If $\dim(M)=3$,
there exists a dense open subset $\cU \subset \cX^1(M)$ such that, for any vector field $X\in \cU$
which can not be approximated by one exhibiting a homoclinic tangency, the chain-recurrent set is the union
of finitely many robustly transitive sets.
\end{Corollary}

\smallskip

When $\dim(M)=3$,
there exists Newhouse domains in $\cX^r(M)$, $r\geq 2$.
But we note that there is no example of a non empty open set $\cU\subset \cX^1(M)$ such that homoclinic tangencies occur on a dense
subset of $\cU$. This raises the following conjecture.

\begin{Conjecture}
If $\dim(M)=3$,
any vector field can be approximated in $\cX^1(M)$ by 
singular hyperbolic ones.
\end{Conjecture}

Even for non-singular vector fields, the conjecture above is open. It claims the density of hyperbolicity
and it has a counterpart for surface diffeomorphisms, sometimes called Smale's conjecture.

The chain-recurrent set naturally decomposes into invariant compact subsets that are called
\emph{chain-recurrence classes} (see~\cite{Con} and Section~\ref{ss.chain}).
The conjecture holds if one shows that for $C^1$-generic vector fields, any chain-recurrence class has a dominated splitting
(see Theorem~\ref{t.GY} below). An important case would be to rule out for $C^1$-generic vector fields
the existence of non-trivial chain-recurrence classes containing a singularity with a complex eigenvalue.

Note that the conjecture also asserts that for typical $3$-dimensional vector fields,
the non-trivial singular behaviors only occur inside Lorenz-like attractors and repellers.

\subsection{Dominated splittings in dimension 3}
The first step for proving the hyperbolicity or the singular hyperbolicity
is to get a dominated splitting for the tangent flow $D\varphi$.
For surface diffeomorphisms far from homoclinic tangencies, this has been proved
in~\cite{PS1}. For vector fields, it is in general
much more delicate.
Indeed one has to handle with sets which may contain both regular orbits
(for which $\RR X$ is a non-degenerate invariant sub-bundle) and singularities:
for instance it is not clear how to extend the tangent splittings at a singularity
and alongs its stable and unstable manifolds.

Since the flow direction does not see any hyperbolicity, it is fruitful to consider
another linear flow that has been defined by Liao~\cite{Lia63}.
To each vector field $X$, one introduces the singular set ${\rm Sing}(X)=\{\sigma\in M:~X(\sigma)=0\}$
and the \emph{normal bundle} $\cN$ which
is the collection of subspaces ${\cal N}_x=\{v\in T_x M:~\left<X(x),v\right>=0\}$ for $x\in M\setminus{\rm Sing}(X)$.
One then defines the \emph{linear Poincar\'e flow} $(\psi_t)_{t\in \RR}$ by projecting orthogonally on $\cN$ the tangent flow:
$$\psi_t(v)={\rm D}\varphi_t(v)-\frac{\left<{\rm D}\varphi_t(v),X(\varphi_t(x))\right>}{\|X(\varphi_t(x))\|^2}X(\varphi_t(x)).$$

If $\Lambda\subset M$ is a (not necessarily compact) invariant set, an invariant continuous splitting
$TM|_{\Lambda}=E\oplus F$ is dominated if it satisfies the first item
of the definition of singular hyperbolicity stated above.
We also say that the linear Poincar\'e flow over $\Lambda\setminus{\rm Sing}(X)$ admits a \emph{dominated splitting}
when there exists a continuous invariant splitting ${\cal N}|_{\Lambda\setminus{\rm Sing}(X)}={\cal E}\oplus {\cal F}$ and a constant $T>0$
such that
$$\forall x\in\Lambda\setminus{\rm Sing}(X), u\in \cE(x)\setminus\{0\}, v\in \cF(x)\setminus\{0\}, \quad \frac{\|{\psi_{T}}.u\|}{\|u\|}\leq \frac 1 2\;\frac{\|{\psi_{T}}.v\|}{\|v\|}.$$

A dominated splitting on $\Lambda$ for the tangent flow $D\varphi$ always extends to the closure of $\Lambda$: for that reason, one usually considers
compact sets. But the dominated splittings of the linear Poincar\'e flow can not always be extended to the closure of the invariant set $\Lambda$ since the closure of $\Lambda$ may contain singularities, where the linear Poincar\'e flow is not defined. It is however natural to consider the linear Poincar\' e flow:
for vector fields away from systems exhibiting a homoclinic tangency, the natural splitting of a hyperbolic saddles is dominated for $\psi$ (see~\cite{GY}),
but this is not the case in general for $D\varphi$. In particular the existence of dominated splitting for the linear Poincar\'e flow
does not imply the existence of a dominated splitting for the tangent flow.

However the equivalence between these two properties holds for $C^1$-generic vector fields on chain-transitive sets (whose definition is recalled in
Section~\ref{ss.chain}).

\begin{theoremalph}\label{Thm-domination}
When $\dim(M)=3$,
there exists a dense G$_\delta$ subset $\cG\subset\cX^1(M)$
such that for any $X\in\cG$ and any chain-transitive set $\Lambda$ (which is not reduced to a periodic orbit or a singularity),
the linear Poincar\'e flow over $\Lambda\setminus {\rm Sing}(X)$ admits a non-trivial dominated splitting
if and only if the tangent flow over $\Lambda$ does.
\end{theoremalph}

The Main Theorem will then follow easily: as already mentioned, far from homoclinic tangencies,
the linear Poincar\'e flow is dominated, hence the tangent flow is also.
The singular hyperbolicity then follows from the domination of the tangent flow,
as it was shown in~\cite{ARH} for chain-recurrence classes without singularities and
in a recent work by Gan and Yang~\cite{GY} for the singular case.

Theorem A is a consequence of a similar result for (non-generic)
$C^2$ vector fields.

\begin{Theorem A'}[Equivalence between dominated splittings]
When $\dim M=3$, Consider
a $C^2$ vector field $X$ on $M$ with a flow $\varphi$ and
an invariant compact set $\Lambda$ with the following properties:
\begin{itemize}
\item[--] Any singularity $\sigma\in \Lambda$
is hyperbolic, has simple real eigenvalues; the smallest one is negative and its
invariant manifold satisfies $W^{ss}(\sigma)\cap \Lambda=\{\sigma\}$.
\item[--] For any periodic orbit in $\Lambda$, the smallest Lyapunov exponent is negative.
\item[--] There is no subset of $\Lambda$ which is a repeller supporting
a dynamics which is the suspension of an irrational circle rotation.
\end{itemize}
Then the tangent flow $D\varphi$ on $\Lambda$ has a dominated splitting
$TM|_{\Lambda}=E\oplus F$ with $\dim(E)=1$
if and only if the linear Poincar\'e flow on $\Lambda\setminus {\rm Sing}(X)$
has a dominated splitting.
\end{Theorem A'}

\subsection{Compactification of the normal flow}\label{ss.compactification}
In this paper, we use techniques for studying flows that may be useful for
other problems.

\paragraph{Local fibered flows.}
In order to analyze the tangent dynamics and to prove the existence of a dominated splitting over a set $\Lambda$,
one needs to analyze the local dynamics near $\Lambda$.
For a diffeomorphism $f$, one usually lifts the local dynamics to the tangent bundle:
for each $x\in M$, one defines a diffeomorphism $\widehat f_x\colon T_xM\to T_{f(x)}M$,
which preserves the $0$-section (i.e. $\widehat f_x(0_x)=0_{f(x)}$ and is locally conjugated to $f$
through the exponential map. It defines in this way a local fibered system on the bundle $TM\to M$.
For flows one introduces a similar notion.

\begin{Definition}[Local fibered flow]\label{d.local-flow}
Let $(\varphi_t)_{t\in \RR}$ be a continuous flow over a compact metric space $K$, and let $\cN\to K$ be a continuous Riemannian vector bundle.
A \emph{local $C^k$-fibered flow} $P$ on $\cN$ is a continuous family of $C^k$-diffeomorphisms $P_{t}\colon \cN_x\to \cN_{\varphi_t(x)}$, for $(x,t)\in K\times \RR$,
preserving $0$-section with the following property.

There is $\beta_0>0$ such that for each $x\in K$, $t_1,t_2\in \RR$, and $u\in \cN_x$ satisfying
$$\|P_{s.t_1}(u)\|\leq \beta_0 \text{ and } \|P_{s.t_2}(P_{t_1}(u))\|\leq \beta_0 \text{ for each }s\in[0,1],$$
then we have
$$P_{t_1+t_2}(u)=P_{t_2}\circ P_{t_1}(u).$$
\end{Definition}

For a vector field $X$, a natural way to lift the dynamics is to define the \emph{Poincar\'e map} by projecting the normal spaces $\cN_x$ and
$\cN_{\varphi_t(x)}$ above two points of a regular orbit using the exponential map\footnote{When $x$ is periodic and $t$ is the period of $x$, this map is defined by Poincar\'e to study the dynamics in a neighborhood of a regular periodic orbit.}.
Then the Poincar\'e map $P_t$ defines a local diffeomorphism from $\cN_x$ to
$\cN_{\varphi_t(x)}$. The advantage of this construction is that the dimension has been dropped by $1$.

\paragraph{Extended flows.}
A new difficulty appears: the domain of the Poincar\'e maps degenerate near the singularities.
For that reason one introduces the \emph{rescaled sectional Poincar\'e flow}:
$$P^*_t(u)=\|X(\varphi_t(x))\|^{-1}\cdot P_t(\|X(x)\|\cdot u).$$

In any dimension this can be compactified
as a fibered lifted flow, assuming that the singularities are not degenerate.

\begin{theoremalph}[Compactification]\label{t.compactification}
Let $X$ be a $C^k$-vector field, $k\geq 1$, over a compact manifold $M$.
Let $\Lambda\subset M$ be a compact set which is invariant by the flow $(\varphi_t)_{t\in\RR}$
associated to $X$ such that $DX(\sigma)$ is invertible at each singularity $\sigma\in \Lambda$.

Then, there exists a topological flow $(\widehat \varphi_t)_{t\in \RR}$ over a compact metric space $\widehat \Lambda$,
and a local $C^k$-fibered flow $(\widehat P^*_t)$ over a Riemannian vector bundle $\widehat {\cN M}\to \widehat \Lambda$
whose fibers have dimension $\dim(M)-1$ such that:
\begin{itemize}
\item[--] the restriction of $\varphi$ to $\Lambda\setminus {\rm Sing}(X)$ embeds in $(\widehat \Lambda, \widehat \varphi)$ through a map $i$,
\item[--] the restriction of $\widehat {\cN M}$ to $i(\Lambda\setminus \sing(X))$ is isomorphic to the normal bundle $\cN M|_{\Lambda\setminus \sing(X)}$
through a map $I$, which is fibered over $i$ and which is an isometry along each fiber,
\item[--] the fibered flow $\widehat P^*$ over $i(\Lambda\setminus \sing(X))$ is conjugated by $I$
near the zero-section to the rescaled sectional Poincar\'e flow $P^*$:
$$\widehat P^*=I\circ P^*\circ I^{-1}.$$
\end{itemize}
\end{theoremalph}

The linear Poincar\'e flow introduced by Liao~\cite{Lia63} was compactified by Li, Gan and Wen \cite{lgw-extended}, who called it \emph{extended linear Poincar\'e flow}.
Liao also introduced its rescaling~\cite{Lia89}. Gan and Yang~\cite{GY}
considered the rescaled sectional Poincar\'e flow and proved some uniform properties.

\paragraph{Identifications structures for fibered flows.}
Since $P^*$ is
defined as a sectional flow over
$\varphi$,
the holonomy by
the flow gives a projection between fibers of points close.
The fibered flow thus comes with an additional structure, that we call
\emph{$C^k$-identification}: let $U$ be an open set in $\Lambda \setminus {\rm Sing}(X)$,
for any points $x,y\in U$ 
close enough, there is a $C^k$-diffeomorphism
$\pi_{y,x}\colon \cN_y\to \cN_x$ which satisfies $\pi_{z,x}\circ \pi_{y,z}=\pi_{y,x}$.
These identifications $\pi_{x,y}$ satisfy several properties (called \emph{compatibility with the flow}),
such as some invariance. See Section~\ref{ss.identifications} for precise definitions.

\subsection{Generalization of Ma\~n\'e-Pujals-Sambarino's theorem for flows}
Let us consider an invariant compact set $\Lambda$ for a $C^2$ flow $\varphi$
such that the linear Poincar\'e flow on $\Lambda\setminus {\rm Sing}(X)$
admits a dominated splitting $\cN=\cE\oplus \cF$.
Under some assumptions,
Theorem A' asserts that the tangent flow is then dominated.
The existence of a dominated splitting $TM|_{\Lambda}=E\oplus F$ with
$\dim(E)=1$ and $X\subset F$ is equivalent to the fact that $\cE$ is uniformly
contracted by the rescaled linear Poincar\'e flow
(see Proposition \ref{p.mixed-domination}).
It is thus reduced to prove that the one-dimensional
bundle $\cE$ of the splitting of the two-dimensional bundle $\cN$ is uniformly contracted by the extended sectional Poincar\'e flow $P^*$.

For $C^2$ surface diffeomorphisms, the existence of a dominated splitting
implies that the (one-dimensional) bundles are uniformly hyperbolic,
under mild assumptions: this is one of the main results of Pujals and Sambarino \cite{PS1}. A result implying the hyperbolicity for
one-dimensional endomorphisms was proved before by Ma\~n\'e~\cite{Man85}.

Our main technical theorem is to extend that technique to the case of
local fibered flows with $2$-dimensional dominated fibers.
As introduced in section~\ref{ss.compactification} we will assume the existence of identifications compatible
with the flow, over an open set $U$. 
We will assume that, on a neighborhood of the complement $\Lambda\setminus U$, the fibered flow contracts the bundle $\cE$:
this is a non-symmetric assumption on the splitting $\cE\oplus \cF$.
See Section~\ref{s.fibered} for the precise definitions.

\begin{theoremalph}[Hyperbolicity of one-dimensional extremal bundle]\label{Thm:1Dcontracting}
Consider a $C^2$ local fibered flow $(\cN,P)$ over a topological flow $(K,\varphi)$ on a compact metric space such that:
\begin{enumerate}
\item there is a dominated splitting $\cN=\cE\oplus \cF$ and $\cE,\cF$
have one-dimensional fibers,
\item there exists a $C^2$-identification compatible with $(P_t)$ on an open set $U$,
\item $\cE$ is uniformly contracted on an open set $V$ containing $K\setminus U$.
\end{enumerate}
Then, one of the following properties occurs:
\begin{itemize}
\item[--] there exists a periodic orbit $\cO\subset K$ such that $\cE|_{\cO}$ is not uniformly contracted,
\item[--] there exists a normally expanded irrational torus,
\item[--] $\cE$ is uniformly contracted above $K$.
\end{itemize}
\end{theoremalph}
\medskip

This theorem is based on the works initiated by Ma\~n\'e~\cite{Man85}
and Pujals-Sambarino~\cite{PS1}, but we have to address additional
difficulties:

\begin{itemize}

\item[--] The time of the dynamical system is not discrete.
This produces some shear between pieces of orbits that remain close.
In the non-singular case, Arroyo and Rodriguez-Hertz~\cite{ARH}
already met that difficulty.

\item[--] Pujals-Sambarino's theorem does not hold in general for
fibered systems. In our setting, the existence of an identification structure
is essential.

\item[--] We adapt the notion of ``induced hyperbolic returns" from~\cite{CP}:
this allows us to work with the induced dynamics on $U$ where the identifications
are defined.

\item[--] In the setting of local flows, we have to replace some global arguments in \cite{PS1,ARH}.
\item[--] The role of the two bundles $\cE$ and $\cF$ is non-symmetric.
In particular we do not have the topological hyperbolicity of $\cF$.
The construction of Markovian boxes (Section~\ref{s.markov})
then requires other ideas, which can be compared to arguments in~\cite{CPS}.

\end{itemize}

\subsection*{Structure of the paper}
In Section~\ref{s.compactification}, we compactify the rescaled sectional
Poincar\'e flow and prove Theorem~\ref{t.compactification}.
Local fibered flows are studied systematically in Section~\ref{s.fibered}.
The proof of Theorem~\ref{Thm:1Dcontracting} occupies Sections~\ref{s.topological-hyperbolicity} to~\ref{s.uniform}. The Theorem A' is obtained in Section~\ref{s.MPS-theo}. The proofs of global genericity results, including the Main Theorem,
Corollary~\ref{c.main} and Theorem~\ref{Thm-domination}
are completed in Section~\ref{s.generic}.

\paragraph{Acknowledgements.}
We are grateful to S. Gan, R. Potrie, E. Pujals and  L. Wen
for discussions related to this work.
We also thank the Universit\'e Paris 11, Soochow University and Pekin University for their hospitality.

\section{Compactification of the sectional flow}\label{s.compactification}
In this section we do not restrict $\dim (M)$
to be equal to $3$ and we prove Theorem~\ref{t.compactification}
(see Theorem~\ref{t.compactified2}).
Let $X$ be a $C^k$ vector field, for some $k\geq 1$, and let $(\varphi_t)_{t\in\RR}$ be its associated
flow. We also assume that $DX(\sigma)$ is invertible at each singularity. In particular ${\rm Sing}(X)$ is finite.

Several flows associated to $\varphi$ have already been used in~\cite{Lia89,lgw-extended,GY}.
We describe here slightly different constructions and introduce the ``extended rescaled sectional Poincar\'e flow".

\subsection{Linear flows}
We associate to $(\varphi_t)_{t\in\RR}$ several $C^{k-1}$ linear and projective flows.

\paragraph{The \emph{tangent flow} $(D\varphi_t)_{t\in \RR}$} is the flow on the tangent bundle
$TM$ which fibers over $(\varphi_t)_{t\in\RR}$ and is obtained by differentiation.

\paragraph{The \emph{unit tangent flow} $(U\varphi_t)_{t\in\RR}$} is the flow on the unit tangent bundle
$T^1M$ obtained from $(D\varphi_t)_{t\in\RR}$ by normalization:
$$U\varphi_t.v=\frac{D\varphi_t.v}{\|D\varphi_t.v\|} \text{ for } v\in T^1M.$$
Sometimes we prefer to work with the projective bundle $PTM$. The unit tangent flow induces a flow on this bundle, that we also denote by  $(U\varphi_t)_{t\in\RR}$ for simplicity.

\paragraph{The \emph{normal flow} $(\cN \varphi_t)_{t\in\RR}$.}
For each $(x,u)\in T^1M$ we denote $\cN T^1_xM$  as the vector subspace orthogonal to
$\RR.u$ in $T_xM$. This defines a vector bundle $\cN T^1M$ over the compact manifold $T^1M$.
We define the normal flow $(\cN \varphi_t)_{t\in\RR}$ on $\cN T^1M$ which fibers above the unit tangent flow
as the orthogonal projection $\cN \varphi_t.v$ of $D\varphi_t.v$ on $(D\varphi_t.u)^\perp$.

\paragraph{The \emph{linear Poincar\'e flow} $(\psi_t)_{t\in\RR}$.}
The normal bundle $\cN(M\setminus {\rm Sing}(X))$ over the space of non-singular points $x$
is the union of the vector subspaces $\cN_x=X(x)^\perp$. It can be identified with the restriction of the
bundle $\cN T^1M$ over the space of pairs $(x,\frac{X(x)}{\|X(x)\|})$ for $x\in M\setminus {\rm Sing}(X)$.
The linear Poincar\'e flow $(\psi_t)_{t\in\RR}$ is the restriction of $(\cN \varphi_t)$ to
$\cN(M\setminus {\rm Sing}(X))$.

\subsection{Lifted and sectional flows}\label{ss.def-flow}
\paragraph{The \emph{sectional Poincar\'e flow} $(P_t)_{t\in\RR}$.}
There exists $r_0>0$ such that the ball $B(0,r_0)$ in
each fiber of the bundle $\cN(M\setminus {\rm Sing}(X))$ projects on $M$ diffeomorphically
by the exponential map. For each $x\in M\setminus {\rm Sing}(X)$, there exists $r_x\in (0,r_0)$
such that for any $t\in [0,1]$, the holonomy map of the flow induces a local diffeomorphism
$P_t$ from $B(0_x,r_x)\subset \cN_x$ to a neighborhood of  $0_{\varphi_t(x)}$ in $B(0_{\varphi_t(x)},r_0)\subset \cN_{\varphi_t(x)}$.
This extends to a local flow $(P_t)_{t\in\RR}$ in a neighborhood of the $0$-section
in $\cN(M\setminus {\rm Sing}(X))$, that is called the sectional Poincar\'e flow. It is tangent to
$(\psi_t)_{t\in\RR}$ at the $0$-section of $\cN(M\setminus  {\rm Sing}(X))$.

The normal bundle and the sectional Poincar\'e flow are $C^k$.

\paragraph{The \emph{lifted flow} $(\cL\varphi_t)_{t\in\RR}$.}
Similarly, for each $t\in [0,1]$ and $x\in M$, the map
$$\cL\varphi_t\colon y\mapsto \exp^{-1}_{\varphi_t(x)}\circ \varphi_t\circ \exp_x(y)$$
sends diffeomorphically a neighborhood of $0$ in $T_xM$ to a neighborhood of $0$
in $B(0,r_0)\subset T_{\varphi_t(x)}M$.
This extends to a local flow $(\cL\varphi_t)_{t\in\RR}$ in a neighborhood of the $0$-section of $TM$, that is called the lifted flow. It is tangent to $(D\varphi_t)_{t\in\RR}$ at the $0$-section.

\paragraph{The \emph{fiber-preserving lifted flow} $(\cL_0\varphi_t)_{t\in\RR}$.}
We can choose not to move the base point $x$ and obtain a fiber-preserving flow $(\cL_0\varphi_t)_{t\in\RR}$, defined by:
$$\cL_0\varphi_t(y)=\exp^{-1}_{x}\circ \varphi_t\circ \exp_x(y).$$
Since the $0$-section is not preserved, this is no ta local flow and it will
be considered only for short times.

\subsection{Rescaled flows}
\paragraph{The \emph{rescaled sectional} and \emph{linear Poincar\'e flows}
$(P^*_t)_{t\in\RR}$, $(\psi^*_t)_{t\in\RR}$.}
Since $DX(\sigma)$ is invertible at each singularity, there exists
$\beta>0$ such that at any $x\in M\setminus {\rm Sing}(X)$
$$r_x>\beta \|X(x)\|.$$
We can thus rescale the sectional Poincar\'e flow. We get for each $x\in M\setminus {\rm Sing}(X)$ and $t\in [0,1]$ a map $P^*_t$ which sends
diffeomorphically $B(0,\beta)\subset \cN_x$ to $\cN_{\varphi_t(x)}$, defined by:
$$P^*_t(y)=\|X(\varphi_t(x))\|^{-1}.P_t(\|X(x)\|.y).$$
Again, this induces a local flow $(P_t^*)_{t\in\RR}$ in a neighborhood of the $0$-section
in $\cN(M\setminus {\rm Sing}(X))$, that is called the \emph{rescaled sectional Poincar\'e flow}.
Its tangent map at the $0$-section defines the rescaled linear Poincar\'e flow $(\psi_t^*)$.

\paragraph{The \emph{rescaled lifted flow} $(\cL\varphi_t^*)_{t\in\RR}$ and
the \emph{rescaled tangent flow} $(D\varphi_t^*)_{t\in\RR}$.} The rescaled lifted flow is defined on a neighborhood of the $0$-section
in $TM$ by
$$\cL\varphi^*_t(y)=\|X(\varphi_t(x))\|^{-1}.\cL\varphi_t(\|X(x)\|.y).$$
Its tangent map at the $0$-section defines the rescaled tangent flow $(D\varphi_t^*)_{t\in\RR}$.

\paragraph{The \emph{rescaled fiber-preserving lifted flow} $(\cL_0\varphi_t^*)_{t\in\RR}$}
is defined similarly:
$$\cL_0\varphi^*_t(y)=\|X(x)\|^{-1}.\cL_0\varphi_t(\|X(x)\|.y).$$

\subsection{Blowup}\label{ss.blow-up}

We will consider a compactification of $M\setminus {\rm Sing}(X)$ and of the tangent bundle
$TM|_{M\setminus \text{Sing}(X)}$ which allows to extend the line field $\RR X$. This is given by the classical blowup.

\paragraph{The manifold $\widehat M$.}
We can blowup $M$
at each singularity of $X$ and get a new compact manifold $\widehat M$
and a projection $p\colon \widehat M\to M$ which is one-to-one above $M\setminus {\rm Sing}(X)$.
Each singularity $\sigma\in{\rm Sing}(X)$ has been replaced by the projectivization
$PT_\sigma M$.

More precisely, at each (isolated) singularity $\sigma$,
one can add $T^1_\sigma M$ to $M\setminus \{\sigma\}$ in order
to build a manifold with boundary. Locally,
it is defined by the chart
$[0, \varepsilon)\times T^1_\sigma M\to (M\setminus \{\sigma\})\cup T^1_\sigma M$
given by:
$$ (s,u)\mapsto 
\begin{cases}
&\exp(s.u) \text{ if }s\neq 0,\\
&u
\text{ if } s=0.
\end{cases}
$$
One then gets $\widehat M$ by identifying points $(0,u)$ and $(0,-u)$
on the boundary.

It is sometimes convenient to
lift the dynamics on $M\setminus \{\sigma\}$ near $\sigma$
and work in the local coordinates
$(-\varepsilon, \varepsilon)\times T^1_\sigma M$.
These coordinates define a double covering of an open subset of the blowup $\widehat M$
and induce a chart from the quotient $(-\varepsilon, \varepsilon)\times T^1_\sigma M /_{(s,u)\sim(-s,-u)}$
to a neighborhood of $p^{-1}(\sigma)$ in $\widehat M$.

\paragraph{The \emph{extended flow $(\widehat \varphi_t)_{t\in \RR}$}.}
The following result is proved in~\cite[section 3]{takens}.
\begin{Proposition}
The flow $(\varphi_t)_{t\in \RR}$ induces a $C^{k-1}$ flow $(\widehat \varphi_t)_{t\in\RR}$
on $\widehat M$ which is associated to a $C^{k-1}$ vector field $\widehat X$. 
For $\sigma\in {\rm Sing}(X)$, this flow preserves  $PT_\sigma M$, and acts on it as the projectivization of $D\varphi_t(\sigma)$;
the vector field $\widehat X$ coincides at $u\in PT_\sigma M$ with $DX(\sigma).u$
in $T_{u}(PT_\sigma M)\cong T_\sigma M/\RR.u$.
\end{Proposition}
In particular the tangent bundle $T\widehat M$ extends $TM|_{M\setminus \text{Sing}(X)}$,
the linear flow  $D\widehat \varphi$ extend $D\varphi$
and the vector field $\widehat X$  extends $X$.
Note that each eigendirection $u$ of $DX(\sigma)$ at a singularity $\sigma$
induces a singularity of $\widehat X$.

\begin{Remark}
In~\cite{takens}, the vector field and the flow are extended locally on the space $(-\varepsilon, \varepsilon)\times T^1_\sigma M$,
but the proof shows that these extensions are invariant under the map $(s,u)\mapsto (-s,-u)$, hence are also defined on $\widehat M$.
\end{Remark}

\paragraph{The \emph{extended bundle} $\widehat {TM}$
and \emph{extended tangent flow} $(\widehat {D\varphi_t})_{t\in \RR}$.}
One associates to $\widehat M$ the bundle $\widehat {TM}$
which is the pull-back of the bundle $\pi\colon TM\to M$  over $M$
by the map $p\colon \widehat M\to M$.
It can be obtained as the restriction of the first projection
$\widehat M\times TM\to \widehat M$ to the set of pairs
$(x,v)$ such that $p(x)=\pi(v)$.
It is naturally endowed with the pull back metric of $TM$ and it is trivial in a neighborhood of preimages
$p^{-1}(z)$, $z\in M$.

The tangent flow $(D\varphi_t)_{t\in \RR}$ can be pull back to $\widehat {TM}$
as a $C^{k-1}$ linear flow $(\widehat {D\varphi_t})_{t\in \RR}$
that we call extended tangent flow.

\paragraph{The \emph{extended line field} $\widehat {\RR X}$.}
The vector field $X$  induces a line field ${\RR X}$ on $M\setminus \text{Sing}(X)$
which admits an extension to $\widehat {TM}$. It is defined locally as follows.

\begin{Proposition}\label{p.extended-field}
At each singularity $\sigma$, let $U$ be a small neighborhood in $M$
and $\widehat U=(U\setminus \{\sigma\})\cup PT_\sigma M$ be a neighborhood of $PT_\sigma M$.
Then, {the map $x\mapsto \frac{\exp_\sigma^{-1}(x)}{\|X(x)\|}$ on $U\setminus{\sigma}$ extends to $\widehat U$ as a $C^{k-1}$-map which coincides at $u\in PT_\sigma M$ with $\frac{u}{\|DX(\sigma).u\|}$,} and 
the map $x\mapsto \frac{\|X(x)\|}{d(x,\sigma)}$ on $U\setminus \{\sigma\}$
extends to $\widehat U$ as a $C^{k-1}$-map which coincides at $u\in PT_\sigma M$
with $\|DX(\sigma).u\|$.

In the local coordinates $(-\varepsilon, \varepsilon)\times T^1_\sigma M$ associated to $\sigma\in \text{Sing}(X)$,
the lift of the vector field $X_1:=X/{\|X\|}$ on $M\setminus \text{Sing}(X)$ extends as a (non-vanishing) $C^{k-1}$ section
$\widehat X_1\colon (-\varepsilon, \varepsilon)\times T^1_\sigma M\to \widehat{TM}$.
For each $x=(0,u)\in p^{-1}(\sigma)$,
one has $$\widehat X_1(x)=\frac{DX(\sigma).u}{\|DX(\sigma).u\|}.$$
\end{Proposition}

A priori, the extension of $X_1$ is not preserved by the symmetry $(-s,-u)\sim(s,u)$
and is not defined in $\widehat{TM}$.
However, the line field $\RR \widehat X_1$ is invariant by the local symmetry $(s,u)\mapsto (-s,-u)$,
hence induces a $C^{k-1}$-line field $\widehat {\RR X}$ on $\widehat {TM}$ invariant by $(\widehat {D\varphi_t})_{t\in \RR}$.

\begin{proof}
In a local chart near a singularity, we have
$$X(x)=\int_{0}^1DX(r.x).x\;dr.$$
Working in the local coordinates $(s,u)\in(-\varepsilon,\varepsilon)\times T^1_\sigma M$, we get
$$X(x)=\int_{0}^1DX(rs.u)\;dr\;. \;s.u.$$
This allows us to define a $C^{k-1}$ section
in a neighborhood of $p^{-1}(\sigma)$ defined by
$$\bar X\colon(s,u)\mapsto  \int_{0}^1DX(rs.u)\;dr. u.$$
This section is $C^{k-1}$, is parallel to $X$ (when $s\neq 0$)
and does not vanish.
Consequently $\frac{\bar X}{\|\bar X\|}$ is $C^{k-1}$ and extends the vector field
$X_1:=X/{\|X\|}$ as required.

Since $\bar X$ extends as $DX(\sigma).u$ at $u\in PT_\sigma M$,
then $X_1$ extends as $DX(\sigma).u/ \|DX(\sigma).u\|$.

Note also that for $s\neq 0$,
$\|\bar X(s.u)\|$ coincides with $\|X(x)\|/d(x,\sigma)$ where $su=x$ is a point
of $M\setminus \{\sigma\}$ close to $\sigma$.
Since $\bar X$ is $C^{k-1}$ and does not vanish,
$(s,u)\mapsto \|\bar X(s.u)\|$ extends as a $C^{k-1}$-function in the local coordinates
$(-\varepsilon,\varepsilon)\times T^1_\sigma M$. It is invariant by the symmetry
$(s,u)\sim (-s,-u)$, hence the maps {$x\to \frac{\exp_\sigma^{-1}(x)}{\|X(x)\|}$ and} $x\mapsto \|X(x)\|/d(x,\sigma)$
for $x\in M\setminus \{\sigma\}$ close to $\sigma$ extends
as a $C^{k-1}$ on a neighborhood of $PT_\sigma M$ in $\widehat M$.
\end{proof}

\paragraph{The \emph{extended normal bundle} $\widehat {\cN M}$ and \emph{extended linear Poincar\'e flow} $(\widehat \psi_t)_{t\in \RR}$.}
The orthogonal spaces to the lines of $\widehat {\RR X}$ define
a $C^{k-1}$ linear bundle $\widehat {\cN M}$. Since
$\widehat {\RR X}$ is preserved by the extended tangent flow, the projection of $(\widehat {D\varphi_t})_{t\in \RR}$
defines the $C^{k-1}$ extended linear Poincar\'e flow $(\widehat \psi_t)_{t\in \RR}$ on $\widehat {\cN M}$.

\paragraph{Alternative construction.}
One can also embed $M\setminus {\rm Sing}(X)$ in $PTM$ by the map
$x\mapsto (x,\frac{X(x)}{\|X(x)\|})$, and take the closure. This set is invariant by the
unit flow. This compactification depends on the vector field $X$ and not only on the finite set ${\rm Sing}(X)$.
It is sometimes called \emph{Nash blowup}, see~\cite{No}.

When $DX(\sigma)$ is invertible at each singularity,
Proposition~\ref{p.extended-field} shows that the closure is homeomorphic to $\widehat M$.
The restriction of the normal bundle $\cN T^1M$ to the closure of $M\setminus {\rm Sing}(X)$ in $PTM$
gives the normal bundle $\widehat {\cN M}$.
This is the approach followed in~\cite{lgw-extended} in order to compactify of the
linear Poincar\'e flow.

\subsection{Compactifications of non-linear local fibered flows}
The rescaled flows introduced above extend to the bundles $\widehat {TM}$
or $\widehat {\cN M}$.
In the following, one will assume that $DX(\sigma)$ is invertible at each singularity
and (without loss of generality) that the metric on $M$
is flat near each singularity of $X$.

{Related to the ``local $C^k$-fibered flow''} in Definition~\ref{d.local-flow},
we will also use the following notion.

\begin{Definition}
Consider a continuous Riemannian vector bundle {$\cN$ over a compact metric space $K$.}
A map $H\colon \cN\to \cN$ is \emph{$C^k$-fibered}, if it fibers over a homeomorphism $h$ of $K$
and if each induced map $H_x\colon \cN_x\to \cN_{h(x)}$ is $C^k$
and depends continuously on $x$ for the $C^k$-topology.
\end{Definition}
\paragraph{The extended lifted flow.}
The following proposition compactifies the rescaled lifted flow $(\cL\varphi^*_t)_{t\in \RR}$
(and the rescaled tangent flow $(D\varphi_t^*)_{t\in \RR}$) as
local fibered flows on $\widehat {TM}$.
\begin{Proposition}\label{p.compactify-lifted}
The rescaled lifted flow $(\cL\varphi_t^*)_{t\in \RR}$
extends as a local $C^k$-fibered flow on $\widehat{TM}$.
The rescaled tangent flow $(D\varphi_t^*)_{t\in \RR}$ extends as a linear flow.

Moreover, there exists $\beta>0$ such that, for each $t\in [0,1]$, $\sigma\in {\rm Sing}(X)$ and $x=u\in p^{-1}(\sigma)$,
on the ball $B(0_x,\beta)\subset \widehat {T_{x}M}$ the map $\cL\varphi_t^*$
writes as:
\begin{equation}\label{e.extend}
y\mapsto \frac{\|DX(\sigma).u\|}{\|DX(\sigma)\circ D\varphi_t(\sigma).u\|}D\varphi_t(\sigma).y.
\end{equation}
\end{Proposition}
\medskip

Before proving the proposition, one first shows:
\begin{Lemma}\label{l.extend}
The function $(x,t)\mapsto \frac{\|X(x)\|}{\|X(\varphi_t(x))\|}$ on $(M\setminus \text{Sing}(X))\times \RR$
extends as a positive $C^{k-1}$ function $\widehat M\times \RR\to \RR_+$
which is equal to $\frac{\|DX(\sigma).u\|}{\|DX(\sigma)\circ D\varphi_t(\sigma).u\|}$
when  $x=u\in p^{-1}(\sigma)$.

The map from $TM|_{M\setminus \text{Sing}(X)}$
into itself which sends $y\in T_xM$ to $\|X(x)\|.y$, extends as a continuous map of $\widehat {TM}$
which vanishes on the set $p^{-1}(\text{Sing}(X))$ and is $C^{\infty}$-fibered.
\end{Lemma}
\begin{proof} From Proposition~\ref{p.extended-field},
in the local chart of $0=\sigma\in{\rm Sing}(X)$,
the map 
$x\mapsto \frac{\|X(x)\|}{\|x\|}$
extends as a $C^{k-1}$ function which coincides at $u\in PT_\sigma M$ with $\|DX(\sigma).u\|$
and does not vanish.
We also extend the map $(x,t)\mapsto \|\varphi_t(x)\|/\|x\|$ as a $C^{k-1}$ map
on $\widehat M\times \RR$ which coincides with $\|D\varphi_t(\sigma).u\|$ when $x=u$.
The proof is similar to the proof of Proposition~\ref{p.extended-field}.
This implies the first part of the lemma.

For the second part, one considers the product of the $C^{k-1}$ function $x\mapsto \frac{\|X(x)\|}{\|x\|}$
with the $C^\infty$-fibered map which extends $y\mapsto \|x\|.y$.
\end{proof}
\medskip

\begin{proof}[Proof of Proposition~\ref{p.compactify-lifted}]
In local coordinates, the local flow $(\cL \varphi_t^*)_{t\in \RR}$ in $T_xM$ acts like:
$$\cL \varphi_t^*(y)=\|X(\varphi_t(x))\|^{-1}\left(\varphi_t(x+\|X(x)\|.y)-\varphi_t(x)\right)$$
\begin{equation}\label{e.tangent-extension}
=\frac{\|X(x)\|}{\|X(\varphi_t(x))\|}
\int_0^1 D\varphi_t\left(x+r\|X(x)\|.y\right).y\;dr.
\end{equation}
By Lemma~\ref{l.extend},
$\frac{\|X(x)\|}{\|X(\varphi_t(x))\|}$ and $\|X(x)\|.y$ extend
as a continuous map and as a $C^\infty$-fibered homeomorphism respectively; hence
$(\cL \varphi_t^*)_{t\in \RR}$ extends continuously at $x=(0,u)\in p^{-1}(\sigma)$
as in~\eqref{e.extend}.
The extended flow is $C^k$ along each fiber. Moreover,
\eqref{e.tangent-extension} implies that it is $C^{k-1}$-fibered.
For $x\in M\setminus {\rm Sing}(X)$, the $k^{th}$ derivative along the fibers is equal to
$$\frac{\|X(x)\|^k}{\|X(\varphi_t(x))\|}D^k\varphi_t(x+\|X(x)\|.y).$$
This converges to $\frac{\|DX(\sigma).u\|}{\|DX(\sigma)\circ D\varphi_t(\sigma).u\|}D\varphi_t(\sigma)$
when $k=1$ and to $0$ for $k>1$. 
The extended rescaled lifted flow is thus a local $C^{k}$-fibered flow defined on a uniform neighborhood of the $0$-section.

From Lemma~\ref{l.extend}, the rescaled linear flow $(D\varphi_t^*)_{t\in \RR}$
extends to $\widehat {TM}$ and coincides at $x=u\in \pi^{-1}(\sigma)$ with
the map defined by~\eqref{e.extend}.
From~\eqref{e.tangent-extension},
it coincides also with the flow tangent to $(\cL \varphi_t^*)_{t\in \RR}$ at the $0$-section.

In order to define $\cL \varphi_t^*$ on the whole bundle $\widehat {TM}$
(and get a fibered flow as in Definition~\ref{d.local-flow}), one first glues each diffeomorphism $\cL \varphi_t^*$
for $t\in [0,1]$ on a small uniform neighborhood of $0$
with the linear map $D\varphi_t^*$ outside a neighborhood of $0$
in such a way that $\cL \varphi_0^*=\id$.
One then defines $\cL \varphi_t^*$ for other times by:
$$\cL \varphi_{-t}^*=\left( \cL \varphi_t^*\right)^{-1} \text{ for } t>0,$$
$$\cL \varphi_{n+t}^*=\cL \varphi_{t}^*\circ \cL \varphi_{1}^* \circ \dots \circ \cL \varphi_{1}^*
\text{ ($n+1$ terms), for } t\in [0,1] \text{ and } n\in \NN.$$
\end{proof}
\medskip

In a same way we compactify the rescaled fiber-preserving lifted flow $(\cL_0\varphi_t^*)$.

\begin{Proposition}
The rescaled fiber-preserving lifted flow $(\cL_0\varphi_t^*)_{t\in \RR}$ extends as a local $C^{k}$-fibered flow
on $\widehat {TM}$.
More precisely, for each $x\in \widehat{TM}$, it defines a $C^k$-map
$(t,y)\mapsto \cL_0\varphi_t^*(y)$
from $\RR\times\widehat {T_{x}M}$ to $\widehat{T_{x} M}$
which depends continuously on $x$ for the $C^k$-topology.

Moreover there exists $\beta>0$ such that, for each $t\in [0,1]$, $\sigma\in {\rm Sing}(X)$ and $x=u\in p^{-1}(\sigma)$,
on the ball $B(0,\beta)\subset \widehat{T_{x}M}$ the map $\cL_0\varphi_t^*$
has the form:
\begin{equation}\label{e.ext-L0}
y\mapsto D\varphi_t(\sigma).y +\frac{D\varphi_t(\sigma).u-u}{\|DX(\sigma).u\|}.
\end{equation}
\end{Proposition}
\begin{proof}
In the local coordinates the flow acts on $B(0,\beta)\subset T_xM$ as:
\begin{align}\label{e.L0}
\cL_0\varphi_t^*(y)&=\|X(x)\|^{-1}\left(\varphi_t(x+\|X(x)\|.y)-x \right)\\
&= \int_0^1D\varphi_t(x+r\|X(x)\|.y).y\;dr \; +\;\frac{\varphi_t(x)-x}{\|X(x)\|}.\label{e.extend-L0}
\end{align}
Arguing as in  Proposition~\ref{p.extended-field} and Lemma~\ref{l.extend},
{for each $t$, the map $x\mapsto \frac{\varphi_t(x)-x}{\|X(x)\|}$}
with $x\neq \sigma$ close to $\sigma$ extends for the $C^0$-topology
 by $\|DX(\sigma).u\|^{-1} (D\varphi_t(\sigma).u-u)$ at points $(0,u)\in PT_\sigma M$.
Since $X$ is $C^k$, these maps are all $C^k$ and depends continuously
with $x$ for the $C^k$-topology.

As before, the integral 
$\int_0^1D\varphi_t(x+r\|X(x)\|.y).y\;dr$ extends as $D\varphi_t(\sigma).y$
at $p^{-1}(\sigma)$.
For each $x$, the map $(t,y)\mapsto\int_0^1D\varphi_t(x+r\|X(x)\|.y).y\;dr$
is $C^k$ (this is checked on the formulas, considering separately the cases $x\in M\setminus  {\rm Sing}(X)$
and $x\in p^{-1}( {\rm Sing}(X))$. Since $X$ is $C^k$, this map depends continuously on $x$
for the $C^{k-1}$-topology.
The $k^\text{th}$ derivative with respect to $y$ is continuous in $x$, for the same reason as in the proof
of Proposition~\ref{p.compactify-lifted}.
For $x\in M\setminus  {\rm Sing}(X)$, the derivative with respect to $t$
of the map above is $(t,y)\mapsto\int_0^1DX(\varphi_t(x+r\|X(x)\|.y)).y\;dr$
and it converges as $x\to \sigma$ towards $(x,y)\mapsto DX(\varphi_t(\sigma)).y$
for the $C^{k-1}$-topology (again using that $X$ is $C^k$).
Hence the first term of~\eqref{e.extend-L0} is a $C^k$-function of $(t,y)$ which depends continuously on $x$
for the $C^k$-topology.

As in Proposition~\ref{p.compactify-lifted}, this proves that $(\cL_0\varphi_t^*)_{t\in \RR}$ extends as a local $C^{k}$-fibered flow having
the announced properties.
\end{proof}

\paragraph{The extended rescaled sectional Poincar\'e flow.}
We also obtain a compactification of the rescaled sectional Poincar\'e flow $(P_t^*)_{t\in \RR}$
(and of the rescaled linear Poincar\'e flow $(\psi^*_t)_{t\in \RR}$).
This implies Theorem~\ref{t.compactification}.
\begin{Theorem}\label{t.compactified2}
If $X$ is $C^k$, $k\geq 1$, and if
$DX(\sigma)$ is invertible at each singularity,
the rescaled sectional Poincar\'e flow $(P_t^*)_{t\in \RR}$ extends as a $C^k$-fibered flow
on a neighborhood of the $0$-section in $\widehat {\cN M}$.
Moreover, there exists $\beta'>0$ such that for each $t\in [0,1]$, $\sigma\in {\rm Sing}(X)$ and
$x=u\in p^{-1}(\sigma)$,
on the ball $B(0,\beta')\subset {\widehat {{\cal N}_{x} M}}$ the map $P^*_t$
writes as:
$$y\mapsto \;\frac{\|DX(\sigma).u\|}{\|DX(\sigma)\circ D\varphi_t(\sigma).u\|}D\varphi_{t+\tau}(\sigma).y\;
+\;\frac{D\varphi_{t+\tau}(\sigma).u-D\varphi_{t}(\sigma).u}{\|DX(\sigma)\circ D\varphi_t(\sigma).u\|},$$
where $\tau$ is a $C^{k}$ function of $(t,y)\in [0,1]\times B(0,\beta')$
which depends continuously on $x$ for the $C^k$-topology,
such that $\tau(x,t,0)=0$.
\smallskip

As a consequence, the rescaled linear Poincar\'e flow $(\psi_t^*)$ extends as a continuous
linear flow: for each $x\in\widehat M$, each $t\in\RR$ and each $v\in\widehat {\cN_{x}M}$
the image $\psi_t^*.v$ coincides with the normal projection of $D\varphi_t^*.v\in {\widehat {T_{\widehat \varphi_t(x)} M}}$ on $\widehat{\cN_{\widehat \varphi_t(x)} M}$.
\end{Theorem}
\begin{proof}
For each singularity $\sigma$, we work with the local coordinates $(-\varepsilon, \varepsilon)\times T^1_\sigma M$
and prove that the rescaled sectional Poincar\'e flow extends as a local $C^k$-fibered flow.
Since the rescaled sectional Poincar\'e flow is invariant by the symmetry $(s,u)\mapsto (-s,-u)$,
this implies the result above a neighborhood of $p^{-1}(\sigma)$ in $\widehat M$,
and hence above the whole manifold $\widehat M$.

The image of $y\in B(0,\beta)\subset \cN_x$ by
the rescaled sectional Poincar\'e flow is the unique point of the curve in $T_{\varphi_t(x)}M$
$$\tau\mapsto \cL_0\varphi_\tau^*\circ \cL\varphi_t^*(y)$$
which belongs to $\cN_{\varphi_t(x)}$.
In the local coordinates $x=(s,u)\in (-\varepsilon, \varepsilon)\times T^1_\sigma M$ near $\sigma$,
it corresponds to the unique value
$\tau=\tau(x,t,y)$ such that the following function vanishes:
$$\Theta(x,t,y,\tau)=\left< \cL_0\varphi_\tau^*\circ \cL\varphi_t^*(y)\;,\;
\frac{X(\varphi_t(x))}{\|X(\varphi_t(x))\|}\right>.$$

From the previous propositions the map $(y,\tau)\mapsto \Theta(x,t,y,\tau)$ is $C^k$ and depends continuously
on $(x,t)$ for the $C^{k}$-topology and is defined at any $x\in\widehat M$.
When $x=u\in p^{-1}(\sigma)$, the first part of the scalar product in $\Theta$ is given by the two propositions above.
According to Proposition~\ref{p.extended-field} the second part
$\widehat X_1(x)$
becomes equal to
$$\frac{DX(\sigma)\circ D\varphi_t(\sigma).u}{\|DX(\sigma)\circ D\varphi_t(\sigma).u\|}.$$

\begin{Claim}
For any $t\in[0,1]$ and for $x\in p^{-1}(\sigma)$, the derivative $\frac{\partial \Theta}{\partial \tau}|_{\tau=0}(x,t,0)$ is non-zero. 
\end{Claim}
\begin{proof}
Indeed, by the previous proposition this is equivalent to
$$\left<\frac{\partial}{\partial\tau}|_{\tau=0}(D\varphi_\tau(\sigma).v)\;,\;
\frac{DX(\sigma)\circ D\varphi_t(\sigma).u}{\|DX(\sigma)\circ D\varphi_t(\sigma).u\|}\right>\neq 0,$$
where $v=\widehat \varphi_t(u)={D\varphi_t(\sigma).u}/{\|D\varphi_t(\sigma).u\|}$.
Thus the condition becomes
$$\frac{\|DX(\sigma)\circ D\varphi_t(\sigma).u\|^2}{\|DX(\sigma)\circ D\varphi_t(\sigma).u\|}\neq 0,$$
which is satisfied.
\end{proof}

By the implicit function theorem and compactness, there exists $\beta'>0$ and, for each $x$, a $C^{k}$ map
$(y,t)\mapsto\tau(x,t,y)$ which depends continuously on $x$ for the $C^k$-topology such that
$$\Theta(x,t,y,\tau(x,t,y))=0,$$
for each $x\in\widehat M$ close to $PT_\sigma M$, each $t\in[0,1]$ and each $y\in B(0,\beta')$.
The rescaled sectional Poincar\'e flow is thus locally given by the composition:
\begin{equation}\label{e.rescaledpoincare}
(x,t,y)\mapsto \cL_0\varphi_{\tau(x,t,y)}^*\circ \cL\varphi_t^*(y),
\end{equation}
which extends as a $C^{k}$-fibered flow. The formula at $x=u\in p^{-1}(\sigma)$
is obtained from the expressions in the previous propositions.
\medskip

We now compute the rescaled linear Poincar\'e flow as the tangent map to the rescaled sectional Poincar\'e flow along the $0$-section. We fix $x\in\widehat M$ and its image
$x'=\widehat \varphi_t(x)$. We take $y\in \cN_{x}$
and its image $y'\in \widehat {T_{\widehat \varphi_t(x)}M}$ by $\cL\varphi^*_t$.
Working in the local coordinates $(-\varepsilon,\varepsilon)\times T^1_\sigma M$
and using Proposition~\ref{p.extended-field} and formulas~\eqref{e.ext-L0} and~\eqref{e.L0} we get
$$\frac{\partial}{\partial \tau}|_{\tau=0}\cL_0\varphi^*_\tau(0_{x'})=\widehat X_1(x').$$
{Note that }$\tau(x',t,0)=0$, we have:
\begin{equation*}
\begin{split}
DP^*_t(0)&=
\left( \frac{\partial}{\partial y'|}_{y'=0}  \cL_0\varphi^*_0(y') \right)\circ\left( \frac{\partial}{\partial y}|_{y=0}\cL\varphi_t^*(y)\right)+
\frac{\partial\tau}{\partial y}|_{y=0}.\left( \frac{\partial}{\partial \tau}|_{\tau=0}\cL_0\varphi^*_\tau(0_{x'})\right)\\
&= D\varphi_t^*(x)+\frac{\partial\tau}{\partial y}|_{y=0}.\widehat X_1(x').
\end{split}
\end{equation*}

On the other hand from~\eqref{e.rescaledpoincare} and the definitions of $\Theta,\tau$ we have
$$\left< DP^*_t(0),\widehat X_1(x')\right>=0,$$
hence $DP^*_t(0)$ coincides with the normal projection of $D\varphi_t^*(x)$
on the linear sub-space of $\widehat {T_{x'}M}$ orthogonal to $\widehat X_1(x')$, which is $\widehat {\cN_{x'}M}$.
\end{proof}
{
\begin{proof}[Proof of Theorem~\ref{t.compactification}]
Let $\Lambda$ be the compact invariant set in the assumption. Recall the blowup $\widehat M$ and the projection $p:~\widehat M\to M$. We define $\widehat\Lambda$ as the closure of $p^{-1}(\Lambda\setminus{\rm Sing}(X)))$ in $\widehat M$. Since the flow $\varphi$ on $M$ induces a flow $\widehat\varphi$ on $\widehat M$, we know that the restriction of $\varphi$ to $\Lambda\setminus{\rm Sing}(X)$ embeds in $({\widehat\Lambda},{\widehat\varphi})$ through the map $i=p^{-1}$. 

The metric on the bundle $\widehat{TM}$ over $\widehat M$ is the pull back metric of $TM$. In other words, if $p({\widehat x})=x$, then $\widehat{T_xM}$ is isometric to $T_x M$ through a map $I$. By Proposition~\ref{p.extended-field}, $\RR.{\widehat X_1}$ is a well defined extension of $\RR.X$. Thus, the restriction of $\widehat{\cN M}$ to $i(\Lambda\setminus{\rm Sing}(X))$ is isomorphic to the normal bundle $\cN M|_{\Lambda\setminus{\rm Sing}(\Lambda)}$ through the map $I$.

Finally Theorem~\ref{t.compactified2} shows that the fibered flow ${\widehat P}^*$ is conjugated by $I$ near the zero-section to the rescaled sectional Poincar\'e flow $P^*$.
\end{proof}
}

\subsection{Linear Poincar\'e flow: robustness of the dominated splitting}
As we mentioned at the end of Section~\ref{ss.blow-up}, the linear Poincar\'e flow has been compactified in~\cite{lgw-extended}
as the normal flow acting on the bundle $\cN T^1M$ over $T^1M$.
This allows in some cases to prove the robustness of the dominated splitting of the linear Poincar\'e flow.

\begin{Proposition}\label{p.robustness-DS}
Let us consider {$X\in{\cal X}^1(M)$, where $\dim M=3$,}
and an invariant compact set $\Lambda$ such that any singularity $\sigma\in \Lambda$ has real eigenvalues $\lambda_1<\lambda_2<0<\lambda_3$
and $W^{ss}(\sigma)\cap \Lambda=\{\sigma\}$.

If the linear Poincar\'e flow on $\Lambda\setminus \sing(X)$ has a dominated splitting,
then there exist neighborhoods $\cU$ of $X$ and $U$ of $\Lambda$
such that for any $Y\in \cU$ and any invariant compact set $\Lambda'\subset U$ for $Y$,
the linear Poincar\'e flow of $Y$ on $\Lambda'\setminus \sing(Y)$ has a dominated splitting.
\end{Proposition}
\begin{proof}
Let us consider the (continuous) unit tangent flow $U\varphi^X$ associated to $X$ and acting on $T^1M$.
The set $M\setminus \sing(X)$ embeds by the map $i_X\colon x\mapsto (x,X(x)/\|X(x)\|)$.
We denote by $S(E)$ the set of unit vectors of a vector space $E$. Thus $S(T_xM)=T^1_xM$.
We introduce the set
$$\Delta_X:=i_X(\Lambda\setminus \sing(X)) \cup \bigcup_{\sigma\in \sing(X)\cap \Lambda} S(E^{cu}(\sigma)).$$
It is compact (by our assumptions on $X$ at the singularities in $\Lambda$) and invariant by $U\varphi$.

\begin{Lemma}\label{l.robust}
Under the assumptions of the proposition,
the closure of $i_Y(\Lambda'\setminus \sing(Y))$ in $T^1M$ is contained in a small neighborhood
of $\Delta_X$.
\end{Lemma}
\begin{proof}
Let us define $B(\Lambda)$ as the set of points $(x,u)\in T^1M$ with $x\in \Lambda$ such that there exists  sequences $Y_n\to X$ in $\cX^1(M)$
and $x_n\in M\setminus \sing(X_n)$ such that
\begin{itemize}
\item[--] $(x_n,Y_n(x_n)/\|Y_n(x_n)\|)\to (x,u)$,
\item[--] the orbit of $x_n$ for the flow of $Y_n$ is contained in the $1/n$-neighborhood of $\Lambda$.
\end{itemize}
For each $\sigma\in \sing(X)\cap \Lambda$,  we have to show that
$$B(\Lambda)\cap T^1_\sigma M\subset S(E^{cu}(\sigma)).$$
Now the property we want is exactly the same as~\cite[Lemma 4.4]{lgw-extended}.
The definition of $B(\Lambda)$ and the setting differ but we can apply the same argument:
{the assumptions of \cite[Lemma 4.4]{lgw-extended} can be replaced by that any  $\sigma\in \Lambda$ has real eigenvalues $\lambda_1<\lambda_2<0<\lambda_3$ and $W^{ss}(\sigma)\cap \Lambda=\{\sigma\}$.}

At each $\sigma\in \Lambda\cap \Sing(X)$, one considers a chart and one fixes a point $z$ in $W^{ss}(\sigma)\setminus \{\sigma\}$.
For each $Y$ close to $X$, each point $y$ close to the continuation $\sigma_Y$
and whose orbit $(\varphi^Y_t(y))$ lies in a neighborhood of $\Lambda$,
let us assume by contradiction that $(y-\sigma_Y)/\|y-\sigma_Y\|$ is not close to the center-unstable plane
of the singularity $\sigma_Y$ of $Y$. After some backward iteration $\varphi_t(y)$ is still close to $\sigma_Y$
and $(\varphi_t(y)-\sigma_Y)/\|\varphi_t(y)-\sigma_Y\|$ gets close to the strong stable direction of $\sigma_Y$.
Iterating further in the past, one gets $\varphi_{s}(y)$ close to $z$: the distance $d(\varphi_s(y),z)$
can be chosen arbitrarily small if $Y$ is close enough to $X$ and if $y$ is close enough to $\sigma_Y$.
Taking the limit, one concludes that $z$ belongs to $\Lambda$: this is in contradiction with $W^{ss}(\sigma)\cap \Lambda=\{\sigma\}$.
\end{proof}

By definition (see Section~\ref{ss.def-flow}), if the linear Poincar\'e flow for $X$ is dominated on $\Lambda\setminus \sing(X)$,
then the normal flow $\cN\varphi^X$ for $X$ is dominated on $i_X(\Lambda\setminus \sing(X))$. Note that $\cN\varphi^X$ is also
dominated on $S(E^{cu}(\sigma))$ (the orthogonal projection of the splitting $E^{ss}\oplus E^{cu}\subset T_\sigma M$ on the fibers $\cN T^1_zM$
for $z\in S(E^{cu}(\sigma))$ is invariant). Consequently the dynamics of the linear cocycle $\cN\varphi^X$ above
the compact set $\Delta_X\subset T^1M$ is also dominated.

By Lemma~\ref{l.robust}
for $Y$ $C^1$-close to $X$ and $\Lambda'$ in a neighborhood of $U$, the set
$i_Y(\Lambda'\setminus \sing(Y))$ is contained in a neighborhood of $\Delta_X$;
moreover the linear flow $\cN\varphi^Y$ associated to $Y$ is close to $\cN\varphi^X$.
This shows that the dynamics of $\cN\varphi^Y$ above $i_Y(\Lambda'\setminus \sing(Y))$ is dominated.
By definition, this implies that the linear Poincar\'e flow associated to $Y$ above $\Lambda'\setminus \sing(Y)$ is dominated.
\end{proof}

\section{Fibered dynamics with a dominated splitting}\label{s.fibered}

In this section we 
introduce an identification structure (for local fibered dynamics) which formalizes the properties
satisfied by the {rescaled} sectional Poincar\'e flow. We then discuss some consequences
of the existence of a dominated splitting inside the fibers.

{\subsection{Dominated splitting for a local fibered flow}}

{ We consider local fibered flows as in Definition~\ref{d.local-flow}. The following notations will be used.}

\smallskip

\noindent
{\bf Notations.} -- One sometimes denotes a point $u\in \cN_x$  as $u_x$  to emphasize the base point $x$.
\hspace{-1cm}
\smallskip

\noindent
-- The length of a $C^1$ curve $\gamma\subset \cN_x$ (with respect to the metric of $\cN_x$)
is denoted by $|\gamma|$.
\smallskip

\noindent
-- A ball centered at $u$ and with radius $r$ inside a fiber $\cN_x$ is denoted by $B(u,r)$.
\smallskip

\noindent
-- For $x\in K$, $t\in \RR$ and $u\in \cN_x$, one denotes
by $DP_t(u)$ the derivative of $P_t$ at $u$ along $\cN_x$.
In particular $(DP_t(0_x))_{t\in \RR, x\in K}$ defines a linear flow over the $0$-section of $\cN$.

\begin{Definition}\label{d.dominated}
The local flow $(P_t)$ admits a \emph{dominated splitting}
$\cN=\cE\oplus \cF$ if $\cE$, $\cF$ are sub-bundles of $\cN$ that are invariant
by the linear flow $(DP_t(0))$ and if there exists $\tau_0> 0$ such that
for any $x\in K$, for any unit $u\in \cE(x)$ and $v\in \cF(x)$ and for any $t\geq \tau_0$ we have:
$$\|DP_t(0_x).u\|\leq \frac 1 2 \|DP_t(0_x).v\|.$$
Moreover we say that $\cE$ is \emph{$2$-dominated} if there exists $\tau_0>0$ such that for any $x\in K$,
any unit vectors $u\in \cE(x)$ and $v\in \cF(x)$, and for any $t\geq \tau_0$ we have:
$$\max(\|DP_t(0_x).u\|,\|DP_t(0_x).u\|^2)\leq \frac 1 2 \|DP_t(0_x).v\|.$$
\end{Definition}

When there exists a dominated splitting $\cN=\cE\oplus \cF$ and $V\subset K$ is an open subset,
one can prove that $\cE$ is uniformly contracted by considering the induced dynamics on $K\setminus V$
and checking that the following property is satisfied.
\begin{Definition}
The bundle $\cE$ is \emph{uniformly contracted on the set $V$} if
there exists $t_0>0$ such that for any $x\in K$ satisfying $\varphi_t(x)\in V$
for any $0\leq t\leq t_0$ we have:
$$\|DP_{t_0}{|\cE(x)}\|\leq \frac 1 2.$$
We say that $\cE$ is \emph{uniformly contracted} if it is uniformly contracted on $K$.
\end{Definition}

\subsection{Identifications}\label{ss.identifications}
\subsubsection{Definition of identifications}
We introduce more structures on the fibered dynamics.

\begin{Definition}
A \emph{$C^k$-identification} $\pi$ on an open set $U\subset K$
is defined by two constants $\beta_0,r_0>0$
and a continuous family of $C^k$-diffeomorphisms $\pi_{y,x}\colon \cN_y\to \cN_x$ associated to pairs of points $x,y \in K$ with $x\in U$ and $d(x,y)< r_0$,
such that:

For any $\{x,y,z\}$ of diameter smaller than $r_0$ with $x,z\in U$
and any $u\in B(0,\beta_0)\subset \cN_y$,
$$\pi_{z,x}\circ \pi_{y,z}(u)=\pi_{y,x}(u).$$
\end{Definition}
In particular $\pi_{x,x}$ coincides with the identity on $B(0,\beta_0)$.
\medskip

\noindent
{\bf Notations.} -- We will sometimes denote $\pi_{y,x}$
by $\pi_x$. Also the projection $\pi_{y,x}(0)=\pi_x(0_y)$ of $0\in \cN_y$ on $\cN_x$ will be denoted
by $\pi_x(y)$.
\medskip

\noindent
-- We will denote by ${\rm Lip}$ be the set of orientation-preserving bi-Lipschitz homeomorphisms $\theta$ of $\RR$
(and by ${\rm Lip}_{1+\rho}$ the set of those whose Lipschitz constant is smaller than $1+\rho$).

\begin{Definition}\label{d.compatible}
The identification $\pi$ on $U$ is \emph{compatible} with the flow $(P_t)$ if:
\begin{enumerate}

\item \emph{Transverse boundary.}
For any segment of orbit $\{\varphi_{s}(x), s\in [-t,t]\}$, $t>0$,
contained in $K\setminus U$ we have $x\in K\setminus \overline U$.

\item \emph{No small period.}
For any $\varkappa>0$, there is $r>0$ such that for any $x\in \overline{U}$
and $t\in [-2,2]$ with $d(x,\varphi_t(x))<r$, then we have $|t|< \varkappa$.

\item \emph{Local injectivity.} For any $\delta>0$, there exists $\beta>0$ such that for any $x,y\in U$:\\
if $d(x,y)<r_0$ and $\|\pi_{x}(y)\|\leq \beta$, then $d(\varphi_t(y),x)\leq \delta$ for some $t\in [-1/4,1/4]$.

\item \emph{Local invariance.} For any $x,y\in U$
and $t\in [-2,2]$ such that $y$ and $\varphi_t(y)$ are $r_0$-close to $x$,
and for any $u\in B(0,\beta_0)\subset \cN_{y}$, we have
$$\pi_{x}\circ P_t(u)=\pi_{x}(u).$$

\item \emph{Global invariance.}
For any $\delta,\rho>0$, there exists $r,\beta>0$ such that:

For any $y,y'\in K$ with $y\in U$ and $d(y,y')<r$,
for any $u\in  \cN_y$, $u'\in \cN_{y'}$ with $\pi_y(u')=u$,
and any intervals $I,I'\subset \RR$ containing $0$ and satisfying
$$\|P_s(u)\|<\beta\text{ and } \|P_{s'}(u')\|<\beta\text{ for any } s\in I \text{ and any }
s'\in I',$$
there is $\theta\in {\rm Lip}_{1+\rho}$ such that $\theta(0)=0$ and
 $d(\varphi_s(y),\varphi_{\theta(s)}(y'))<\delta$ for any $s\in I\cap \theta^{-1}(I')$.
Moreover, for any $v\in \cN_y$ such that
$ \|P_s(v)\|<\beta$  for each $s\in I\cap \theta^{-1}(I')$ then
\begin{itemize}
\item[--] $v'=\pi_{y'}(v)$ satisfies
$\|P_{\theta(s)}(v')\|<\delta$ for each such $s$,
\item[--] if $\varphi_s(y)\in U$
for some $s$, then $\pi_{\varphi_s(y)}\circ P_{\theta(s)}(v')=P_s(v)$.
\end{itemize}
\end{enumerate}

\end{Definition}

\begin{Remarks-numbered}\label{r.identification}\rm
a) These definitions are still satisfied if one reduces
$r_0$ or $\beta_0$. Their value will be reduced in the following sections in order to
satisfy additional properties.
\smallskip

One can also rescale the time and keep a compatible identification: the flow $t\mapsto \varphi_{t/C}$ for $C>1$
still satisfies the definitions above, maybe after reducing the constant $r_0$.

The main point to check is that the time $t$ in the Local injectivity can still
be chosen in $[-1/4,1/4]$. Indeed, this is ensured by the ``No small period" assumption applied with $\varkappa=1/4C$: if $r_0$ is chosen smaller
and if both $d(\varphi_t(y),x),d(y,x)$ are less than $r_0$ for $t\in [-1/4,1/4]$, then $|t|$ is smaller than $\varkappa$.
Now the time $t$ in the Local injectivity property belongs to $[-\varkappa,\varkappa]$ for the initial flow, hence to $[-1/4,1/4]$ for the time-rescaled flow.
\medskip

\noindent
b) The ``No small period" assumption (which does not involve the projections $\pi_x$) is equivalent to the non-existence of periodic orbits of period $\leq 2$
which intersect $\overline U$.
In particular, by reducing $r_0$, one can assume the following property:
\smallskip

\emph{For any $x\in U$ and any $t\in [1,2]$, we have $d(x,\varphi_t(x))\ge r_0$.}
\medskip

\noindent
c) For $x\in U$, the Local injectivity prevents the existence of  $y\in U$ that is $r_0$-close to $x$,
is different from $\varphi_t(x)$ for any $t\in [-1/4,1/4]$, and
such that $\pi_x(y)=0_x$.
In particular:
\smallskip

\emph{If $x,\varphi_t(x)\in U$ and $t\not\in (-1/2,1/2)$ satisfy
$\pi_x(\varphi_t(x))=0_x$, then $x$ is periodic.}\hspace{-1cm}\mbox{}
\medskip

\noindent
d) The Global invariance says that when two orbits $(P_s(u))$ and
$(P_s(u'))$ of the local fibered flow are close to
the $0$ section of $\cN$ and have two points which are identified by $\pi$,
then they are associated to orbits of the flow $\varphi$ that are close
(up to a reparametrization $\theta$).
In this case, any orbit of $(P_t)$ close to the zero-section
above the first $\varphi$-orbit can be projected to an orbit of $(P_t)$ above the second $\varphi$-orbit.
\medskip

\noindent
e) The Global invariance can be applied to pairs of points $y,y'$ where the condition
$d(y,y')<r$ has been replaced by a weaker one $d(y,y')<r_0$. In particular, this gives:
\medskip

\emph{For any $\delta,\rho>0$, there exist $\beta>0$ such that:
if  $y,y'\in K$, $u\in  \cN_y$, $u'\in \cN_{y'}$ and the intervals $I,I'\subset \RR$ containing $0$ satisfy
\begin{itemize}
\item[--] $d(y,y')<r_0$ and $y\in U$,
\item[--] $\pi_y(u')=u$,
$\|P_s(u)\|<\beta\text{ and } \|P_{s'}(u')\|<\beta\text{ for any } s\in I \text{ and any }
s'\in I',$
\end{itemize}
there is $\theta\in {\rm Lip}_{1+\rho}$ such that $|\theta(0)|\leq 1/4$ and $d(\varphi_s(y),\varphi_{\theta(s)}(y'))<\delta$ for any $s\in I\cap \theta^{-1}(I')$.}
\medskip

Indeed provided that $\beta>0$ has been chosen small enough,
one can apply the Local injectivity and the Local invariance
in order to replace $y'$ and $u'$ by $y''=\varphi_t(y')$ and $u''=P_s(u')$
for some $t\in [-1/4,1/4]$ such that $d(y,y'')<r$.
The assumptions for the Global invariance then are satisfied by $y,y''$ and $u,u''$.
It gives a  $\theta\in {\rm Lip}_{1+\rho}$
satisfying $d(\varphi_s(y),\varphi_{\theta(s)}(y'))<\delta$ for $s\in I\cap \theta^{-1}(I')$
but the condition $\theta(0)=0$ has been replaced by
$\theta(0)=t$; in particular $|\theta(0)|<1/4$.
\end{Remarks-numbered}

\noindent
{\bf Fundamental example.}
One may think that $(\varphi_t)$ is the compactified flow $\widehat \varphi$
on an invariant compact set $K\subset \widehat M$ as in Section~\ref{s.compactification},
that $\cN$ is the compactified normal bundle over $K$,
and that $(P_t)$ is the extended rescaled sectional Poincar\'e flow.

\subsubsection{No shear inside orbits}
The next property states that one cannot find two reparametrizations
of a same orbit, that shadow each other, coincide for some parameter and differ by at least $2$ for another parameter.

\begin{Proposition}\label{p.no-shear}
If $r_0$ is small enough,
for any $x\in U$, any increasing homeomorphism $\theta$ of $\RR$, any interval $I$ containing $0$ satisfying $\varphi_{\theta(0)}(x)\in U$ and $d(\varphi_t(x),\varphi_{\theta(t)}(x))\leq r_0,\;\forall t\in I$,
then:
\begin{itemize}
\item[--] $\theta(0)> 1/2$ implies that $\theta(t)> t+2$, $\forall t\in I$ such that $\varphi_t(x), \varphi_{\theta(t)}(x)\in U$;
\item[--] $\theta(0)\in [-2,2]$ implies that $\theta(t) \in [t-1/2,t+1/2]$, $\forall t\in I$ such that $\varphi_t(x),\varphi_{\theta(t)}(x)\in U$;
\item[--] $\theta(0)< -1/2$ implies that $\theta(t)<t-2$, $\forall t\in I$ such that $\varphi_t(x),\varphi_{\theta(t)}(x)\in U$.
\end{itemize}
\end{Proposition}
\begin{proof}
Let $\Delta$ be the set of points $x$ such that $\varphi_t(x)\not\in U$ for every $|t|\leq 1/2$.
By the ``Transverse boundary" assumption (and up to reduce $r_0$), it is compact and its distance to $\overline U$
is larger than $2r_0$.
Let us choose $\varepsilon\in (0,1/2)$ small enough so that
$\varphi_s(\overline U)$ is disjoint from the $r_0$-neighborhood of $\Delta$ when $|s|\leq \varepsilon$.
Still reducing $r_0$, one can assume that:
\begin{itemize}
\item[(a)] any piece of orbit
$\{\varphi_s(y),s\in [0,b]\}\subset K\setminus U$, with $y,\varphi_b(y)$ in the $r_0$-neighborhood of $\Delta$ and $b\leq 1/2$, is disjoint from the $r_0$-neighborhood of $\overline U$,
\item[(b)] if $\varphi_s(y)$ is $r_0$-close to $y\in\overline U$ for $|s|\leq 2$, then $|s|\leq \varepsilon$.
\end{itemize}
The first condition is satisfied by small $r_0>0$ since
otherwise letting $r_0\to 0$ one would construct $y,\varphi_b(y)\in \Delta$
and $\varphi_{s}(y)\in \overline U$ where $0\leq s\leq b\leq 1/2$, contradicting the definition of $\Delta$.
The second condition is  a consequence of the ``No small period" assumption.

\begin{Claim}
If $\theta(0)\geq -2$ then $\theta(t)\geq t-1/2$ for any $t\in I$ satisfying $\varphi_t(x)\in U$.
\end{Claim}
\begin{proof}
The case $t=0$ is a consequence of the ``No small period" assumption.
We deal with the case {that} $t$ is positive.
The case {that} $t$ is negative can be deduced by applying the positive case
to $\theta^{-1}$ and to the point $\varphi_{\theta(0)}(x)$.

Let $J$ be the interval of $t\in I$ satisfying
for all $s\in J$ one has either $\varphi_s(x)\not\in \overline U$,
or $\theta(s)\geq s-1/2$. Let $t_1\in  [0,t_0]\cap J$ be the largest time satisfying $\varphi_{t_1}(x)\in \overline U$. It exists since if $(t_k)$ is an increasing sequence in $[0,t_0]\cap J$ satisfying
$\varphi_{t_k}(x)\in \overline U$, then we have $\theta(t_k)\geq t_k-1/2$.
So the limit $\overline t$ satisfies $\theta(\overline t)\geq \overline t-1/2$ (and belongs to $J$)
and $\varphi_{\overline t}(x)\in \overline U$.

By property (b), $\theta(t_1)\geq t_1-\varepsilon$.
For $s>t_1$ close to $t_1$, we thus have $s\in J$. Since $t_1$ is maximal
we also get  $\varphi_{s}(x)\not\in \overline U$. Since $\varphi_{t_0}(x)\in U$,
there exists a minimal $t_2\in [ t_1,t_0]$ such that $\varphi_{t_2}(x)$ belongs to the boundary of $U$.
In particular $\varphi_{\theta(t_2)}(x)$ is $r_0$-close to the boundary of $U$.
Note that $[t_1,t_2)\subset J$. By maximality of $t_1$ one has $\theta(t_2)<t_2-1/2$.

The ``No small period" assumption implies $\theta(t_2)<t_2-2$.
In particular,
$$t_2>\theta(t_2)+2>\theta(t_1)+2\geq t_1+2-\varepsilon.$$
This shows that $\varphi_{s}(x)\in \Delta$, $\forall s\in [t_1+1/2,t_2-1/2]$.
Since $\varphi_{\theta(t_1+1/2)}(x)$ is $r_0$-close to $\varphi_{t_1+1/2}(x)\in \Delta$,
and $\varphi_{t_1}(x)\in\overline U$,
one has $|\theta(t_1+1/2)-t_1|> \varepsilon$ (by our choice of $\varepsilon$).
Since $\theta(t_1)\geq t_1-\varepsilon$, this gives
$\theta(t_1+1/2)> t_1+\varepsilon$. Hence $(\theta(t_1+1/2),t_1+1/2]$ has length smaller than $1/2$.

If $\theta(t_2)\in (\theta(t_1+1/2),t_1+1/2]$, since
$\varphi_{\theta(t_1+1/2)}(x)$ belongs to the $r_0$-neighborhood of $\Delta$
and since $\varphi_{t_1+1/2}(x)\in \Delta$, the property (a) implies that
$\varphi_{\theta(t_2)}(x)$ is disjoint from the $r_0$-neighborhood of $\overline U$.
Otherwise $\theta(t_2)\in [t_1+1/2,t_2-1/2]$ and then $\varphi_{\theta(t_2)}(x)\in \Delta$
is $r_0$ far from $\overline U$.
This is a contradiction since we have proved before that $\varphi_{\theta(t_2)}(x)$ is $r_0$-close to the boundary of $U$.
\end{proof}

Arguing in a same way we deduce that $\theta(0)\leq 2$ implies $\theta(t)\leq t+1/2$ for any $t\in I$ such that $\varphi_t(x)\in U$. One deduces the second item of the proposition.
Note that if $\theta(t_0)\leq t_0+2$ for some $t_0$ with $\varphi_{t_0}(x)\in U$, then one deduces
that $\theta(t)\leq t+1/2$ for all other $t$ and in particular $\theta(0)\leq 2$;
this gives the first item. The third one is similar.
\end{proof}

\subsubsection{Closing lemmas}
The following closing lemma is an example of properties given by identifications.
\begin{Lemma}\label{l.closing0}
Let us assume that $\beta_0,r_0$ are small enough.
Let us consider:
\begin{itemize}
\item[--] $x\in U$ having an iterate $y=\varphi_T(x)$ in $U\cap B(x,r_0)$ with $T\geq 4$,
\item[--] a fixed point $p\in \cN_x$ for $\widetilde P_T:=\pi_x\circ P_T$
such that $\|P_t(p)\|<\beta_0$ for each $t\in [0,T]$,
\item[--]  a sequence $(y_k)$ in a compact set of $U\cap B(x,r_0/2)$
such that $\pi_x(y_k)$ converges to $p$.
\end{itemize}
Then there exists a sequence $(s_k)$ in $[-1,1]$
such that $\varphi_{s_k}(y_k)$ converges to a periodic
point $y$ of $K$ having some period $T'$ such that $\pi_x(y)=p$
and $$DP_{T'}(0_y)= D\pi_{x}(0_y)^{-1}\circ D\widetilde P_T(p) \circ D\pi_x(0_y).$$
\end{Lemma}
\begin{proof}
Up to extract a subsequence, $(y_{k})$ converges to a point
$y\in U\cap B(x,r_0/2)$ such that $\pi_x(y)=p$.
By the Local injectivity,
since the sequence $(\pi_y(y_k))$ converges to $0_y$,
there exists $(s_k)$ in $[-1,1]$ such that $\varphi_{s_k}(y_k)$ converges to $y$.

By the Global invariance, there exists $(T_k)$ satisfying
$\frac 1 2 T\leq T_k\leq 2T$ such that $\varphi_{T_k}(y_k)$
is in $B(x,r_0/2)$ and projects by $\pi_x$ on $\widetilde P_T(\pi_x(y_k))$.

In particular $(\pi_x\circ \varphi_{T_k}(y_k))$ converges to $p$
and $(\pi_y\circ \varphi_{T_k}(y_k))$ converges to $0_y$.
One deduces (up to modify $T_k$ by adding a real number in $[-1,1]$)
that $\varphi_{T_k}(y_k)$ converges to $y$.
Since $T\geq 4$, the limit value $T'$ of $T_k$ is larger than $1$ and
one deduces that $y$ is $T'$-periodic.
Moreover, $\pi_x(y)=p$ so that by the Global invariance
$D\widetilde P_T(p)$ and $DP_{T'}(0_y)$ are conjugated by $D\pi_{x}(0_y)$.
\end{proof}

For the next statement, we consider an open set $V$ containing $K\setminus U$
so that points in $K\setminus V$ are separated from the boundary of $U$ by a distance much larger than $r_0$.
\begin{Corollary}\label{c.closing0}
Assume that $\beta_0,r_0$ are small enough.
If $x\in K\setminus V$ has an iterate $y=\varphi_T(x)$ in $B(x,r_0)$ with $T\geq 4$ and if there exists
a subset $B\subset \cN_x$ containing $0$ such that
\begin{itemize}
\item[--] $P_t(B)\subset B(0_{\varphi_t(x)},\beta_0)$ for any $0<t<T$,
\item[--] $\widetilde P_T:=\pi\circ P_T$ sends $B$ into itself,
\item[--] the sequence $\widetilde P_T^k(0)$ converges to a fixed point $p\in B$,
\end{itemize}
then the positive orbit of $x$ by $\varphi$ also converges to a periodic orbit.
\end{Corollary}
\begin{proof}
From the Global invariance, there exists a sequence $T_k\to +\infty$
such that $y_k:=\varphi_{T_k}(x)$ projects by $\pi_x$ on $\widetilde P_T^k(0)$
and $|T_{k+1}-T_k|$ is uniformly bounded in $k$.

Since $(\widetilde P_T^k(0)=\pi_x(y_k))$ converges to $p$,
we can apply the previous lemma so that $(\varphi_{s_k}(y_k))$ converges to a $T'$-periodic point $y\in K$
for some $s_k\in [-1,1]$. Since $|T_{k+1}-T_k|$ is uniformly bounded in $k$, this proves that the $\omega$-limit set of $x$ is the orbit of $y$.
\end{proof}

\subsubsection{Generalized orbits}
The identifications $\pi$ allow us to introduce generalized orbits.
In the 
case where $K$ is a non-singular invariant set and
$(P_t)$ is the sectional Poincar\'e flow on 
$\cN$,
these orbits correspond to the orbits of the flow contained in the maximal
invariant set in a neighborhood of $K$.

\begin{Definition}[Generalized orbit]
A (piecewise continuous) path $\bar u=(u(t))_{t\in \RR}$ in $\cN$ is a \emph{generalized orbit} if there is a
sequence $(t_n)_{n\in\ZZ}$ in $\RR$ such that if $y(t)$ denotes the projection of $u(t)$ to $K$ by the bundle map $\cN\to K$, then for each $n\in \ZZ$:
\begin{itemize}

\item[--] $t_{n+1}-t_n\geq 1$,

\item[--] $\|u(t)\|< \beta_0$ and $u(t)=P_{t-t_n}(u_n)$ for $t\in [t_n,t_{n+1})$,

\item[--]  $\varphi_{t_{n+1}-t_n}(y_n),y_{n+1}$ belong to $U$, are $r_0$-close
and $\pi_{y_{n+1}}(P_{t_{n+1}-t_n}(u_n))=u_{n+1}$.

\end{itemize}
\emph{The projection of $u(t)$ from $\cN$ to $K$ defines a pseudo-orbit $(y(t))$ of $\varphi$ in $K$.}
\end{Definition}

\begin{Remarks-numbered}\label{rk.generalized-orbit} a) If $(u(t))$ is a generalized orbit, then $(u(t+s))_{t\in \RR}$ is also for any $s\in \RR$.
\smallskip

\noindent
b) The generalized orbits satisfying $u(t)=0_{y(t)}$ for any $t$
can be identified to the orbits of $\varphi$
on $K$ which meet $U$ for arbitrarily large positive and negative times $t_n$.
\end{Remarks-numbered}

\begin{Definition}[Topology on generalized orbits]
Let us fix $\bar u$.
For $T>0$ large and $\eta>0$ small, we say that
a generalized orbit $\bar u'$ is $(T,\eta)$-close to $\bar u$ if
$u(t)$ and $u'(t)$ are $\eta$-close for each $t\in [-T,T]$.
\end{Definition}

For the next notion, we fix an open set $V$ containing $K\setminus U$.
\begin{Definition}[Neighborhood of K]
A generalized orbit \emph{belongs to the $\eta$-neighborhood of $K$} (or of the $0$-section
of $\cN$) if the additional conditions hold:
\begin{itemize}
\item[--] $d(y(t_{n+1}),\varphi_{t_{n+1}-t_n}(y(t_n)))\leq \eta$, for any $n\in \ZZ$,
\item[--]  $d(y(t_n), K\setminus V)<\eta$,
for each $n\in \ZZ$ such that $y(t_{n})\neq \varphi_{t_{n}-t_{n-1}}(y(t_{n-1}))$,
\item[--] $\|u(s)\|< \eta$ for any $s\in\RR$.
\end{itemize}
\end{Definition}

\begin{Definition}[Generalized flow]\label{d.generalized-flow}
We associate, to any generalized orbit $\bar u=(u(t))$
and any $s,t\in \RR$, a diffeomorphism $\bar P_t$
from a neighborhood of $u(s)$ in $\cN_{y(s)}$ to a neighborhood of $u(s+t)$ in $\cN_{y(s+t)}$
which for any $t,t'$ satisfies  $\bar P_{t'}\circ \bar P_{t}=\bar P_{t+t'}$. It is defined by:
\begin{itemize}
\item[--] by $\bar P_t=P_t$ when $t_n\leq s\leq t+s<t_{n+1}$,
\item[--] by $\bar P_t=  P_{t+s-t_{n+1}}\circ \pi_{y(t_{n+1})}\circ  P_{t_{n+1}-s}$
when $t_n\leq s< t_{n+1}\leq  t+s<t_{n+2}$,
\item[--] and by applying inductively the flow relation in the other cases.
\end{itemize}
\end{Definition}

The generalized flow acts on generalized orbits:
$\bar P_t(\bar u)$ coincides at time $s$ with $\bar u(s+t)$.
When $\bar u$ can be identified to an orbit of $\varphi$ (as in Remark~\ref{rk.generalized-orbit}), 
$\bar P_t$ coincides with the flow $P_t$.

\paragraph{Half generalized orbits.} The previous definitions may be extended to any (piecewise continuous)
path $(u(t))_{t\in I}$ parametrized by an interval $I$ of $\RR$. When $I=[0,+\infty)$ or $(-\infty,0]$ one gets
the notion of half generalized orbits. The generalized semi-flow $(\bar P_{t})_{t\geq 0}$ (resp. $(\bar P_{t})_{t\leq 0}$)
acts on half generalized orbits parametrized by $[0,+\infty)$ (resp. $(-\infty,0]$).

\subsubsection{Normally expanded irrational tori}\label{ss:tori}

We give a setting of
a dominated splitting $\cE\oplus \cF$ such that $\cE$ is not uniformly contracted.
\begin{Definition}A \emph{normally expanded irrational torus} is
an invariant compact subset $\cT\subset K$ such that
\begin{itemize}
\item[--] the dynamics of $\varphi|_{\cT}$
is topologically equivalent to an irrational flow on $\TT^2$,
\item[--] there exists a dominated splitting $\cN|_{\cT}=\cE\oplus \cF$ and $\cE$ has one-dimensional fibers,
\item[--] for some $x\in U\cap \cT$ and $r>0$, $\pi_x(\{z\in K, d(x,z)<r\})$ is a $C^1$-curve
tangent to $\cE(x)$.
\end{itemize}
\end{Definition}

The name is justified as follows.
\begin{Lemma}\label{l.torus}
For any normally expanded irrational torus $\cT$,
the Lyapunov exponent along $\cE$ of the (unique) invariant measure of $\varphi$ on $\cT$
is equal to zero; in particular $\cF$ is uniformly expanded (i.e uniformly
contracted by backward iterations).
\end{Lemma}
\begin{Remark-numbered}\label{r.torus}
With the technics of Section~\ref{s.topological-hyperbolicity}, one
can also prove that the $\alpha$-limit set of any point $z$ in a neighborhood
coincides with $\cT$.
\end{Remark-numbered}
\begin{proof}
Let us choose a global transversal $\Sigma\simeq \TT^1$ containing $x$
for the restriction of $\varphi$ to $\cT$. The dynamics is conjugated to a suspension
of an irrational rotation of $\Sigma$.
We consider the sequence $(t_k)$ of positive returns of the orbit of $x$ inside a neighborhood
of $x$ in $\Sigma$. Note that $|t_{k+1}-t_k|$ is uniformly bounded.
For every $y\in \cT$ close to $x$
there exists a sequence $(t'_k)$ such that $|t_{k+1}-t_k|$ and $|t'_{k+1}-t'_k|$ are close and
$\varphi_{t'_k}(y)$ is close to $\varphi_{t_k}(x)$ and belongs to $\Sigma$.
In particular by choosing $y$ close enough to $x$,
there exist $\varepsilon_1,\varepsilon_2>0$ arbitrarily small such that
$$\varepsilon_1\leq d(\varphi_{t'_k}(y), \varphi_{t_k}(x))\leq \varepsilon_2.$$
Let $I\subset \Sigma$ be the interval bounded by $x,y$.
The transversal $\Sigma$ is mapped homeomorphically inside an interval of the $C^1$-curve
$\gamma=\pi_x(\{z\in K, d(x,z)<r\})$ and by the Global invariance
$\pi_x\circ P_{t_k}$ sends $\pi_x(y)$ to
$\pi_x(\varphi_{t'_k}(y))$ and similarly $\pi_x(I)\subset \gamma$ inside $\gamma$.
Moreover, there exists $\varepsilon'_1,\varepsilon'_2>0$ arbitrarily small such that
for each $k$,
$$\varepsilon'_1\leq |\pi_x\circ P_{t_k}\circ\pi_x(I)|\leq \varepsilon'_2.$$

Since $t_{k+1}-t_k$ is bounded, it implies that the Lyapunov exponent along $\cE$ vanishes.
\end{proof}

\subsubsection{Contraction on periodic orbits and criterion for $2$-domination}
When there exists a dominated splitting $\cE\oplus \cF$ where $\cE$ is one-dimensional,
the uniform contraction of the bundle $\cE$ above each periodic orbit of $K$,
implies that it is $2$-dominated.
\begin{Proposition}\label{p.2domination}
Let us assume that
\begin{itemize}
\item[--] there exists a dominated splitting $\cN=\cE\oplus \cF$ and the fibers of $\cE$ are one-dimensional,
\item[--] $\cE$ is uniformly contracted on an open set $V$ containing $K\setminus U$.
\end{itemize}
Then either the bundle $\cE$ is $2$-dominated, or there exists a periodic orbit
$\cO$ in $K$ whose Lyapunov exponents are all positive.
\end{Proposition}
\begin{proof}
If there is no $2$-domination, there exists a sequence $(x_n)$ in $K$,
such that
$$\|DP_n{|\cE(x_n)}\|^2\geq \frac 1 2 \|DP_n{|\cF(x_n)}\|.$$
One can extract a $\varphi$-invariant measure from the sequence
$$\mu_n:=\frac 1 n \int_{t=0}^{n}\delta_{\varphi_{t}(x_n)}\; dt$$
and the maximal Lyapunov exponents $\lambda^\cE,\lambda^\cF$
along $\cE,\cF$ satisfy
$2\lambda^\cE\geq \lambda^\cF$.
In particular, $\lambda^\cF> \lambda^\cE\geq \lambda^\cF-\lambda^\cE>0$.
Since $\cE$ is one-dimensional,
one deduces that $\varphi$ admits an ergodic measure $\mu$ whose Lyapunov exponents are both positive.
Since $\cE$ is uniformly contracted on $V$, the support of $\mu$ has to intersect $U$.
For $\mu$-almost every point $x\in K$, there exists a neighborhood
$V_x$ of $0$ in $\cN_x$ such that $\|(DP_{-t})|_{V_x}\|$ decreases exponentially
as $t\to +\infty$. In particular, one can take $x\in U$ recurrent
and find a large time $T>0$ such that
$\widetilde P_{-T}:=\pi_x\circ P_{-T}$ sends
$V_x$ into itself as a contraction.
{By} Corollary~\ref{c.closing0}, 
there is a periodic point $y$ in $K$ with some period $T'>0$
and a fixed point $p\in V_x$ for $\widetilde P_{-T}$
such that the tangent map $DP_{-T'}(0_y)$ is conjugate to the tangent map
$D\widetilde P_{-T}(p)$ (by Lemma~\ref{l.closing0}).
Hence the Lyapunov exponents of $y$ are all positive.
\end{proof}

\subsection{Plaque families}\label{s.plaque}

We now introduce center-stable plaques  $\cW^{cs}(x)$ that are candidates to be the local stable manifolds
of the dynamics tangent to $\cE$.
A symmetric discussion gives center-unstable plaques $\cW^{cu}$ tangent to $\cF$.

\subsubsection{Standing assumptions}\label{ss.assumptions}
In this Section~\ref{s.plaque}, we consider:
\begin{itemize}
\item[--] a bundle $\cN$ with $d$-dimensional fibers, a local fibered flow $(\cN,P)$ over a topological flow $(\varphi,K)$ and an identification $\pi$
on an open set $U$, compatible with $(P_t)$,
\item[--] a dominated splitting $\cN=\cE\oplus \cF$ of the bundle $\cN$,
\item[--] an open set $V$ containing $K\setminus U$.
\end{itemize}
Reducing $r_0$ we assume that the distance between $K\setminus V$ and $K\setminus U$ is much larger than $r_0$.

We also fix an integer $\tau_0\geq 1$ satisfying Definition~\ref{d.dominated} of the domination.
We choose $\lambda>1$ such that $\lambda^{4\tau_0}<2$.
 In particular
for any $x\in K$, for any unit $u\in \cE(x)$ and $v\in \cF(x)$, we have:
\begin{equation}\label{e.domination}
\forall t\geq\tau_0,~~~\|DP_t(0).u\|\leq \lambda^{-2t} \|DP_t(0).v\|.
\end{equation}

In each space $\cN_x=\cE(x)\oplus \cF(x)$, $x\in K$, we introduce the constant cone
$$\cC^\cE(x):=\{u=u^\cE+u^\cF\in \cN_x=\cE(x)\oplus \cF(x), \|u^\cE\|> \|u^\cF\|\},$$
and $\cC^\cF(x)$ in a symmetric way. They vary continuously with $x$.
Moreover the dominated splitting implies that for any $t\geq \tau_0$ the cone fields are contracted:
$$DP_t(0_x).\overline{\cC^\cF(x)}\subset \cC^\cF(\varphi_t(x))
\text{ and } DP_{-t}(0_x).\overline{\cC^\cE(x)}\subset \cC^\cE(\varphi_{-t}(x)).$$

\subsubsection{Plaque family for fibered flows}
\begin{Definition}
A \emph{$C^k$-plaque family tangent to $\cE$} is a continous fibered embedding
$\psi\in {\rm Emb}^k(\cE,N)$, that is a family of $C^k$-diffeomorphisms onto their image $\psi_x\colon \cE(x)\to \cN_x$ such that $\psi_x(0_x)=0_x$, the image of $D\psi_x(0_x)$ coincides with $\cE(x)$ and such that $\psi_x$ depends continuously on $x\in K$ for the $C^k$-topology.

\emph{ For $\alpha>0$, we denote by $\cW^{cs}_{\alpha}(x)$ the ball centered at $0_x$ and of radius $\alpha$
inside $\psi_x(\cE(x))$ with respect to the restriction of the metric on $\cN_x$ and we denote by $\cW^{cs}(x)=\cW^{cs}_1(x)$.}

The plaque family $\psi$ is \emph{locally invariant} by the time-one map of the flow $(P_t)$ if there exist $\alpha_\cE>0$
such that for  any $x\in K$ we have
$$P_1(\cW^{cs}_{\alpha_\cE}(x))\subset \cW^{cs}(\varphi_1(x)).$$
\end{Definition}

Hirsch-Pugh-Shub's plaque family theorem~\cite{hirsch-pugh-shub} generalizes to local fibered flows.
\begin{Theorem}\label{t.plaque}
For any local fibered flow $(\cN,P)$ admitting a dominated splitting $\cN=\cE\oplus \cF$
there exists a $C^1$-plaque family tangent to $\cE$ which is locally invariant by $P_1$.

If the flow is $C^2$ and if $\cE$ is $2$-dominated, then the plaque family can be chosen $C^2$.
\end{Theorem}

\subsubsection{Plaque family for generalized orbits}\label{ss.plaque-genralized}

The previous result extends to generalized orbits.
\begin{Theorem}\label{t.generalized-plaques}
For any local fibered flow $(\cN,P)$ admitting a compatible identification and a domination $\cN=\cE\oplus \cF$,
there exists $\eta,\alpha_\cE>0$ with the following property.

For any $x\in K$ and any half generalized orbit $\bar u=(u(t))_{t\in [0,+\infty)}$ in the $\eta$-neighborhood of the zero-section with $u(0)\in \cN_x$,
there exists a $C^1$-diffeomorphism onto its image $\psi_{\bar u}\colon \cE(x)\to \cN_x$
such that $\psi_{\bar u}(0_x)=u(0)$ and
the image of $D\psi_{\bar u}(0_x)$ is contained in the cone $\cC^\cE(x)$.
Moreover $\psi_{\bar u}$ depends continuously for the $C^1$-topology on $\bar u$. {Denote by $\cW^{cs}(\bar u)=\psi_{\bar u}(\cE(x))$.}

The family of plaques is locally invariant by $\bar P_1$:
$$\bar P_1(\cW^{cs}_{\alpha_\cE}(\bar u))\subset \cW^{cs}(\bar P_1(\bar u)),$$
where $\cW^{cs}_{\alpha_\cE}(\bar u)$ denotes as before the ball centered at $u(0)$ and of radius $\alpha_\cE$. 

When the identification and the flow are $C^2$ and when $\cE$ is $2$-dominated, then the plaques can be chosen $C^2$
and the family $\psi_{\bar u}$ depends continuously on $\bar u$ for the $C^2$-topology.
\end{Theorem}

The tangent space to $\cW^{cs}(\bar u)$ at $\bar u$ in $\cN_x$ is denoted by $\cE(\bar u)$.
By construction it varies continuously with $\bar u$ and coincides with $\cE(x)$ when $\bar u$ is the
half orbit $(0_{\varphi_{t}(x)})_{t\geq 0}$.

\begin{Remark}
We have $D\bar P_t(\cE(\bar u))=\cE(\bar P_t(\bar u))$ for any $t\in \RR$.
\emph{Indeed, this holds for $t\in \NN$ by local invariance of the plaque family.
The forward invariance of $\cC^\cF$ implies that $\bar P_t(\cE(\bar u))$ is tangent to $\cC^\cE$ for any $t\geq 0$.
The dominated splitting implies that any vector $v$ tangent to $\cN_x$ at $\bar u$
whose forward iterates are all tangent to $\cC^\cE$ belongs to $\cE(\bar u)$
(this can be also seen from the sequence of diffeomorphisms introduced in the next section).
This characterization implies the invariance of $\cE$ for any time.}
\end{Remark}

\subsubsection{Plaque family for sequences of diffeomorphisms}
The proofs of Theorems~\ref{t.plaque} and~\ref{t.generalized-plaques} are very similar to~\cite[Theorem 5.5]{hirsch-pugh-shub}.
It is a consequence of a more general result that we state now.
We denote by $d=d^\cE+d^\cF$ the dimensions of the fibers of $\cN,\cE,\cF$ and endow $\RR^d$ with the standard euclidean metric. For $\chi>0$ let us define
the horizontal cone
$$\cC_\chi=\{(x,y)\in\RR^{d^\cE}\times \RR^{d^\cF}, \chi\|x\|\geq \|y\|\}.$$
\begin{Definition}\label{d.sequence}
A \emph{sequence of $C^k$-diffeomorphisms of $\RR^d$ bounded by constants $\beta,C>0$}
is a sequence $\underbar F$ of diffeomorphisms $f_n\colon U_n\to V_n$, $n\in \NN$, where
$U_n,V_n\subset \RR^d$ contain $B(0,\beta)$, such that $f_n(0)=0$ and
such that the $C^k$-norms of $f_n, f_n^{-1}$ are bounded by $C$.

We denote by $\sigma(\underbar F)$ the shifted sequence $(f_{n+1})_{n\geq 0}$ associated to $\underbar F=(f_n)_{n\geq 0}$.
\smallskip

The sequence has a \emph{dominated splitting} if there exists $\tau_0\in \NN$
such that for any $n\in \NN$ and for any 
{$z\in B(0,\beta)\cap f_n^{-1}(B(0,\beta))\cap \dots \cap f_{n+\tau_0-1}^{-\tau_0}(B(0,\beta))$,}
the cone $\cC_1$ is mapped by $D(f_{n+\tau_0-1}\circ \dots \circ f_n)^{-1}(z)$
inside the smaller cone $\cC_{1/2}$.
\smallskip

The center stable direction of the dominated splitting is \emph{$2$-dominated}
if there exists $\tau_0\in \NN$ such that for any 
{$z\in B(0,\beta)\cap f_n^{-1}(B(0,\beta))\cap \dots \cap f_{n+\tau_0-1}^{-\tau_0}(B(0,\beta))$}
and for any unit vectors $u,v$ satisfying
$D(f_{n+\tau_0-1}\circ \dots \circ f_n)(z).u\in\cC_1$ and $v\in \RR^d\setminus \cC_1$,
then
$$ \|D(f_{n+\tau_0-1}\circ \dots \circ f_n)(z).u\|^2\leq \frac 1 2 \|D(f_{n+\tau_0-1}\circ \dots \circ f_n)(z).v\|. $$
\end{Definition}

\begin{Theorem}\label{t.generalizedplaque}
For any $C,\beta,\tau_0$, there exists $\alpha\in (0,\beta)$, and for any sequence of $C^1$-diffeomor\-phisms
$\underbar F=(f_n)$ of $\RR^d$ bounded by $\beta,C$ with a dominated splitting, associated to the constant $\tau_0$, there
exists a $C^1$-map $\psi=\psi(\underbar F)\colon \RR\to \RR$ such that:
\begin{itemize}
\item[--] For any $z,z'$ in the graph $\{(x,\psi(\underbar F)(x)), x\in \RR^{d^\cE}\}$, the difference $z'-z$ is contained in $\cC_{\frac 1 2}$.
\item[--] (Local invariance.) $f_0\left( \{(x,\psi(\underbar F)(x)), |x|<\alpha\}\right) \subset \{(x,\psi(\sigma(\underbar F))(x)), x\in \RR^{d^\cE}\}$.
\item[--] The function $\psi$ depends continuously on $\underbar F$ for the $C^1$-topology:
for any $R,\varepsilon>0$, there exists $N\geq 1$ and $\delta>0$
such that if the two sequences $\underbar F$ and $\underbar F'$ satisfy
$$\|(f_n-f_n')|_{B(0,\beta)}\|_{C^1}\leq \delta \text{ for } 0\leq n \leq N$$
then $\|(\psi(\underbar F)-\psi(\underbar F'))|_{B(0,R)}\|_{C^1}$ is smaller than $\varepsilon$.
\item[--] For sequences of $C^2$-diffeomorphisms $\underbar F$ such that the center stable direction of the dominated splitting is $2$-dominated (still for the constant $\tau_0$), the function $\psi(\underbar F)$ is
$C^2$ and depends continuously on $\underbar F$ for the $C^2$-topology.
\end{itemize}
\end{Theorem}
\noindent
The proof of this theorem is standard. It is obtained by
\begin{itemize}
\item[--] introducing a sequence of diffeomorphisms
$(\widehat f_n)$ defined on the whole plane $\RR^d$ which coincide with the diffeomorphisms
$f_n$ on a uniform neighborhood of $0$ and with the linear diffeomorphism
$Df_n(0)$ outside a uniform neighborhood of $0$,
\item[--] applying a graph transform argument.
\end{itemize}

Theorem~\ref{t.plaque} is a direct consequence of Theorem~\ref{t.generalizedplaque}:
for each $x\in K$, we consider the sequence of local diffeomorphisms
$P_1\colon \cN_{\varphi_n(x)}\to \cN_{\varphi_{n+1}(x)}$.
There are bounded linear isomorphisms which identify $\cN_{\varphi_n(x)}$
with $\RR^d$ and send the spaces $\cE(\varphi_n(x))$ and $\cF(\varphi_n(x))$
to $\RR^{d^\cE}\times \{0\}$ and $\{0\}\times \RR^{d^\cF}$. Since the isomorphisms are bounded
we get a sequence of diffeomorphisms as in Definition~\ref{d.sequence}.
Theorem~\ref{t.generalizedplaque} provides a plaque in $\cN_x$ and which depends continuously on $x$.

Theorem~\ref{t.generalized-plaques} is proved similarly:
let $\bar u$ be a half generalized orbit and
$(y(t))_{t\in [0,\infty)}$ be its projection to $K$ and $\bar P$ the generalized flow;
we consider the  local diffeomorphisms
$\bar P_1\colon \cN_{y(n)}\to \cN_{y(n+1)}$ in a neighborhood of
the points $u(n)$ and $u(n+1)$ respectively, for each $n\in \NN$.
\medskip

\subsubsection{Uniqueness}
There is no uniqueness in Theorem~\ref{t.generalizedplaque},
but once we have fixed the way of choosing $\widehat f_n$, the invariant graph becomes unique
(and this is used to prove the continuity in Theorem~\ref{t.generalizedplaque}).
Also the following classical lemma holds.

\begin{Proposition}\label{p.uniqueness}
In the setting of Theorem~\ref{t.generalizedplaque}, up to reduce $\alpha$, the following property holds.
If there exists $z'$ in the graph of $\psi(\underbar F)$ and $z\in \RR^d$ such that  for any $n\geq 0$
\begin{itemize}
\item[--] the iterates of $z$ and $z'$ by $f_n\circ\dots\circ f_0$ are defined and belong to
$B(0,\alpha)$,
\item[--] $(f_n\circ\dots\circ f_0(z))-(f_n\circ\dots\circ f_0(z'))\in \cC_1$,
\end{itemize}
then $z$ is also contained in the graph of $\psi(\underbar F)$.
\end{Proposition}
\begin{proof}
Let us assume by contradiction that $z=(x,y)$ is not contained in the graph of $\psi$ and
let us denote $\widehat z=(x,\psi(x))$.
The line containing the iterates of $z$ and $\widehat z$
by the sequence $f_{n-1}\circ\dots\circ f_0$, $n\geq 1$,
remains tangent to the cone $\RR^d\setminus \cC_1$ (by the dominated splitting).
The line containing the iterates of $z$ and $z'$
and the line containing the iterates of $\widehat z$ and $z'$ are tangent to $\cC_1$
by our assumption, and by the two items of Theorem~\ref{t.generalizedplaque}.

The domination implies that the distances between the $n^\text{th}$ iterates of $z,\widehat z$
gets exponentially larger than their distance to the $n^\text{th}$ iterate of $z'$.
This contradicts the triangular inequality.
\end{proof}

\begin{Remark-numbered}\label{r.plaque-invariance}
The plaque family $\cW^{cs}$ given by Theorem~\ref{t.plaque} is a priori only invariant
by the time-$1$ map $P_1$ of the flow but the previous proposition shows that
if $\alpha_\cE>0$ is small enough, for any $x\in K$ and $z\in \cW^{cs}(x)$ such that
$P_{n}(z)\in \cW^{cs}_{\alpha_\cE}(\varphi_n(x))$ for each $n\in\NN$,
then we have $P_{t}(z)\in \cW^{cs}(\varphi_t(x))$ for any $t>0$.
Indeed, by invariance of the cone $\cC^\cF$, the point $P_{t}(z)$ belongs to $\cC^\cE(\varphi_t(x))$ for any $t\geq 0$.

The same property holds for half generalized orbits parametrized by $[0,+\infty)$.
\end{Remark-numbered}

\subsubsection{Coherence}
The uniqueness allows us to deduce that when plaques intersect then they have to \emph{match}, i.e.
to be contained in a larger sub-manifold.

\begin{Proposition}\label{p.coherence}
Fix a plaque family $\cW^{cs}$ as given by Theorem~\ref{t.generalized-plaques}.
Up to reduce the constants $\eta,\alpha_\cE>0$ the following property holds.
Let us consider any half generalized orbits $\bar u,\bar u'$ (parametrized by $[0,+\infty)$),
and any sets $X\subset \cW^{cs}_{\alpha_\cE}(\bar u)$, $X'\subset \cW^{cs}_{\alpha_\cE}(\bar u')$
such that:
\begin{enumerate}
\item $\bar u$, $\bar u'$ belong to the $\eta$-neighborhood of $K$ and satisfy $\bar u\in X$, $\bar u'\in X'$,
\item the points $y,y'$ satisfying $u(0)\in \cN_y$, $u'(0)\in \cN_{y'}$
are $r_0$ close and $y$ belongs to the $2r_0$-neighborhood of $K\setminus V$,
\item the projection $(y(t))_{t\in [0,+\infty)}$ of $\bar u$ to $K$ has arbitrarily large iterates in the
$r_0$-neighborhood of $K\setminus V$,
\item $\pi_{y}(X')\cap X\neq \emptyset$ and
$\diam (\bar P_t(X)), \diam (\bar P_t(X'))\leq \alpha_\cE$ for all $t\geq 0$,
\end{enumerate}
then $\pi_{y}(X')$ is contained in $\cW^{cs}(\bar u)$.
\end{Proposition}
\begin{proof}
Let us denote by $(y(t))$ and $(y'(t))$ the projections of $\bar u$ and $\bar u'$ to $K$.
Proposition~\ref{p.uniqueness} gives $\alpha>0$.
Provided $\eta,\alpha_\cE$ are small enough,
the proof consists in checking that some Global invariance extends to generalized orbits:

\begin{Claim} \label{c.coherence}
There exist $T,T'>10$ such that
\begin{itemize}
\item[(a)] the points $y(T),y'(T')$ are $r_0$ close and
$y(T)$ belongs to the $2r_0$-neighborhood of $K\setminus V$,
\item[(b)] $\bar P_{T}(\pi_{y}(X'))=\pi_{y(T)}(\bar P_{T'}(X'))$,
\item[(c)]  $\diam (\bar P_t(\pi_{y}(X')))\leq \alpha/2$ for any $t\in [0,T]$.
\end{itemize}
\end{Claim}
\begin{proof}[Proof of the Claim]
Let $C$ be a large constant which bounds:
\begin{itemize}
\item[--] the Lipschitz constant of the projections $\pi_z$, for $z$ in
the $2r_0$-neighborhood of $K\setminus V$,
\item[--] the norms $\|DP_s\|$ for $|s|\leq 1$.
\end{itemize}
Since 
$\varphi$ is continuous,
there exists $\tilde \eta>\eta$ such that
for each $z,z'\in K$ satisfying $d(z,z')<\eta$
we have $\sup_{s\in[-1,1]}d(\varphi_s(z),\varphi_s(z'))<\tilde \eta$.
Taking $\eta$ small allows to choose $\tilde \eta$ small as well.

We will apply a first time the Global invariance, with the constants
$\rho=2$ and $\delta_1:=\min(r_0/4,\alpha/4)$: it gives us 
constants $\beta_1, r_1$.
Then we will apply a second time the Global invariance (the version of Remark~\ref{r.identification}.(e)),
with the constant $\rho=2$ and $\delta_2:=r_1/4$: it gives us 
constant $\beta_2$.
Take $\eta$ and $\alpha_\cE$ small so that
$(1+C)^2(\tilde \eta+\alpha_\cE)<\min(\beta_1,\beta_2)$
and $\eta<r_0/4$, $\tilde \eta<r_1/2$.
\smallskip

For proving the claim, it is enough to prove the existence of $T,T'>0$
satisfying the properties (a), (b), (c) above and such that one of  the following properties occurs:
\begin{itemize}
\item[--] $T,T'>10$, 
\item[--] $[0,T]$ contains a discontinuity of $(y(t))$,
\item[--] $[0,T']$ contains a discontinuity of $(y'(t))$.
\end{itemize}
Indeed by definition the discontinuities of generalized orbits are separated
in time by at least $1$. It is thus enough to apply the argument below 20 times in order
to get a pair of times $(T,T')$ such that $T,T'>10$.
\medskip

We now explain how to obtain the pair $(T,T')$.
The diameters of $\{0_{y}\}\cup X$ and of $\{0_{y'}\}\cup X'$
are smaller than $\eta+\alpha_\cE$. Moreover $\pi_y(X')$ meets $X$.
Hence $\|\pi_y(y')\|<(1+C)(\eta+\alpha_\cE)$.
With our choice of $\eta,\alpha_\cE$, this gives $\|\pi_y(y')\|<C^{-1}\min(\beta_1,\beta_2)$
and then $\|P_s\circ \pi_y(y')\|<\beta_2$ for any $s\in [-1, 1]$.
The Global invariance (Remark~\ref{r.identification}.(e)) gives $\theta_2\in \lip_{1+\rho}$ such that
$|\theta_2(0)|\leq 1/4$, and
$d(\varphi_{s}(y),\varphi_{\theta_2(s)}(y))<\delta_2$ for any $s\in [-1, 1]$.

\noindent
\paragraph{\it First case: $\theta_2(0)\geq 0$.} We estimate
$$d(y, y'(\theta_2(0))\leq d(y, \varphi_{\theta_2(0)}(y'))+d(\varphi_{\theta_2(0)}(y'), y'(\theta_2(0))).$$
We have $d(y, \varphi_{\theta_2(0)}(y'))= d(\varphi_0(y), \varphi_{\theta_2(0)}(y'))<\delta_2$.
Note that either $\varphi_{\theta_2(0)}(y')=y'(\theta_2(0))$, or, the generalized orbit $\bar u'$ has one discontinuity
at some time $s$ which belongs to $[0, \theta_2(0)]\subset [-1,1]$. Since $\bar u'$ is in the $\eta$-neighborhood of
$K$, we have $d(\varphi_s(y'), y'(s))<\eta$. One thus gets that $d(\varphi_{\theta_2(0)}(y'), y'(\theta_2(0)))<\tilde \eta$
and $d(y, y'(\theta_2(0))<\delta_2+\tilde \eta<r_1$.

We can thus apply the Global invariance to the points $y\in U$ and $y'(\theta_2(0))$
and to points $v$ in $X$, $v'\in X'$ such that $\pi_y(v')=v$.
Using the local invariance and a previous estimate we get:
$\|\pi_y(y'(\theta_2(0))\|=\|\pi_y(y')\|<\beta_1$. Consequently, there exists
$\theta_1\in \lip_{1+\rho}$ and an interval $[0,a)$ such that
$d(\varphi_t(y),\varphi_{\theta_1(t)}(y'(\theta_2(0)))<\delta_1=r_0$ for any $t\in [0,a)$.
Here $[0,a)$ is any interval such that
$\|P_t(v)\|<\beta_1$ and $\|P_{\theta_1(t)+\theta_2(0)}(v')\|<\beta_1$.
Since $\bar u$, $\bar u'$ are in the $\eta$-neighborhood of $K$
and by the assumption (4) above, this is ensured if $[0,a)$
is the maximal interval of time $t$ such that
$\varphi_t(y)=y(t)$ and $\varphi_{\theta_1(t)}(y'(\theta_2(0)))=y'(\theta_2(0)+\theta_1(t))$.
\medskip

If $a<+\infty$, we set $T=a$ and $T'=\theta_2(0)+\theta_1(a)$.
By definition of generalized orbits, either $y(T)$ or $y'(T')$ is in the $\eta$-neighborhood
of $K\setminus V$ and they are at distance smaller than $\delta_1+2\eta<r_0$.
In particular both $y(T)$ and $y'(T')$ belong to
the $2r_0$-neighborhood of $K\setminus V$.
Using again the condition $\|u(t)\|+\diam(\bar P_t(X'))<\eta+\alpha_\cE<\min (\beta_1,\beta_2)$,
the Global invariance gives the condition (b) and $\|\bar P_t(w)\|\leq \delta_1=\alpha/4$
for each $t\in [0,T]$ and each $w\in \pi_y(X')$.
The conditions on $T,T'$ are thus satisfied.

If $a=\infty$, we use the assumption (3) to find a large time $T$
such that $y(T)$ belongs to the $r_0$-neighborhood of $K\setminus V$ and we set $T'=\theta_2(0)+\theta_1(a)$.
The conditions on $(T,T')$ are checked similarly (this case is simpler).

\noindent
\paragraph{\it Second case: $\theta_2(0)< 0$.} We follow the argument of the first case.
As above,  $d(y',y(\theta_2^{-1}(0))<r_1$ and we apply the Global invariance
to the points $y'\in U$ and $y(\theta_2^{-1}(0)$. This gives $\theta_1\in \lip_{1+\rho}$.
We choose $T'>0$ and set $T=\theta_2^{-1}(0)+\theta_1^{-1}(T)$
such that either $y(T)$ or $y'(T')$ is a discontinuity of the family $(y(t))$ or $(y'(t'))$,
or $y(T)\in K\setminus V$.
\end{proof}

\smallskip
Applying the Claim inductively, we find two increasing sequences of times $T_n,T'_n\to +\infty$.
Indeed having defined $T_n,T'_n$, the generalized orbits
$\bar P_{T_n}(\bar u)$ and $\bar P_{T'_n}(\bar u')$ satisfy the assumptions of  Proposition~\ref{p.coherence}
and the Claim associate a pair $T,T'$; we then set $T_{n+1}=T_n+T$ and $T'_{n+1}=T'_n+T'$.

In order to conclude Proposition~\ref{p.coherence},
one considers $z'\in \pi_y(X')\cap X$ and any $z\in \pi_y(X')$
and use Proposition~\ref{p.uniqueness}. Let us check that its assumptions are satisfied:
\begin{itemize}
\item[--] By our assumptions, and requiring $\alpha_\cE<\alpha$ we have
$\|\bar P_t(z')-u(t)\|<\alpha/2$ for any $t\geq 0$.
\item[--] Since $\bar P_t(z'), \bar P_t(z)\in \bar P_t(\pi_{y}(X'))$, the first item of the lemma implies that
we also have 
$\|\bar P_t(z)-u(t)\|\leq \alpha$ for any $t\geq 0$.
\item[--] Since the complement of the cone field $\cC^\cE$ is invariant by forward iterations,
it only remains to check that $\bar P_{t}(z)-\bar P_{t}(z')\in \cC^\cE(y(t))$ for a sequence of arbitrarily large times $t$.

From the second item of the lemma,
the projections of the points $\bar P_{T_n}(z)$, $\bar P_{T_n}(z')$ by $\pi_{y'(T'_n)}$ belong
to $\bar P_{T'_n}(\cW^{cs}(\bar u'))$, and hence to
$\cW^{cs}(\bar P_{T'_n}(\bar u'))$ by Remark~\ref{r.plaque-invariance}.
So their difference belongs to
$\cC^\cE_{1/2}(y'(T'_n))$ (by Theorem~\ref{t.generalizedplaque}).
The continuity of the cone field and the fact $d(y(T_n),y'(T'_n))<\delta_0$
gives $\bar P_{T_n}(z)-\bar P_{T_n}(z')\in \cC^\cE(y(T_n))$.
\end{itemize}
Hence Proposition~\ref{p.uniqueness} applies and concludes the proof of Proposition~\ref{p.coherence}.
\end{proof}

\subsubsection{Limit dynamics in periodic fibers}

We state a consequence of the existence of plaque families. It will be used for the center-unstable plaques $\cW^{cu}$.
Note that any $u\in \cN$ such that $\|P_{-t}(u)\|$ is small for any $t\geq 0$ is a half generalized orbit,
hence has a plaque $\cW^{cu}(u)$.

\begin{Proposition}\label{p.fixed-point}
For any local fibered flow  $(P_t)$ on a bundle $\cN$ admitting a dominated splitting $\cN=\cE\oplus \cF$
where $\cE$ is one-dimensional, there exists $\delta>0$ with the following property.

For any periodic point $z\in K$ with period $T$ and any $u\in \cN_z$ satisfying $\|P_{-t}(u)\|\leq \delta$ for all $t>0$
and $0_z\not\in \cW^{cu}(u)$,
there exists $p\in \cN_z$ such that $P_{2T}(p)=p$ and $P_{-t}(u)$ converges to the orbit of $p$
when $t$ goes to $+\infty$.
\end{Proposition}
\begin{proof}

Let $\alpha_\cE,\alpha_\cF$ be the constants associated to $\cE,\cF$ as in Theorem~\ref{t.generalized-plaques}.
The plaques $\cW^{cu}(u)$ and $\cW^{cs}(0_z)$
intersect at a (unique) point $y$.

Since $\|u\|$ is small, by the local invariance of the plaque families, $y$ is also
the intersection between $P_{1}(\cW^{cs}_{\alpha_E}(\varphi_{-1}(z)))$ and
$\cW^{cu}_{\alpha_\cF}(u)$.
One deduces that $P_{-1}(y)$
is the (unique) intersection point between
the plaques $\cW^{cu}(P_{-1}(u))$ and $\cW^{cs}(\varphi_{-1}(z))$.
Repeating this argument inductively,
one deduces that the backward orbit of $y$ by $P_{-1}$
remain in the plaques $\cW^{cu}(P_{-k}(u))$ and $\cW^{cs}(\varphi_{-k}(z))$.
From Remark~\ref{r.plaque-invariance}, any backward iterate $P_{-t}(y)$
belongs to $\cW^{cu}(P_{-t}(u))$.

Since $y\neq 0_z$, the domination implies that
$d(P_{-k}(u),P_{-k}(y))$ is exponentially smaller than
$d(P_{-k}(y), 0_{\varphi_{-k}(z)})$ as $k\to +\infty$.
So $d(P_{-k}(u),P_{-k}(y))$ goes to $0$.
The same argument applied to $P_{-s}(u)$,
$s\in [0,1]$ shows that the distance of $P_{-t}(u)$
to $\cW^{cs}(\varphi_{-t}(z))$ converges to $0$ as $t\to +\infty$.
In particular, the limit set of the orbit of $u$ under $P_{-2T}$
is a closed subset $L$ of $\cW^{cs}(z)$.
In the case $L$ is a single point $p$, the conclusion of the proposition follows.

We assume now by contradiction  that $L$ is not a single point.
There exists $q\neq 0_z$ invariant by $P_{2T}$
in $L$ such that $L$ intersects the
open arc $\gamma$ in $\cW^{cs}(z)$ bounded by $0_z$ and $q$.
Note that the forward iterates $P_k(\gamma)$ by $P_1$ remain small,
hence in $\cW^{cs}(\varphi_k(z))$ by the local invariance of $\cW^{cs}$.
From Remark~\ref{r.plaque-invariance}, any iterate
$P_t(\gamma)$ is contained in $\cW^{cs}(\varphi_t(z))$, $t\in \RR$.

Up to replace $u$ by a backward iterate, one can assume that $y$ belongs to $\gamma$.
This shows that $P_{-t}(y)$ is the intersection between
$\cW^{cu}(P_{-t}(u))$ and $\cW^{cs}(P_{-t}(z))$ for any $t\geq 0$.
The set $L$ is thus the limit set of the orbit of $y$ under $P_{-2T}$.
This reduces to a one-dimensional dynamics
for an orientation preserving diffeomorphism,
and $L$ has to be a single point, a contradiction.
\end{proof}

\subsubsection{Distortion control}

The following lemma restates the classical Denjoy-Schwartz argument in our setting.

\begin{Lemma}\label{Lem:schwartz}
Let us assume that $(P_t)$ is $C^2$, that $\cE$ is one-dimensional and that
$\cW^{cs}$ is a $C^2$ locally invariant plaque family.
Then, there is $\beta_S>0$ and for any $C_{Sum}>0$, there are $C_S,\eta_S>0$ with the following property.

For any $x\in K$, for any
interval $I\subset \cW^{cs}(x)$ and any $n\in \NN$ satisfying
$$
\forall m\in \{0,\dots,n\},\; P_m(I)\subset B(0,\beta_S) \text{ and }
\sum_{m=0}^{n} |P_{m}(I)|\leq C_{Sum},$$
then (1) for any $u,v\in I$ we have
\begin{equation}\label{e.distortion}
C_S^{-1}\leq \frac {\|DP_n(u)|_{I}\|}{\|DP_n(v)|_{I}\|}\leq C_S;
\end{equation}
in particular $\|DP_n(u)|_{I}\|\leq C_S \frac{|P_n(I)|}{|I|}$;
\smallskip

\noindent
(2) any interval $\widehat I\subset \cW^{cs}(x)$ containing $I$
with $|\widehat I |\leq (1+\eta_S)|I|$ satisfies
$|P_{n}(\widehat I)|\leq 2|P_{n}(I)|$;
\smallskip

\noindent
(3) any interval $\widehat I\subset \cW^{cs}(x)$ containing $I$
with $|P_{n}(\widehat I)|\leq (1+\eta_S)|P_{n}(I)|$ satisfies
$|\widehat I |\leq 2|I|$.
\end{Lemma}

The proof is similar to~\cite[Chapter I.2]{dMvS}.

\subsection{Hyperbolic iterates}\label{ss:hyperbolicreturns}
We continue with the setting of Section~\ref{s.plaque} and
we fix two locally invariant plaques families $\cW^{cs}$ and $\cW^{cu}$ tangent to $\cE$ and $\cF$ respectively,
and two constants $\alpha_\cE,\alpha_\cF$ controlling  the geometry and the dynamics inside these plaques
as in the previous sections.
The plaques are defined at points of $K$ but also at half generalized orbits $\bar u$ contained in a $\eta$-neighborhood of $K$
and parametrized by $[0,+\infty)$ and $(-\infty, 0]$ respectively.
The quantities $\eta,\alpha_\cE,\alpha_\cF>0$ may be reduced in order to satisfy further properties below.

In case $\cE$ is $2$-dominated, $\cW^{cs}$ will be a $C^2$-plaque family. Remember that $\tau_0,\lambda$ are the constants
associated to the domination, as introduced in Section~\ref{ss.assumptions}.

\subsubsection{Hyperbolic points}
We introduce a first notion of hyperbolicity.
\begin{Definition}
Let us fix $C_\cE,\lambda_\cE>1$.
A piece of orbit $(x,\varphi_{t}(x))$ in $K$ is \emph{$(C_\cE,\lambda_\cE)$-hyperbolic for $\cE$} if
for any $s\in (0,t)$, we have
$$\|DP_{s}{|\cE(x)}\| \leq C_\cE\lambda_\cE^{-s}.$$
A point $x$ is \emph{$(C_\cE,\lambda_\cE)$-hyperbolic for $\cE$} if
$(x,\varphi_{t}(x))$ is $(C_\cE,\lambda_\cE)$-hyperbolic for $\cE$ for any $t>0$.
We have similar definitions for the bundle $\cF$ (considering the flow $t\mapsto P_{-t}$).
\end{Definition}

By the continuity of the fibered flow for the $C^1$-topology, the hyperbolicity extends to orbits close
(the proof is easy and omitted).
\begin{Lemma}\label{l.shadowingandhyperbolicity}
Let us assume that $\cE$ is one-dimensional.
For any $\lambda'>1$, there exist $C',\delta,\rho>0$
such that for any $x,y\in K$, $t>0$ and $\theta\in \lip_{1+\rho}$ satisfying
 $\theta(0)=0$ and $$d(\varphi_s(x),\varphi_{\theta(s)}(y))<\delta \text{ for each } s\in [0,t],$$
then
$$\|DP_{\theta(t)}|\cE(y)\|\leq C'{\lambda'}^{t}\|DP_t|\cE(x)\|.$$
\end{Lemma}

Hyperbolicity implies summability for the iterations inside one-dimensional plaques.

\begin{Lemma}[Summability]\label{l.summability-hyperbolicity}
Let us assume that $\cE$ is one-dimensional and consider $\lambda_\cE,C_\cE>1$.
Then, there exists $C'_\cE>1$ and $\delta_\cE>0$ with the following property.

For any piece of orbit $(x,\varphi_t(x))$ which is $(C_\cE,\lambda_\cE)$-hyperbolic for $\cE$
for any interval $I\subset \cW^{cs}(x)$ containing $0$
whose length $|I|$ is smaller than $\delta_\cE$, and for any interval $J\subset I$ one has
$$|P_t(J)|\leq C_\cE\lambda_\cE^{-t/2}\;|J| \text{ and }
\sum_{0\leq m\leq [t]} |P_m(J)|\leq C'_\cE\; |J|.$$
\end{Lemma}
\begin{proof}
Let $\eta>0$ be small such that $1+\eta<\lambda_\cE^{1/2}$.
Then, there exists $\delta_0$ such that for any $y\in K$ and any interval $I_0\subset \cW^{cs}(y)$
containing $0$ whose length is smaller than $\delta_0$, one has 
$$\forall s\in[0,1],~~~|P_{s}(I_0)|\leq (1+\eta)\|DP_{s}{|\cE(y)}\|\; |I_0| .$$
Let us choose $\delta_\cE$ satisfying $\delta_\cE C_\cE\lambda_\cE^{1/2}<\delta_0$.
One checks inductively that the length of $P_k(I)$ is smaller than $\delta_0$ for each $k\in [0,t]$.
The conclusion of the lemma follows.
\end{proof}

\subsubsection{Pliss points}
We introduce a more combinatorial notion of hyperbolicity (only used for $\cF$).

\begin{Definition}\label{d.pliss}
For $T\geq 0$ and $\gamma>1$, we say that a piece of orbit $(\varphi_{-t}(x),x)$ is a \emph{$(T,\gamma)$-Pliss string}
(for the bundle $\cF$)
if there exists an integer $s\in[0,T]$ such that
$$\text{for any integer } m\in \bigg[0,\frac {t-s}{\tau_0}\bigg],\quad  \prod_{n=0}^{m-1}\|DP_{-\tau_0}|{\cF(\varphi_{-(n\tau_0+s)}(x))}\|\leq \gamma^{-m\tau_0}.$$
A point $x$ is $(T,\gamma)$-Pliss (for $\cF$) if  $(\varphi_{-t}(x),x)$ is a $(T,\gamma)$-Pliss string for any $t>0$.

For simplicity, a piece of orbit $(\varphi_{-t}(x),x)$ is a \emph{$T$-Pliss string} if it is a $(T,\lambda)$-Pliss string and $x$ is \emph{$T$-Pliss} if it is $(T,\lambda)$-Pliss, where $\lambda$ is the constant for the domination.
\end{Definition}

For any $T>0$, there exists $C>1$ such that
any piece of orbit which is a $T$-Pliss string is also $(C,\lambda)$-hyperbolic for $\cF$.
Pliss lemma gives a kind of converse result.

\begin{Lemma}\label{Lem:hyperbolic-Pliss}
Assume $\dim\cF=1$.
For any 
$\lambda_1>\lambda_2>1$ and $C>1$, there is $T_\cF>0$ such that any piece of orbit $(\varphi_{-t}(x),x)$ which is $(C,\lambda_1)$-hyperbolic for $\cF$
is also a $(T_\cF,\lambda_2)$-Pliss  string.
\end{Lemma}
\begin{proof}
Let us define $a_i=\log\|DP_{-\tau_0}|{\cF(\varphi_{-i\tau_0}(x))}\|$ for $0\le i\le \frac t {\tau_0}-1$. Then by applying a usual Pliss lemma, one can get the conclusion
(see for instance~\cite[Lemma 2.3]{MCY}).
\end{proof}

Pliss strings extend to orbits close.

\begin{Lemma}\label{l.continuity-Pliss}
If $\cF$ is one-dimensional,
there exist $\beta>0$ and $T\geq 1$ with the following property:
if $(x,\varphi_{t}(x))$ is a $T_\cF$-Pliss string with $x\in U$ and $T_\cF\geq T$
and if $y$ satisfies:
$$d(y,x)<r_0 \text{ and for any } 0\leq s \leq t \text{ one has }
\|P_s\pi_{x}(y)\|<\beta,$$
then there exists a homeomorphism $\theta\in{\rm Lip}_2$  such that
\begin{itemize}
\item[--] $|\theta(0)|\leq 1/4$ and $d(\varphi_s(x),\varphi_{\theta(s)}(y))<r_0/2$ for each $s\in [-1,t+1]$,
\item[--] $\pi_{\varphi_s(x)}\circ \varphi_{\theta(s)}(y)=P_s\circ \pi_x(y)$ when $s\in[0,t]$ and $\varphi_s(x)\in U$,
\item[--] the piece of orbit $(y,\varphi_{\theta(t)+a}(y))$ is a $(2T_\cF,\lambda^{1/2})$-Pliss string for any $a\in [-1,1]$.
\end{itemize}
\end{Lemma}
\begin{proof}
The continuity of the flow for the $C^1$-topology implies that
there are $\delta>0$ and $\rho\in(0,1/20)$ such that for any $z,z'\in K$ and any $\theta\in{\rm Lip}_{1+\rho}$, if $d(\varphi_s(z),\varphi_{\theta(s)}(z'))<\delta$ for any $0\le s\le \tau_0$, then
$$\|DP_{\theta(0)-\theta(\tau_0)}|\cF(\varphi_{\theta(\tau_0)}(z'))\|\le \lambda^{\tau_0/5}\|DP_{-\tau_0}|\cF(z')\|.$$
The Global invariance associates to $\delta,\rho>0$ some constants $r,\beta>0$. 

Consider $z,z'\in K$, $n\ge 0$ and $\theta\in {\rm Lip}_{1+\rho}$ such that $(z,\varphi_{n\tau_0}(z))$ is $0$-Pliss
and
$$d(\varphi_s(z),\varphi_{\theta(s)}(z'))<\delta,~~~\forall 0\le s\le n\tau_0.$$
By the choice of $\delta,\rho$, we have for any $0\le k\le n-1$,
$$\prod_{j=k+1}^{n}\|DP_{\theta((j-1)\tau_0)-\theta(j\tau_0)}|\cF(\varphi_{\theta(j\tau_0)}(z'))\|<
\lambda^{-4/5(n-k)\tau_0}<\lambda^{-3/4 [\theta(n\tau_0)-\theta(k\tau_0)]}.$$
Thus there is $C>0$ depending on $\lambda$ and $\sup_{s\in [0,2\tau_0]} \|DP_{-s}\|$
such that
$$\|DP_{-s}|\cF(\varphi_{\theta(n\tau_0)-d}(z'))\|\le C\lambda^{-3/4 s},~~~\forall s\in[0,\theta(n\tau_0)-d+\tau_0], \; d\in [0,1].$$

Now, by our assumptions and the Global invariance, there is $\theta\in{\rm Lip}_{1+\rho}$ such that
$$d(\varphi_s(x),\varphi_{\theta(s)}(y))<\delta,~~~\forall s\in[-1, t+1].$$
Consider $T$ associated to $\lambda^{3/4},\lambda^{1/2},C$ by Lemma~\ref{Lem:hyperbolic-Pliss}.
For any $T_\cF>2T$, the fact that $(x,\varphi_{t}(x))$ is $(T_\cF,\lambda)$-Pliss implies that there is an integer $b\in[0,T_\cF]$ such that
$(x,\varphi_{t-b}(x))$ is $0$-Pliss. One chooses $d\in [0,1]$ such that
$\theta(t)+a-\theta(t-b)+d$ is a nonnegative integer.

One deduces that for any $s\in [0,\theta(t-b)-d]$,
$$\|DP_{-s}|\cF(\varphi_{\theta(t-b)-d}(z))\|\le C \lambda^{-3/4 s}.$$
By Lemma~\ref{Lem:hyperbolic-Pliss}, there is an integer $e\in [0,T]$ such that $(z,\varphi_{\theta(t-b)-d-e}(z))$ is $(0,\lambda^{1/2})$-Pliss.
Now $\theta(t-b)-d-e$ differs from $\theta(t)+a$ by an integer smaller than $(1+\rho)T_\cF +2 +T$, which is smaller than $2T_\cF$.
Hence $(z,\varphi_{\theta(t)}(z))$ is $(2T_\cF,\lambda^{1/2})$-Pliss for $\cF$.
\end{proof}
\medskip

The next proposition will allow us to find iterates that are Pliss points and belong to $U$.

\begin{Proposition}\label{l.summability}
Let us assume that $\cE$ is one-dimensional and
let $W$ be a set such that $\cE$ is uniformly contracted on $\bigcup_{s\in [0,1]}\varphi_s(W)$.
Then, there exist $C_\cE,\lambda_\cE>1$ such that any $T_\cF\geq 0$ large enough has the following property.

If there are $x\in K$ and integers $0\leq k\leq \ell$ such that
\begin{itemize}
\item[--] $(\varphi_{-\ell}(x),\varphi_{-k}(x))$ is a $T_\cF$-Pliss string,
\item[--] for any $j\in \{1,\dots,k-2\}$
either $\varphi_{-j}(x)\in W$ or the piece of orbit $(\varphi_{-\ell}(x),\varphi_{-j}(x))$ is not a $T_\cF$-Pliss string,
\end{itemize}
then $(\varphi_{-k}(x),x)$ is $(C_\cE,\lambda_\cE)$-hyperbolic for $\cE$.

\medskip

Similarly if there are $x\in K$ and $k\geq 0$ such that
\begin{itemize}
\item[--] $\varphi_{-k}(x)$ is $T_\cF$-Pliss,
\item[--] for any $j\in \{1,\dots,k-2\}$
either the point $\varphi_{-j}(x)$ is in $W$ or $\varphi_{-j}(x)$ is not $T_\cF$-Pliss,
\end{itemize}
then $(\varphi_{-k}(x),x)$ is $(C_\cE,\lambda_\cE)$-hyperbolic for $\cE$.
\end{Proposition}

\begin{proof}
The proof is essentially contained in \cite[Lemma 9.20]{CP}.
Recall that $\lambda>1$ is the constant for the domination.
There exist $C_0,\lambda_0>1$ such that for any piece of orbit
$(y, \varphi_t(y))$ in $\bigcup_{s\in [0,1]}\varphi_s(W)$, one has
$\|DP_t{|\cE(y)}\|\leq C_0\lambda_0^{-t}$.
Let $C_1>1$ such that
\begin{equation}\label{C1}
\forall y\in K,~\forall s\in [-\tau_0,\tau_0],~~~\|DP_{s}{|\cE(y)}\|\leq C_1\lambda^{-s}.
\end{equation}
One takes $C_\cE=C_0^2C_1^3$ and $\lambda_\cE>1$ smaller than $\min(\lambda,\lambda_0)$.
One then chooses $T_\cF\geq 0$ large such that $C_0C_1\lambda^{-T_\cF}<\lambda_\cE^{-T_\cF}$.
\smallskip

Let $x\in K$, $0\leq k\leq \ell$ be as in the statement of the lemma. We introduce the set
$$\cP=\bigg\{j\in\{1,\dots,k-2\},\; (\varphi_{-\ell}(x),\varphi_{-j}(x)) \text{ is a } T_\cF\text{-Pliss string}\bigg\}.$$

The set $\cP$ decomposes into intervals $\{a_i,1+a_i,\dots,b_i\}\subset \{1,\dots,k-2\}$, with $i=1,\dots,i_0$, such that
$b_i+1<a_{i+1}$. By convention we set $b_0=0$.

\begin{Claim-numbered}
$b_i-a_i\geq T_\cF$ unless $\{a_i,\dots,b_i\}$ contains $1$ or $k-2$.
\end{Claim-numbered}
\begin{proof} By maximality of each interval, $(\varphi_{-\ell}(x),\varphi_{-b_i}(x))$ has to be a $0$-Pliss string.
\end{proof}

\begin{Claim-numbered}\label{c.pliss}
Consider $n_1,n_2\in \{0,\dots,k\}$ with $n_1<n_2$ and such that
$(\varphi_{-\ell}(x),\varphi_{-n_2}(x))$ is a $0$-Pliss string and $(\varphi_{-\ell}(x),\varphi_{-j}(x))$ is not a $0$-Pliss string
for $n_1<j<n_2$. Then for any $0\leq m < (n_2-n_1)/\tau_0$,
$$\|DP_{m\tau_0}{|\cE(\varphi_{-n_2}(x))}\|\leq \lambda^{-m\tau_0}.$$
\end{Claim-numbered}
\begin{proof}
One checks inductively that
\begin{equation}\label{e.pliss}
\prod_{n=0}^{m-1}\|DP_{\tau_0}{|\cF(\varphi_{-n_2+n\tau_0}(x))}\|\leq \lambda^{m\tau_0}.
\end{equation}
Indeed if this inequality holds up to an integer $m-1$ and fails for $m$,
the piece of orbit $(\varphi_{-n_2}(x),\varphi_{-n_2+m\tau_0}(x))$ is a $0$-Pliss string.
It may be concatenate with $(\varphi_{-\ell}(x),\varphi_{-n_2}(x))$, implying that $(\varphi_{-\ell}(x),\varphi_{-n_2+m\tau_0}(x))$
is a $0$-Pliss string. This is a contradiction since $n_2-m\tau_0> n_1$.

The estimate of the claim follows from~\eqref{e.pliss} by domination.
\end{proof}

The proposition will be a consequence of the following properties.
\begin{Claim-numbered}\label{c.return0}
\begin{enumerate}

\item 
The piece of orbit $(\varphi_{-b_i}(x), \varphi_{-b_{i-1}}(x))$ is $(C_0C_1,\lambda_\cE)$-hyperbolic for $\cE$,
for any $i\in \{1,\dots,i_0\}$, unless $i=i_0$ and $b_{i_0}=k-2$.

Moreover if $i\neq 1$, then  $\|DP_{b_i-b_{i-1}}{|\cE(\varphi_{-b_i}(x))}\|\leq \lambda_\cE^{-(b_{i}-b_{i-1})}$.

\item If $b_{i_0}=k-2$, then $(\varphi_{-k}(x), \varphi_{-b_{i_{0}-1}}(x))$ is $(C_0C_1,\lambda_\cE)$-hyperbolic for $\cE$.

\item If $b_{i_0}<k-2$, then $(\varphi_{-k}(x), \varphi_{-b_{i_0}}(x))$ is $(C_0C_1,\lambda_\cE)$-hyperbolic for $\cE$.
\end{enumerate}
\end{Claim-numbered}
\begin{proof}
In order to check the first item, one introduces the smallest $j\in \{a_i,\dots,b_i\}$ such that
$(\varphi_{-\ell}(x),\varphi_{-j}(x))$ is a $0$-Pliss string: it exists unless $i=i_0$ and $b_{i_0}=k-2$.
By our assumptions, the piece of orbit $(\varphi_{-b_i}(x), \varphi_{-j}(x))$ is contained in $\bigcup_{s\in [0,1]}\varphi_s(W)$, hence is $(C_0,\lambda_\cE)$-hyperbolic for $\cE$.
Then Claim~\ref{c.pliss} gives
$\|DP_{m\tau_0}{|\cE(\varphi_{-j}(x))}\|\leq \lambda^{-m\tau_0}$ for any $m\in \{0,\dots, (j-b_{i-1})/\tau_0\}$.
One concludes the first part of item 1 by combining these estimates with~\eqref{C1}.
\smallskip

Note that one also gets the estimate:
$$\|DP_{b_i-b_{i-1}}|{\cE}(\varphi_{-b_i}(x))\|\leq C_0C_1\lambda_\cE^{-(b_{i}-b_{i-1})}\bigg(\frac \lambda {\lambda_\cE}\bigg)^{-(j-b_{i-1})}.$$
If $i\geq 2$, one gets $j-b_{i-1}\geq j-a_i= T_\cF$, hence by our choice of $T_\cF$:
$$\|DP_{b_i-b_{i-1}}|{\cE}(\varphi_{-b_i}(x))\|\leq C_0C_1\lambda_\cE^{-(b_{i}-b_{i-1})}\bigg(\frac \lambda {\lambda_\cE}\bigg)^{-T_\cF}
\leq \lambda_\cE^{-(b_{i}-b_{i-1})}.$$
This gives the second part of item 1.
\smallskip

The proofs of items 2 and 3 are similar to the proof of item 1:
for item 2, one introduces the smallest $j\geq a_{i_0}$ such that $(\varphi_{-\ell}(x),\varphi_{-j}(x))$ is a $0$-Pliss string;
for item 3, one introduces the smallest $j\geq k$ such that $(\varphi_{-\ell}(x),\varphi_{-j}(x))$ is a $0$-Pliss string and use the fact that
$(\varphi_{-\ell}(x),\varphi_{-k}(x))$ is a $T_\cF$-Pliss string.
\end{proof}

From Claim~\ref{c.return0}, one first checks that for each $i\in \{1,\dots,i_0\}$
$$\|DP_{k-b_i}{|\cE}(\varphi_{-k}(x))\|\leq C_0C_1\lambda_\cE^{k-b_i}.$$
For each $m\in \{0,\dots,k\}$, either there exists $i\in \{1,\dots,i_0\}$ such that $b_{i-1}\leq m\leq b_i$
or $b_{i_0}\leq m \leq k$. Using items 1 or 3, one concludes
$$\|DP_{k-m}{|\cE}(\varphi_{-k}(x))\|\leq C_0^2C_1^2\lambda_\cE^{k-m}.$$
Combining with~\eqref{C1}, one gets the required bound on $\|DP_{k-t}{|\cE}(\varphi_{-k}(x)\|$ for any $t\in [0,k]$.

This prooves the lemma for pieces of orbits $(\varphi_{-\ell}(x),x)$ such that $(\varphi_{-\ell}(x),\varphi_{-k}(x))$ is
a $T_\cF$-Pliss string.
The proof for half orbits $\{\varphi_{-t}(x),t>0\}$ such that $\varphi_{-k}(x)$ is $T_\cF$-Pliss is similar.
The proposition is proved.
\end{proof}

\subsubsection{Hyperbolic generalized orbits}
The hyperbolicity extends to half generalized orbits.
(Recall that if $\bar u$ is parametrized by $(-\infty,0]$, one can define the space $\cF(\bar u)$,
see Subsection~\ref{ss.plaque-genralized}.)

\begin{Definition}
Let us fix $C_\cF,\lambda_\cF>1$, $T_\cF\geq 0$ and consider 
a half generalized orbit $\bar u$ parametrized by $(-\infty,0]$.
\smallskip

\noindent
$\bar u$ is  \emph{$(C_\cF,\lambda_\cF)$-hyperbolic for $\cF$} if
for any $t\geq 0$, we have $\|D\bar P_{t}{|\cF(\bar u)}\| \leq C_\cF\lambda_\cF^{-t}$.
\smallskip

\noindent
$\bar u$ is  \emph{$(T_\cF,\lambda_\cF)$-Pliss (for $\cF$)} if
there exists an integer $s\in [0,T_\cF]$ such that for any $m\in \NN$,
$$ \prod_{n=0}^{m-1}\|D\bar P_{-\tau_0}{|\cF(\bar P_{-(n\tau_0+s)}(\bar u))}\| \leq \lambda_\cF^{-m\tau_0}.$$

\end{Definition}

One can define similarly hyperbolicity and Pliss property for pieces of generalized orbits.
By continuity and invariance of the spaces $\cF(\bar u)$, one gets

\begin{Lemma}\label{l.cont3}
For any $\lambda_\cF\in (1,\lambda)$ and $T_\cF\geq 0$, there exists $\eta>0$ such that
for any $y\in K$ which is $T_\cF$-Pliss and for 
any half generalized orbit $\bar u$ parametrized by $(-\infty,0]$, if
\begin{description}
\item \quad \quad $\bar u$ is in the $\eta$-neighborhood of $K$ and
its projection on $K$
is $(\varphi_{-t}(y))_{t\geq 0}$,
\end{description}
then $\bar u$ is $(T_\cF,\lambda_\cF)$-Pliss.
\end{Lemma}

\begin{Lemma}\label{l.cont4}
For any $C_\cF, \lambda_\cF>1$, there exists $\eta>0$ such that
for any piece of orbit $(\varphi_{-t}(y),y)$ which is $(C_\cF/2,\lambda_\cF^2)$-hyperbolic for $\cF$
and for  any half generalized orbit $\bar u$ parametrized by $(-\infty,0]$, if
\begin{description}
\item \quad \quad $\bar u$ is in the $\eta$-neighborhood of $K$ and
$u(-s)\in \cN_{\varphi_{-s}(y)}$ for each $s\in (0,t)$,
\end{description}
then $(\bar P_{-t}(\bar u),\bar u)$ is $(C_\cF,\lambda_\cF)$-hyperbolic for $\cF$.
\end{Lemma}
{The proofs of Lemmas~\ref{l.cont3} and~\ref{l.cont4} are standard by continuity, hence are omitted.}

\subsubsection{Unstable manifolds}\label{ss.unstable}
Pliss points have uniform unstable manifolds 
in the plaques
(see e.g.~\cite[Section 8.2]{abc-measure}).

\begin{Proposition}\label{p.unstable}
Consider $\eta>0$ and a center-unstable plaque family $\cW^{cu}$ as given by Theorem~\ref{t.generalized-plaques}.
For any $C_\cF,\lambda_\cF>1$, $\beta_\cF>0$ and $T_\cF\geq 0$, there exists $\alpha>0$ such that
for any half generalized orbit $\bar u$ parametrized by $(-\infty, 0]$, in the $\eta$-neighborhood of $K$,
if $\bar u$ is  $(T_\cF,\lambda_\cF)$-Pliss,
or if $\cF$ is one-dimensional and $\bar u$ is $(C_\cF,\lambda_\cF)$-hyperbolic for $\cF$,
then:
$$\forall t\geq 0,\quad \diam(\bar P_{-t}(\cW^{cu}_{\alpha}(\bar u)))\leq \beta_\cF\lambda_\cF^{-t/2}.$$
In particular, from Remark~\ref{r.plaque-invariance}, the image $\bar P_{-t}(\cW^{cu}_{\alpha}(\bar u))$ is contained in $\cW^{cu}(\bar P_{-t}(\bar u))$.
\end{Proposition}

\subsubsection{Lipschitz holonomy and rectangle distortion}\label{ss.rectangle}
It is well-known that for a one-codimensional invariant foliation
whose leaves are uniformly contracted, the holonomies between transversals are Lipschitz.
In order to state a similar property in our setting we define the notion of rectangle.

\begin{Definition}\label{d.rectangle}
A \emph{rectangle} $R\subset \cN_x$ is a subset which is homeomorphic to $[0,1]\times B_{d-1}(0,1)$
by a homeomorphism $\psi$,
where $B_{d-1}(0,1)$ is the $\dim(\cN_x)-1$-dimensional unit ball such that:
\begin{itemize}
\item[--] the set $\psi(\{0,1\}\times B_{d-1}(0,1))$ is a union of two $C^1$-discs tangent to $\cC^\cF$,
and is called the  \emph{$\cF$-boundary $\partial^{\cF}R$},
\item[--] the curve $\psi([0,1]\times \{0\})$ is $C^1$ and tangent to $\cC^\cE$.
\end{itemize}
A rectangle $R$ has \emph{distortion bounded by $\Delta>1$}
if for any two $C^1$-curves $\gamma,\gamma'\subset R$ tangent to $\cC^\cE$
with endpoints in the two connected components of $\partial^{\cF}R$, then
$$\Delta^{-1} |\gamma|\leq |\gamma'| \leq \Delta |\gamma|.$$
\end{Definition}

\begin{Proposition}\label{p.distortion}
Assume that the local fibered flow is $C^2$.
For any $C_\cF,\lambda_\cF>1$,
there exist $\Delta>0$ and $\beta>0$ with the following property.
For any $y,\varphi_{-t}(y)\in K$ and $R\subset \cN_y$ such that:
\begin{itemize}
\item[--] $(\varphi_{-t}(y),y)$ is $(C_\cF,\lambda_\cF)$-hyperbolic for $\cF$,
\item[--] $P_{-s}(R)$ is a rectangle and has diameter smaller than $\beta$ for each $s\in [0,t]$,
\item[--] if $D_1,D_2$ are the two components of $\partial^\cF(P_{-t}(R))$, then
$$d(D_1,D_2)>10.\max(\diam(D_1),\diam(D_2)),$$
\end{itemize}
then the rectangle $R$ has distortion bounded by $\Delta$.
\end{Proposition}
\begin{proof}
It is enough to prove the version of this result stated for a sequence of $C^2$-diffeomorphisms
with a dominated splitting. Then the argument is the same as~\cite[Lemma 3.4.1]{PS1}.
\end{proof}

\begin{Remark-numbered}\label{r.distortion}
In the previous statement, it is enough to replace the second condition by the weaker one:
\emph{$R$ is contained in $B(0_{y},\beta)$}.

Indeed, the proof considers a backward iterate $P_{-s}(R)$ such that
$d(P_{-s}(D_1),P_{-s}(D_2))$ is comparable to $\max(\diam(P_{-s}(D_1),\diam(P_{-s}(D_2))$.
The second condition ensures that the backward iterates of $R$ exist and remain small
until such a time.

If we know that $R\subset B(0_{y},\beta)$ for $\beta$ small, this can be verified as follow:
by the hyperbolicity for $\cF$, the diameter of the $\cF$-boundary
of $P_{-s}(R)$ decreases exponentially as $s$ increases; by the domination,
the ratio between the diameter of the $\cF$-boundary and the distance between the
two $\cF$-boundaries also increases exponentially. One deduces that the diameter of $P_{-s}(R)$ remains small until
the first time $s$ such that the $\cF$-boundary of $P_{-s}(R)$ becomes much smaller than
$d(P_{-s}(D_1),P_{-s}(D_2))$
\end{Remark-numbered}

\section{Topological hyperbolicity}\label{s.topological-hyperbolicity}

\noindent
{\bf Standing assumptions.}
In the whole section, $(\cN,P)$ is a $C^2$ local fibered flow over a topological flow $(K,\varphi)$
and $\pi$ is an identification compatible with $(P_t)$ on an open set $U$
such that:
\begin{enumerate}
\item[(A1)] there exists a dominated splitting $\cN=\cE\oplus \cF$ and the fibers of $\cE$ are one-dimensional,
\item[(A2)] $\cE$ is uniformly contracted on an open set $V$ containing $K\setminus U$,
\item[(A3)] $\cE$ is uniformly contracted over any periodic orbit $\cO\subset K$.
\end{enumerate}
From the last item and Proposition~\ref{p.2domination}, the bundle $\cE$ is $2$-dominated.
By Theorem~\ref{t.plaque}, one can fix a $C^2$-plaque family $\cW^{cs}$ tangent to $\cE$
and by Theorem~\ref{t.generalized-plaques}, there exists a $C^1$-plaque family $\cW^{cu}$ for
half generalized orbits parametrized by $(-\infty,0]$ that are in a small neighborhood of $K$.
Both are locally invariant by the time-one maps $P_1$ and $\bar P_{-1}$ respectively.
\medskip

The goal of this section is to prove the following theorem (see Subsection~\ref{ss.conclusion-topological}):

\begin{Theorem}\label{Thm:topologicalcontracting}
Under the assumptions above, one of the following properties occur:
\begin{enumerate}
\item[--] There exists a non empty proper invariant compact subset $K'\subset K$ such that $\cE|_{K'}$ is not uniformly
contracted.
\item[--] $K$ is a normally expanded irrational torus.
\item[--] $\cE$ is \emph{topologically contracted}: there is $\varepsilon_0>0$ such that the image $P_t({\cal W}^{cs}_{\varepsilon_0}(x))$
is well-defined for any $t\ge 0$, $x\in K$, and
$$\lim_{t\to+\infty}\sup_{x\in K} |(P_t({\cal W}^{cs}_{\varepsilon_0}(x))|=0.$$
\end{enumerate}
\end{Theorem}
\medskip

\noindent
{\bf Choice of constants.}
Let us remark that by assumption (A2), the bundle $\cE$ is also uniformly contracted on a neighborhood of ${\rm Closure}(V)$.
Hence, on the set $\bigcup_{s\in [0,\varepsilon]}\varphi_s(V)$, for some $\varepsilon>0$ small.
By Remark~\ref{r.identification}.(a), one can rescale the time so that $\varepsilon=1$ and assume:
\begin{enumerate}
\item[(A2')] $\cE$ is uniformly contracted on $\bigcup_{s\in [0,1]}\varphi_s(V)$, where $V$ is an open set containing $K\setminus U$.
\end{enumerate}

As introduced in Subsection~\ref{ss.assumptions} we denote by $\tau_0\in \NN$ and $\lambda>1$ the constants associated to the $2$-domination $\cE\oplus \cF$.
Proposition~\ref{l.summability} associates to the set $W:=V$
the constants $T_\cF\geq 0$ defining Pliss points for $\cF$ and $C_{\cE},\lambda_{\cE}$
defining the hyperbolicity for $\cE$. We also choose arbitrarily $\lambda_{\cF}\in (1,\lambda)$.
Sections~\ref{s.plaque}
gives some constants $\alpha_\cF$ controlling size of the unstable manifold at $(T_\cF,\lambda_{\cF})$-Pliss points of $K$,
and also at half generalized orbits parametrized by $(-\infty,0]$ in the $\eta$-neighborhood of $K$, provided $\eta$ is small enough.
There exists $C_{\cF}>0$ such that any $(T_\cF,\lambda_{\cF})$-Pliss generalized orbit $\bar u$ is also $(C_{\cF},\lambda_{\cF})$-hyperbolic for $\cF$.
Proposition~\ref{p.unstable} gives $\alpha>0$ such that for any generalized orbit in the $\eta$-neighborhood of $K$ which is
$(C_{\cF},\lambda_{\cF})$-hyperbolic for $\cF$, the backward iterates $P_{-t}(\cW_\alpha^{cu})$ have diameter smaller than
$\alpha_\cF \lambda_{\cF}^{-t/2}$.

We also consider small constants $\beta_\cF,r,\delta_0,\alpha'>0$ which will be chosen in this order during this section:
they control distances inside the spaces $\cN_x$, $K$, or $\cW^{cs}_x$.

\subsection{Topological stability and $\delta$-intervals}

\subsubsection{Dynamics of $\delta$-intervals}
We introduce a crucial notion for this section.

\begin{Definition}\label{Def:deltainterval} Consider $\delta\in (0,\delta_0]$.
A curve $I\subset {\cal W}^{cs}_x$ (not reduced to a single point),  for $x\in K$,
is called a \emph{$\delta$-interval}
if $0_x\in I$ and for any $t\ge 0$, one has
$$|P_{-t}(I)|\le\delta.$$
\end{Definition}

One example of $\delta$-interval is given by a periodic point $z\in K$
together with a non-trivial interval in $\cW^{cs}(z)$ that is periodic for $(P_t)$
and contains $0_z$.

\begin{Definition}
A $\delta$-interval $I$ is \emph{periodic}
if it coincides with $P_{-T}(I)$ for some $T>0$.

\noindent We say that a $\delta$-interval $I$ at $x$
is \emph{contained in the unstable set of some
periodic $\delta$-interval} if:
\begin{enumerate}
\item the $\alpha$-limit set $\alpha(x)\subset K$ of $x$ is the orbit of a periodic point $y$,
\item $y$ admits a periodic $\delta$-interval $\widehat I_y$,
\item $P_{-t}(I)$ accumulates as $t\to +\infty$ on the orbit of a (maybe trivial) interval $I_y\subset \widehat I_y$.
\end{enumerate}
\end{Definition}

The next property will be proved in Section~\ref{ss.Lyapunov}.
\begin{Lemma}\label{l.periodic}
{There is $\delta_0>0$ such that for any $\delta\in(0,\delta_0]$,}
for any periodic $\delta$-interval $I\subset \cN_q$, there exists $\chi>0$ with the following property.

Let $z$ be close to $q$, let $L\subset \cN_z$ be an arc which is close to $I$ in the Hausdorff topology
and contains $0_z$ and let $T>0$ such that $|P_{-t}(L)|\leq \delta$ for any $t\in [0,T]$.
Then $|P_{-T}(L)|>\chi$.
\end{Lemma}
Proposition~\ref{Prop:dynamicsofinterval} describes dynamics of $\delta$-intervals. It is an analogue to ~\cite[Theorem 3.2]{PS2}.

\begin{Proposition}\label{Prop:dynamicsofinterval}
{There is $\delta_0>0$ such that if there is a $\delta$-interval $I\subset \cW^{cs}_x$ for $\delta\in(0,\delta_0]$}
then
\begin{itemize}
\item[--] either $K$ contains a normally expanded irrational torus,
\item[--] or $I$ is contained in the unstable set of some periodic $\delta$-interval.
\end{itemize}
\end{Proposition}

\begin{Remark}
In the first case one can even show that $\alpha(x)$ is a normally expanded irrational torus.
We will not use it.
\end{Remark}
\smallskip

\noindent
{\it Strategy of the proof of Proposition~\ref{Prop:dynamicsofinterval}.}
The next five subsections are devoted to the proof:
\begin{itemize}
\item[--] One introduces a \emph{limit} $\delta$-interval $I_\infty$
from the backward orbit of $I$ (Section~\ref{ss.limit}).
\item[--] $I_\infty$ has returns close to itself
(Section~\ref{ss.return-I-infty}). Under some ``non-shifting" condition one gets a periodic $\delta$-interval
(Section~\ref{ss.criterion-periodic}) and the last case of the proposition holds.
\item[--] If the ``non-shifting" condition does not hold, there exists a normally expanded irrational torus
which attracts $x$, $I$ and $I_\infty$ by backward iterations
(Sections~\ref{ss.aperiodic} and~\ref{ss.topological}).
\end{itemize}
The conclusion of the proof is given in Section~\ref{ss.Lyapunov}.

\subsubsection{Topological stability}\label{sss.lyapunov-stable}
Before proving Lemma~\ref{l.periodic} and Proposition~\ref{Prop:dynamicsofinterval},
we derive a consequence.

\begin{Proposition}\label{Pro:lyapunovstablity}
If there is no normally expanded irrational torus,
then $\cE$ is \emph{topologically stable}:
there is $\varepsilon_0>0$ and for any $\varepsilon_1\in(0,\varepsilon_0)$, there is $\varepsilon_2>0$ such that 
$$\forall x\in K~\textrm{and}~\forall t>0,~~~P_t({\cal W}^{cs}_{\varepsilon_2}(x))\subset {\cal W}^{cs}_{\varepsilon_1}(\varphi_t(x)).$$
\end{Proposition}
\begin{proof}
By Remark~\ref{r.plaque-invariance} it is enough to check that 
$|P_t({\cal W}^{cs}_{\varepsilon_2}(x))|$ is bounded by $\varepsilon_1$.
{One can choose $\delta_0$ small so that Proposition~\ref{Prop:dynamicsofinterval} holds.} We argue by contradiction. If the topological stability does not hold, then there exist {$\delta\in(0,\delta_0]$}
a sequence $(x_n)$ in $K$
and, for each $n$, an interval $I_n\subset {\cal W}^{cs}(x_n)$ containing $0$
and a time $T_n>0$ such that:
\begin{itemize}
\item[--] $|I_n|\to 0$ as $n\to +\infty$.
\item[--] $|P_{T_n}(I_n)|=\delta$
and $|P_{t}(I_n)|<\delta$ for all $0<t<T_n$.
\end{itemize}
Taking a subsequence, one can assume that $(\varphi_{T_n}(x_n))$
converges to a point $x\in K$ and $(P_{T_n}(I_n))$
to an interval $I$. We have $|I|=\delta$ and
$|P_{-t}(I)|\le\delta$ for all $t>0$, so that
$I$ is an $\delta$-interval.

The second case of the proposition is satisfied by
$I$ and $L_n=P_{T_n}(I_n)$, $n$ large.
Lemma~\ref{l.periodic} implies that $P_{-T_n}(L_n)$
has length uniformly bounded away from $0$.
This contradicts the fact that the length of $I_n=P_{-T_n}(L_n)$
goes to $0$ when $n\to \infty$. 
\end{proof}

\subsection{Limit $\delta$-interval $I_\infty$}\label{ss.limit}
One can obtain infinitely many $\delta$-intervals
with length uniformly bounded away from zero at points of the backward orbit
of a $\delta$-interval.
The goal of this section is to prove this property.

\begin{Proposition}\label{p.limit}
If $\delta>0$ is small enough,
for any $x\in K$ and any $\delta$-interval $I$,
there exists an increasing sequence $(n_k)$ in $\NN$
and $\delta$-intervals $\widehat I_k$ at $\varphi_{-n_k}(x)$ such that:
\begin{itemize}
\item[--] $P_{-n_k}(I)$ and $P_{n_\ell-n_k}(\widehat I_{\ell})$
are contained in $\widehat I_{k}$ for any $\ell\leq k$,
\item[--] $\varphi_{-n_k}(x)$ is $T_\cF$-Pliss
and belongs to $K\setminus V$,  for any $k\geq 0$,
\item[--] $(\widehat I_k)$ converges to some $\delta$-interval $I_\infty$ at some point $x_\infty\in K\setminus V$.
\end{itemize}
\end{Proposition}

\subsubsection{Existence of hyperbolic returns}
\begin{Lemma}\label{Lem:hyperbolicreturns}
If $\delta>0$ is small enough,
any point $x\in K$ which admits a $\delta$-interval
has infinitely many backward iterates $\varphi_{-n}(x)$, $n\in \NN$, in $K\setminus V$ that are $T_\cF$-Pliss.
\end{Lemma}

\begin{proof}
Since the backward iterates $P_{-n}(I)$, $n\in \NN$, of a $\delta$-interval $I$
are still $\delta$-intervals, it is enough to show that
any point $x$ has at least one backward iterate by $\varphi_{1}$ in $K\setminus V$ that is $T_\cF$-Pliss. The proof is done by contradiction.

{Let $\chi>0$ such that
$1+\chi<\min(\lambda,\lambda_{\cE}).$}
If $\delta_1>0$ is small enough, then for any $\delta$-interval $I$ with $\delta<\delta_1$,
at any point $x$, we have
$$\|DP_{-t}{|\cE(x)}\|\leq (1+\chi)^t\;\frac{|P_{-t}(I)|}{|I|}\leq\frac{(1+\chi)^t\;\delta}{|I|},~~~\forall t\ge0.$$
With the domination estimate~\eqref{e.domination} of Section~\ref{ss.assumptions}, one gets for any $k\geq 0$:
$$\prod_{j=0}^{k-1}\|DP_{-\tau_0}{|\cF(\varphi_{-j\tau_0}(x))}\| \leq \frac{(1+\chi)^{k\tau_0}\;\lambda^{-2k\tau_0}\;\delta}{|I|}.$$
Using $(1+\chi)<\lambda$ and
Pliss lemma (see~\cite[Lemma 11.8]{Man87}), there exists an arbitrarily large integer $i$ such that for any $k\geq 0$ one has:
$$\prod_{j=0}^{k-1}\|DP_{-\tau_0}{|\cF({\varphi_{-(j+i)\tau_0}(x)})}\| \leq \lambda^{-k\tau_0}.$$
This proves that $x$ has arbitrarily large $0$-Pliss
backward iterates by $\varphi_{\tau_0}$.

Let us fix any of these $0$-Pliss backward iterates $\varphi_{-k\tau_0}(x)$.
By contradiction, we assume that there is no iterate $\varphi_{-n}(x)$ in $K\setminus V$ that is
$T_\cF$-Pliss
with $1\leq n\leq k\tau_0$. We can thus apply Proposition~\ref{l.summability} with $W=V$; for any such $n$, one gets
$$\|DP_{n}{|\cE(\varphi_{-k\tau_0}(x))}\|\leq C_{\cE}\lambda_{\cE}^{-n}.$$
As before, this implies that
$$|I|\leq |P_{-k\tau_0}(I)|C_{\cE}\lambda_{\cE}^{-k\tau_0}(1+\chi)^{k\tau_0}\leq \delta\; C_{\cE}\; (1+\chi)^{k\tau_0}\; \lambda_{\cE}^{-k\tau_0}.$$
Since $(1+\chi)/\lambda_{\cE}<1$ and $k$ is arbitrarily large, one gets $|I|=0$ which is a contradiction.
\end{proof}

\subsubsection{Rectangles associated to $\delta$-intervals of Pliss points}

When $\delta>0$ is smaller than $\eta$, any point $u$ in a $\delta$ interval $I$ has a well defined backward orbit under $(P_t)$,
which satisfies the definition of generalized orbit in the $\eta$-neighborhood of $K$.
Consequently, it has a space $\cF(u)$ and a center-unstable plaque $\cW^{cu}(u)$.
Moreover by Lemma~\ref{l.cont3}, if the point $x\in K$ such that $u\in \cN_x$ is $T_\cF$-Pliss,
then $u$ is $(T_\cF,\lambda_{\cF})$-Pliss.

We will need to build a rectangle $R(I)$ foliated by unstable plaques for each $\delta$-interval $I$
above a Pliss point $x\in K$. As before $d$ denotes the dimension of the fibers $\cN_x$.

\begin{Proposition}\label{p.rectangle}
Fix $\beta_\cF\in (0,\beta_0/4)$.
There exist $\alpha_{min},\delta_0,C_{R}>0$  such that
for any $\delta\in (0,\delta_0]$, any $T_\cF$-Pliss point $x\in K\setminus V$ and any $\delta$-interval $I$ at $x$,
we associate a rectangle $R(I)$ which is the image of $[0,1]\times B_{d-1}(0,1)$
by a homeomorphism $\psi$ such that:
\begin{enumerate}
\item $\psi\colon [0,1]\times \{0\}\to \cN_x$ is a $C^1$-parametrization of $I$,
\item each $u\in R(I)$ belongs to a (unique) leaf
$\psi(\{z\}\times B_{d-1}(0,1))$ denoted by $W^u_{R(I)}(u)$;
it is contained in $\cW^{cu}(u)$
(and $\cW^{cu}(\psi(z,0))$), and it contains a disc with radius $\alpha_{min}$,

\item ${\rm Volume}(R(I))\geq C_{R}.|I|$,
\item for any $t>0$, ${\rm Diam}(P_{-t}(W^u_{R(I)}(u)))<\beta_\cF\lambda_{\cF}^{-t/2}$; hence $P_{-t}(R(I))\subset B(0,2\beta_\cF)\subset \cN_{\varphi_{-t}(x)}$.
\end{enumerate}
Moreover, if $x,x'$ are $T_\cF$-Pliss points with $\delta$-intervals $I,I'$ and if $t,t'>0$
satisfy:
\begin{itemize}
\item[--] $\varphi_{-t}(x),\varphi_{-t'}(x')$ belong to $K\setminus V$ and are $r_0$-close,
\item[--] $P_{-t}(R(I))$ and the projection of $P_{-t'}(R(I'))$ by $\pi_{\varphi_{-t}(x)}$ intersect,
\end{itemize}
then, the foliations of $P_{-t}(R(I))$ and $\pi_{\varphi_{-t}(x)}\circ P_{-t'}(R(I'))$
coincide:
if $P_{-t}(W^u_{R(I)}(u))$ and $\pi_{\varphi_{-t}(x)}P_{-t'}(W^u_{R(I')}(u'))$ intersect, they are contained
in a same $C^1$-disc tangent to $\cC^\cF$.
\end{Proposition}
\begin{Remark-numbered}\label{r.intersection}
Assume {that} $\beta_\cF$ and $r>0$ {are} small enough.
By the Global invariance,
under the assumptions of the last part of Proposition~\ref{p.rectangle},
and assuming $d(\varphi_{-t}(x),\varphi_{-t'}(x'))<r$,
there exists $\theta\in \lip$ such that $\theta(0)=0$ and 
\begin{itemize}
\item[--] the distance $d(\varphi_{-\theta(s)-t'}(x'),\varphi_{-s-t}(x))$ for $s>0$
remains bounded (and arbitrarily small if $\beta_\cF>0$ and
$r$ have been chosen small enough),
\item[--]  $P_{-s}\circ \pi_{\varphi_{-t}(x)}=\pi_{\varphi_{-s-t}(x)}\circ P_{-\theta(s)}$ on $P_{-t'}(R(I'))$
when $\varphi_{-s-t}(x)$ and $\varphi_{-\theta(s)-t'}(x')$ are $r_0$-close and belong to $U$.
\end{itemize}
\end{Remark-numbered}

\begin{proof}[Proof of Proposition~\ref{p.rectangle}]
By Proposition~\ref{p.distortion}, one associates to $C_{\cF},\lambda_{\cF}$
some constants $\Delta,\beta$. One can reduce $\beta_\cF$ to ensure $2\beta_\cF<\beta$.

By Proposition~\ref{p.unstable} one associates to $C_{\cF},\lambda_{\cF},T_\cF$ and $\beta_\cF$,
some quantities $\alpha',\eta'$. We will assume that $\delta$ is smaller than $\eta'$
so that for any $u$ in the $\delta$-interval $I$ and any $t\geq 0$ we have
an exponential contraction of the unstable plaque of size $\alpha'$:
$$\forall t\geq 0,\quad \diam(\bar P_{-t}(\cW^{cu}_{\alpha'}(\bar u)))\leq \beta_\cF\lambda_{\cF}^{-t/2}.$$

One parametrizes the curve $I$ by $[0,1]$.
From the plaque-family Theorem~\ref{t.generalized-plaques}, there exists a continuous
map $\psi\colon [0,1]\times \RR^{d-1}\to \cN_x$
such that $\psi([0,1]\times \{0\})=I$.
Up to rescale $\RR^{d-1}$, the image $\psi(\{s\}\times B_{d-1}(0,1))$ is small,
hence is contained in $\cW^{cu}_{\alpha'}(\psi(s,0))$ for each $s\in [0,1]$.

Note that two plaques $\psi(\{s\}\times B_{d-1}(0,1)),\psi(\{s'\}\times B_{d-1}(0,1))$, for $s\neq s'$,
can not be contained in a same $C^1$-disc tangent to $\cC^\cF$ since they intersect
the transverse curve $I$ at two different points.
By coherence, they are disjoint: indeed since the backward orbit of
$x$ has arbitrarily large backward iterates in $K\setminus V$
(see Lemma~\ref{Lem:hyperbolicreturns}), Proposition~\ref{p.coherence} applies.

Hence $\psi$
is injective on $[0,1]\times B_{d-1}(0,1)$ and is a homeomorphism on its image $R(I)$ by the invariance of domain theorem.
In particular $R(I)$ satisfies Definition~\ref{d.rectangle} and is a rectangle.
The coherence again implies that for any $u\in \psi(\{s\}\times B_{d-1}(0,1))$,
the plaque $\cW^{cu}(u)$ contains $\psi(\{s\}\times B_{d-1}(0,1))$.
By compactness, there exists $\alpha_{min}$ (which does not depend on $x$ and $I$) such that
$\psi(\{s\}\times B_{d-1}(0,1))$ contains $\cW^{cu}_{\alpha_{min}}(\psi(s,0))$ for any $s\in [0,1]$.

The rectangle $R(I)$ has distortion bounded by $\Delta$ (from Proposition~\ref{p.distortion}),
hence one bounds ${\rm Volume}(R(I))$ from below by using Fubini's theorem and integrating along curves
tangent to $\cC^\cE$. 
This gives the four items of the lemma.

The last part is a direct consequence of the coherence (Proposition~\ref{p.coherence}).
\end{proof}

\subsubsection{Maximal $\delta$-intervals}
In the setting of Proposition~\ref{p.limit},
let us consider all the integers $n_0=0< n_1<n_2<\cdots<n_k<\cdots$ such that
the backward iterate $x_k:=\varphi_{-n_k}(x)$ belongs to $K\setminus V$
and is $T_\cF$-Pliss.
We introduce inductively some maximal $\delta$-intervals $\widehat{I}_{k}$ at these iterates
such that $I\subset \widehat I_0$ and $P_{n_k-n_{k+1}}(\widehat I_k)\subset \widehat I_{k+1}$.
We denote $R_k=R(\widehat{I}_{k})$.

Lemma~\ref{l.summability-hyperbolicity} associates to $C_{\cE},\lambda_{\cE}$ some constants
$C'_{\cE},\delta_E$. Let us assume $\delta<\delta_E$.
From the definition of the sequence $(n_k)$, Proposition~\ref{l.summability}
implies that $(\varphi_{-n_{k+1}}(x),\varphi_{-n_{k}}(x))$ is $(C_{\cE},\lambda_{\cE})$-hyperbolic for $\cE$ and for each $k$.
Lemma~\ref{l.summability-hyperbolicity} then gives
\begin{equation}\label{e.bounded}
\sum_{m=n_k}^{n_{k+1}}|P_{n_k-m}(\widehat I_{k})|
\leq C'_{\cE}\; |\widehat I_{k+1}|.
\end{equation}

\subsubsection{Non-disjointness}

\begin{Lemma}[Existence of intersections]\label{Sub:disjointcase}
If $\delta_0>0$ is small enough, {then} for any $\delta\in (0,\delta_0]$, for any $r>0$
there exist $k<\ell$ arbitrarily large such that
$d(x_k,x_\ell)<r$ and the interior of $\pi_{x_k}({R}_\ell)$ and the interior of $R_k$ intersect.
\end{Lemma}
\begin{proof}
Since $(\pi_{x,y})$ is a continuous family of diffeomorphisms on the open set
$U$ containing $K\setminus V$, the following property holds provided $r$ is small enough:

For any $x,y\in K\setminus V$ with $d(x,y)<r$, the projection $\pi_{y}\colon \cN_x\to\cN_y$ satisfies:
$$\pi_{y}(B(0,\beta_0/2))\subset B(0,\beta_0),$$
$${\rm det}(D\pi_{x,y})(u)\leq 2 \text{ for any } u\in B(0,\beta_0/2).$$
By compactness, one can find a finite set $Z\subset U$
such that any $x\in K\setminus V$ satisfies $d(x,z)<r/2$ for some $z\in Z$.
For each point $x_k$ we associate some $z_k\in Z$ such that
$d(x_k,z_k)<r/2$. Since $R_k\subset B(0,2\beta_\cF)\subset \cN_{x_k}$ with $\beta_\cF<\beta_0/4$ and from
Proposition~\ref{p.rectangle}, we have
$$\pi_{z_{k}}(R_k)\subset B(0,\beta_0)\subset \cN_{z_k},$$
$${\rm Volume}(\Interior(\pi_{z_{k}}(R_k)))\geq \frac 1 2 {\rm Volume}(\Interior(R_k)) \geq \frac {C_R} {2} |\widehat I_k|.$$

Let us assume by contradiction that the statement of the lemma does not hold.
One deduces that there exists $s\geq 0$ such that for any $z\in Z$ and any $k,\ell\geq s$
such that $z_k=z_\ell=z$ we have
$$\pi_{z}(\Interior(R_k))\cap \pi_{z}(\Interior( R_\ell))=\emptyset.$$
In particular if $C_{Vol}$ denotes the supremum of $\text{Volume}(B(0_x,\beta_0))$ over $x\in K$,
$$\sum_{k=1}^\infty|\widehat{I}_k|\le 2C_R^{-1}\;C_{Vol}\text{Card}(Z).$$

With~\eqref{e.bounded} we get for any $k$:
\begin{equation}\label{e.sum}
\sum_{m=0}^{+\infty}|P_{-m}(\widehat I_k)|\leq
C'_{\cE}\sum_{\ell=k}^\infty|\widehat{I}_\ell|\leq C_{Sum}:=2C_R^{-1}\;C_{Vol}\text{Card}(Z)\; C'_{\cE}.
\end{equation}
By Denjoy-Schwartz Lemma~\ref{Lem:schwartz} one gets $\eta_S>0$ and for $k$ large one can
introduce an interval $J\subset \cW^{cs}(x_k)$ containing $\widehat I_k$ and of length
equal to $(1+\eta_S)|\widehat I_k|$. One gets
$$|P_{-m}(J)|\leq 2|P_{-m}(\widehat I_k)|,~~~\forall m\ge 0.$$
From~\eqref{e.sum}, $\sup_{m\geq 0} |P_{-m}(\widehat I_k)|$ is arbitrarily small for $k$ large,
hence $|P_{-t}(J)|$ is smaller than $\delta$ for any $t>0$. This proves
that $J$ is a $\delta$-interval, contradicting the maximality of $\widehat I_k$.
\end{proof}

\subsubsection{Non-shrinking property}
\begin{Lemma}
If $\delta_0$ is small enough, 
the length $|\widehat I_k|$ does not go to zero as $k\to \infty$.
\end{Lemma}
\begin{proof}
We first introduce an integer $N\geq 1$ such that $\beta_\cF\lambda_{\cF}^{N/2}$ is much smaller
than $\alpha_{min}$.

We argue by contradiction. Assume that $|\widehat I_k|$ is arbitrarily small as $k$ is large.
From~\eqref{e.bounded} we deduce that for any
$\delta'\in (0,\delta)$, if $k$ is large enough, then $\widehat I_k$
is a $\delta'$-interval.

By Lemma~\ref{Sub:disjointcase},
there exist $k\neq \ell$ large such that $x_k,x_\ell$ are arbitrarily close
and we have that $\pi_{x_k}(\Interior({R}_\ell))\cap \operatorname{Interior}({R}_k)\neq\emptyset$.
By Remark~\ref{r.intersection}, the unstable foliations of
$\pi_{x_k}({R}_\ell)$ and ${R}_k$ coincide on the intersection,
hence one of the following cases occurs (see Figure~\ref{f.shrinking}).
\begin{enumerate}

\item There exists an endpoint $u$ of $\widehat I_k$
such that $W^{u}_{R_k}(u)$ intersects
$\pi_{x_k}(\widehat I_\ell)$ at a point which is not endpoint.

\item The endpoints of $\widehat I_k$ and $\pi_{x_k}(\widehat I_\ell)$
have the same unstable manifolds.

\end{enumerate}

\begin{figure}[ht]
\begin{center}
\includegraphics[width=13cm]{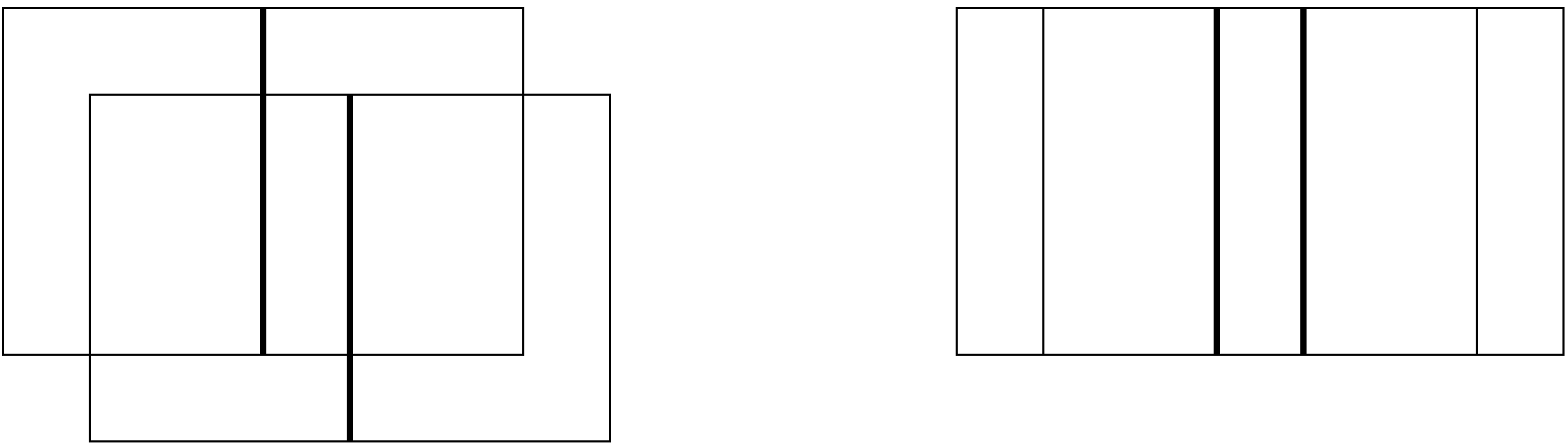}
\begin{picture}(0,0)
\put(-320,110){\small$\widehat I_k$}
\put(-317,12){\small$u$}
\put(-100,110){\small$\widehat I_k$}
\put(-75,110){\small$\pi_{x_k}(\widehat I_\ell)$}
\put(-300,88){\small$\pi_{x_k}(\widehat I_\ell)$}
\end{picture}
\end{center}
\caption{\label{f.shrinking}}
\end{figure}

In the first case, by Remark~\ref{r.intersection}, there exists
a homeomorphism $\theta$ of $[0,+\infty)$ such that
$\varphi_{-t}(x_k)$ and $\varphi_{-\theta(t)}(x_\ell)$ are close for any $t>0$,
and $P_{-t}(\pi_{x_k}(R( I_\ell)))$ remains in a neighborhood
of $\varphi_{-t}(x_k)$ which is arbitrarily small
if $\delta'$ and $d(x_k,x_\ell)$ are small enough.
The rectangle $\pi_{x_k}(R(\widehat I_\ell))$
intersects $\cW^{cs}(x_k)$ along an interval $J$,
which meets $\widehat I_k$.
This proves that the union of
$J$ with $\widehat I_k$ is a $\delta$-interval
and contradicts the maximality since $J$ is not contained
in $\widehat I_k$ in this first case.

In the second case, without loss of generality, we assume $n_\ell>n_k$ and set
$T:=n_\ell-n_k$.
We introduce the map $\widetilde P_{-T}:=\pi_{x_k}\circ P_{n_k-n_\ell}$.
Since $r$ has been chosen small enough and since the endpoints of $\widehat I_k$ and $\pi_{x_k}(\widehat I_\ell)$
have the same unstable manifolds, the iterates $\widetilde P_{-T}^i(\widehat I_k)$, $i\geq 0$,
are all contained in $R_k$. 
Hence by the Global invariance, there exists
a sequence of times $0<t_1<t_2<\dots$ going to $+\infty$
such that
\begin{itemize}
\item[--] $\varphi_{-t_1}(x_k)=x_\ell$,
\item[--] $(\varphi_{-t}(x_k))_{t_i\leq t\leq t_{i+1}}$ shadows $(\varphi_{-t}(x_k))_{0\leq t\leq t_{1}}$,
\item[--] $\varphi_{-t_i}(x_k)$ is close to $x_k$
and projects by $\pi_{x_k}$ in $R_k$.
\end{itemize}
Note that the differences $t_{i+1}-t_i$ are uniformly bounded in $i$ by some constant
$T_0$. Since $\delta'$ can be chosen arbitrarily small (provided $k$ is large),
for $t=t_i$ arbitrarily large,
the interval $J=\pi_{\varphi_{-t_i}(x_k)}(R_k)\cap \cW^{cs}(\varphi_{-t_i}(x_k))$
contains $0$ and is a $\delta/2$-interval.
One can choose $t_i$ and a backward iterate $x_j$
such that $t_i\leq n_j\leq t_i+T_0$.
Since $P_{t_i-n_j}(J)$ is a $\delta/2$-interval and
$\widehat I_j$ is a $\delta'$-interval,
$\widehat I_j\cup P_{t_i-n_j}(J)$ is a $\delta$-interval.
As $j$ can be chosen arbitrarily large, $|\widehat I_j|$
is arbitrarily small, whereas $|P_{t_i-n_j}(J)|$ is uniformly bounded away from zero
(since $n_j-t_i$ is bounded). Consequently $\widehat I_j\cup P_{t_i-n_j}(J)$
is strictly larger than $\widehat I_j$, contradicting the maximality.
\end{proof}

\subsubsection{Existence of limit intervals}
The Proposition~\ref{p.limit} now follows easily from the previous lemmas,
up to extract a subsequence from the sequence of hyperbolic times $(n_k)$.
\qed

\subsection{Returns of $\delta$-intervals}\label{ss.return-delta}

\subsubsection{Definition of returns and of shifting returns}

We now introduce the times which will allow to induce the dynamics near a $\delta$-interval.
\begin{Definition}\label{d.return}
Let $x\in K\setminus V$ be a $T_\cF$-Pliss point
and $I$ be a $\delta$-interval at $x$.
A time $t>0$ is a \emph{return} of $I$ if
\begin{itemize}
\item[--] $\varphi_{-t}(x)$ and $x$ are $r_0$-close, and $\Interior(R(I))\cap \Interior(\pi_x\circ P_{-t}(R(I)))\neq \emptyset$,
\item[--] for any $z,z'\in I$ such that
$\pi_x\circ P_{-t}(W^{u}_{R(I)}(z))\cap W^{u}_{R(I)}(z')\neq \emptyset$, we have
$$\pi_x\circ P_{-t}(W^{u}_{R(I)}(z))\subset W^{u}_{R(I)}(z').$$
\end{itemize}
We then denote by $\widetilde P_{-t}$ the map $\pi_x\circ P_{-t}\colon R(I)\to \cN_x$.
\smallskip

\noindent
A sequence of returns $(t_n)$ is \emph{deep} if
$t_n\to +\infty$ and if one can find one sequence $(x_n)$ in $K$ with
$\pi_x(x_n)\in R(I)$ such that
{ $\widetilde P_{-t_n}\circ \pi_x(x_n)\to 0_x$ as $n\to +\infty$}.

\end{Definition}

\begin{Remark}\label{r.deep}
When $(t_n)$ is a sequence of deep returns,
$\widetilde P_{-t_n}\circ \pi_x(R(I))$ gets arbitrarily close to $\cW^{cs}(x)$ as $n\to +\infty$.
\emph{ Indeed, since $t$ is large, $\widetilde P_{-t_n}\circ \pi_x(R(I))$ is thin and contained in a small neighborhood of
$\pi_x(\cW^{cs}(\varphi_{-t'_n}(x_n)))$, where $\widetilde P_{-t_n}\circ \pi_x(x_n)=\pi_x(\varphi_{-t'_n}(x_n))$.
Moreover as $\widetilde P_{-t_n}\circ \pi_x(x_n)\to 0_x$, the plaque $\pi_x(\cW^{cs}(\varphi_{-t'_n}(x_n)))$
gets close to $\cW^{cs}(x)$.}
\end{Remark}

\begin{Lemma}\label{l.return-existence}
If $\delta_0$ is small enough, 
a $T_\cF$-Pliss point $x\in K\setminus V$ with a $\delta$-interval $I$ {for $\delta\in (0,\delta_0]$},
some $t>2\log(\beta_\cF/\delta_0)/\log(\lambda_{\cF})$ satisfying $d(x,\varphi_{-t}(x))<r_0$ and  $z\in I$ satisfying $\pi_{x}\circ P_{-t}(z)\in \Interior(R(I))$
and $d(\pi_{x}\circ P_{-t}(z),I)<\delta_0$, {then  the time $t$ is a return.} 
\end{Lemma}
\begin{proof}
The leaves $W^u_{R(I)}(z)$ are tangent to $\cC^\cF$ and have uniform size $\alpha_{min}$.
The interval $I$ is tangent to $\cC^\cE$ and has size less than $\delta_0$, chosen much smaller than $\alpha_{min}$.
Assuming that $t>0$ is large enough, the images $P_{-t}(W^u_{R(I)}(z))$ have diameter smaller than $\beta_\cF\lambda_{\cF}^{t/2}<\delta_0$,
so $P_{-t}(R(I))$ has diameter smaller than $2\delta_0$. If $d(\pi_{x}\circ P_{-t}(z),I)$ is smaller than $\delta_0$,
the images $P_{-t}(W^u_{R(I)}(z))$ can not intersect the boundary of the leaves $W^u_{R(I)}(z')$.
From the last part of Proposition~\ref{p.rectangle},
we get that $P_{-t}(W^u_{R(I)}(z))$ is disjoint or contained in $W^u_{R(I)}(z')$ for each $z,z'\in I$.
\end{proof}
\medskip

To each return, one associates a one-dimensional map $S_{-t}:I\to \cW^{cs}_x$ as follows.
\begin{Proposition}\label{p.one-dimensional}
Assume that $\beta_\cF,r,\delta_0$ are small enough and that $\delta$ is much smaller than $\delta_0$.
Consider a return $t$ of a $\delta$-interval $I\subset \cN_x$ and $d(\varphi_{-t}(x),x)<r$.

Then there exists a $\delta_0$-interval $J\subset \cW^{cs}(x)$ containing $I$
and a continuous injective map $S_{-t}\colon I\to J$ such that
for each $u\in I$, the point $\widetilde P_{-t}(u)$ belongs to $W^{u}_{R(J)}(S_{-t}(u))$.
\end{Proposition}
Notice that from Definition~\ref{d.return}, $S_{-t}(I)$ intersects $I$ along a non-trivial interval.
\begin{proof}
Assuming $\beta_\cF,r$ small enough,
By Remark~\ref{r.intersection},
the backward orbits
of $\varphi_{-t}(x)$ and $x$ stay at an arbitrarily small distance.
The $\delta$-interval $P_{-t}(I)$ projects by $\pi_x$
on a set $X$ whose backward iterates by $P_{-s}$ are contained in
$B(0_{\varphi_{-s}(x)},\delta_0/2)$ (using the Global invariance).
Since $x$ is a $T_\cF$-Pliss point, any $u\in X$ is $(T_\cF,\lambda_{\cF})$-Pliss (see Lemma~\ref{l.cont3})
and has an unstable plaque $\cW^{cu}_{\alpha}(u)$
whose backward iterates under $P_{-t}$ have diameter smaller than $\alpha_\cF\lambda_{\cF}^{-t/2}$.
One can thus project {the arc} $X$ along the plaques
$\cW^{cu}_\alpha(u)$ of points $u\in X$ to a connected set $I'\subset \cW^{cs}(x)$ which intersects $I$.
If $\delta_0$ is small enough,
and $P_{-s}(I\cup I')$ has diameter smaller than $2\diam(P_{-s}(R(I)\cup \widetilde P_{-t}(I)))<\delta_0$.
Thus $J=I\cup I'$ is a $\delta_0$-interval.

By the coherence (Proposition~\ref{p.rectangle}, item 2),
the plaques $\cW^{cu}_\alpha(u)$ of 
$u\in X$ intersect $R(J)$ along a leaf $W^u_{R(J)}(u)$.
Note that each plaque intersect $I'\subset J$ by construction.
Moreover since $d(x, \varphi_{-t}(x))$ is small, 
$X$ does not intersect the boundary of the
leaves $W^u_{R(J)}(u)$. Thus $X\subset R(J)$.

Any point in $X=\widetilde P_{-t}(I)$ can thus be projected to $J$ along the leaves of $R(J)$.
The map $S_{-t}$ is the composition of this projection with $\widetilde P_{-t}$.
\end{proof}

Deepness has been introduced for the following statement.
\begin{Lemma}\label{l.return}
Let $(t_n)$ be a deep sequence of returns of a $\delta$-interval $I$
and let $J$ be a $\delta_0$-interval containing $I$ such that
$S_{-t_n}(I)\subset J$ for each $n$.
Then, there exists $n_0\geq 1$ with the following property.
If $n(1),\dots,n(\ell)$ is a sequence of integers with $n(i)\geq n_0$
and if there exists $u$ in the interior of $J$ satisfying for each $1\leq i\leq \ell$
$$S_{-t_{n(i)}}\circ\dots\circ S_{-t_{n(0)}}(u)\in \operatorname{Interior}(J),$$
then there exists a return $t>0$ such that
$S_{-t}=S_{-t_{n(\ell)}}\circ\dots\circ S_{-t_{n(0)}}$.
\end{Lemma}
The return $t$ will be called \emph{composition}
of the returns $t_{n(0)},\dots,t_{n(\ell)}$.
\begin{proof}
Note first that the image $\widetilde P_{-t}(R(I))$ associated to a large return $t$
has diameter smaller than $2\delta$. So if this set contains a point $\widetilde P_{-t}(u)$
$\delta'$-close to $I$, then $\widetilde P_{-t}(R(I))$ is contained in the $2\delta+\delta'$-neighborhood of $I$.

The lemma is proved by induction.
The composition $S_{-t_{n(\ell-1)}}\circ\dots\circ S_{-t_{n(0)}}$ is associated
to a return $t'>0$. The point $\widetilde P_{-t'}(u)$
belongs to the image $\widetilde P_{-t_{n(\ell-1)}}(R(I))$,
hence (since $t_{n(\ell-1)}$ is large by deepness and Remark~\ref{r.deep}),
is $2\delta$-close to $I$.
By the Local injectivity and Remark~\ref{r.intersection}, there exists an increasing homeomorphism
$\theta$ such that $|\theta(0)|\leq 1/4$ and
$\varphi_{\theta(-s)-t'}(x)$ shadows $\varphi_{-s}(x)$.
Hence for $t=t'+\theta(t_n(\ell))$ the point $\varphi_{-t}(x)$ is close to $x$
and $\widetilde P_{-t}=\widetilde P_{-t_{n(\ell)}}\circ \widetilde P_{-t'}$
by the Global invariance. The first item of Definition~\ref{d.return}
is satisfied.

The second one is implied by Proposition~\ref{p.rectangle}:
if $\widetilde P_{-t}(W^u_{R(I)}(z))$ intersects $W^u_{R(I)}(z')$,
then these two discs match. The set $\widetilde P_{-t}(R(I))$ has diameter smaller than $2\delta$;
it contains the point $\widetilde P_{-t_{n(\ell)}}\circ \widetilde P_{-t'}(u)$;
this last point also belongs to $\widetilde P_{-t_{n(\ell)}}(R(I))$ which is included in
the $2\delta$-neighborhood of $I$. Hence $\widetilde P_{-t}(R(I))$ is contained in the $4\delta$-neighborhood of $I$
and can not intersect the boundary of the disc $W^u_{R(I)}(z')$.
This gives $\widetilde P_{-t}(W^u_{R(I)}(z))\subset W^u_{R(I)}(z')$.

This proves that $t$ is a return
such that $\widetilde P_{-t}=\widetilde P_{-t_{n(\ell)}} \circ \widetilde P_{-t'}$:
the one-dimensional
map associated to $t$ coincides with the composition of the one-dimensional maps
of the returns $t_{n(\ell)}$ and $t'$ as required.
\end{proof}

\begin{Definition}
A return $t$ is \emph{shifting} if the one-dimensional map $S_{-t}$ has no fixed point.

\emph{Let us fix an orientation on $\cW^{cs}(x)$.
It is preserved by $S_{-t}$ when $t$ is shifting.}
\smallskip

\noindent
A return \emph{shifts to the right} (resp.
\emph{to the left}) if it is a shifting return and if there exists
$u\in I$ that can be joined to $S_{-t}(u)$ by a positive arc
(resp. negative arc)  of $\cW^{cs}(x)$.
\end{Definition}

\subsubsection{Criterion for the existence of periodic $\delta$-intervals}\label{ss.criterion-periodic}

The following proposition shows that under the setting of Proposition~\ref{p.limit},
if the interval $I$ has a large non-shifting return, then 
$I$ is contained in the unstable set of some periodic $\delta$-interval.

\begin{Proposition}\label{p.periodic-return}
Let $I$ be a $\delta$-interval at a point $x\in K$ and
let $J$ be a $3\delta$-interval at a $T_\cF$-Pliss point $y\in K\setminus V$
having large non-shifting returns.
If $\pi_y\circ P_{-s}(I)$ intersects $R(J)$ for some $s\geq 0$,
then $I$ is contained in the unstable set of some periodic $\delta$-interval.
\end{Proposition}
\begin{proof}
By assumption there exist $t>0$ large and $u\in J$
such that $W^{u}_{R(J)}(u)$ is mapped into itself by
$\widetilde P_{-t}$. This implies that $\widetilde P_{-t}$ has a fixed point $v$ in $R(J)$.

Since $t$ is a large return, assuming the $\beta_\cF$ is small enough, there exists (by the local injectivity)
$t_1\in [t-1/4,t+1/4]$ such that $d(\varphi_{-t_1}(y),y)<r/2$. This allows (up to modify $t$)
to assume that $d(\varphi_{-t}(y),y)<r/2$.

\begin{Claim}
There is a periodic point 
$q\in K$ that is 
$r_0$ close to $y$ such that $\pi_y(q)$ is a fixed point of $\widetilde P_{-t}$ in $R(J)$.
\end{Claim}

\begin{proof}[Proof of the claim.]
We build inductively 	a homeomorphism $\theta$ of $[0,+\infty)$
such that:
\begin{itemize}
\item[--] for each $k\geq 0$ and each $s\in [0,t]$, we have
$d(\varphi_{-\theta(s)}(y),\varphi_{-\theta(kt+s)}(y))<r_0/2$,
\item[--] $d(\varphi_{-t_k}(y),y)<r$ where $t_k:=\theta(kt)$ and $r$ is the constant in Remark~\ref{r.intersection},
\item[--] we have $\pi_{y}\circ P_{-t_k}=(\widetilde P_{-t})^k$.
\end{itemize}
Since $d(\varphi_{-t}(y),y)<r/2$, one defines $t_1=t$ and $\theta(s)=s$ for $s\in [0,t]$.

Let us then assume that $\theta$ has been built on $[0,kt]$.
Since $P_{-s}(v)\in P_{-s}(R(J))$ is $\beta_\cF$-close to $0_{\varphi_{-s}(y)}$ for any $s\geq 0$
and since $\pi_{\varphi_{-t_k}(y)}(v)=P_{-t_k}(v)$, Remark~\ref{r.intersection} applies and defines the homeomorphism $\theta$
on $[kt,(k+1)t]$ such that
$d(\varphi_{-s}(y),\varphi_{-\theta(kt+s)}(y))<r_0/4$ for $s\in [0,t]$.
By the local injectivity, one can choose $t_{k+1}$ with
$|t_{k+1}-\theta(kt)|\leq 1/4$,
$d(y,\varphi_{t_{k+1}}(y))<r$ and one can
modify $\theta$ near $(k+1)t$
so that $\theta((k+1)t)=t_{k+1}$.

Since $t_1=t$ is large,
and since
$d(\varphi_{-\theta(kt+s)}(y),\varphi_{-\theta((k+1)t+s)}(y))<r_0$,
the No shear property (Proposition~\ref{p.no-shear}) implies inductively $t_k-t_{k-1}\ge 2$ for each $k\ge 1$.
In particular $t_k\to+\infty$.

By the dominated splitting,
the limit set $\Lambda$ of the curves $\pi_{y}\circ P_{-t_k}(J)=\widetilde P_{-t}^k(J)$
is a union of (uniformly Lipschitz) curves in $\cN_{y}$, containing $v$
and tangent to $\cC^\cE(y)$.

One can apply Proposition~\ref{p.uniqueness}
to a periodic sequence of diffeomorphisms
$f_0,\dots,f_{[t]}$ where $f_i$
coincides with $P_{-1}$ on $\cN_{\varphi_{-m}(y)}$
for $0\leq i <t-1$
and $f_{[t]}$ coincides with
$\pi_{y}\circ P_{t-[t]+1}$ on $\cN_{\varphi_{-[t]+1}(y)}$.
We have $\widetilde P_{-t}=f_{[t]}\circ \dots\circ f_0$
and any curve in the limit set $\Lambda$ remain bounded and
tangent to $\cC^\cE(y)$ under iterations by
$(\widetilde P_{-t})^{-1}$. Consequently, they are all contained in a
same Lipschitz arc, invariant by $\widetilde P_{-t}$.
One deduces that
a subsequence of
$(\widetilde P_{-t})^k(0_y)$
converges in $\cN_{y}$
to a fixed point $p$ of $R(J)$.
\medskip

By Lemma~\ref{l.closing0}, there exists a periodic point $q\in K$
such that $\pi_{y}(q)=p\in R(J)$ is the fixed point of $\widetilde P_{-t}$.
\end{proof}

Now we prove that $I$ is contained in the unstable set of some periodic $\delta$-interval. Without loss of generality, we take $s=0$.
Since $\pi_y(q)$ and $\pi_y(x)$ belong to $R(J)$, by choosing $\beta_\cF$ small enough and by the local invariance,
one can assume that the distances $d(y,q)$ and $d(y,x)$ are smaller than any given constant.
In particular, the projection by $\pi_y$ of center-unstable plaques $\cW^{cu}(u)$ of point $u\in \cN_q$
is tangent to the cone $\cC^\cF$.

\begin{Claim}
There exists a periodic point $q_0\in K$ such that $\alpha(x)$ is the orbit of $q_0$.
\end{Claim}
\begin{proof}[Proof of the Claim.]
By the assumptions, there is $u_0\in I$ such that $\pi_y(u_0)$ is contained in $R(J)$. 
Since $\beta_\cF$ and $\delta$ are small,
by the Global invariance, there is $\theta_y\in\lip$ such that $|\theta_y(0)|\leq 1/4$ and
$d(\varphi_{-s}(x),\varphi_{-\theta_y(s)}(y))<r_0/2$ for any $s>0$.

In a similar way, there exists $\theta_q\in\lip$ such that $|\theta_q(0)|\leq 1/4$ and
$d(\varphi_{-t}(y),\varphi_{-\theta_q(t)}(q))$
remains small for any $t>0$.
By defining $\theta=\theta_q\circ \theta_y$, one deduces that
$d(\varphi_{-\theta(t)}(q),\varphi_{-t}(x))$ is small too for any $t>0$.

By {using the} Global invariance twice,
$\|P_{-t}(\pi_{y}(x))\|$ remains small in the fiber $\cN_{\varphi_{-t}(y)}$
and 
$\|P_{-t}(\pi_{q}(x))\|$ remains small in the fiber $\cN_{\varphi_{-t}(q)}$.
In particular $(P_{-t}(\pi_q(x)))_{t>0}$ is a half generalized orbit and has
a center unstable plaque $\cW^{cu}(\pi_q(x))$.

Let us first assume that $0_q\in\cW^{cu}(\pi_q(x))$.
Since $\|P_{-t}(\pi_{q}(x))\|$ remains small when $t\to +\infty$,
the local invariance of the plaque families implies that
$0_{\varphi_{-t}(q)}\in P_{-t}(\cW^{cu}(\pi_q(x)))$ for any $t>0$.
Projecting to the fiber of $\varphi_{-\theta_q^{-1}(t)}(y)$,
and using the Global invariance, one deduces that
$P_{-t}(\pi_y(x))$ and $P_{-t}(\pi_y(q))$ 
is connected by a small arc tangent to $\cC^\cF$ for any $t>0$.
By Proposition~\ref{p.uniqueness}, this shows that $\pi_y(x)$ and $\pi_y(q)$ belong to a same leaf $W^u_{R(J)}(u)$ of $R(J)$.
Consequently, $d(P_{-t}(\pi_y(x)),P_{-t}(\pi_y(q)))\to 0$ as $t\to +\infty$.
The Global invariance shows that $d(P_{-t}(\pi_q(x)),0)\to 0$ as $t\to +\infty$.
The Local injectivity then implies that $\varphi_{-t}(x)$ converges to the orbit of $q$.

If $\cW^{cu}(\pi_q(x))$ does not contains $0_q$,
by Proposition~\ref{p.fixed-point},
there exists a point $p\in \cN_q$ which is fixed by $P_{-2T}$ (where $T$ is the period of $q$),
such that $P_{-t}(\pi_{q}(x))$ converges to the orbit of $p$ as $t\to +\infty$.
By the Global invariance,
$P_{-\ell T}(\pi_{q}(x))$ coincides with $\pi_q(\varphi_{-\theta^{-1}(\ell T)}(x))$
for $\ell\in\NN$.
Lemma~\ref{l.closing0} implies that there exists a periodic point $q_0\in K$ such that
$\varphi_{-\theta^{-1}(\ell T)}(x)$ converges to $q_0$ as $\ell\to  +\infty$.
Since $\theta$ is a bi-Lipschitz homeomorphism, one deduces that
the backward orbit of $x$ converges to the orbit of the periodic point $q_0$.
\end{proof}

Up to replace $x$ and $I$ by large backward iterates, there is $\theta\in \lip$
s.t. $d(\varphi_{-t}(x),\varphi_{-\theta(t)}(q_0))$ is small for any $t>0$
and by the Global invariance, 
the intervals $P_{-t}\circ \pi_{q_0}(I)$
are curves tangent to the cone $\cC^\cE(\varphi_{-t}(q_0))$
which remain small for any $t>0$.
When $t=\ell T'$ where $T'$ is the period of $q_0$,
they converge to a limit set which is a union of curves
tangent to the cone $\cC^\cE(q_0)$ and contain $0_{q_0}$.
By Proposition~\ref{p.uniqueness}, this limit set is an interval
$ I_{q_0}\subset \cW^{cs}(q_0)$ that is fixed by $P_{T'}$.
By the Global invariance,
$I_{q_0}$ is the limit of $\pi_y(P_{-\theta^{-1}(\ell T')}(I))$
as $\ell\to +\infty$.
This implies that the backward orbit of $I$ converges to the orbit
of the $\delta$-interval $I_{q_0}$.
\end{proof}

\subsubsection{Returns of limit intervals}\label{ss.return-I-infty}
We now prove that the interval $I_\infty$ has always returns.
Moreover, in the case $I_{\infty}$ is not contained in a $3\delta$-interval having arbitrarily large non-shifting return
(so that Proposition~\ref{p.periodic-return} can not be applied to some $\widehat I_k$ and an interval $J$ containing $I_{\infty}$),
we also prove that there exist shifting returns for $I_\infty$, both to the right and to the left.

\begin{Proposition}\label{p.shifting-return}
Under the setting of Proposition~\ref{p.limit}, 
\begin{itemize}
\item[--] either $I_\infty$ is contained in a $3\delta$-interval having arbitrarily large non-shifting return,
\item[--] or $I_{\infty}$ has {two deep sequences of shifting returns { one shifting to the left and the other one to the right.}}
\end{itemize}
\end{Proposition}
\begin{proof} We assume that the first case does not hold and one chooses an orientation, hence an order on
$\cW^{cs}_{x_\infty}$. Denote by $a_\infty<b_\infty$ the endpoints of $I_\infty$.
There exists a $\frac 3 2\delta$-interval $L$ in $\cW^{cs}_{x_\infty}$ such that $R(L)$ contains all the
$\pi_{x_\infty}(\widehat I_k)$, $k$ large.
One can thus project $\pi_{x_\infty}(\widehat I_k)$ to $\cW^{cs}_{x_\infty}$ by the unstable holonomy
and denotes $a_k<b_k$ the endpoints of this projection. In the same way one denotes
$c_k\in [a_k,b_k]$ the projection of $\pi_{x_\infty}(x_k)$.
By Lemma~\ref{l.return-existence}, for any $k<\ell$ such that $k$ and $\ell-k$ are large,
there exists a (large) return $t>0$ such that $\widetilde P_{-t}(\pi_{x_\infty}(x_k))=\pi_{x_\infty}(x_\ell)$.
Since all large returns of $L$ are shifting, such iterates $x_k,x_\ell$ of $x$ do not project by $\pi_{x_\infty}$ to a same
unstable manifold.
Hence one can assume without loss of generality that
for each $k<\ell$ we have
$c_k>c_\ell$.

\begin{Claim}
$I_\infty$ admits a deep sequence of returns $t>0$ shifting to the left.
\end{Claim}
\begin{proof}
For each $k<\ell$, there exists
a return $t>0$ of $L$ such that
$\widetilde P_{-t}(\pi_{x_\infty}(x_k))=\pi_{x_\infty}(x_\ell)$.
This return shifts to the left.
We claim that it is a return for $I_\infty$.
Denote $a'_k=S_{-t}(a_k)$.
The return $t$ has been chosen so that
$\pi_{x_\infty}\circ P_{-n_\ell+n_k}(\widehat I_k)= \widetilde P_{-t}\circ\pi_{x_\infty}(\widehat I_k)$.
Since $\widehat I_\ell$ contains $P_{-n_\ell+n_k}(\widehat I_k)$,
this gives $a_\ell\leq a'_k$ and in particular $a_\ell<a_k$.
Repeating this argument, one gets a decreasing subsequence $(a_i)$ containing $a_k$ and $a_\ell$.
It converges to $a_\infty$ so that $a_\infty< a'_k<a_k$.
This implies that both $a_k$ and $a'_k$ belong to $I_\infty$ and that $t>0$ is also a return for $I_\infty$.
Note that when $k$ is large and $\ell$ much larger, the time $t>0$ is large and the intervals
$\pi_{x_\infty}(\widehat I_k), \pi_{x_\infty}(\widehat I_\ell)$ are close to $\cW^{cs}_{x_\infty}$.
In particular the sequence of returns $t>0$ one obtains by this construction is deep.
\end{proof}
\medskip

Let us fix a return $t$ shifting left as given by the previous claim.
We then choose $k\geq 1$ large and $\ell$ much larger and build a return which shifts to the right.
As explained above we have $a_\infty<a_\ell<a_k$. Let us denote by
$\bar a_\infty<\bar a_\ell<\bar a_k$ their images by $S_{-t}$. Since $S_{-t}$ shifts left, for $k$ large
$\bar a_k$ which is close to $\bar a_\infty$ satisfies $\bar a_k<a_\infty$.
See Figure~\ref{f.return}.
Let us denote $t'>0$ a time such that the pieces of orbits $\varphi_{[-t',0]}(x_k)$
and $\varphi_{[-t,0]}(x_\infty)$ remain close (up to reparametrization), so that
$\pi_{x_\infty}\circ \varphi_{-t'}(x_k)=\widetilde P_{-t}\circ \pi_{x_\infty}(x_k)$.

\begin{figure}[ht]
\begin{center}
\includegraphics[width=10cm]{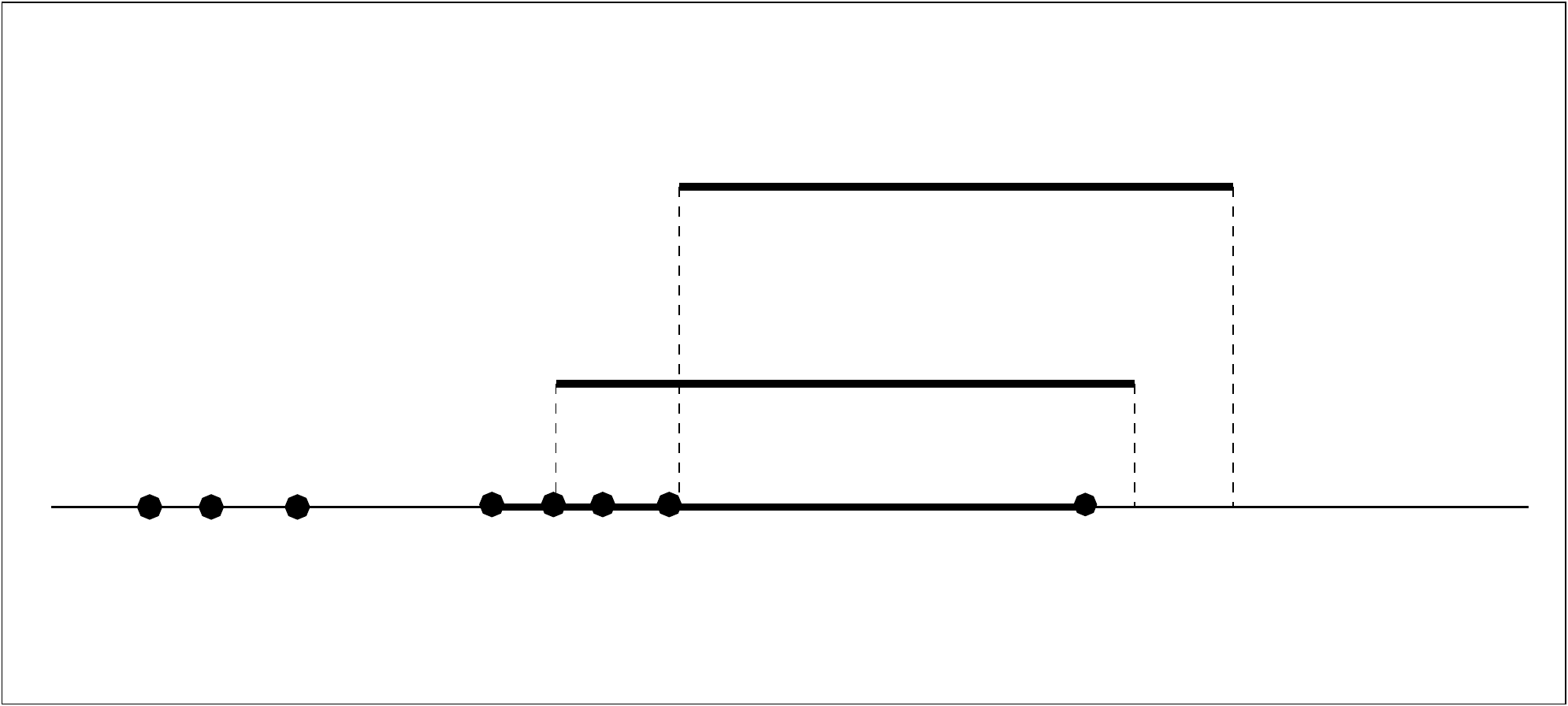}
\begin{picture}(0,0)
\put(-207,25){\scriptsize$a_\infty$}
\put(-192,25){\scriptsize$a_\ell$}
\put(-182,25){\scriptsize$a'_k$}
\put(-170,25){\scriptsize$a_k$}
\put(-95,25){\scriptsize$b_\infty$}
\put(-265,25){\scriptsize$\bar a_\infty$}
\put(-251,25){\scriptsize$\bar a_\ell$}
\put(-238,25){\scriptsize$\bar a_k$}
\put(-150,40){$I_\infty$}
\put(-130,65){$\pi_{x_\infty}(\widehat I_\ell)$}
\put(-110,100){$\pi_{x_\infty}(\widehat I_k)$}
\put(-280,110){$\cN_{x_\infty}$}
\end{picture}
\end{center}
\caption{\label{f.return}}
\end{figure}

The rectangle associated to the $3\delta$-interval $L\cup S_{-t}(L)$
contains $\pi_{x_\infty}\circ \varphi_{-t'}(x_k)$ and $\pi_{x_\infty}(x_\ell)$.
In particular there exists a return $s>0$ such that
$\widetilde P_{-s}(\pi_{x_\infty}\circ \varphi_{-t'}(x_k))=\pi_{x_\infty}(x_\ell)$.
Since $\ell-k$ is large, $s$ is large and is a shifting return by our assumptions.
We denote $a_\infty'<a_k'$ the images of $\bar a_\infty<\bar a_k$ by $S_{-s}$.
Note that there exists a time $s'>0$ such that
$\pi_{x_\infty}\circ \varphi_{-s'}(\varphi_{-t'}(x_k))=\widetilde P_{-s}\circ\pi_{x_\infty}(P_{-t'}(x_k))=\pi_{x_\infty}(x_\ell)$.
By the local injectivity, one can choose $s'$ such that
$t'+s'=n_\ell-n_k$. In particular $S_{-s}\circ S_{-t}$
coincides with the one-dimensional map $S_{-\widetilde t}$ associated to the return $\widetilde t>0$
sending $\pi_{x_\infty}(x_k)$ to $\pi_{x_\infty}(x_\ell)$.
Note that $t'$ is a return as considered in the proof of the previous claim;
in particular we have proved that $a_\infty< S_{-\widetilde t}(a_k)<a_k$.
Hence $a_\infty< a'_k<a_k$.

We have obtained $\bar a_k<a_\infty<a'_k$, so that $S_{-s}$ shifts to the right
and $a_\infty<a'_\infty$.
On the other hand $a'_\infty<a'_k<a_k<b_\infty$.
So this gives $a'_\infty\in (a_\infty, b_\infty)$ which implies that $S_{-s}$
is a return of $I_\infty$ which shifts to the right as required.
{ Since $\pi_{x_\infty}\circ \varphi_{-t'}(x_k)\in R(I)$ and $\pi_{x_\infty}(x_\ell)\to 0_{x_\infty}$,
the sequence of returns $s$ one may build by this construction is deep.}
\end{proof}

\subsection{Aperiodic $\delta$-intervals}\label{ss.aperiodic}
We introduce the $\delta$-intervals which will produce normally expanded irrational tori.

\begin{Definition}\label{d.aperiodic}
A $\delta$-interval $J$ at $x\in K\setminus V$
is \emph{aperiodic} if there exist returns $t_1,t_2>0$
and intervals $J_1,J_2\subset J$ such that:
\begin{enumerate}
\item[--] $J_1,J_2$ have disjoint interior and $J=J_1\cup J_2$,
\item[--] $\widetilde P_{-t_1}(J_1), \widetilde P_{-t_2}(J_2)$
have disjoint interior and $J=\widetilde P_{-t_1}(J_1)\cup \widetilde P_{-t_2}(J_2)$,
\item[--] any non-empty compact set $\Lambda\subset J$ such that
$\widetilde P_{-t_1}\left(\Lambda\cap J_1\right)\subset \Lambda
\text{ and } \widetilde P_{-t_2}\left(\Lambda\cap J_2\right)\subset \Lambda$
coincides with $J$.
\end{enumerate}
\end{Definition}

We prove here that the second case of the Proposition~\ref{p.limit} gives aperiodic $\delta$-intervals.
\begin{Proposition}\label{p.aperiodic0}
Let $x\in K\setminus V$ be a $T_\cF$-Pliss point
and let $I$ be a $\delta$-interval whose large returns are all shifting
and which admit a deep sequence of returns $(t_n^l)$ shifting to the left
and another one $(t^r_n)$ shifting to the right.

Then $\alpha(x)$ contains a point $x'\in K\setminus V$
having an aperiodic  $\delta$-interval.
\end{Proposition}

We first select good returns for $I$.
\begin{Lemma}\label{l.aperiodic0}
Under the setting of Proposition~\ref{p.aperiodic0} there exists a $\delta$-interval $L\subset I$, returns $t_1,t_2>0$
and intervals $L_1,L_2\subset L$
such that
\begin{itemize}
\item[--] $L_1,L_2$ have disjoint interiors and $L=L_1\cup L_2$,
\item[--] $S_{-t_1}(L_1), S_{-t_2}(L_2)$ have disjoint interior and
$L=S_{-t_1}(L_1)\cup S_{-t_2}(L_2)$,
\item[--] any non-empty compact set $\Lambda\subset L$ such that
$S_{-t_1}\left(\Lambda\cap L_1\right)\subset \Lambda
\text{ and } S_{-t_2}\left(\Lambda\cap L_2\right)\subset \Lambda$
coincides with $L$.
\end{itemize}
Moreover $t_1,t_2$ are composition of large returns inside the sequences
$(t_n^l)$ and $(t^r_n)$.
\end{Lemma}
\begin{proof}
Let us consider two large returns $s,t>0$ shifting to the right and to the left respectively.
They can be chosen inside the sequences $(t_n^r)$ and $(t^l_n)$, hence by Lemma~\ref{l.return}
the compositions of the maps $S_{-s}, S_{-t}$, when they are defined, have no fixed point in $I$.
Let us denote $D_s=I\cap S_{-s}^{-1}(I)$ and $I_s=S_{-s}(D_s)$
and similarly let us denote $D_t, I_t$  the domain and the image of $S_{-t}$ in $I$.
\medskip

\emph{Step 1.} One will reduce the interval $I$ so that the assumptions of the Proposition~\ref{p.aperiodic0} still hold
but moreover either $D_s\cup D_t$ or $I_s\cup I_t$ coincides with $I$. See Figure~\ref{f.reduced-interval}.

\begin{figure}[ht]
\begin{center}
\includegraphics[width=12cm]{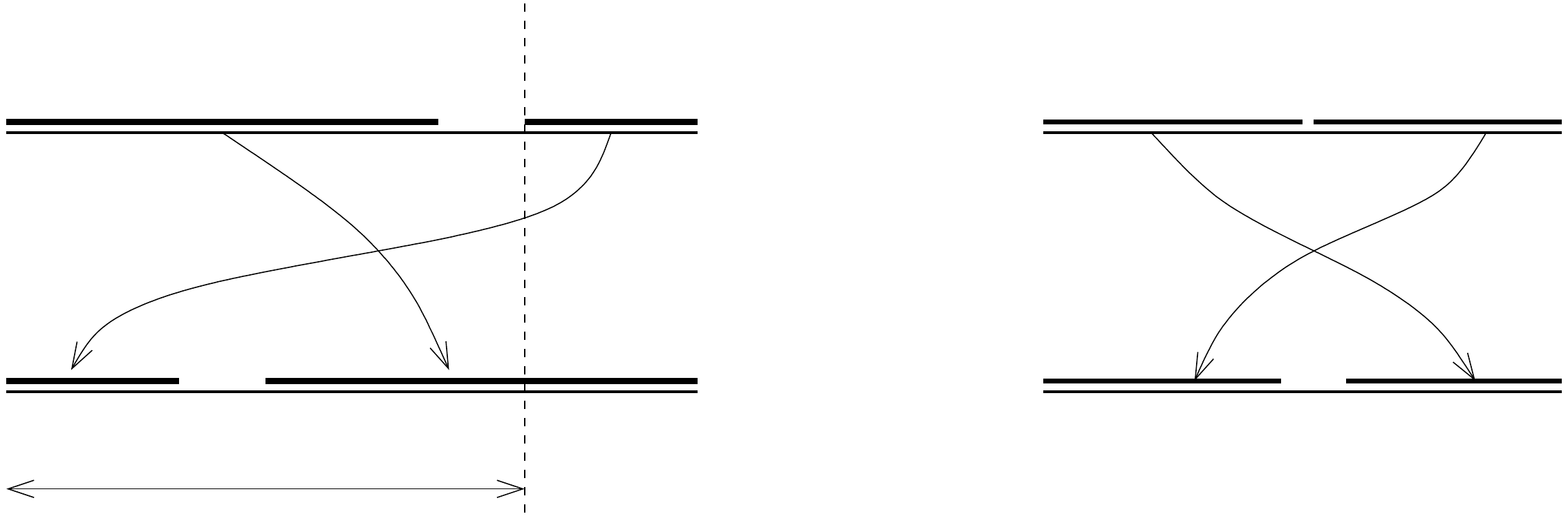}
\begin{picture}(0,0)
\put(-310,93){\scriptsize$D_s$}
\put(-215,93){\scriptsize$D_t$}
\put(-325,20){\scriptsize$I_t$}
\put(-245,20){\scriptsize$I_s$}
\put(-100,93){\scriptsize$D'_s$}
\put(-25,93){\scriptsize$D'_t$}
\put(-95,20){\scriptsize$I'_t$}
\put(-30,20){\scriptsize$I'_s$}
\put(-360,85){$I$}
\put(-360,28){$I$}
\put(-130,85){$I'$}
\put(-130,28){$I'$}
\put(-342,45){$S_{-t}$}
\put(-250,45){$S_{-s}$}
\put(-290,-10){$I'$}
\put(-132,45){$S_{-s}\circ S_{-t}$}
\put(-30,45){$S_{-s}$}
\end{picture}
\end{center}
\caption{\label{f.reduced-interval}}
\end{figure}

Note that both of these two sets contain the endpoints of $I$.
If both are not connected one reduces $I$ in the following way.
Without loss of generality $0_x$ does not belong to $D_t$.
One can thus moves the right point of $I$ (and $D_t)$ inside $D_t$ to the left and reduce $I$ which still contains $0_x$.
This implies that the right endpoints of $D_s,D_t,I_s,I_t,I$ move to the left whereas the left endpoints
remain unchanged. At some moment one of the two intervals $D_s,D_t$ becomes trivial.
We obtain this way new intervals $I',D'_s,D'_t,I'_s,I'_t$.
Note that both can not be trivial simultaneously since otherwise $S_{-s}\circ S_{-t}$ preserves the right endpoint
of the new interval: this one dimensional map is associated to a large return of $I$ which has a fixed point - a contradiction.

Let us assume for instance that $D'_t$ (and $I'_t$) has become a trivial interval (the case
$D'_s$ is trivial is similar): the interval $I'$ is bounded by the left endpoint of $I$
and the left endpoint of $D_t$. Moreover the map $S_{-t}$ sends the right endpoint of $I'$ to its left endpoint.
The map $S_{-t}\circ S_{-s}$ is associated to a return $t'$ of $I'$ which sends the right endpoint of $D'_s$
to the left endpoint of $I'$, hence which shifts left and whose domain coincides with $I'\setminus D'_s$.
We keep the return $s$. We have shown that $D'_s\cup D'_{t'}$ coincides with $I'$.
\medskip

\emph{Step 2.} One again reduces 
$I$ so that the assumptions of Proposition~\ref{p.aperiodic0} still hold, still $D_s\cup D_t$ or $I_s\cup I_t$ coincide with $I$ and furthermore $D_s,D_t$
(resp. $I_s,I_t$) have disjoint interior.

One follows the same argument as in step 1. For instance one moves the right endpoint of $I$ to the left.
Different cases can occur:
\begin{itemize}
\item[--] the point $x$ becomes the right endpoint of $I$,
we then exchange the orientation of $I$ and reduce $I$ again; this case will not
appear anymore;
\item[--] the new domains or the new images have disjoint interior; when this occurs
either $D_s\cup D_t$ or $I_s\cup I_t$ coincide with $I$.
\end{itemize}
Note that the domains $D_s,D_t$ can not become trivial as in step 1. Indeed
if for instance $D_t$ (and $I_t$) becomes trivial, since $I$ still coincides
with $D_s\cup D_t$ or $I_s\cup I_t$, one gets that $I$ coincides with $D_s$ or $I_s$,
proving that $R(I)$ contains a fixed unstable manifold - a contradiction.
\medskip

\emph{Step 3.} We have now obtained a $\delta$-interval $L\subset I$ and two
returns $t_1,t_2$, shifting to the left and the right respectively, whose domains $L_1,L_2$
and images $S_1(L_1)$, $S_2(L_2)$ by $S_1:=S_{-t_1}$ and $S_2:=S_{-t_2}$
have disjoint interior and which satisfy one of the two cases:
\begin{description}
\item[Case 1.] $L_1\cup L_2=L$,
\item[Case 2.] $S_1(L_1)\cup S_2(L_2)=L$.
\end{description}
After identifying the two endpoints of $L$,
one can define an increasing map $f$ on the circle which
coincides with $S_1$ (resp. $S_2$) on the interior of $I_1$ (resp. $I_2$):
in the first case it is injective and has one discontinuity (and $f^{-1}$ is continuous)
whereas in the second case $f$ is continuous.
By our assumptions on $s$ and $t$, $f$ has no periodic point.
We will prove that $f$ is a homeomorphism
conjugated to a minimal rotation. This will conclude the proof of the lemma.

We discuss the case 2 (the first case is very similar, arguing with $f^{-1}$ instead of $f$).
Poincar\'e theory of orientation preserving circle homeomorphisms extends to continuous increasing maps.
Since $f$ has no periodic point, there exists a unique minimal set $K$.
Let us assume by contradiction
that $K$ is not the whole circle. Let $J$ be a component of its complement.
It is disjoint from its preimages $J_{-n}=f^{-n}(J)$ and (up to replace $J$ by one of its backward iterate),
$f\colon J_{-n}\to J_{-n+1}$ is a homeomorphism for each $n\geq 0$.
We now use the Denjoy-Schwartz argument to find a contradiction.

Let us fix $n\geq 0$.
For each $0\leq k\leq n$, the interval $f^k(J_{-n})$ is contained in one of the domains
$L_1$ or $L_2$. Hence $f^k|_{J_{-n}}$ coincides with a composition of the maps $S_1$ and $S_2$
and is associated to a return $s_k>0$ of $L$. Note that $\widetilde P_{-s_k}(J_{-n})=J_{k-n}$
is a $C^1$-curve in $R(L)$ tangent to the cone field $\cC^\cE$.
By Proposition~\ref{p.distortion}, there exists $\Delta>0$ such that any sub-rectangle of $R(L)$, bounded by two leaves
$W^u_{R(L)}(u), W^u_{R(L)}(u')$ has distortion bounded by $\Delta$.
Hence
$\widetilde P_{-s_k}(J_{-n})$ has length bounded by $\Delta.|J_{k-n}|$.
Consequently the sum $\sum_{k=0}^n |P_{-s_k}(J_{-n})|$ is uniformly bounded
as $n$ increases. The difference $s_{k+1}-s_k$ is uniformly bounded also,
so that there exists a uniform bound $C_{Sum}$ satisfying
$$\sum_{0\leq m <s_n} |P_{-m}(J_{-n})|<C_{Sum}.$$
From Lemma~\ref{Lem:schwartz}, there exists an interval $\widehat J_{-n}\subset \cW^{cs}(x)$
containing $J_{-n}$
satisfying $|\widehat J_{-n}|\leq 2|J_{-n}|$ such that
each component of
$P_{-s_n}(\widehat J_{-n}\setminus J_n)$ has length
$\eta_S|P_{-s_n}(J_{-n})|$,
where $\eta_S>0$ is a small uniform constant. Note that the projection through unstable holonomy
of $\widetilde P_{-s_n}(\widehat J_{-n})$ in $L$ contains a uniform neighborhood $\widehat J=S_{-s_n}(\widehat J_{-n})$ of $J$.

The large integer $n$ can be chosen so that the small interval
$J_{-n}$ is arbitrarily close to $J$.
Consequently, $\widehat J_{-n}$ is contained in $\widehat J$. This means
$S_{-s_n}(\widehat J_{-n})\supset \widehat J_{-n}$ implying that $\widehat J_{-n}$ contains
a fixed point of the map $S_{-s_n}$.
This is a contradiction since we have assumed that all the returns are shifting.
As a consequence $f$ is a minimal homeomorphism, which implies the lemma.
\end{proof}
\smallskip

\begin{proof}[Proof of Proposition~\ref{p.aperiodic0}]
Let us consider some intervals $L_1,L_2,L$ and some returns $t_1,t_2$ as in
Lemma~\ref{l.aperiodic0}. We denote $\widetilde P_i=\widetilde P_{-t_i}$
and $S_i=S_{-t_i}$, $i=1,2$.
Since $t_1,t_2$ are large inside deep sequences of returns,
the compositions $S_{i_k}\circ\dots\circ S_{i_1}$,
when they are defined, are associated to returns of $L$
(see Lemma~\ref{l.return}).
\medskip

\begin{Claim}
$\widetilde P_1$ and $\widetilde P_2$ commute.
\end{Claim}
\begin{proof}
Both compositions $\widetilde P_1\circ \widetilde P_2$ and $\widetilde P_2\circ \widetilde P_1$
are associated to returns $s,s'>0$.
Note that the common endpoint of $L_1$ and $L_2$ has the same image by
$S_1\circ S_2$ and $S_2\circ S_1$.
If the two compositions do not coincide, the two returns are different,
for instance $s'>s$, but there exists two points $u,u'$ in $R(L)$
in a same unstable manifold such that $u=\pi_x\circ P_{-s'+s}\circ \pi_{\varphi_{-s}(x)}(u')$.
Note that since $t_1,t_2$ can be obtained from deep sequences of returns,
$u,u'$ are arbitrarily close to $L$. The time $s'-s$ can be assumed to be
arbitrarily large: otherwise, we let $t_1,t_2$ go to $+\infty$
in the deep sequences of returns; if $s'-s$ remain bounded,
the points $u,u'$ converge to a point of $L$
a point which is fixed by some limit of the $\pi_x\circ P_{-s'+s}$. This is a contradiction
since the returns of $I$ are non-shifting.
Since $s'-s$ is large, {Lemma~\ref{l.return-existence}} implies that there is a large return sending
$u$ on $u'$, which is a contradiction since large returns are shifting.
Consequently the two returns are the same and the compositions coincide.
\end{proof}
\medskip

By the properties on $S_1,S_2$, there is $(i_k)\in \{1,2\}^\NN$
such that for each $k$
$$\widetilde P_{i_k}\circ \dots\circ\widetilde P_{i_1}(x)\in R(L).$$
From Lemma~\ref{l.return}, for each $k$ there exists a return $s_k$ such that
$$\widetilde P_{-s_k}=\widetilde P_{i_k}\circ \dots\circ\widetilde P_{i_1}.$$
The dynamics of $S_1,S_2$ is minimal in $L$,
hence the iterates $x_k:=\varphi_{-s_k}(x)$ have
a subsequence $x_{k(j)}$ converging
to a point $x'\in K\setminus V$ such that $\pi_x(x')$ belongs to the unstable manifold of $x$.
{We define the intervals $J,J_1,J_2$} as limits of the iterates of $L,L_1,L_2$
by the maps $P_{-s_{k(j)}}$. In particular $J$ is a $\delta$-interval and the
first item of Definition~\ref{d.aperiodic} holds.

By Remark~\ref{r.intersection}, there exist times $t'_1,t'_2$
such that backward orbit of $x$ during time $[-t_i,0]$
is shadowed by the backward orbit of $x'$ during the time $[-t'_i,0]$.
By the Global invariance, the maps
$\pi_{x'}\circ P_{-t'_i}$, $i=1,2$, from a neighborhood of $0_{x'}$ to $\cN_{x'}$
are conjugated to $\widetilde P_i$ by the projection $\pi_x$ and will be still
denoted by $\widetilde P_i$.
In particular, the map {$P_{-s_{k(j)}}$} from a neighborhood of $0_x\in \cN_x$
to $\cN_{x'}$ coincides with $\widetilde P_{i_{k(j)}}\circ \dots\circ\widetilde P_{i_1}\circ \pi_{x'}$
and $\pi_{x'}\circ \widetilde P_{i_{k(j)}}\circ \dots\circ\widetilde P_{i_1}$.
\medskip

\begin{Claim}
$J$ is aperiodic.
\end{Claim}
\begin{proof}
The projections by $\pi_{x'}$ of $L_1,L_2$ and $S_1(L_1), S_2(L_2)$
have images under the maps {$P_{-s_{k(j)}}$}
which converge to subintervals of $J$:
{by definition,} the first ones are $J_1,J_2$ whereas
{by Global invariance} the last ones
denoted by $J'_1,J'_2$ have disjoint interior and satisfy $J=J'_1\cup J'_2$.
{Since $S_1(L_1)$ and $\widetilde P_1(L_1)$ have the same projection by unstable holonomy,}
Note that $J'_1,J'_2$ are also the limits of
$\widetilde P_1(L_1)$ and $\widetilde P_2(L_2)$
under the maps {$P_{-s_{k(j)}}$}.
Since $\widetilde P_1$ and $\widetilde P_2$ commute this implies that
$J'_1=\widetilde P_1(J_1)$ and $J'_2=\widetilde P_2(J_2)$.
Consequently the second item of the Definition~\ref{d.aperiodic} holds.

Let us consider the projection $\varphi$ from $L$
to $J$ obtained as composition of $\pi_{x'}$ with the unstable holonomy.
Note that it conjugates the orbit of $0_x$ in $L$ by $S_1,S_2$ with the orbit
of $0_{x'}$ in $J$ by $\widetilde P_1,\widetilde P_2$.
Passing to the limit $\varphi$ induces a conjugacy between the dynamics of $S_1,S_2$
on $L$ and $\widetilde P_1,\widetilde P_2$ on $J$.
The third item is thus a consequence of Lemma~\ref{l.aperiodic0}
\end{proof}

The proof of Proposition~\ref{p.aperiodic0} is now complete.
\end{proof}

\subsection{Construction of a normally expanded irrational torus}\label{ss.topological}
The whole section is devoted to the proof of the next proposition.
\begin{Proposition}\label{p.aperiodic}
If $x\in K\setminus V$ has an aperiodic $\delta$-interval $J$,
then the orbit of $x$ is contained in a minimal set $\cT\subset K$
which is a normally expanded irrational torus, and $\pi_{x}(\cT\cap B(x,r_0/2))\supset J$.
\end{Proposition}

\begin{proof}
The definition of {aperiodic $\delta$-interval}
is associated to two returns $t_1,t_2>0$.
As before we denote $\widetilde P_1=\widetilde P_{-t_1}$ and $\widetilde P_2=\widetilde P_{-t_2}$.

Let us define $\cC$ to be the compact set of points $z\in K$ such that
$$d(z, x)\leq r_0/2 \text{ and } \pi_{x}(z)\in J.$$
The map $z\mapsto \pi_{x}(z)$ is continuous from $\cC$ to $J$.
We also define the invariant set $\cT$ of points whose orbit meets $\cC$.

For any $u\in J$, we define the sets
$$\Lambda_u^+=\{v\in J,~\exists n\in\NN,~k(1),\dots,k(n)\in\{1,2\},~{\rm s.t.}~\widetilde P_{k(n)}\circ \widetilde P_{k(n-1)}\circ \widetilde P_{k(1)}u=v\},$$
$$\Lambda_u^-=\{v\in J,~\exists n\in\NN,~k(1),\dots,k(n)\in\{1,2\},~{\rm s.t.}~\widetilde P^{-1}_{k(n)}\circ \widetilde P^{-1}_{k(n-1)}\circ \widetilde P^{-1}_{k(1)}u=v\}.$$

By the definition of aperiodic interval and of $\Lambda_u^+$ and $\Lambda_u^-$,
the set of accumulation points of $\Lambda_u^+$ and $\Lambda_u^-$ are $J$.

\begin{Claim-numbered}\label{c.surjective}
The map $z\mapsto \pi_{x}(z)$ from $\cC$ to $J$ is surjective.
\end{Claim-numbered}
\begin{proof} 
This follows from the fact that the closure of $\Lambda^+_{0_x}$ is $J$,
and that $\widetilde P_{k(n)}\circ \widetilde P_{k(n-1)}\circ \widetilde P_{k(1)}.0_x$
belongs to the projection of the orbit of $x$ by $\pi_x$.
\end{proof}

\begin{Claim-numbered}\label{c.return}
For any non-trivial interval $J'\subset J$ there is $T>0$ such that any $z\in \cC$ has iterates
$\varphi_{t^+}(z)$, $\varphi_{-t^-}(z)$ with $t^+,t^-\in (1, T)$ which belong to $\cC$
and project by $\pi_{x}$ in $J'$.
\end{Claim-numbered}
\begin{proof} For $z\in \cal C$, since $\Lambda_z^+$ is dense in $J$, there {are}
$u_0=\pi_{x}(z)$, $u_1$,\dots, $u_{\ell(z)}$ in $J$ such that
\begin{itemize}

\item[--] $u_{\ell(z)}$ belongs to the interior of $J'$,

\item[--] For $0\le i\le \ell(z)-1$, there is $k(i)\in\{1,2\}$ such that $\widetilde P_{k(i)}(u_i)=u_{i+1}$.
\end{itemize}

Thus by compactness, there is a uniform $\ell$ such that $\ell(z)\le \ell$ for any $z\in{\cal C}$. Moreover by the Global invariance, there is $t(i)>0$ such that $\pi_x\varphi_{-t(i)}(z_i)$ is close to $u_i$ and $t(i+1)-t(i)$ is smaller than $2\max(t_1,t_2)$.

Since $\pi_{x}(z)\in J$ there exists
a $2$-Lipschitz homeomorphism $\theta$ of $[0,+\infty)$ such that
$\varphi_{-\theta(t)}(z)$ is close to $\varphi_{-t}(x)$ for each $t\geq 0$.
This proves that $\varphi_{-\theta(t(\ell))}(z)$ projects by $\pi_{x}$ on $u_\ell$,
hence belongs to $\cC$. Since $\ell$ and the differences $t(i+1)-t(i)$ are bounded
and since $\theta$ is $2$-Lipschitz, the time $t^-:=\theta(t(\ell))$ is bounded by a constant which only
depends on $J'$, as required.

The time $t^+$ is obtained in a similar way since $\Lambda_z^-$ is dense in $J$ for any $z\in J$.
\end{proof}

Set $\cT=\cup_{t\in\RR}\varphi_t(\cC)$. From the previous claim, there exists $T>0$ such that
$\cT=\varphi_{[0,T]}(\cC)$. Since $\cC$ is compact, $\cT$ is also compact.
Note that (using Claim~\ref{c.return}) any point in $\cT$ has arbitrarily large forward iterates in $\cC$,
whose projection by $\pi_x$ belongs to $J\subset \cN_x$.
Since $x\in K\setminus V$,
by choosing $\delta>0$ small enough, the local injectivity implies:
\begin{Claim-numbered}\label{c.large-returns}
Any point in $\cT$ has arbitrarily large forward iterates in the $r_0$-neighborhood of $K\setminus V$.
\end{Claim-numbered}

\begin{Claim-numbered}
There exists a curve $\overline \gamma\subset \cC\cap B(x,r_0/2)$ which
projects homeomorphically by $z\mapsto \pi_{x}(z)$ on a non trivial compact interval of $\operatorname{Interior}(J)$ .
\end{Claim-numbered}
\begin{proof}
We note that:
\begin{enumerate}
\item\label{i.delta} By the ``No small period" assumption, for $k\in\NN$, there exists $\varkappa_k>0$ such that for any $y,y'\in \cC$
and $t\in [0,1]$ satisfying $d(y,y')<\varkappa_k$ and $d(y,\varphi_t(y'))<\varkappa_k$,
then for any $s\in [0,t]$ we have $d(y,\varphi_s(y'))<2^{-k-1}r_0$.

\item\label{i.epsilon} For each $\varkappa_k>0$ as in Item~\ref{i.delta}, there exists $\varepsilon_k>0$ with the following property.
For any $y\in \cC$ satisfying $B(y,\varkappa_k)\subset B(x,r_0)$ and
for any $u\in J$ such that $d(u,\pi_{x}(y))<\varepsilon_k$,
then there is $y'\in B(y,\varkappa_k/2)\cap\cC$ such that $\pi_x(y')=u$. Indeed, by the Local injectivity property, there is $\beta_k>0$ such that for any points $w\in B(y,r_0)$, if $\|\pi_y w\|<\beta_k$, then there is $t\in[-1/2,1/2]$ such that $d(\varphi_t(w),y)<\varkappa_k/2$. By the uniform continuity of the identification $\pi$, for $\varepsilon_k>0$ such that for any $v_1,v_2\in {\cal N}_x$, if $\|v_1-v_2\|<\varepsilon_k$, then $\|\pi_y(v_1)-\pi_y(v_2)\|<\beta_k$.  Now for any $u\in J$ such that $d(u,\pi_{x}(y))<\varepsilon_k$, by Claim~\ref{c.surjective} there is $y_0\in B(x,r_0/2)\cap\cC$  such that $\pi_x(y_0)=u$ and $d(\pi_x(y_0),\pi_x(y))<\varepsilon_k$. Hence $d(\pi_y(y_0),0_y)=d(\pi_y\circ\pi_x(y_0),\pi_y\circ\pi_x(y))<\beta_k$. By using the Local injectivity property, $d(\varphi_t(y_0),y)<\varkappa_k/2$ for some $t\in[-1/2,1/2]$. Set $y'=\varphi_t(y_0)$. By Local invariance, we have that $\pi_x(y')=u$.

\item\label{i.lift} By letting $k\to\infty$ and $\beta_k\to 0$, one deduces that
for any $u\in J$ and any $y\in \cC$ such that $\pi_{x}(y)$ is close to $u$,
there exists $y'$ close to $y$ such that $\pi_{x}(y')=u$.
\end{enumerate}

We build inductively an increasing sequence of finite sets $Y_k$ in $\cC$
satisfying:
\begin{itemize}
\item[--] $\pi_{x}$ is injective on $Y_k$,
\item[--] for any $y,y'\in Y_k$ such that
$\pi_{x}(y)$, $\pi_{x}(y')$ are consecutive points of $\pi_{x}(Y_k)$ in $J$,
$$d(\pi_{x}(y),\pi_{x}(y'))<\varepsilon_k.$$
\item[--] if moreover $y''\in Y_{k+1}$ satisfies that $\pi_{x}(y'')$ is
between  $\pi_{x}(y)$ and $\pi_{x}(y')$ in $J$,
then $y''$ is $2^{-k}r_0$-close to $y$ and $y'$.
\end{itemize}
Let us explain how to obtain $Y_{k+1}$ from $Y_k$.
We fix $y,y'\in Y_k$ such that $\pi_{x}(y)$, $\pi_{x}(y')$ are consecutive points of $\pi_{x}(Y_k)$ in $J$ and define the points in $Y_{k+1}$
which project between $\pi_{x}(y)$, $\pi_{x}(y')$.
By items~\ref{i.epsilon},
we introduce a finite set $y''_1,y''_2,\dots,y''_j$
of points in $B(y, \varkappa_k/2)$ such that

\begin{itemize}

\item[--] $y_1''=y$ and $y_j''=y'$,
\item[--] $\pi_{x}\{y_1'',y_2'',\dots,y_j''\}$
is $\varepsilon_{k+1}$-dense in the arc of $J$ bounded by $\pi_{x}(y)$ and $\pi_{x}(y')$,
\item[--] $\pi_{x}(y''_i),\pi_{x}(y''_{i+1})$ 
are consecutive points of the projections in $J$ for $i\in \{1,\dots,j-1\}$,

\end{itemize}

The distance of $y_i''$ and $y_{i+1}''$ may be larger than $r_0 2^{-k-2}$. We need to modify the construction. By Item~\ref{i.epsilon}, there is $t_i\in[-1,1]$ such that $z_i=\varphi_{t_i}(y_i'')$ is $\varkappa_{k+1}/2$-close to $y_{i+1}''$. Choose $n\in\NN$ large enough. Consider the points $\{\varphi_{mt_i/n}(y_i'')\}_{m=0}^n$.  By using the item~\ref{i.lift}
there exists a finite collection $X_i$ of points that
are arbitrarily close to the set $\{\varphi_{mt_i/n}(y_i'')\}_{m=0}^n$ such that they
have distinct projection by $\pi_{x}$
and such that any two such point with consecutive projections are 
$2^{-k-2}r_0$-close.
By the item~\ref{i.delta}, the set $X_i$
is contained in $B(y,2^{-k-1}r_0)$.
Since $d(y,y')<2^{-k-1}r_0$, it is also contained in $B(y',2^{-k}r_0)$.
The set  of points of $Y_{k+1}$which project between $\pi_{x}(y)$, $\pi_{x}(y')$ 
is the union of the $\{y_i'',y_{i+1}''\}\cup X_i$ for any $i$.

Let us define $Y=\cup Y_k$.
The restriction of $\pi_{x}$ to $Y$ is injective
and has a dense image in a non-trivial interval $J'$
contained in $\operatorname{Interior}(J)$.
Its inverse $\chi$ is uniformly continuous:
indeed for any $k$, any $y,y'\in Y_k$
with consecutive projection, and any
$y''\in Y$ such that $\pi_{x}(y'')$
is between $\pi_{x}(y)$, $\pi_{x}(y')$,
the distance $d(y'',y)$ is smaller than $2^{-k+1}r_0$.
As a consequence $\chi$ extends continuously to $J'$
as a homeomorphism such that $\pi_{x}\circ \chi=\id_{J'}$.
The curve $\bar \gamma$ is the image of $\chi$.
\end{proof}

Let $\gamma$ be the open curve obtained by removing the endpoints of $\overline \gamma$.
By the Local invariance,
for $\varepsilon>0$ small,
$\varphi$ is injective on $[-\varepsilon,\varepsilon]\times \overline\gamma$;
its image is contained in $\cC$ and is homeomorphic to the ball $[0,1]^2$. This is because a continuous bijective map from a compact space to a Hausdorff space is a homeomorphism. Thus $\varphi$ is a homeomorphism from $(-\varepsilon,\varepsilon)\times \gamma$ to its
image.

\begin{Claim-numbered}\label{c.open}
The set $\varphi_{(-\varepsilon,\varepsilon)}(\gamma)$ is open in $\cT$.
\end{Claim-numbered}
\begin{proof} Let us fix $z_0\in \varphi_{t}(\gamma)$
for some $t\in (-\varepsilon,\varepsilon)$ and let us consider any point $z\in \cT$ close to $z_0$.
We have to prove that $z$ belongs to the open set $\varphi_{(-\varepsilon,\varepsilon)}(\gamma)$.
Note that
$d(z, x)<r_0$ and $\pi_{x}(z)$ belongs to $R(J)$.
By Claim~\ref{c.return} and the definition of $\cT$,
the point $z$ has a negative iterate $z'=\varphi_{-s}(z)$ in $\cC$, with $s$ bounded, so that $\pi_x(z')\in J$.

One deduces that $\pi_{x}(z)$ belongs to $J$.
Otherwise, $\pi_x(z)\in R(J)\setminus J$.
Since $z$ and $x$ have arbitrarily large returns in the $r_0$-neighborhood of $K\setminus V$
(Claim~\ref{c.large-returns}), the Proposition~\ref{p.coherence} can be applied and
$\pi_x\circ P_{s}\circ \pi_{z'}(J)$ is disjoint from $J$.
Since $s$ is bounded, the distance $d(\pi_x\circ P_{s}\circ \pi_{z'}(J),J)$
is bounded away from zero. We have $\pi_x(z_0)\in J$ whereas
$\pi_x(z)\in P_{s}\circ \pi_{z'}(J)$ which contradicts the fact that
$z$ is arbitrarily close to $z_0$.

Since $z$ is close to $z_0$ and the curve $\gamma$ is open,
we have $\pi_{x}(z)\in \pi_{x}(\gamma)$
and $z\in \varphi_{t'}(\gamma)$ for some $|t'|<1/2$ by the Local injectivity.

By the Local invariance $\varphi_{-t}(z)$ and $\varphi_{-t'}(z)$
have the same projection by $\pi_{x}$.
If $z$ is close enough to $z_0$, both are arbitrarily close to $\gamma$.
Since $\pi_{x}$ is injective on $\gamma$ one deduces
that $d(\varphi_{-t}(z),\varphi_{-t'}(z))$ is arbitrarily small.
By the ``No small period" assumption,  this implies that $|t'-t|$ is arbitrarily small.
Hence $|t'|<\varepsilon$
proving that $z\in \varphi_{(-\varepsilon,\varepsilon)}(\gamma)$, as required.
\end{proof}

By Claim~\ref{c.return}, any point $z\in \cT$ has a backward  iterate in $\cC$
which project by $\pi_{x}$ in $\pi_{x}(\gamma)$, hence 
by Local injectivity has an iterate $\varphi_{-t}(z)$ in $\gamma$. Since $\varphi_{-t}$
is a homeomorphism, one deduces that $z$ has an open neighborhood of the form
$\varphi_{t}(\varphi_{(-\varepsilon,\varepsilon)}(\gamma))$ which is homeomorphic
to the ball $(0,1)^2$. As a consequence, $\cT$ is a compact topological surface.

Since any forward and backward orbit of $\cT$ meets the small open set $\varphi_{(-\varepsilon',\varepsilon')}(\gamma')$, for any $\varepsilon'\in (0,\varepsilon)$
and $\gamma'\subset \gamma$,
the dynamics induced by $(\varphi_{t})$ on $\cT$ is minimal.

From classical results on foliations on surfaces (see~\cite[Theorem 4.1.10, chapter I]{hector-hirsch-foliation}),
we get:
\begin{Claim-numbered}
$\cT$ is homeomorphic to the torus $\TT^2$ and
the induced dynamics of $(\varphi_t)_{t\in \RR}$ on $\cT$
is topologically conjugated to the suspension of an irrational rotation of the circle.
\end{Claim-numbered}

By Claim~\ref{c.open},
$\varphi_{(-\varepsilon,\varepsilon)}(\gamma)$
is a neighborhood of $x$ in $\cT$
and 
$\pi_{x}(\gamma)$
is contained in $\cW^{cs}(x)$, hence is a $C^1$-curve
tangent to $\cE(x)$ at $0_{x}$.
Hence $\cT$ is a normally expanded irrational torus.
Moreover $\pi_{x}(\cT\cap B(x,r_0/2))$ contains $J$
by Claim~\ref{c.surjective}.
This ends the proof of Proposition~\ref{p.aperiodic}.
\end{proof}

\subsection{Proof of the topological stability}\label{ss.Lyapunov}
We prove Proposition~\ref{Prop:dynamicsofinterval} and Lemma~\ref{l.periodic} that imply the topological stability
in Section~\ref{sss.lyapunov-stable}.

\begin{proof}[Proof of Proposition~\ref{Prop:dynamicsofinterval}]
By Proposition~\ref{p.limit}, there is a sequence of $\delta$-intervals $(\widehat I_k)$
at $T_\cF$-Pliss backward iterates $\varphi_{-n_k}(x)$ in $K\setminus V$ and converging to a
$\delta$-interval $I_\infty$ at a $T_\cF$-Pliss point $x_\infty$.
\medskip

Let us assume first that all large returns of $I_\infty$ are shifting.
By Proposition~\ref{p.shifting-return}, $I_\infty$ has two deep sequences of returns,
one shifting to the right and one shifting to the left.
Proposition~\ref{p.aperiodic0} then implies that $\alpha(x_\infty)$ contains a point $x'\in K\setminus V$
with an aperiodic $\delta$-interval.
By Proposition~\ref{p.aperiodic}, the orbit of $x'$ is contained in a normally expanded
irrational torus $\cT$. We have proved that the first case of Proposition~\ref{Prop:dynamicsofinterval} holds.
\hspace{-1cm}\mbox{}
\medskip

Let us assume then that $I_\infty$ admits arbitrarily large non-shifting returns.
Proposition~\ref{p.periodic-return} implies that $I$ is contained in the unstable set of some periodic $\delta$-interval $J$.
\end{proof}

\begin{proof}[Proof of Lemma~\ref{l.periodic}]
Let us consider the rectangle $R(I)$: it is mapped into itself by $P_{-2s}$, where $s$ is the period of $q$.
By assumption (A3), $\cE$ is contracted over the orbit of $q$.
One deduces that there exists a neighborhood $B$ of $0_q$ in $R(I)$
such that for any $u\in B\setminus \cW^{cu}(0_q)$, the backward orbit of $u$
by $(P_t)$ converges towards a periodic point of $I$ which is not contracting along $\cE$.
So Lemma~\ref{l.closing0} implies that $\pi_q(K)$ is disjoint from $B\setminus \cW^{cu}(0_q)$.

For any interval $L\subset \cN_z$ as in Lemma~\ref{l.periodic}, the point $\pi_q(z)$
does not belong to $B\setminus \cW^{cs}(0_q)$. Hence $\pi_q(L)$ contains an interval $J$
which is contained in $B$,
meets $\cW^{cu}(0_q)$ and has a length larger than some constant $\chi_0>0$.
The backward iterate $P_{-2s}(J)$ still contain an interval $J'$ having this property
since $q$ is attracting along $\cE$.

Since $\pi_q(J)$ intersects $R(I)$, there exists $\theta\in \lip$
such that $d(\varphi_{-t}(z),\varphi_{-\theta(t)}(q))$ remain small for any $t\in [0,T]$.
By the Global invariance, if $k$ is the largest integer such that $\theta(T)>ks$,
then $\varphi_{-\theta^{-1}(ks)}(L))$ contains an interval of length larger than $\chi_0/2$.
Since $T_\theta^{-1}(T)$ is bounded, this implies that $\varphi_{-T}(L)$ has length larger
than some constant $\chi>0$.
\end{proof}

\subsection{Proof of the topological contraction}
We prove a proposition that will allow us to conclude the topological contraction.

\begin{Definition}
$K$ admits \emph{arbitrarily small periodic intervals} if for any $\delta>0$, there is a periodic point $p\in K$, whose orbit supports a periodic $\delta$-interval.
\end{Definition}

\begin{Proposition}\label{Pro:smallperiodic-interval}

Under the assumptions of Theorem~\ref{Thm:topologicalcontracting}, if $K$ does not contain a normally expanded irrational torus,
if it admits arbitrarily small periodic intervals and is transitive, then there are 
$C_0>0$, $\varepsilon_0>0$, and 
a non-empty open set $U_0\subset K$ such that for any $z\in U_0$, we have

$$\sum_{k\in\NN}|P_k(\cW^{cs}_{\varepsilon_0}(z))|<C_0.$$

\end{Proposition}

In the next 3 sections, we assume that the setting of Proposition~\ref{Pro:smallperiodic-interval} holds.
We also assume that $\cE$ is not uniformly contracted (since otherwise Proposition~\ref{Pro:smallperiodic-interval} holds by Lemma~\ref{l.summability-hyperbolicity}).
Proposition~\ref{Pro:smallperiodic-interval} is proved in Section~\ref{ss.sum}, and then
Theorem~\ref{Thm:topologicalcontracting} is proved in Section~\ref{ss.conclusion-topological}.

\subsubsection{The unstable set of periodic points}

\begin{Lemma}\label{l.unstable-set}
For any $\beta>0$, there exist:
\begin{itemize}
\item[--] a periodic point $p\in K\setminus V$ (with period $T$),
\item[--] a point $x\in K\setminus \{p\}$ which is $r_0/2$-close to $p$ and whose $\alpha$-limit set is the orbit of $p$,
\item[--] $r_x>0$ and a connected component $Q$ of $B(0_x,r_x)\setminus \cW^{cu}(x)$ in $\cN_x$ {such that $Q\cap \pi_x(K)=\emptyset$ and the diameter of $P_{-t}(Q)$ is smaller than $\beta$ for each $t>0$.}
\end{itemize}
\end{Lemma}

\begin{proof} 
Since $K$ admits arbitrarily small periodic intervals, there is a periodic point $p\in K$ with period $T>0$ and a periodic $\delta$-interval
$I\subset {\cal N}_p$ for $\delta$ small.
Since $\cE$ is uniformly contracted in $V\supset K\setminus U$ (assumption (A2)), one can replace $p$ by one of its iterates so that $p\in K\setminus V$.
By Lemma~\ref{Lem:hyperbolicreturns} (and the continuity of $(P_t)$),
\begin{itemize}
\item[--] the restriction of $\cF$ to the orbit of $I$ by ($P_t)$ is an expanded bundle. 
\end{itemize}
Since $\cE$ is uniformly contracted over the orbit of $0_p$ (by (A3)),
one can assume that:

\begin{itemize}
\item[--] Only the endpoints of $I$ are fixed by $P_{T}$. One is $0_p$, the other one attracts any point of $I\setminus \{0_p\}$
by negative iterations of $P_T$.
\item[--] There is $\beta_p>0$ such that for any $u\in \cN_p$ satisfying $\|P_{-t}(u)\|<\beta_p$ for each $t>0$,
then $u$ is in the unstable manifold of $0_p$ for $P_t$. 
\end{itemize}
Let $r_p>0$ such that $B(p,r_p)\subset U$ and $P_{-t}\circ \pi_p(B(p,r_p))\subset B(0_{\varphi_{-t}(p)},\beta_p)$ for each $t\in [0,T]$.\hspace{-1cm}\mbox{}
\medskip

Since $K$ is transitive, there are sequences $(x_n)$ in $K$ and $(t_n)$ in $(0,+\infty)$ such that
\begin{itemize}
\item[--] $\lim_{n\to\infty}x_n=p$ and $\lim_{n\to\infty}t_n=+\infty$,
\item[--] $d(\varphi_{t}(x_n), \orb(p))<r_p$ for $t\in (0,t_n)$ and $d(\varphi_{t_n}(x_n), \orb(p))=r_p$ for each $n$.
\end{itemize}
Taking a subsequence if necessary, we let $x:=\lim \varphi_{t_n}(x_n)$.
By the Global invariance, $P_{-t}\circ \pi_p(x)\subset B(0_{\varphi_{-t}(p)},\beta_p)$ for each $t>0$, hence $\pi_p(x)$ lies in the unstable manifold of $0_p$.
Combined with the Local injectivity, there exists $\theta\in \lip$ such that $\theta(0)=0$ and $d(\varphi_{\theta(t)}(x),\varphi_t(p))\to 0$ as $t\to -\infty$.
Hence the $\alpha$-limit set of $x$ is $\orb(p)$.
\medskip

We now consider the dynamics of $P_{T}$ in restriction to $\cN_p$. The periodic interval $I$ is normally expanded.
Consequently, for $r_x$ small, one of the components $Q$ of $B(0_x,r_x)\setminus \cW^{cu}(x)$ in $\cN_x$ has an image by $\pi_x$ contained in the unstable set of $I\setminus \{0_p\}$.

Let us assume by contradiction that there exists $y\in U$ such that $\pi_x(y)\in Q$.
The backward orbit of $\pi_p(y)$ by $P_T$ converges to the endpoint $v$ of $I$ which is not attracting along $\cE$.
Lemma~\ref{l.closing0} and the Global invariance imply that the backward orbit of $y$ converges to a periodic orbit in $K$ whose eigenvalues
at the period for the fibered flow coincide with the eigenvalues at the period of the fixed point $v$ for $P_T$. This is a contradiction since all the eigenvalues of
$v$ are non-negative whereas $\cE$ is uniformly contracted over the periodic orbits of $K$.
So $Q$ is disjoint from $\pi_x(K)$.
\medskip

One can choose $p$ and $I$ such that any iterates $P_{-t}(I)$ has diameter smaller than $\beta/2$.
Reducing $r_x$, one can ensure that the iterates of $P_{-t}(Q)$ have diameter smaller than $\beta$
until some time, where it stays close to the orbit of $I$ and has diameter smaller than $2\sup_t\diam(P_{-t}(I))$.
This concludes the proof of the lemma.
\end{proof}

\subsubsection{Wandering rectangles}\label{ss.wandering-rectangles}
One chooses $\delta\in (0,r_0/2)$,  $\beta,\varepsilon>0$, and a component $Q$
as in Lemma~\ref{l.unstable-set} such that
\begin{itemize}
\item[--] if $\varphi_{-t}(x)$, $t>1$, is $2\delta$-close to $x$, then
$Q\cap \pi_x(P_{-t}(Q))=\emptyset$;
\item[--] if $y,z\in K$ are close to $x$ and $\theta\in \lip$ satisfies
$d(\varphi_{\theta(t)}(y),\varphi_t(z))<\delta$ for $t\in [-2,0]$, then $\theta(0)-\theta(-2)>3/2$;
\item[--] $\beta>0$ associated to a shadowing at scale $\delta$ as in the Global invariance (Remarks~\ref{r.identification}.e);
\item[--] for any point $z$, the forward iterates of $\cW^{cs}_\varepsilon(z)$ have length smaller than $\beta/3$ and than
 the constant $\delta_E$ given by Lemma~\ref{l.summability-hyperbolicity}
 (this is possible since $\cE$ is topologically stable);
\item[--] the backward iterates of $Q$ have diameter smaller than $\beta/2$.
\end{itemize}

For $z$ close to $x$, one considers the closed curve $J(z)\subset \cW^{cs}(z)$ of length $\varepsilon$
bounded by $0_z$, such that $\pi_x(J(z))$ intersects $Q$.
Since $\pi_x(z)$ does not belong to $Q$ (by Lemma~\ref{l.unstable-set}), the unstable manifold of $0_p$
intersects $\pi_p(J(z))$, defining two disjoint arcs $J(z)=J^0(z)\cup J^1(z)$ such that $J^1(z)$ is bounded by $0_z$, disjoint from $Q$
and $\pi_x(J^0(z))\subset Q$.
\medskip

Let $H(z)$ denote the set of integers $n\geq 0$ such that $\varphi_n(z)\in K\setminus V$ and
$(z,\varphi_n(z))$ is $T_\cF$-Pliss. For $n\in H(z)$,
we set $J_n(z)=P_n(J(z))$ and $J^{0}_n(z)=P_n(J^{0}(z))$.

\begin{Lemma}\label{l.construction-rectangle}
There is $C_R>0$ with the following property.
For any $z$ close to $x$ and any $n\in H(z)$,
there exists a rectangle $R_n(z)\subset \cN_{\varphi_n(z)}$ which is the image of $[0,1]\times B_{d-1}(0,1)$ by a homeomorphism $\psi_n$ such that
(see Figure~\ref{f.topological-hyperbolicity}):
\begin{enumerate}

\item\label{i.construction-rectangle} $\text{Volume\;}(R_n(z))>C_R.|J^0_n(z)|$,
\item the preimages $P_{-t}(R_n(z))$ for $t\in [0,n]$ have diameter smaller than $\beta/2$,
\item $\pi_x\circ P_{-n}(R_n(z))$ is contained in $Q$.
\end{enumerate}
\end{Lemma}
\begin{figure}[ht]
\begin{center}
\includegraphics[width=10cm]{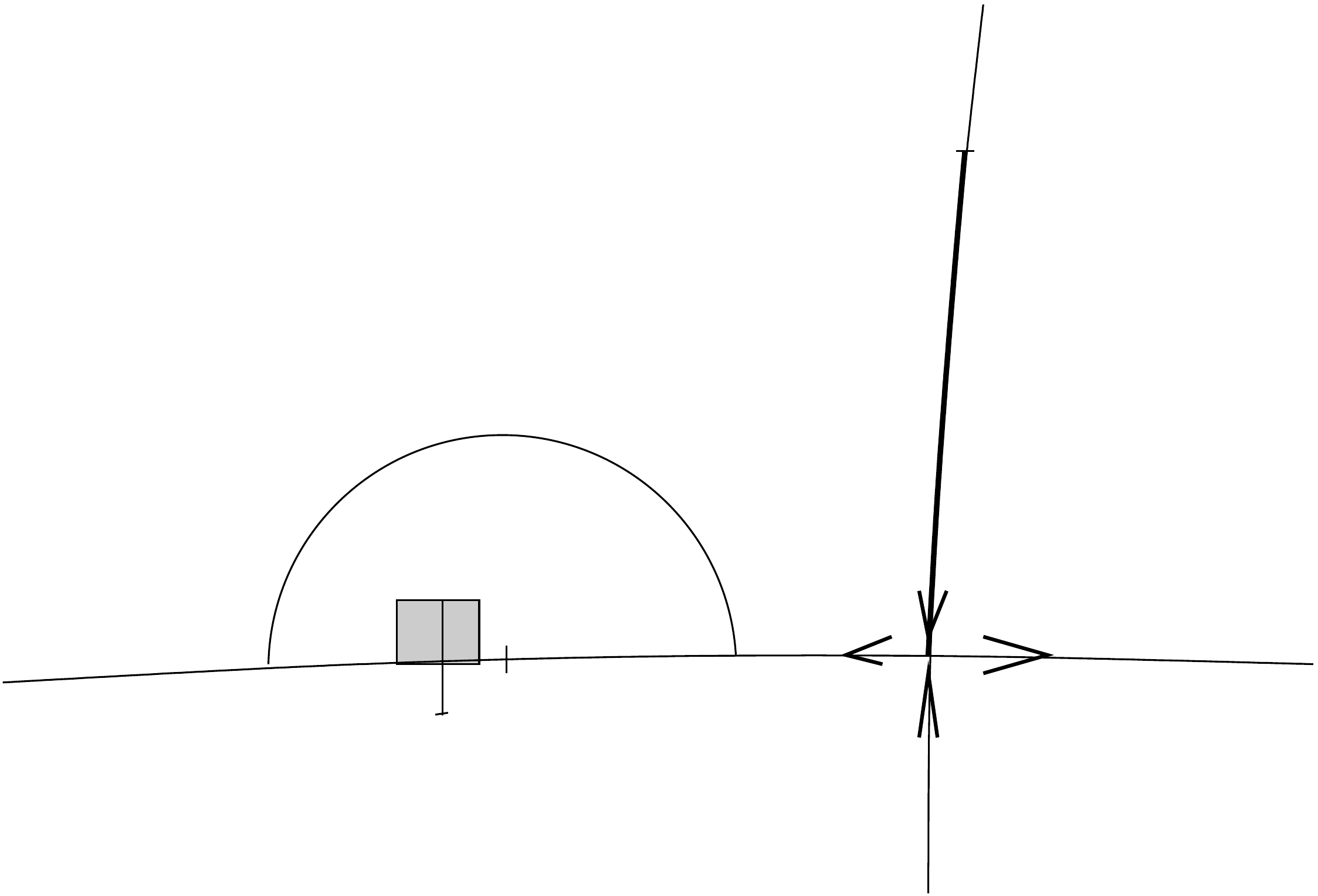}
\begin{picture}(0,0)
\put(-82,41){$p$}
\put(-75,120){$I$}
\put(-182,40){$\pi_p(x)$}
\put(-222,35){$\pi_p(z)$}
\put(-210,70){\scriptsize$\pi_p(R_0(z))$}
\put(-210,105){\scriptsize$\pi_p(Q)$}
\end{picture}
\end{center}
\caption{\label{f.topological-hyperbolicity}}
\end{figure}
\begin{proof}
The construction is very similar to the proof of Proposition~\ref{p.rectangle}.
Let us fix $\alpha'>0$ and $\alpha_{min}>0$ much smaller.
One considers a rectangle in $\cN_z$ whose interior projects by $\pi_x$ in $Q$,
given by a parametrization $\psi_0$ such that $\psi_0([0,1]\times \{0\})=J^0(z)$
and each disc $\psi_0(\{u\}\times B_{d-1})$ is tangent to the center-unstable cones,
has diameter smaller than $\alpha'$ and contains a center-unstable ball centered at $J^0(z)$
and with radius much larger than $\alpha_{min}$. Moreover $\pi_p\circ \psi_0(\{0_p\}\times B_{d-1})$ is contained in
the unstable manifold of $p$.

Since the center-unstable cone-field is invariant,
at any iterate $\varphi_n(z)$ such that $(z,\varphi_n(z))$ is $T_\cF$-Pliss,
one can build a similar rectangle $R_n$ such that $P_{-t}(R_n)$, $0<t<n$, has center-unstable disc
of diameter smaller than $\alpha'$ and $P_{-n}(R_n)\subset R_0$.

Since the forward iterates of $J^0(z)$ have length smaller than $\beta/3$,
by choosing $\alpha'$ small enough, one guaranties that the $P_{-t}(R_n)$, $0<t<n$,
have diameter smaller than $\beta/2$, and that $P_{-n}(R_n)\subset R_0(z)$ are contained in $\pi_z(Q)$.
The estimate on the volume is obtained from Fubini theorem and the distortion estimate
given by Proposition~\ref{p.distortion}.
\end{proof}

\begin{Lemma}\label{l.rectangle-disjoint}
For each $z$ close to $x$ and each $n< m$ in $H(z)$,
if $d(\varphi_n(z),\varphi_m(z))<\delta$, then $\pi_{\varphi_n(z)}(R_m(z)) \cap R_n(z)=\emptyset$.
\end{Lemma}
\begin{proof}
Let us assume by contradiction that $\pi_{\varphi_m(z)}(R_m(z))$ and $R_n(z)$ intersect.

Since $P_{-t}(R_n(z)\cup J_n(z))$, $t\in [0,n]$, and $P_{-t}(R_m(z)\cup J_m(z))$, $t\in [0,m]$, have diameter smaller than $\beta$,
the Global invariance (Remark~\ref{r.identification}.(e)) applies: there is $\theta\in \lip$ such that
\begin{itemize}
\item[--] $|\theta(n)-m|<1/4$,
\item[--] for any $t\in [0,n]$, one has $d(\varphi_{t}(z),\varphi_{\theta(t)}(z))<\delta$,
\item[--] $\pi_x\circ P_{\theta(0)-m}(R_m(z))$ intersects $\pi_x\circ P_{-n}(R_n(z))$, hence $Q$.
\end{itemize}
In particular $\theta(n)>n+1/2$ and Proposition~\ref{p.no-shear} gives $\theta(0)>2$.

Since the backward iterates of $Q$ by $P_t$ have diameter smaller than $\beta$,
the Global invariance (item (e) in Remarks~\ref{r.identification})
can be applied to the points $x$ and $\varphi_{\theta(0)}(z)$.
It gives $\theta'\in \lip$ with $|\theta'(\theta(0))|\leq 1/4$ such that
$d(\varphi_{\theta'(t)}(x),\varphi_{t}(z))<\delta$ for each $t\in [0,\theta(0)]$.
Moreover $\pi_x\circ P_{\theta'(0)}(Q)$ intersects $R_0(z)$, hence $Q$.
We have $d(\varphi_{\theta'(0)}(x),x)<2\delta$ and $1/4>\theta'(\theta(0))>\theta'(0)+3/2$ by our choice of $\delta$
at the beginning of Section~\ref{ss.wandering-rectangles}

We proved that $\pi_x\circ P_{-t}(Q)\cap Q\neq \emptyset$ for some $t>1$ such that $d(\varphi_{-t}(x),x)<2\delta$.
This contradicts the choice of $Q$. The rectangles $\pi_{\varphi_n(z)}(R_m(z))$ and $R_n(z)$ are thus disjoint.
\end{proof}

As a consequence of Lemma~\ref{l.rectangle-disjoint}, one gets

\begin{Corollary}\label{c.volume-bounded}
There exits $C_H>0$ such that for any $z$ close to $x$,
$$\sum_{n\in H(z)}|J^0_{n}(z)|< C_H.$$
\end{Corollary}
\begin{proof}
As in the proof of Lemma~\ref{Sub:disjointcase}, one fixes a finite set $Z\subset U$ such that
any point $z\in U$ is $\delta/2$-close to a point of $Z$. Let $C_{Vol}$ be a bound on the volume
of the balls $B(0_z,\beta_0)\subset \cN_z$ over $z\in K$ and let $C_H=2C_R^{-1}C_{Vol}\text{Card} (Z)$.
Since identifications are $C^1$, up to reduce $r_0$, one can assume that the modulus of their Jacobian is smaller than $2$.

The statement now follows
from the item~\ref{i.construction-rectangle} of Lemma~\ref{l.construction-rectangle} and the disjointness of the rectangles (Lemma~\ref{l.rectangle-disjoint}).
\end{proof}

\subsubsection{Summability. Proof of Proposition~\ref{Pro:smallperiodic-interval}}\label{ss.sum}

\begin{Lemma}\label{l.summability-topological-hyperbolicity}
There exists $C_{sum}>0$ such that for any point $z$ close to $x$,  
we have
$$\forall n\geq 0,~~~\sum_{k=0}^n |P_k(J^0(z))|< C_{sum}.$$
\end{Lemma}

\begin{proof}

Let us denote by $n_0=0<n_1<n_2<\dots$ the integers in $H(z)$.
By Proposition~\ref{l.summability}, for any $i$, the piece of orbit $(\varphi_{n_i}(z),\varphi_{n_{i+1}}(z))$
is $(C_{\cE},\lambda_{\cE})$-hyperbolic for $\cE$.

Let $\delta_E, C'_{\cE}$ be the constants associated to $C_{\cE},\lambda_{\cE}$ by Lemma~\ref{l.summability-hyperbolicity}.
We have built $J(z)$ such that any forward iterate has length smaller than $\delta_E$.
Hence Lemma~\ref{l.summability-hyperbolicity} implies that
$$\sum_{k=n_i}^{n_{i+1}} |P_k(J^0(z))|< C_{\cE}'|J^0_{n_{i+1}}|.$$
With Corollary~\ref{c.volume-bounded}, one deduces
$$\sum_{k=0}^n |P_k(J^0(z))|< C_{Sum}:=C_{\cE}'C_{H}.$$
\end{proof}

We can now end the proof of the proposition.

\begin{proof}[Proof of the Proposition~\ref{Pro:smallperiodic-interval}]

Let $\eta_S>0$ be the constant associated to $C_{Sum}$ by Lemma~\ref{Lem:schwartz}.
If $z$ belongs to a small neighborhood $U_0$ of $x$
and if $\varepsilon_0$ is small enough,
the intervals $\cW^{cs}_{\varepsilon_0}(z)$ and $J^0(z)$ are both contained in an interval $\widehat J(z)\subset \cW^{cs}(z)$
such that $|\widehat J(z)|\leq (1+\eta_S)|J^0(z)|$.
Combining Lemma~\ref{Lem:schwartz} with Lemma~\ref{l.summability-topological-hyperbolicity}, one gets
$$\sum_{k=0}^n |P_k(\cW^{cs}_{\varepsilon_0}(z))|\leq  \sum_{k=0}^n |P_k(\widehat J(z))|< 2\sum_{k=0}^n |P_k(J^0(z))|<C_0:=2C_{sum}.$$
\end{proof}

\subsubsection{Proof of Theorem~\ref{Thm:topologicalcontracting}}\label{ss.conclusion-topological}
We consider the set $K$ as in the Theorem~\ref{Thm:topologicalcontracting}, and we assume by contradiction that
none of the three properties in the statement of the theorem holds.
In particular $K$ does not contain a normally expanded irrational torus.

\begin{Claim}
$K$ is transitive.
\end{Claim}
\begin{proof}
Since $\cE$ is not uniformly contracted, there exists an ergodic measure whose Lyapunov exponent along $\cE$ is non-negative.
Since $\cE$ is uniformly contracted on each proper invariant subset of $K$, the support of the measure coincides with $K$.
Hence 
$K$ is transitive.
\end{proof}

Let us fix $\delta>0$ arbitrarily small. Since $\cE$ is topological stable,
there is $\varepsilon>0$ such that for any $x\in K$ and any $t>0$, one has
$$P_t({\cal W}^{cs}_{\varepsilon}(x)) \subset {\cal W}^{cs}_{\delta}(\varphi_t(x)).$$
Since the topological contraction fails, there are
 $(x_n)$ in $K$, $(t_n)\to +\infty$ and $\chi>0$ such that
$$\chi<|P_{t_n}({\cal W}^{cs}_{\varepsilon}(x_n))| \text{ and } |P_{t}({\cal W}^{cs}_{\varepsilon}(x_n))|<\delta,~\forall t>0.$$
Let $I=\lim_{n\to\infty}P_{t_n}({\cal W}^{cs}_{\varepsilon}(x))$. It is a $\delta$-interval and by Proposition~\ref{Prop:dynamicsofinterval},
it is contained in the unstable set of a periodic $\delta$-interval since $K$ contains no normally expanded irrational tori.
This proves that $K$ admits arbitrarily small periodic intervals and Proposition~\ref{Pro:smallperiodic-interval} applies.

One gets a non-empty open set $U_0\subset K$ such that at any $z\in U_0$ a summability holds in the $\cE$ direction.
With Lemma~\ref{Lem:schwartz}, one deduces that for any $x\in U_0$,
$$\lim_{n\to\infty}\|DP_n|\cE(x)\|=0.$$
Now for any $z\in K$,
\begin{itemize}
\item[--] either there is $t>0$ such that $\varphi_t(z)\in U_0$ and then $\lim_{n\to\infty}\|DP_{n}|\cE(z)\|=0$;
\item[--] or the forward orbit of $z$ does not meet $U_0$,
then $\cE$ is contracted on the proper invariant compact set $\omega(z)$ and we also have $\lim_{n\to\infty}\|DP_{n}|{\cE(z)}\|=0$.
\end{itemize}
By using a compactness argument, one deduces that $\cE$ is uniformly contracted on $K$. This contradicts our assumptions on $K$.
Theorem~\ref{Thm:topologicalcontracting} is now proved.
\qed

\section{Markovian boxes}\label{s.markov}
We will build boxes with a Markovian property
for  $C^2$ local fibered flow having a dominated splitting
$\cN=\cE\oplus \cF$ with two-dimensional fibers, such that
$\cE$ is topologically contracted.
\medskip

\noindent
{\bf Standing assumptions.}
Keep assumptions (A1), (A2), (A3) of Section~\ref{s.topological-hyperbolicity}
and add furthermore:
\begin{enumerate}
\setcounter{enumi}{3} 
\item[(A4)] $\cE$ is topologically contracted and $\cF$ is one-dimensional,
\item[(A5)] there exists an ergodic measure $\mu$ for $(\varphi_t)$ whose Lyapunov exponent along $\cF$ is positive,
and whose support is not a periodic orbit and intersects $K\setminus \overline V$.
\end{enumerate}

\subsection{Existence of Markovian boxes}\label{ss.existence-boxes}
We fix a non-periodic point $x\in K\setminus \overline V$ in the support of $\mu$. In particular
taking $r_0$ small enough, the ball $U(x,r_0)$ centered at $x$ and with radius $r_0$ in $K$ is contained in $U$.
We also denote $\mu_x=(\pi_x)_{*}(\mu|_{U(x,r_0)})$ and fix some $\beta_x>0$.

\begin{Definition}
A \emph{box} $B\subset \cN_x$ is the image by a homeomorphism $\psi$ such that:
\begin{itemize}
\item[--] $\partial^{\cal F}B:=\psi(\{0,1\}\times [0,1])$ is $C^1$, tangent to $\cC^\cF$ (and called the \emph{${\cal F}$-boundary}),
\item[--] $\partial^{\cal E}B:=\psi([0,1]\times \{0,1\})$ is $C^1$, tangent to $\cC^\cE$ (and called the \emph{${\cal E}$-boundary}).
\end{itemize}
A \emph{center-stable sub-box} (resp. \emph{center-unstable sub-box})
is a box $B'\subset B$ such that $$\partial^{\cal F}B'\subset \partial^{\cal F}B
\text{ (resp. } \partial^{\cal E}B'\subset \partial^{\cal E}B \text{).}$$
\end{Definition}
In particular a box is a rectangle as defined in Section~\ref{ss.rectangle}.

\begin{Definition} Let us fix some constants $C_{\cF},\lambda_{\cF}>1$.
\smallskip

\noindent
A \emph{transition} between boxes $B,B'\subset \cN_x$ is defined by some $y\in K$ and $t>2$ such that:
\begin{itemize}
\item[--] $y$ and $\varphi_t(y)$ are $r_0/2$-close to $x$,
\item[--] $(y,\varphi_t(y))$ is $(C_{\cF},\lambda_{\cF})$-hyperbolic for $\cF$,
\item[--] $y$ projects by $\pi_x$ in the interior of $B$ and $\varphi_t(y)$ in the interior of $B'$.
\end{itemize}
\medskip

\noindent
Two boxes $B,B'$ have \emph{Markovian transitions} if
for any transition $(y,t)$ between $B$ and $B'$,
there exist a center-stable sub-box $B^{cs}\subset B$ and a center-unstable sub-boxes $B^{cu}\subset B'$
whose interior contain $\pi_x(y)$ and $\pi_x(\varphi_t(y))$ respectively, such that
$\pi_xP_t\pi_y(B^{cs})=B^{cu}.$
\end{Definition}
The remainder of Section~\ref{s.markov} is devoted to prove the following main result of this section.
\begin{Theorem}[Existence of Markovian boxes]\label{t.existence-box}
Under the assumptions above,
there exists a box $R$ in $B(0,\beta_x)\subset \cN_x$
whose interior has positive $\mu_x$-measure,
such that for any $C_{\cF},\lambda_{\cF}>1$ (defining the transitions)
and $\beta_{box}>0$,
the box $R$ contains finitely many boxes $B_1,\dots,B_k$ and $t_{box},\Delta_{box}>0$
satisfying the following properties:
\begin{enumerate}
\item\label{i.markov-boundary} Any 
$z$ which is $r_0/2$-close to $x$ and whose projection $\pi_x(z)$
belongs to $\partial^{\cE}R$ (resp. to $\partial^{\cF}R$)
has a forward orbit (resp. a backward orbit) which accumulates on a periodic orbit of $K$.
\item\label{box1} The boxes $B_1,\dots,B_k$ have disjoint interiors, their union contains $R\cap \pi_x(K)$,
and their boundaries have zero measure for $\mu_x$.
\item\label{box2} Any transition $(y,t)$ with $t>t_{box}$ between any boxes $B_i,B_j$ is Markovian.
\item\label{box3} The sub-boxes $B^{cs},B^{cu}$ associated to any transition $(y,t)$ with $t>t_{box}$
between any $B_i,B_j$ satisfy:
\begin{itemize}
\item[--] $Diam(P_s\circ \pi_y(B^{cs}))<\beta_x$ for any $s\in (0,t)$,
\item[--] $Diam(P_s\circ \pi_y(B^{cs}))<\beta_{box}$ for any $s\in (t_{box},t)$,
\item[--] $B^{cu}$ has distorsion bounded by $\Delta_{box}$.
\end{itemize}
\item\label{box4} For any two transitions $(y_1,t_1)$ and $(y_2,t_2)$ with $t_1, t_2>t_{box}$ such that
the interiors of the sub-boxes $B^{cu}_1$, $B^{cu}_2$ intersect, we have $B^{cu}_1\subset B^{cu}_2$ or $B^{cu}_2\subset B^{cu}_1$.

More precisely, $B^{cu}_2\subset B^{cu}_1$ holds if there exists $\theta\in \lip_2$ satisfying
\begin{itemize}
\item[--] $\theta(t_1)\geq t_2-2$,  $\theta^{-1}(t_2)\geq t_1-2$ and $\theta(0)\geq -1$,
\item[--] $d(\varphi_t(y_1), \varphi_{\theta(t)}(y_2))<r_0/2$ for
$t\in [0,t_1]\cap \theta^{-1}([0,t_2])$.
\end{itemize}
Up to exchange $(y_1,t_1)$ and $(y_2,t_2)$, there exists such a $\theta$ satisfying $|\theta(t_1)-t_2|\leq 1/2$.
\item\label{box5} For any two transitions $(y_1,t_1)$ and $(y_2,t_2)$ with $t_1, t_2>t_{box}$ such that
$y_1=y_2$ and $t_2>t_1+t_{box}$, we have $B^{cs}_2\subset B^{cs}_1$.
\end{enumerate}
\end{Theorem}

\subsection{Construction of boxes}
\subsubsection{Notations, choices of constants}\label{ss.constants}
One will consider some small numbers $\alpha_0,\eta, \alpha_x,\beta_x>0$, that are chosen in this order,
according to the properties stated in this subsection.
The constant $\alpha_0$ will bound the size of the plaques.
In the whole Section~\ref{s.markov}, one will work with generalized orbits $\bar u=(u(t))$ in the $\eta$-neighborhood of $K$.
The number $\alpha_x$ controls the hyperbolicity inside the center-unstable plaques.
At last, $\beta_x>0$ is the constant introduced at the beginning of Section~\ref{ss.existence-boxes}, that can be reduced to be much smaller
than the other quantities.

\paragraph{The plaques $\cW^{cs},\cW^{cu}$.}

\begin{Lemma}\label{l.plaques-box}

There exist center-stable and center-unstable plaques $\cW^{cs}(\bar u),\cW^{cu}(\bar u)$
that have length smaller than $\alpha_0$, depend continuously on $\bar u$ and satisfy moreover:
\begin{itemize}
\item[--] The center-unstable plaques are \emph{locally invariant}: there is $\alpha_\cF\in (0,\alpha_0)$ such that
$$\bar P_{-1}(\cW^{cu}_{\alpha_\cF}(\bar u))\subset \cW^{cu}(\bar P_{-1}(\bar u)).$$
\item[--] The center-unstable plaques are \emph{coherent}: the statement of Proposition~\ref{p.coherence}
holds for $\cW^{cu}$, the constants $\eta,\alpha_\cF$ (and the flow $(P_{-t})$).
\item[--] The center-stable plaques are
\emph{trapped for time $s>0$}: 
$$\forall s>0,~~~\bar P_s(\overline{\cW^{cs}(\bar u)})\subset \cW^{cs}(\bar P_s(\bar u)).$$
\item[--] The center-stable plaques are \emph{coherent}:
let $\bar u,\bar v$ be any generalized orbits  with $u(0)\in \cN_y$, $v(0)\in \cN_{y'}$ such that $y,y'$ are $r_0/2$-close to $x$
and the projection $(y(t))_{t\in \RR}$
of $\bar u$ in $K$ has arbitrarily large positive iterates in the $r_0$-neighborhood of $K\setminus V$;
if $\pi_x (\cW^{cs}(\bar u))$ and $\pi_x (\cW^{cs}(\bar v))$ intersect, then they are contained in a same $C^1$ curve.
\end{itemize}
\end{Lemma}

\begin{proof}
The first item is given by Proposition~\ref{t.generalized-plaques}.
It gives also a  locally-invariant center-stable plaque family $\cW^{cs,0}$. The second item is obtained by Proposition~\ref{p.coherence} directly.

Since $\cE$ is topologically contracted over the $0$-section,
there exists $\varepsilon>0$ and $T_*>0$
such that $\bar P_s(\cW^{cs,0}_\varepsilon(\bar u))\subset \cW^{cs,0}(\bar P_s(\bar u))$ for any $s>0$ and which is trapped for times $s\geq T_*$.
Let us choose $b>0$ small and define
$$
\cW^{cs}(\bar u)=\bigcup_{0\leq t\leq T_*} 
\bar P_t(\cW^{cs,0}_{\varepsilon+bt}(\bar P_{-t}(\bar u))).
$$
These plaques are open sets in $\cW^{cs,0}$ and we have to check
the trapping property at any time $s>0$. Note that it is enough to choose $s\in (0,T_*)$.

Let us consider $t\in [0,T_*]$. In the case $t\geq T_*-s$ we set $t'=s+t-T_*$ and we have
$$\bar P_s(\bar P_t(\cW^{cs,0}_{\varepsilon+bt}(\bar P_{-t}(\bar u)))
=\bar P_{t'}\circ \bar P_{T_*}(\cW^{cs,0}_{\varepsilon+bt}(\bar P_{-t-s}\circ \bar P_s(\bar u))).$$
By the trapping property at time $T_*$, provided $b$ has been chosen small enough,
one has
$$\bar P_{T_*}(\cW^{cs,0}_{\varepsilon+bt}(\bar P_{-t-s}\circ \bar P_s(\bar u))\subset
\cW^{cs,0}_{\varepsilon+bt'-bs}(\bar P_{T_*-s-t}\circ \bar P_s(\bar u)).$$
Hence $\bar P_s(\bar P_t(\cW^{cs,0}_{\varepsilon+bt}(\bar P_{-t}(\bar u)))$
is contained in $\bar P_{t'}(\cW^{cs,0}_{\varepsilon+bt'-bs}(\bar P_{-t'}\circ \bar P_s(\bar u)))$.

{In the case $t< T_*-s$, we set $t'=t+s>t$ and we have
\begin{eqnarray*}
\bar P_s(\bar P_t(\cW^{cs,0}_{\varepsilon+bt}(\bar P_{-t}(\bar u)))
&=&\bar P_{s+t}(\cW^{cs,0}_{\varepsilon+bt}(\bar P_{-(t+s)}\circ \bar P_s(\bar u)))\\
&=& \bar P_{t'}(\cW^{cs,0}_{\varepsilon+bt'-bs}(\bar P_{-t'}\circ \bar P_s(\bar u)))
\end{eqnarray*}}
The closure of $\cup_{0\leq t'\leq T_*} 
\bar P_{t'}(\cW^{cs,0}_{\varepsilon+bt'-bs}(\bar P_{-t}\circ \bar P_s(\bar u)))$ is contained in
$\cW^{cs}(\bar P_s(\bar u))$.

Combining these two cases, one thus gets the third item:
$$\forall s>0,~~~\bar P_s({\rm Closure}(\cW^{cs}(\bar u)))\subset \cW^{cs}(\bar P_s(\bar u)).$$

For the fourth item, one uses the Local injectivity: since the plaque are small, one can assume that $y,y'$ are $\eta$-close.
Then Proposition~\ref{p.coherence} applies and gives the coherence.
\end{proof}

\paragraph{The constants $C_x,\lambda_x$.}
For $C_x,\lambda_x>1$, we introduce the set $\cH$ of points $y\in K$ that are $(C_x/2,\lambda_x^2)$-hyperbolic for $\cF$.
Since the Lyapunov exponent of $\mu$ along $\cF$ is positive,
we can fix $C_x,\lambda_x$ such that
the following set has positive $\mu$-measure, for any $\beta_x>0$:
$$H_x:=\{y\in \cH,\; d(y,x)<r_0/2 \text{ and } \pi_x(y)\in B(0_x,\beta_x)\}.$$
The choice of $\beta_x$ will be fixed later.

\paragraph{The constant $\alpha_x$.}
Up to reduce $\eta>0$, there exists $\alpha_{x}>0$ (depending on $C_x,\lambda_x$) such that
for any generalized orbit $\bar u$, if
$\bar u$ is $(C_x,\lambda_x)$-hyperbolic for $\cF$, then the set
$\bar P_{-t}(\cW^{cu}_{\alpha_x}(\bar u))$ is defined for any $t>0$, has
a diameter smaller than $\min(\eta,\alpha_\cF)\lambda_x^{-t/2}$
(see Proposition~\ref{p.unstable})
and is contained in $\cW^{cu}(\bar P_{-t}(\bar u))$ (by invariance, Remark~\ref{r.plaque-invariance}).
\medskip

Up to reduce $\eta,\alpha_x$, a stronger coherence for center-unstable plaques is satisfied:

\begin{Lemma}\label{l.coherence-alpha}
Consider two generalized orbits $\bar u,\bar v$ and $t\geq 0$ such that:
\begin{itemize}
\item[--] $u(0)\in \cN_y$, $v(-t)\in \cN_{y'}$ with $y,y'$ $r_0/2$-close to $x$,
and  the projection $(y(t))_{t\in \RR}$ of $\bar u$ in $K$ has arbitrarily large negative
iterates in the $r_0$-neighborhood of $K\setminus V$,
\item[--] $\bar v$ is $(C_x,\lambda_x)$-hyperbolic for $\cF$ and
$\pi_x(u(0))\in \pi_x(\bar P_{-t}(\cW^{cu}_{\alpha_{x}}(\bar v)))$.
\end{itemize}
Then $\pi_x(\bar P_{-t}(\cW^{cu}_{\alpha_{x}}(\bar v)))\subset \pi_x(\cW^{cu}(\bar u))$.
\end{Lemma}
\begin{proof}
This is a direct consequence of Proposition~\ref{p.coherence} (and the second item
of Lemma~\ref{l.plaques-box}) applied to the sets
$X=\{u(0)\}$ and $X'=\bar P_{-t}(\cW^{cu}_{\alpha_{x}}(\bar v))$.
Indeed, the hyperbolicity of $\bar v$ ensures that the diameter
of the sets $\bar P_{-s}(X')$ is smaller than $\alpha_\cF$ for any $s\geq 0$.
\end{proof}

The choice of $\eta$ is fixed now.
One can build generalized orbits as local product between generalized orbits
and orbits of $K$.
Up to reduce $\alpha_x$ and $\beta_x$, one gets:

\begin{Lemma}\label{l.product1}
For any generalized orbit $\bar p$
contained in the $\eta/2$-neighborhood of $K$,
and for any $y\in K$, $t\geq 0$ satisfying
\begin{itemize}
\item[--] $y$ and $z$ are $r_0/2$-close to $x$, where $z\in K$ is the point such that $p(-t)\in \cN_z$,
\item[--] $\bar p$ is $(C_x,\lambda_x)$-hyperbolic for $\cF$
and $p(-t-s)=P_{1-s}(p(-t-1))$ for any $s\in (0,1]$,
\item[--] $\pi_x(\cW^{cs}_{\beta_x}(y))$ and $\pi_x(\bar P_{-t}(\cW^{cu}_{\alpha_x}(\bar p)))$ intersect
and $d(z,y)<\eta/2$,
\end{itemize}
then there exists a generalized orbit $\bar u$ (in the $\eta$-neighborhood of $K$) satisfying:
\begin{itemize}
\item[--] $u(s)\in \cW^{cs}(\varphi_s(y))$ for every $s\geq 0$,
\item[--] $\pi_x(u(0))\in \pi_x(\bar P_{-t}(\cW^{cu}_{\alpha_x}(\bar p)))$
and $u(-s)\in \bar P_{-t-s}(\cW^{cu}_{\alpha_x}(\bar p))$ for $s>0$.
\end{itemize}
\end{Lemma}
\begin{Remark}\label{r.product}
In the case $d(z,y)<\eta/2$ does not hold, one can choose by the Local injectivity some $\tau\in [-1/4,1/4]$
such that $d(\varphi_\tau(z),y)<\eta/2$. One can then define in a same way a generalized orbit
which satisfies
$u(-s)\in \bar P_{-t-s+\tau}(\cW^{cu}_{\alpha_x}(\bar p))$ for $s>\max(0,\tau)$.
\end{Remark}
\begin{proof}
Define $u(0)$ to be the intersection point between $\cW^{cs}_{\beta_x}(y)$ and
$\pi_y(\bar P_{-t}(\cW^{cu}_{\alpha_x}(\bar p)))$. Since $\beta_x$ is small, the topological contraction of $\cE$
implies that $|u(s)|<\eta$ for each $s\geq0$, where
 $u(s):=P_s(u(0))$.
The generalized flow $\bar P$ associated to the generalized orbit $\bar p$ (see Definition~\ref{d.generalized-flow})
allows to define $u(s):=\bar P_{s}(\pi_z(u(0)))$ for each $s< 0$.
Note that $u(s)$ belongs to $\bar P_{s}(\cW^{cu}_{\alpha_x}(\bar p))$, whose diameter is smaller
than $\eta$ for any $s< 0$, provided $\alpha_x$ is chosen small enough since
$\bar p$ is $(C_x,\lambda_x)$-hyperbolic for $\cF$.
Note also that the projection of $(u(s))_{s\in [-1,0)}$ on $K$ is continuous.

We have thus defined in this way a generalized orbit, whose projection on $K$ coincides with the projection of $\bar p$
for times $s<0$ and with $\varphi_s(y)$ for times $s\geq 0$. 
In order to check that it is contained in the $\eta$-neighborhood of $K$, it remains to
show that:
\begin{itemize}
\item[--] $y\in K\setminus V$: this follows from our assumptions,
\item[--] the projection of $p(s)$ on $K$ for $s<0$ close to $0$
is $\eta$-close to $y$: this follows from the fact that $\bar p$ is in the $\eta/2$-neighborhood of $K$
and that $d(z,y)<\eta/2$.
\end{itemize}
\end{proof}
\paragraph{The constant $\beta_x$.}
The constant $\beta_x$ chosen for Theorem~\ref{t.existence-box} can be reduced to be smaller than
$\alpha_0,\eta,\alpha_x$ and to satisfy (using the trapping):
\begin{itemize}
\item[--] for any $y\in K$ and $t\geq 1/4$ such that
$y, \varphi_t(y)$ are $r_0/2$-close to $x$, the two components of
$\pi_x(\cW^{cs}(\varphi_t(y))\setminus P_t(\cW^{cs}(y)))$ have length larger than $2\beta_x$.
\end{itemize}

\paragraph{Other constants.}
Once 
$R$ is constructed, one can choose other numbers $C_{\cF},\lambda_{\cF},\beta_{box}$
as in Theorem~\ref{t.existence-box}.
Another constant $\alpha_{box}>0$ (which is a relaxed analogue of $\alpha_x$) will be introduced later in Section~\ref{ss.moreconstants},
depending on these choices.

\subsubsection{A shadowing lemma}

\begin{Proposition}\label{p.shadowing}
For any $\delta>0$, there exist $r>0$ and $T_0\geq 1$ such that
for any points $y,\varphi_T(y)\in H_x$ that are $r$-close with $T\geq T_0$, there exists
$p\in \cN_x$ such that
\begin{itemize}
\item[--] $P_s(\pi_y(p))$ is defined and is contained in $B(0_{\varphi_s(y)},\delta)\subset \cN_{\varphi_s(y)}$ for any $s\in (0,T)$,
\item[--] $p$ is fixed by $\widetilde P_T:=\pi_x\circ P_T \circ \pi_y$.
\end{itemize}
\end{Proposition}
We could give an argument which uses the domination $\cE\oplus \cF$, similar to
the construction in~\cite[Proposition 9.6]{CP}.
We propose here a more topological proof.
\begin{proof}
Let us choose $\varepsilon>0$ much smaller than $\delta$.
\medskip

\begin{Claim}
If $T_0$ is large enough,
for any $y,\varphi_T(y)\in H_x$ with $T\geq T_0$,
there exists a box $B\subset \cN_y$ such that:
\begin{itemize}
\item[--] $P_s(B)$ is a box of $B(0,\delta)$ in $\cN_{\varphi_s(y)}$ for any $s\in [0,T]$,
\item[--] $B$ contains $\cW^{cs}_{\varepsilon}(y)$ and  $P_T(B)$ contains $\cW^{cu}_{\varepsilon}(\varphi_T(y))$.
\end{itemize}
\end{Claim}
\begin{proof}
Since $\cE$ is topologically contracted,
if $\varepsilon$ has been chosen small, the iterates $P_s(\cW^{cs}_\varepsilon(y))$
are smaller than $\delta/10$ for any $s\in [0,T]$.
We can choose two disjoint arcs $L^-_0,L^+_0$ of length $1$,
tangent to $\cC^\cF$, centered at the endpoints
of $\cW^{cs}_\varepsilon(y)$ and disjoint from $\cW^{cu}(y)$.
Let us consider $L^-\subset L^-_0$ (resp. $L^+\subset L^+_0$) the maximal arc whose iterates by $P_{s}$,
$s\in [0,T]$, remain at distance smaller than $\delta$ from $0_{\varphi_s(y)}$.
Since $(y,\varphi_T(y))$ is $(C_x/2,\lambda_x^2)$ hyperbolic for $\cF$,
and since the endpoints of $P_s(\cW^{cs}_\varepsilon(y))$ are close to
$0_{\varphi_s(y)}$, we deduce that $P_T(L^-)$ and $P_T(L^+)$
have length larger than $10\varepsilon$.

Let us note that $P_T(L^-)$ and $P_T(L^+)$ are disjoint from
$\cW^{cu}_\varepsilon(\varphi_T(y))$: otherwise
$L^-$ (or $L^+$) would intersect $P_{-T}(\cW^{cu}_\varepsilon(\varphi_T(y)))$,
but these three curves have a length arbitrarily small
if $T$ is large (by hyperbolicity along $\cF$) and contain respectively
the endpoints and the center of $\cW^{cs}_\varepsilon(y)$
which are separated by a uniform distance (of order $\varepsilon$).

We then build two disjoint curves $J^-,J^+$
through the endpoints of $\cW^{cu}_\varepsilon(\varphi_T(y))$,
tangent to $\cC^\cE$, disjoint from $\cW^{cs}(\varphi_T(y))$
and connecting $P_T(L^-)$ to $P_T(L^+)$.
The curves $P_{-T}(J^-)$ and $P_{-T}(J^+)$
are still tangent to $\cC^\cE$, so that with $L^-,L^+$ they bound a box
$B$ with the required properties.
The claim is thus proved.
\end{proof}

Let us choose $r$ small.
If $y,\varphi_T(y)\in H_x$ are $r$-close with $T\geq T_0$,
the claim can be applied and moreover the projection of the boxes $\pi_x(B)$ and $\pi_x(P_T(B))$
intersect. If $T_0$ has been chosen large enough, using the uniform expansion along $\cF$
and the topological contraction along $\cE$, one deduces that $B$, $P_T(B)$ are contained in small
neighborhoods of $\cW^{cs}(y)$ and $\cW^{cu}(P_T(y))$ respectively.
Consequently, the union of their projection on $\cN_x$ is diffeomorphic to
$([-2,2]\times [-1,1])\cup ([-1,1]\times [-2,2])$ in $\RR^2$:
$\partial^{\cF}B$ is identified with $\{-2,2\}\times [-1,1]$
and $\partial^{\cF}P_T(B)$ is identified with $\{-1,1\}\times [-2,2]$.

The map $\widetilde P_T:=\pi_xP_T\pi_y$ is defined from $\pi_x(B)$ to $\pi_x(P_T(B))$.
We can deform continuously the restriction $\widetilde P_T\colon \partial( \pi_x(B))\to \partial(\pi_x\circ P_T(B))$
so that it coincides after deformation with the restriction
of a linear map $A\colon [-2,2]\times [-1,1]\to [-1,1]\times [-2,2]$
where $A=\begin{pmatrix}
\pm 1/2 & 0 \\
0 & \pm 2
\end{pmatrix},$
proving that the degree of the map
$$\Theta\colon z\mapsto {\widetilde P_T(z)-z}/{\|\widetilde P_T(z)-z\|}$$
from $\partial( \pi_x(B))$ to $S^1$ (for the canonical orientations
of $\partial( \pi_x(B))\subset \RR^2$ and $S^1$) is non-zero.
This proves that $\widetilde P_T$ has a fixed point $p$:
{otherwise,} one can consider the degree of $\Theta$
on each circle $D_t:=\partial([-2t,2t]\times [-t,t])$ for $t\in (0,1]$; it does not depend on $t$, hence is non-zero
and it is a contradiction since for $t$ small the disk bounded by $D_t$ is disjoint from its image.

The point $\pi_y(p)$ belongs to $B$; hence any iterate $P_s(\pi_y(p))$, $s\in [0,T]$, is
contained $B(0,\delta)\subset \cN_{\varphi_s(y)}$ by construction of $B$.
\end{proof}

\subsubsection{Construction of the box $R$}\label{ss.boxR}

The box $R$ is built from the next proposition.
It implies the item~\ref{i.markov-boundary} of Theorem~\ref{t.existence-box}.

\begin{Proposition}\label{p.boxR}
For $\mu$-almost every point $y\in H_x$,
and any $\eta_x>0$, there exist a generalized orbit $\bar p=(p(t))$ which is periodic
and has two iterates $\bar p_1,\bar p_2$ satisfying:
\begin{itemize}
\item[--] $\bar p$ is contained in the $\eta_x$-neighborhood of $K$,
\item[--] $\bar p_1,\bar p_2$ are $\eta_x$-close to $0_y$ and $(C_x,\lambda_x)$-hyperbolic for $\cF$.
\end{itemize}
The box $R\subset \cN_x$ bounded by
the four curves $\pi_x(\cW^{cs}(\bar p_i))$ and $\pi_x(\cW^{cu}_{\alpha_x}(\bar p_i))$, where $i\in \{1,2\}$,
has
\begin{itemize}
\item[--] interior $\operatorname{Int}(R)$ with positive $\mu_x$-measure,
\item[--] boundary $\partial(R)$ with zero $\mu_x$-measure.
\end{itemize}
Moreover if $z\in K$ is $r_0/2$-close to $x$ and $i\in \{1,2\}$, then
\begin{itemize}
\item[--] if $\pi_x(z)\in \pi_x(\cW^{cs}(\bar p_i))$, the forward orbit of $z$ accumulates on a periodic orbit of $K$.

\item[--] if $\pi_x(z)\in \pi_x(\cW^{cu}_{\alpha_x}(\bar p_i))$, the backward orbit of $z$ accumulates on a periodic orbit of $K$.
\end{itemize}

\end{Proposition}
Since $\bar p_1,\bar p_2$ are arbitrarily close to $y\in H_x$, we have $R\subset B(0_x,\beta_x)$.
{A point $p(t)$ in the generalized orbit $\bar p$ is a \emph{return of $\bar p$ at $x$} if its projection in $K$ is $r_0/2$-close to $x$.}

\smallskip

\begin{proof}[Proof of Proposition~\ref{p.boxR}]
Recall the coherences of the center-stable plaques $\cW^{cs}(\bar u)$ defined for generalized orbits
and the center-unstable plaques $\cW^{cu}_{\alpha_x}(\bar u)$
at $(C_x,\lambda_x)$-hyperbolic points for $\cF$.

\begin{Lemma}\label{l.measure}
Consider any generalized orbit $\bar u$.
If the projection $(y(t))$ of $(u(t))$ on $K$ has arbitrarily large negative iterates in the $r_0$-neighborhood of $K\setminus V$
and if $d(y(0),x)<r_0/2$, then, the projection $\pi_x(\cW^{cs}(\bar u))$ has zero $\mu_x$-measure.

If $\bar u$ is $(C_x,\lambda_x)$-hyperbolic for $\cF$, then
the same holds for the projection $\pi_x(\cW^{cu}_{\alpha_x}(\bar u))$.
\end{Lemma}
\begin{proof}
Assume by contradiction that $\pi_x(\cW^{cs}(\bar u))$ has positive $\mu_x$-measure:
there exists a measurable set $A\subset K$ such that
$\pi_x(A)\subset \pi_x(\cW^{cs}(\bar u))$ and $\mu(A)>0$.
Hence there exist a positively recurrent point $z\in A$ and arbitrarily large $T>0$ such that
$\varphi_T(z)$ belongs to $A$, is arbitrarily close to $z$
and $P_T(\cW^{cs}(z))$ has arbitrarily small diameter (by topological hyperbolicity of $\cE$).
Since $y,z, \varphi_T(z)$ are $r_0/2$-close to $x$,
the coherence implies that  $\pi_x(\cW^{cs}(z))$, $\pi_x\circ P_T(\cW^{cs}(z))$
and $\pi_x(\cW^{cs}(u))$ are all contained in a same $C^1$-curve.
Hence, $\pi_x(\cW^{cs}(z))$ contains $\pi_x\circ P_T(\cW^{cs}(z))$.
We have proved that $\cW^{cs}(z)$ contains $\widetilde P_T(\cW^{cs}(z))$,
where $\widetilde P_T=\pi_z\circ P_T$,
so that the sequence $(\widetilde P_T^k(0_z))$ converges to a fixed point of $\widetilde P_T$ contained in
$\cW^{cs}(z)$. By Corollary~\ref{c.closing0}, the orbit of $z$ converges to a periodic orbit of $K$.
This is a contradiction since $\mu$-almost every point $z$ in $A$ has a forward orbit which is dense
in the support of $\mu$, which is not a periodic orbit by assumption.

The proof for center-unstable plaques is similar.
If there exists a measurable set $A\subset K$ such that
$\pi_x(A)\subset \pi_x(\cW^{cu}_{\alpha_x}(\bar u))$ and $\mu(A)>0$.
Since the Lyapunov exponent of $\mu$ along $\cF$ is positive,
up to reduce the set $A$, there exists $C'>0$, $\lambda'>1$ such that
any point $z\in A$ is $(C',\lambda')$-hyperbolic for $\cF$.
By Proposition~\ref{p.unstable}, there exists  $\beta>0$ such that
$P_{-t}(\cW^{cu}_{\beta}(z))$ is defined for any $t>0$ and has
a diameter smaller than $\alpha_{\cF}\lambda^{-t/2}$.
One ends the argument by considering
$z$ and $\varphi_{-T}(z)$ arbitrarily close to $z$.
By the coherence in Lemma~\ref{l.coherence-alpha}, the plaque $\cW^{cu}_{\alpha_x}(\bar u)$ is contained in
$\cW^{cu}(z)$ and in $\cW^{cu}(\varphi_T(z))$. Hence if $T$ is large enough,
$P_{-T}(\cW^{cu}_{\beta}(z))$ is arbitrarily small and contained
in $\cW^{cu}_{\beta}(z)$. We conclude as before.
\end{proof}

Let us build a first approximation of $R$.
\begin{Lemma}\label{l.first-box}
$\mu$-almost every point $y\in H_x$ has arbitrarily large iterates $\varphi_t(y)\in H_x$
close to $y$ such that the projection of the four plaques
$\cW^{cs}(y)$, $\cW^{cs}(\varphi_t(y))$, $\cW^{cu}_{\alpha_x}(y)$, $\cW^{cu}_{\alpha_x}(\varphi_t(y))$
by $\pi_x$ in $\cN_x$ bound a small box $R_y\subset B(0_x,\beta_x)$ whose measure for $\mu_x$ is positive.
\end{Lemma}
\begin{proof}
Let us choose $y\in H_x$ whose forward and backward orbits have dense sets of iterates in $H_x$
and such that $\pi_x(\{z\in H_x,\; d(z,y)<r\})$ has positive $\mu_x$-measure for any $r>0$.

For $r>0$ small, let us consider the four connected components of
$B(\pi_x(y),r)\setminus \pi_x( \cW^{cs}(y))\cup \pi_x (\cW^{cu}_{\alpha_x}(y))$.
Since $\pi_x( \cW^{cs}(y)), \pi_x( \cW^{cu}_{\alpha_x}(y))$ have zero measure for $\mu_x$,
for one of these connected components $Q$, the measure
$\mu_x(Q\cap H_x\cap B(\pi_x(y),r'))$ is positive for any $r'\in (0,r)$.

Choose $\varphi_t(y)\in H_x$ close to $y$ in $Q$. By the coherence, the plaques
$\cW^{cs}(y)$, $\cW^{cs}(\varphi_t(y))$ have disjoint projection by $\pi_x$;
by Lemma~\ref{l.coherence-alpha}, the same holds for the plaques $\cW^{cu}_{\alpha_x}(y)$, $\cW^{cu}_{\alpha_x}(\varphi_t(y))$.
Hence they bound a small box $R_y\subset B(0_x,\beta_x)$ whose measure for $\mu_x$ is positive.
\end{proof}

\noindent
\emph{End of the construction of the box $R$.}
By Lemma~\ref{l.first-box}, for $\mu$-almost every point $y\in H_x$, there exists
$t>0$ large such that the projection of the four plaques
$\cW^{cs}(y)$, $\cW^{cs}(\varphi_t(y))$, $\cW^{cu}_{\alpha_x}(y)$, $\cW^{cu}_{\alpha_x}(\varphi_t(y))$
by $\pi_x$ in $\cN_x$ bound a small box $R_y\subset B(0_x,\beta_x)$ with positive $\mu_x$-measure.

We then choose $T>0$ much larger than $t$ such that $\varphi_T(y)\in H_x$ is very close to $y$
and we apply Proposition~\ref{p.shadowing}.
We get a point $p\in \cN_x$ arbitrarily close to $\pi_x(y)$
and the repetition of the piece of orbit
$\{p(s)=P_s(\pi_y(p)), s\in [0,T)\}$ gives a generalized periodic orbit $\bar p$
which is $\eta_x$-close to $K$.
Since $y$ and $\varphi_t(y)$ are $(C_x/2,\lambda_x^2)$-hyperbolic for $\cF$,
one deduces from Lemma~\ref{l.cont4} that $\bar p$ and $\bar P_t(\bar p)$ are $(C_x,\lambda_x)$-hyperbolic
for $\cF$ and have projection by $\pi_x$ close to $\pi_x(y)$
and $\pi_x(\varphi_t(y))$.

The box $R\subset B(0_x,\beta_x)$ bounded by the projection of the four plaques
$\cW^{cs}(\bar p)$, $\cW^{cs}(\bar P_t(\bar p))$, $\cW^{cu}_{\alpha_x}(\bar p)$, $\cW^{cu}_{\alpha_x}(\bar P_t(\bar p))$
by $\pi_x$ in $\cN_x$ is close to $R_y$,
hence has positive $\mu_x$-measure.
By Lemma~\ref{l.measure} the boundary of $R$ has zero $\mu_x$-measure.
\medskip

Assume $z$ is a point satisfying $\pi_x(z)\in \pi_x(\cW^{cs}(\bar p_i))$.
There exists $s\in [-1/4,1/4]$ such that $z':=\varphi_s(z)$ is $r_0/2$-close to $y$
and still satisfies $\pi_y(z')\in \cW^{cs}(\bar p_i)$.
By the Global invariance, there exists
$T'>0$ such that
the forward orbit of $z'$ under $\pi_{z'}\circ P_{T'}$ is semi-conjugated by $\pi_y$
with the forward orbit of $\pi_y(z')$ under $\pi_{y}\circ P_{T}$. The later converges to the fixed point $p_i:=p_i(0)$ of the orbit
$\bar p_i=(p_i(t))$.
Corollary~\ref{c.closing0} applies and
implies that the forward orbit of $z'$ and $z$ converges to a periodic orbit.

A similar argument holds when $\pi_x(z)\in \pi_x(\cW^{cu}_{\alpha_x}(\bar p_i))$.
This concludes Proposition~\ref{p.boxR}.
\end{proof}

\begin{Remark}\label{r.lengthR}
We can choose the diameter of the rectangle $R$ 
 much smaller than $\beta_x$.
In particular, by topological hyperbolicity of $\cE$,
if $y$ is $r_0/2$-close and $\pi_x(y)\in \Interior(R)$, then
any arc $I\subset \cW^{cs}(y)$ satisfying $\pi_x(I)\subset R$
has forward iterates of length much smaller than $\beta_x$.

Moreover, for any $\beta_{box}>0$, there exists a uniform time $t_1>0$ such that for any such $y,I$
the length of $|P_t(I)|$ is much smaller than $\beta_{box}$ when $t>t_1$.
\end{Remark}

\subsubsection{New choices of constants}\label{ss.moreconstants}
In the previous section we have built the box $R$.
Before building the boxes $B_1,\dots,B_k$, 
we introduce $C_{\cF},\lambda_{\cF},\beta_{box}$ as
in the statement of Theorem~\ref{t.existence-box}, and another constant
 $\alpha_{box}>0$. One can reduce these numbers in order to satisfy the following properties:
\begin{itemize}
\item[--] $C_{\cF},\lambda_{\cF}$:
by relaxing constants, one can require that
for any $i\in\{1,2\}$ and $s\in \RR$,
any generalized orbit $\bar u$
satisfying $u(-t)\in \bar P_{-t-s}(\cW^{cu}_{\alpha_x}(\bar p_i))$ for any $t\geq 0$
is $(C_{\cF},\lambda_{\cF})$-hyperbolic for $\cF$.
\item[--] $\beta_{box}$: Proposition~\ref{p.distortion} associates to $C_{\cF},\lambda_{\cF}$ some
constants $\Delta,\beta$. We can reduce $\beta_{box}$ so that $\beta_{box}<\beta$.
\item[--] $\Delta_{box}$: it is chosen so that the projection by the local diffeomorphisms
$\pi_x$ of any box with distortion $\Delta$ is a box with distortion $\Delta_{box}$. 
\item[--] $\alpha_{box}$: one chooses $\alpha_{box}$ small so that the two following properties are satisfied
(for the same reasons as in Section~\ref{ss.constants} for choosing $\alpha_x$).

\item[] \emph{Backwards contraction.} For any generalized orbit $\bar u$
which is $(2C_{\cF},{\lambda_{\cF}}^{1/2})$-hyperbolic for $\cF$,
the set $\bar P_{-t}(\cW^{cu}_{\alpha_{box}}(\bar u))$
is defined for any $t\geq 0$, is contained in $\cW^{cu}(\bar P_{-t}(\bar u))$
and has diameter smaller than $\min(\beta_{box},\alpha_\cF)\lambda_{\cF}^{-t/2}$.

\item[--] \emph{Coherence.} Consider two generalized orbits $\bar u,\bar v$ and $t\geq 0$ such that:
\begin{itemize}
\item[--] $u(0)\in \cN_y$, $v(-t)\in \cN_{y'}$ with $y,y'$ $r_0/2$-close to $x$,
and the projection $(y(t))_{t\in \RR}$ of
$\bar u$ in $K$ has arbitrarily large negative iterates in the $r_0$-neighborhood of $K\setminus V$,
\item[--] $\bar v$ is $(2C_{\cF},\lambda_{\cF}^{1/2})$-hyperbolic for $\cF$ and
$\pi_x(u(0))\in \pi_x(\bar P_{-t}(\cW^{cu}_{\alpha_{box}}(\bar u)))$.
\end{itemize}
Then $\pi_x(\bar P_{-t}(\cW^{cu}_{\alpha_{box}}(\bar u)))\subset \pi_x(\cW^{cu}(\bar v))$.
\end{itemize}

\subsubsection{Construction of the sub-boxes $B_1,\dots, B_k$}\label{ss.subbox}

\begin{Proposition}\label{p.subbox}
There exists $\delta>0$ and finitely many sub-boxes $B_1,\dots,B_k\subset R$ with disjoint interiors,
whose union contains $R\cap \pi_x(K)$, whose boundary has zero $\mu_x$-measure,
{having the following properties.}

\begin{enumerate}
\item[(i)]\label{subbox1} \emph{Geometry.} If $\gamma_1,\gamma_2$ are the two components of $\partial^{\cal F}(B_j)$, then
$$d(\gamma_1,\gamma_2)>10.\max(\diam(\gamma_1),\diam(\gamma_2)).$$
\item[(ii)]\label{subbox2} \emph{${\cal F}$-boundary.}
Any component $\gamma$ of $\partial^{\cal F}(B_j)$
coincides with $\pi_x(I)$ of an arc $I$ contained in $\pi_x(\bar P_{-t}(\cW^{cu}_{\alpha_{x}}(\bar p_{i})))$, $i\in \{1,2\}$,
where $\bar P_{-t}(\bar p_i)$ is a return of $\bar p_i$ at $x$.
Moreover, when it is defined, $\pi_x\circ \bar P_{-s}(I)$ for $s\geq 0$ is disjoint from all the
$\Interior(B_\ell)$, $\ell\in \{1,\dots,k\}$.

\item[(iii)]\label{subbox2b} \emph{${\cal E}$-boundary.}
Any component $\gamma$ of $\partial^{\cal E}(B_j)$
satisfies one of the following properties:
\begin{itemize}
\item[--] $\pi_x(\cW^{cs}(y))$ does not intersect the $\delta$-neighborhood of $\gamma$ for any $y\in \pi_x(K)\cap B_j$;
in particular $\gamma\cap \pi_x(K)=\emptyset$.
\item[--] $\gamma$ is the projection by $\pi_x$ of an arc $I$ contained in the center-stable plaque $\cW^{cs}(\bar q)$
of a periodic generalized orbit $\bar q$.  Moreover when it is defined, $\pi_x\circ \bar P_{s}(I)$ for $s\geq 0$ is disjoint from all the
$\Interior(B_\ell)$, $\ell\in \{1,\dots,k\}$.
\end{itemize}

\item[(iv)]\label{subbox4} \emph{Coherence with plaques.}
Consider $y$ that is $r_0/2$-close to $x$ such that $\pi_x(y)\in B_i$.
Then $\pi_x(\cW^{cs}(y))\cap B_i$ is an arc connecting the two components of 
$\partial^{\cal F} B_i$.

Consider $y$ that is $r_0/2$-close to $x$ and a generalized orbit $\bar u$ with
$u(0)\in \cN_y$, such that
$\bar u$ is $(2C_{\cF},\lambda_{\cF}^{1/2})$-hyperbolic for $\cF$ and  $\pi_x(\bar u)\in B_i$.
Then $\pi_x(\cW^{cu}_{\alpha_{box}}(\bar u))\cap B_i$ is an arc connecting the two components of
$\partial^{\cal E} B_i$.
\end{enumerate}
\end{Proposition}

This subsection is devoted to prove Proposition~\ref{p.subbox}, which implies Item~\ref{box1} of Theorem~\ref{t.existence-box}.
\medskip

\noindent
\emph{Unstable curves.}
Recall that $R$ has been built from the plaques of $\bar p_1,\bar p_2$ in the orbit of $\bar p$.
By the Local invariance, one chooses a finite set of iterates $\bar p_1=\bar p$, $\bar p_2=\bar P_{t_2}(\bar p)$, $\bar p_3=\bar P_{t_3}(\bar p)$,\dots,
of the generalized orbit $\bar p$ such that:
\begin{itemize}
\item[--] Each $\bar p_i$ is a return of $\bar p$ at $x$.
\item[--] Consider a return $\bar P_{-t}(\bar p)$ of $\bar p$ at $x$.
Then there exists $\bar p_j$ and $s\in [-1/4,1/4]$ such that $\bar P_{-t}(\bar p)=\bar P_{s}(\bar p_j)$.
\item[--] If $\bar P_{-t}(\bar p_i)=\bar p_j$ for $|t|\leq 1/4$, then $i=j$.
\end{itemize}
One considers some $C^1$-curves $\gamma^u_i\subset \cW^{cu}(\bar p_i)$ and $\delta_1>0$ with the following properties:
\begin{itemize}
\item[(a)] The union of the curves $\pi_x(\gamma^u_i)$ is a $C^1$-submanifold.
\item[(b)] $\partial^{\cal F}R\subset \pi_x(\gamma^u_1)\cup \pi_x(\gamma^u_2)$.
\item[(c)] Each $\gamma^u_i$ locally coincides with $\bar P_{-s}(\cW^{cu}_{\alpha_x}(\bar p_1))$
or $\bar P_{-s}(\cW^{cu}_{\alpha_x}(\bar p_2))$ for some $s\geq 0$.
\item[(d)] If $\bar p_i=\bar P_{-s}(\bar p_j)$ for some $s\geq 0$ and $i\neq j$ (so that $s\geq 1/4$),
then $\bar P_{-s}(\gamma^u_j)\subset \gamma^u_i$.

Moreover $\pi_x(\gamma^u_i\setminus \bar P_{-s}(\gamma^u_j))$
is the union of two arcs of length larger than $\delta_1$.

\item[(e)] For each $i$ and each $t\geq 0$ such that $\bar P_{-t}(\bar p_i)$ is a return of $\bar p_i$ at $x$,
there exists $j$ and $s\in [t-1/4,t+1/4]$ such that $\bar p_j=\bar P_{-s}(\bar p_i)$
and $\pi_x\circ \bar P_{-s}(\gamma_i^u)=\pi_x\circ \bar P_{-t}(\gamma_i^u)$.
\end{itemize}

Property (d) shows that if one chooses sub-curves $\tilde \gamma_i^u\subset \gamma_i^u$
such that the components of $\pi_x(\gamma_i^u\setminus \tilde \gamma_i^u)$ have length smaller than $\delta_1$,
then the inclusion $\bar P_{-s}(\tilde \gamma^u_j)\subset \tilde \gamma^u_i$ still holds when
$\bar p_i=\bar P_{-s}(\bar p_j)$ for some $s\geq 0$ and $i\neq j$.
\medskip

Let us explain how to build it.
Let $\sigma_1\subset \cW^{cu}_{\alpha_x}(\bar p_1)$ such that
$\pi_x(\sigma_1)$ is a component of $\partial^{\cal F} R$.
The returns of the backward iterates of $\bar p_1$ by $\bar P_t$
inside the finite set $\{\bar p_1,\bar p_2,\dots\}$
defines an infinite periodic sequence $\bar z_1,\bar z_2,\cdots$.
There exists a minimal $s_k>0$ such that $\bar P_{-s_k}(\bar z_k)=\bar z_{k+1}$.

We inductively define $\sigma_k$ as the curve in $\cW^{cu}(\bar z_k)$
such that $\pi_x(\sigma_k)$ is the $\delta_1$-neighborhood of
$\pi_x( \bar P_{-s_k}(\sigma_{k-1}))$.
We then define $\gamma_i^1$ as the union of the $\sigma_k$ such that $\bar z_k=\bar p_k$.
Since $\bar P_{-s}(\cW^{cu}_{\alpha_x}(\bar p_1))$ decreases exponentially as $s\to +\infty$,
by choosing $\delta_1$ small enough,
we get:
\begin{itemize}
\item[--] For any $i$, there is $s\geq 0$ such that $\gamma_i^1\subset \bar P_{-s}(\cW^{cu}_{\alpha_x}(\bar p_1))$.
\item[--] If $\bar p_i=\bar P_{-s}(\bar p_j)$ for some $s\geq 0$ and $i\neq j$,
then $\bar P_{-s}(\gamma^1_j)\subset \gamma^1_i$ and $\pi_x(\gamma^1_i\setminus \bar P_{-s}(\gamma^1_j))$
is the union of two arcs of length larger than $\delta_1$.
\end{itemize}
We repeat the same construction starting with the curve in $\cW^{cu}_{\alpha_x}(\bar p_2)$ which projects
to the other component of $\partial^{\cal F} R$. One obtains another family of curves $\gamma^2_i$.
One then set $\gamma^u_i=\gamma^1_i\cup \gamma^2_i$ so that items (b), (c), (d) are satisfied.

In order to check item (a), it is enough to notice that (from item (c) and Lemma~\ref{l.coherence-alpha},
the union of any two curves $\pi_x(\gamma^u_i)\cup \pi_x(\gamma^u_j)$ is a $C^1$-submanifold.

Let us prove item (e): if $\bar P_{-t}(\bar p_i)$ is a return of $\bar p_i$ at $x$,
by definition of the family $\bar p_1,\bar p_2,\bar p_3,\dots$,
there exists $j$ and $s\in [t-1/4,t+1/4]$ such that $\bar P_{-s}(\bar p_i))=\bar p_j$.
By the Local invariance, $\pi_x\circ P_{-t}(\gamma^u_i)=\pi_x\circ \bar P_{-s}(\gamma^u_i)$.

\medskip

\noindent
\emph{Stable curves.}
Let us consider $\delta_2>0$ much smaller than $\min( \delta_1,\alpha_{box})$. One also considers $\varepsilon,\delta_3>0$ that will be fixed later.

Choose a point $y$ in a subset of full $\mu$-measure of $H_x\cap \pi_x^{-1}(\Interior(R))$ and
a forward iterate $\varphi_T(y)\in H_x$ close to $y$ with $T>0$ large.
By Proposition~\ref{p.shadowing}, we build a periodic generalized orbit $\bar q$ which is $\varepsilon$-close
to the zero-section of $\cN$ for the Hausdorff distance. As for the generalized orbit $\bar p$,
one chooses a finite set of iterates $\bar q_1$, $\bar q_2$, $\bar q_3$,\dots,
of the generalized orbit $\bar q$ such that:
\begin{itemize}
\item[--] Each $\bar q_i$ is a return of $\bar q$ at $x$.
\item[--] Assume that $\bar P_{-t}(\bar q)$ is a return of $\bar q$ at $x$.
Then there exists $\bar q_j$ and $s\in [-1/4,1/4]$ such that $\bar P_{-t}(\bar q)=\bar P_{s}(\bar q_j)$.
\item[--] If $\bar P_{-t}(\bar q_i)=\bar q_j$ for $|t|\leq 1/4$, then $i=j$.
\end{itemize}
We denote by $\{z_1,\dots,z_m\}$ the collection of $\bar p_i$, $\bar q_j$ and build curves $\gamma_1^s,\dots,\gamma_m^s$ such that:
\begin{itemize}
\item[(a')] The union of the curves $\pi_x(\gamma^s_i)$ is  a $C^1$-submanifold.
\item[(b')] $\partial^{\cal E}R\subset \pi_x(\gamma^s_1)\cup \pi_x(\gamma^s_2)$.
\item[(c')] Each $\gamma^s_i$ is contained in $\cW^{cs}(z_i)$.
\item[(d')] If $z_i=\bar P_{s}(z_j)$ for some $s\geq 0$ and $i\neq j$,
then $\bar P_{s}(\gamma^s_j)\subset \gamma^s_i$.
\item[(e')] For each $i$ and each $t\geq 0$ such that $\bar P_{-t}(z_i)$ is a return of $z_i$ at $x$,
there exists $j$ and $s\in [t-1/4,t+1/4]$ such that $z_j=\bar P_{-s}(z_i)$
and $\pi_x\circ \bar P_{-s}(\gamma_i^s)=\pi_x\circ \bar P_{-t}(\gamma_i^s)$.
\item[(f')] If $\gamma^s_i\cap R$ is non-empty, then it is an arc which connects the two components of $\partial^{\cal F} R$.
\end{itemize}

The center-stable plaques $\cW^{cs}(z_i)$ satisfy the properties of items (a') to (e') above, by coherence, trapping,
the Local invariance and definition of $R$. Note that these properties are still satisfied if one replace
$\cW^{cs}(z_i)$ by a sub-arc $\gamma^s_i$ such that the components of
$\pi_x(\cW^{cs}(z_i)\setminus \gamma^s_i)$ have length smaller than $\beta_x$.
(For the property (d') this comes from the choice of $\beta_x$ and the fact that  $z_i=\bar P_{s}(z_j)$ for some $s\geq 0$ and $i\neq j$
implies $s\geq 1/4$.) By construction, $R$ is contained in $B(0,\beta_x)$, so one can find such sub-arcs $\gamma_i^s$ which satisfy (f').
\medskip

\noindent
\emph{Strips.}
The set $\Interior (R)\setminus (\pi_x(\gamma_1^s)\cup\dots\cup\pi_x(\gamma^s_m))$
has finitely many connected components, whose closures are center-stable sub-boxes of $R$
that we call strips.

We distinguish two kinds of strips:
\begin{itemize}
\item[--] \emph{thin strips}: strips whose ${\cal E}$-boundaries are $\delta_2$-close to each-other:
any $C^1$-curve in the strip which connects the two components of the ${\cal E}$-boundary and which is
tangent to $\cC^\cF$ has length smaller than $\delta_2$,
\item[--] \emph{thick strips}: the other ones.
\end{itemize}

\begin{Lemma}\label{l.real-rectangle} The minimal distance between the components of the ${\cal E}$-boundary of the thick strips
is bounded away from $0$ (uniformly in the choice of the periodic orbit $\bar q$).
\end{Lemma}
\begin{proof}
Otherwise, there exist $z_i,z_j$ and $\gamma^s_i,\gamma_j^s$ whose projections by $\pi_x$ have two points arbitrarily close, and
there exists a transverse arc tangent to $\cC^\cF$
of length larger than $\delta_2$ connecting these two curves.

Taking the limit, one gets two points $\bar z,\bar z'$
which still belong to generalized orbits ($\eta$-close to the $0$-section)
whose center-stable plaques intersect but are not contained in a same
$C^1$-submanifold tangent to $\cC^\cE$. This contradicts the coherence.
\end{proof}

The next lemma fixes the constant $\varepsilon>0$.
\begin{Lemma}\label{l.substrip}
Let us choose $\delta_3>0$. If $\varepsilon$ is small enough,
then for any thick strip $S$ and any $y\in K$ which is $r_0/2$-close to $x$
and satisfies $\pi_x(y)\in S$, $\pi_x(\cW^{cs}(y))\cap S$ is $\delta_3$-close to $\partial^{\cal E}S$.
\end{Lemma}
\begin{proof}
First we show that when $\varepsilon$ goes to zero, the distance between $\pi_x(K)\cap S$ and $\partial^{\cal E} S$
goes to zero.
Otherwise there exists a thick strip $S$ and a point $y\in K$ that is $r_0/2$-close to $x$
such that $\pi_x(y)$ belongs to $S$ and is at a distance from $\partial^{\cal E}S$, larger than a
uniform constant $e>0$. For $\varepsilon>0$ small enough, there exists a point
$z_i$ close to $y$, defining a curve $\gamma_i^s$ which intersects $S$ but
is at a bounded distance from $\partial^{\cal E} S$. This contradicts the definition
of the strips $S$ as the closures of connected components of
$\Interior(R)\setminus (\pi_x(\gamma^s_1)\cup\dots\cup\pi_x(\gamma^s_m))$.

We then conclude that all the set $\pi_x(\cW^{cs}(y))\cap S$ is close to $\partial^{\cal E}S$.
Otherwise, one can take a limit $\bar y$ of such points $y$ and a limit $\gamma$
of components of $\partial^{\cal E} S$, such that $\pi_x(\cW^{cs}(\bar y))$ and
$\gamma$ intersect at $\pi_x(\bar y)$ but are not contained in a $C^1$-curve.
Since $\gamma$ is contained in a center-stable plaque, this contradicts the coherence.
\end{proof}

By Lemmas~\ref{l.real-rectangle} and~\ref{l.substrip}, we take $\delta_3\in (0,\delta_2)$ smaller than half of the minimal distance between the components
of $\partial^{\cal E} S$ of any thick strips $S$ and choose $\varepsilon$ such that in each thick strip $S$,
the sets $\pi_x(\cW^{cs}(y))\cap S$, for any $y\in S$, is at a distance to $\partial^{\cal E} S$ smaller than $\delta_3$.
This allows to build in each thick strip S, two disjoint center-stable sub-boxes (sub-strips) $S_-,S_+$ such that:
\begin{itemize}
\item[--] $S_-\cup S_+$ contains $\pi_x(K)\cap S$,
\item[--] the two components of $\partial^{\cal E} S_-$ (resp. $\partial^{\cal E} S_+$) are $\delta_2$-close,
\item[--] one component of $\partial^{\cal E} S_-$ (resp. $\partial^{\cal E} S_+$)
coincides with a component of $\partial^{\cal E}S$, the other one is disjoint from the $\delta_3$-neighborhood of $\partial^{\cal E} S$.
\end{itemize}
In particular, there exists $\delta>0$ such that for any thick strip $S$ and any $y\in \pi_x(K)\cap S$, the plaque $\pi_x(\cW^{cs}(y))$
is disjoint from the $\delta$-neighborhood of the component $\gamma$ of $\partial^{\cal E}S_{\pm}$ which is not
contained in $\partial^{\cal E}S$.
\medskip

\noindent
\emph{The sub-boxes $B_0,\dots,B_k$.}
One chooses sub-curves $\tilde \gamma^u_i\subset \gamma_i^u$
such that the components of $\pi_x(\gamma_i^u\setminus \tilde \gamma_i^u)$ have length smaller than $\delta_1$
and whose endpoints do not belong to the interior of thin strips nor to boxes $S_\pm$ associated to a thick strip $S$.

The sub-boxes $B_0,\dots,B_k$ are obtained from a thin strips $S_0=S$ or a sub-strips $S_0\in\{S_-,S_+\}$ as follow:
We consider the connected components
of
$$\Interior(S_0)\setminus (\pi_x(\tilde \gamma^u_1)\cup\dots\cup \pi_x(\tilde \gamma^u_m))$$ and take their closures.
They  have disjoint interiors
and their union contains
$R\cap \pi_x(K)$. The item (i) holds by the choice of $\delta_2$, much smaller than the distance between pair of curves $\gamma^u_i$.

Each component of the ${\cal F}$-boundary of these boxes is the projection $\pi_x(I)$ of an arc $I$ contained in
a curve $\tilde \gamma^u_i$. In particular the ${\cal F}$-boundary has zero $\mu_x$-measure.
Moreover, by the properties on the curves $\tilde \gamma^u_i$ and the Local invariance,
for each return $\bar P_{-s}(\bar p_i)$ at $x$, $s\geq 0$,
the iterate $\pi_x(\bar P_{-s}(I))$ is contained in a curve $\tilde \gamma^u_j$, hence is disjoint
from the interior of the boxes $B_\ell$. The item (ii) is thus satisfied.

For each component $\gamma$ of the ${\cal E}$-boundary of these boxes $B_\ell$,
either the $\delta$-neighborhood of $\gamma$ is disjoint from $\pi_x(\cW^{cs}(y))$
for any $y\in B_\ell$,
or $\gamma$ is the projection $\pi_x(I)$ of an arc $I$ contained in a curve $\tilde \gamma^s_i$.
In particular the ${\cal E}$-boundary has zero $\mu_x$-measure.
Moreover, in this second case by the properties on the curves $\tilde \gamma^s_i$ and Local invariance,
for each return $\bar P_{-s}(\bar z_i)$ at $x$, $s\geq 0$,
the iterate $\pi_x(\bar P_{-s}(I))$ is contained in a curve $\tilde \gamma^s_j$, hence is disjoint
from the interior of the boxes $B_\ell$.
The item (iii) is thus satisfied.

We have proved that the boundary of the boxes has zero $\mu_x$-measure.
\medskip

\noindent
\emph{Coherence with plaques.}
Consider $y\in K$ that is $r_0/2$-close to $x$ and such that $\pi_x(y)\in B_i$.
Let $\gamma$ be a component of $\partial^{\cal E}B_i$. There are two cases.
\begin{itemize}
\item[--] If $\gamma$ is contained in a curve $\gamma_i^s$ (hence in a center-stable plaque),
then by coherence, $\pi_x(\cW^{cs}(y))$ is disjoint from or contains $\gamma$.
\item[--] Otherwise $B_i$ is built from a sub-box $S_-$ or $S_+$ of a thick strip $S$,
and $\gamma$ is disjoint from $\pi_x(\cW^{cs}(y'))$ for any $y'\in \pi_x(K)\cap B_i$.
So $\pi_x(\cW^{cs}(y))\cap B_i$ is disjoint from $\gamma$.
\end{itemize}
We have obtained the first part of item (iv).

In order to check the second part, one recalls that
the components of the ${\cal F}$-boundary of each box $B_i$ are in a curve $\pi_x(\bar P_{-s}(\cW^{cu}_{\alpha_x}(\bar p_i)))$,
where $\bar p_i$ is $(C_x,\lambda_x)$-hyperbolic for $\cF$ by item (c) above; the length of all their backwards iterates of
$\cW^{cu}_{\alpha_x}(\bar p_i)$ remain small.
If $y$ is $r_0/2$-close to $x$ and $\bar u$ is a generalized orbit with $u(0)\in \cN_y$,
such that $\pi_x(\bar u)\in B_i$ and $\bar u$ is $(2C_{\cF},\lambda_{\cF}^{1/2})$-hyperbolic for $\cF$,
then the lengths of all their backwards iterates of $\cW^{cu}_{\alpha_{box}}(\bar u)$ remain small.
Hence Lemma~\ref{l.coherence-alpha} implies that the union of $\pi_x(\cW^{cu}(\bar u))$
and of $\partial^{\cal F}B_i$ is a submanifold. Consequently,
each component of $\partial^{\cal F}B_i$
is either disjoint from or contained in $\pi_x(\cW^{cu}(\bar u))$.
Since the distance between the two components of $\partial^{\cal E} B_i$ is much smaller than $\alpha_{box}$,
the curve $\pi_x(\cW^{cu}_{\alpha_{box}}(\bar u))$ meets both of them.
This gives the second part of item (iv) and thus completes the proof of Proposition~\ref{p.subbox}. \qed

\subsection{The Markovian property}

We have proved the items~\ref{i.markov-boundary} and~\ref{box1} of Theorem~\ref{t.existence-box} in
Sections~\ref{ss.boxR} and~\ref{ss.subbox}. We now prove the other items.

\paragraph{Items~\ref{box2} and~\ref{box3} of Theorem~\ref{t.existence-box}.} Let us consider a transition $(y,t)$ between two
sub-boxes $B,B'\in \{B_1,\dots,B_k\}$
(associated to the constants $C_{\cF},\lambda_{\cF}$)
and such that $t>t_{box}$, where $t_{box}$ is a large constant to be chosen later.
By item (iv) of Proposition~\ref{p.subbox}, there exists an arc $I\subset \cW^{cs}(y)$
containing $0_y$, whose projection by $\pi_x$ is contained in $B$
and connects the two components of the ${\cal F}$-boundary of $B$.

\begin{Lemma}
The image $\pi_x(P_t(I))$ is contained in $B'$.
\end{Lemma}
\begin{proof}
Otherwise, by item (iv) of Proposition~\ref{p.subbox},
there exists $u'\in P_t(I)$ which is not an endpoint and
which projects by $\pi_x$ inside the ${\cal F}$-boundary of $B'$.
Hence by the item (iii) of Proposition~\ref{p.subbox} and the Global invariance,
$u:=P_{-t}(u')$ projects by $\pi_x$ is contained in $\partial^{\cal F}(B)$.
Since $u$ is not an endpoint of $I$, the arc $\pi_x(I)$ is not contained in $B$.
This is a contradiction.
\end{proof}

\begin{Lemma}
There exists $T_1$ such that provided $t>T_1$,
each endpoint of $I$ belongs to a generalized orbit $\bar u$
such that $\bar P_{t}(\bar u)$ is $(2C_{\cF},{\lambda_{\cF}}^{1/2})$-hyperbolic for $\cF$.
\end{Lemma}
\begin{proof}
Let $u$ be an endpoint of $I$. Note that $u\in \cW^{cs}_{\beta_x}(y)$.
By item (ii) of Proposition~\ref{p.subbox},
$\pi_x(u)$ belongs to $\pi_x(\bar P_{-\tau}(\cW^{cu}_{\alpha_{x}}(\bar p_i)))$, for some return $\bar P_{-\tau}(\bar p_i)$, $\tau>0$ of
$\bar p_i$, $i=1,2$ at $x$. By definition of the $\bar p_i$, one can assume that there is no discontinuity in the orbit
$\bar p_i(-\tau-s)$, with $s\in [-1,0)$. 
Lemma~\ref{l.product1} and Remark~\ref{r.product} apply and define the generalized orbit $\bar u$.
It is $(C_{\cF},\lambda_{\cF})$-hyperbolic for $\cF$ by our choice of $C_{\cF}$ and $\lambda_{\cF}$.

For any $\varepsilon>0$ small, there exist uniform $T_0,C_0>0$ such that the following holds.
\begin{itemize}
\item[--] Since $\cE$ is topologically contracted,
the points $P_s(y)$ and $u(s)$ are close for any $s>T_0$.
Moreover, $(y,\varphi_t(y))$ is $(C_{\cF},\lambda_{\cF})$-hyperbolic for $\cF$. Hence, for $s\in (T_0,t)$,
$$\|DP_{s-t}{|\cF}(u(t))\|\leq \|DP_{s-t}{|\cF}(\varphi_t(y))\|(1+\varepsilon)^{(t-s)}\leq C_{\cF}\lambda_{\cF}^{s-t}(1+\varepsilon)^{(t-s)},$$
\item[--] $\|DP_{T_0-s}{|\cF}(u(T_0))\|\leq C_0$ for any $s\in [0,T_0]$.
\end{itemize}
So $\bar P_{t}(\bar u)$ is $(2C_{\cF},{\lambda_{\cF}}^{1/2})$-hyperbolic for $\cF$ if $T_1$
satisfies
$$C_0\lambda_{\cF}^{T_0}(1+\varepsilon)^{-T_0}<(1+\varepsilon)^{-T_1}{\lambda_{\cF}}^{T_1/2}.$$
\end{proof}

Let $\bar u$ and $\bar v$ be the generalized orbits associated to each endpoint of $I$.
By coherence (stated in Section~\ref{ss.moreconstants}), the projection of the plaques $\cW^{cu}_{\alpha_{box}}(\bar P_{t}(\bar u))$, $\cW^{cu}_{\alpha_{box}}(\bar P_{t}(\bar v))$
are disjoint and by item (iv) of Proposition~\ref{p.subbox} they cross $B'$.
We thus obtain a center-unstable sub-box $B^{cu}\subset B'$ bounded by
these curves.

For all $0\leq s\leq t$, the iterates $P_{-s}\circ \pi_{\varphi_t(y)}(B^{cu})$
are contained in $B(0, 2\alpha_0)\subset \cN_{\varphi_{t-s}(y)}$, where $\alpha_0$
is an upper bound on the size of the plaques $\cW^{cs}$ and $\cW^{cu}$.
Moreover, the four edges remain tangent to the cones $\cC^\cE$ or $\cC^\cF$.
We denote $B^{cs}:=\pi_x\circ P_{-t}\circ \pi_{\varphi_t(y)}(B^{cu})$.

\begin{Lemma}\label{l.diameter}
The sets $P_{s}(\pi_{y}(B^{cs}))$ have diameter smaller than $\beta_x$
for each $s\in [0,t]$. Their diameter is smaller than $\beta_{box}$
when $s$ is larger than some uniform constant $T_2$.
\end{Lemma}
\begin{proof}
There exists a uniform $C>0$ such that
the diameter of $P_{s}(\pi_{y}(B^{cs}))=P_{s-t}(\pi_{\varphi_{t}(y)}(B^{cu}))$
is smaller than $C\max(|P_{s}(I)|, |P_{s-t}(\partial^{\cal F}B^{cu})|)$.

By our choice of $\alpha_{box}$, the curve
$P_{s-t}(\cW^{cu}_{\alpha_{box}}(u(t)))$ has a size much smaller than $\beta_{box}$
and it contains $P_{s-t}(\partial^{\cal F}B^{cu})$.
By Remark~\ref{r.lengthR}, the length of $P_{s}(I)$ is much smaller than $\beta_x$,
implying the first property, since $\cE$ is topologically contracted. This length is much smaller than $\beta_{box}$
when $t$ is larger than a constant $T_2$, giving the second property.
\end{proof}

The following lemmas end the proof of item~\ref{box2} of Theorem~\ref{t.existence-box}.

\begin{Lemma}
When $t$ is larger than some $T_3$,
the $\delta$-neighborhood of $\pi_x(I)$ contains $B^{cs}$.
\end{Lemma}
\begin{proof}
For each point $z$ in $B^{cs}$, there exists a curve $\gamma\subset \pi_y(B^{cs})$ tangent to $\cC^\cF$
which connects $\pi_y(z)$ to a point $z'$ in $I$.
The iterates $P_s(\gamma)$ for $s\in [0,t]$ are still tangent to
$\cC^\cF$ and, for any such $s$ which is larger than some uniform $T$, one has:
\begin{itemize}
\item[--] $P_s(\gamma)$ is contained in $P_s(\pi_y(B^{cs}))$
hence in a small neighborhood of $0_{\varphi_s(y)}$ by Lemma~\ref{l.diameter},
\item[--] the tangent spaces to $P_s(\gamma)$ are close to $\cF({\varphi_s(y)})$.
\end{itemize}
The derivative of $P_{t-s}$ along $P_s(\gamma)$ and $\cF({\varphi_s(y)})$
can be compared. Since $(y,\varphi_t(y))$ is $(C_{\cF},\lambda_{\cF})$-hyperbolic for $\cF$,
$|P_s(\gamma)|$ is exponentially small in $t-s$ for $s\geq T$.
If $t$ is large enough, this implies that the length of $\gamma$ is exponentially small
in $t$. In particular it is smaller than $\delta$.
\end{proof}

\begin{Lemma}
When $t>T_3$, the box $B^{cs}$ is a center-stable sub-box of $B$.
\end{Lemma}
\begin{proof}
By coherence (stated in Section~\ref{ss.moreconstants}), the union of the ${\cal F}$-boundary of $B$ and $B^{cs}$
is contained in the union of two disjoint $C^1$-curves.
If one assumes by contradiction that $B^{cs}$ is not contained in $B$,
there exists a point $u$ in the ${\cal E}$-boundary of $B$ that belongs to $\Interior(B^{cs})$.
By item (iii) of Proposition~\ref{p.subbox}, two cases are possible:
\begin{itemize}
\item[--] $u$ belongs to the projection by $\pi_x$ of an arc $J\subset \cW^{cs}(\bar q)$ of a generalized periodic orbit $\bar q$ and no forward iterate of $J$
projects to the interior of any box $B_1,\dots,B_k$. This is a contradiction since
$\pi_x\circ P_t \circ\pi_y(u)$ belongs to the interior of $B^{cu}\subset B'$.
\item[--] $u$ is $\delta$-far from $\pi_x(\cW^{cs}(y))$:
it contradicts the previous lemma.
\end{itemize}
We have proved that $B^{cs}\subset B$ and $\partial^{\cal F} B^{cs}\subset \partial^{\cal F} B$ as required.
\end{proof}

The following ends the proof of the item~\ref{box3} of Theorem~\ref{t.existence-box}.
\begin{Lemma}
The distortion of $B^{cu}$
is bounded by $\Delta_{box}$.
\end{Lemma}
\begin{proof}
Since $B^{cs}$ is a center-stable sub-box of $B$
and since $B$ satisfies the item (i) of Proposition~\ref{p.subbox},
the box $B^{cs}$ satisfies this condition too.
Since $(y,\varphi_t(y))$ is $(C_{\cF},\lambda_{\cF})$-hyperbolic for $\cF$
and since the diameter of $\pi_y(B^{cu})$ is smaller than $\beta_{box}$,
Proposition~\ref{p.distortion}
and Remark~\ref{r.distortion} (and the choice of the constants $\Delta,\beta_{box}$)
imply that $\pi_y(B^{cu})$ has distortion bounded by $\Delta$.
From our choice of $\Delta_{box}$,
the box $B^{cu}$ has distortion bounded by $\Delta_{box}$.
\end{proof}

\begin{Lemma}\label{l.otherinclusion}
Consider $0<s<t$ such that $s$ and $t-s$ are larger than some $T_4$
and such that $\varphi_s(y)$ is $r_0/2$-close to $x$ and $\pi_x(\varphi_t(y))$ belongs to some box $B_i$.
Then the interior of $\pi_x\circ P_s\circ \pi_y(B^{cs})$
does not meet the ${\cal E}$-boundaries of $B_i$.
\end{Lemma}
\begin{proof}
Note that if $T_4$ is large enough, then the diameter of
$\pi_x\circ P_s\circ \pi_y(B^{cs})$ is arbitrarily small: indeed,
from the topological contraction of $\cE$, the distance between the two components of
the ${\cal F}$-boundary of $P_s\circ \pi_y(B^{cs})$ is arbitrarily small if $s$ is large enough.
These components are contained in the backward image of plaques $\cW^{cu}_{\alpha_{box}}(\bar P_{t}(\bar u))$, $\cW^{cu}_{\alpha_{box}}(\bar P_{t}(\bar v))$
whose lengths are exponentially small in $t-s$.

Assume by contradiction that the interior of $\pi_x\circ P_s\circ \pi_y(B^{cs})$
meets some component $\gamma$ of $\partial^{\cal E} B_i$.
By item (iii) of Proposition~\ref{p.subbox}, $\gamma$ satisfies one of the two next cases.
\begin{itemize}
\item[--] $\gamma$ is disjoint from the $\delta$-neighborhood of $\pi_x(K)\cap B_i$: it is a contradiction since
$\pi_x\circ P_s\circ \pi_y(B^{cs})$ contains $\pi_x(\varphi_s(y))\in \pi_x(K)\cap B_i$ and has arbitrarily small diameter.
\item[--] $\gamma$ is  the projection by $\pi_x$ of an arc $I$ contained in the center-stable plaque $\cW^{cs}(\bar q)$
of a periodic generalized orbit $\bar q$ and $\pi_x\circ \bar P_{\tau}(I)$ for $\tau\geq 0$ is disjoint from all the
$\Interior(B_\ell)$, $\ell\in \{1,\dots,k\}$. This is a contradiction since by the Global invariance,
there exists an iterate $\pi_x\circ \bar P_{\tau}(I)$ which intersects the interior of $B^{cu}$.
\end{itemize}
\end{proof}

We take $t_{box}$ equal to the supremum of the $T_i$ for $i=1,2,3,4$.

\paragraph{Proof of items~\ref{box4} and~\ref{box5} of Theorem~\ref{t.existence-box}.}
We consider two transitions $(y_1,t_1)$, $(y_2,t_2)$ with $t_1,t_2>t_{box}$ such that
the interiors of the boxes $B^{cu}_1$ and $B^{cu}_2$ intersect.

\begin{Lemma}\label{l.contained}
Assume that there exists $\theta\in \lip_2$
such that:
\begin{itemize}
\item[--] $\theta(t_1)\geq t_2-2$, $\theta^{-1}(t_2)\geq t_1-2$ and $\theta(0)\geq -1$,
\item[--] $d(\varphi_t(y_1),\varphi_{\theta(t)}(y_2))<r_0/2$
for $t\in [0,t_1]\cap \theta^{-1}([0,t_2])$.
\end{itemize}
Then $B^{cu}_2\subset B^{cu}_1$.
\end{Lemma}
\begin{proof}
Let us assume by contradiction that $\partial^\cF(B^{cu}_1)\cap{\rm Interior}(B^{cu}_2)\neq\emptyset$.
The two transitions are associated to boxes $B_1,B_2$ (containing $y_1,y_2$ respectively) and $B'$ containing both $B^{cu}_1$ and $B^{cu}_2$.
We also denote $y'_1:=\varphi_{t_1}(y_1)$ and $y'_2:=\varphi_{t_2}(y_2)$. Moreover we set
$[a,b]=[0,t_1]\cap \theta^{-1}([0,t_2])$.

By the Global invariance
$$\pi_x\circ P_{a-b}\circ \pi_{\varphi_b(y_1)}(B^{cu}_i)=\pi_x\circ P_{\theta(a)-\theta(b)}\circ \pi_{\varphi_{\theta(b)}(y_2)}(B^{cu}_i).$$
By our assumptions, $|\theta(b)-t_2|\leq 2$
and $|b-t_1|\leq 2$.
By the Local invariance,
$$ \pi_{\varphi_b(y_1)}(B^{cu}_i)=  P_{b-t_1}\circ \pi_{y'_1}(B^{cu}_i)
\text{ and }  \pi_{\varphi_{\theta(b)}(y_2)}(B^{cu}_i)=  P_{\theta(b)-t_2}\circ \pi_{y'_2}(B^{cu}_i).$$

Since $\theta$ is $2$-Lipschitz and $\theta(0)\geq -1$, we check that $|a|\leq 2$ and that $\varphi_a(y_1)$ is $r_0$-close to $x$.
Hence by the Local invariance,
$$\pi_x\circ P_{a-b}\circ \pi_{\varphi_b(y_1)}(B^{cu}_i)= \pi_x\circ P_{-b}\circ \pi_{\varphi_b(y_1)}(B^{cu}_i).$$
This shows that
$$\pi_x\circ P_{a-b}\circ \pi_{\varphi_b(y_1)}(B^{cu}_1)= B^{cs}_1
\text{ and }
\pi_x\circ P_{\theta(a)-\theta(b)}\circ \pi_{\varphi_{\theta(b)}(y_2)}(B^{cu}_2)=\pi_x\circ P_{\theta(a)}\circ \pi_{y_2}(B^{cs}_2).$$
Consequently, the interior of $\pi_x\circ P_{\theta(a)}\circ \pi_{y_2}(B^{cs}_2)$
meets the ${\cal F}$-boundary of $B^{cs}_1$. We denote by $\gamma$ the corresponding component of $\partial^{\cal F}B^{cs}_1$.
By item (ii) of Proposition~\ref{p.subbox} we have $\gamma=\pi_x(I)$ where $I$ is an arc in
$\bar P_{-t}(\cW^{cu}_{\alpha_{x}}(\bar p_1))$ or in $\bar P_{-t}(\cW^{cu}_{\alpha_{x}}(\bar p_2))$
for some $t\geq 0$.

In the case $\theta(a)\in [0,1]$, the Local invariance shows that $\pi_x\circ P_{\theta(a)}\circ \pi_{y_2}(B^{cs}_2)=B^{cs}_2$.
Hence the interior of $B^{cs}_2$ meets $\pi_x(I)$, contradicting the item (ii) of Proposition~\ref{p.subbox}.

In the other case, $\theta(a)>1$.
The Global invariance (Remark~\ref{r.identification}, item (e)) shows that
since $\pi_x(I)$ intersects ${\rm Interior}(\pi_x\circ P_{\theta(a)}\circ \pi_{y_2}(B^{cs}_2))$,
there is an iterate $\bar P_{-s}(I)$, $s\geq 0$, whose projection by $\pi_x$ meets 
${\rm Interior}(B^{cs}_2)$. Again, this contradicts the item (ii) of Proposition~\ref{p.subbox}.
\end{proof}

In order to prove item~\ref{box4}, we have to check that, up to exchange $(y_1,t_1)$
and $(y_2,t_2)$, the conditions of Lemma~\ref{l.contained} are satisfied by some $\theta$
which furthermore satisfies $|\theta(t_1)-t_2|\leq 1/2$.
By the Global invariance (Remark~\ref{r.identification}, item (e)),
there exists $\theta\in \lip_2$
such that:
\begin{itemize}
\item[--] $\theta(t_1)\in [t_2-1/4,t_2+1/4]$,
\item[--] $d(\varphi_t(y_1),\varphi_{\theta(t)}(y_2))<r_0/2$
for $t\in [0,t_1]\cap \theta^{-1}([0,t_2])$.
\end{itemize}
Note that it also gives $|\theta^{-1}(t_2)- t_1|\leq 1/2$ since $\theta$ is $2$-bi-Lipschitz.
Up to exchange $y_1$ and $y_2$, we can suppose $0\in \theta^{-1}([0,t_2])$,
hence the assumptions of the Lemma~\ref{l.contained} are satisfied.
This gives $B^{cu}_2\subset B^{cu}_1$ and ends the proof of item 5.
\medskip

The proof of item~\ref{box5} uses similar ideas.
Let $B^{cs}_i,B^{cu}_i$, $i=1,2$, be the boxes associated to the transitions
$(y,t_1)$ and $(y,t_2)$ where $y=y_1=y_2$ such that the interior of
$B^{cs}_1$ and $B^{cs}_2$ intersect and that $t_2>t_1+t_{box}>2 t_{box}$.
Recall that $y$ belongs to the interior of some box
$B\in \{B_1,\dots,B_k\}$ which contains $B^{cs}_1$ and $B^{cs}_2$.
Let $I$ be the arc in $\cW^{cs}(y)$ connecting the two components of $\partial^{\cal F} B$
and let $\bar u,\bar v$ be the two generalized orbits associated to the endpoints of $I$
as in the proof of items 3 and 4 above.

Let us consider the two boxes $B^{cu}_1=\pi_xP_{t_1}\pi_y(B^{cs}_1)$ and $\pi_xP_{t_1}\pi_y(B^{cs}_2)$:
their interior intersect (since they contain $\pi_x\varphi_{t_1}(y)$).
By construction $\partial^{\cal F} B^{cu}_1$ is contained in the union of  $\pi_x(\cW^{cu}(\bar P_{t_1}(\bar u)))$
and $\pi_x(\cW^{cu}(\bar P_{t_1}(\bar v)))$. By construction, the ${\cal F}$-boundary of $P_{t_2}\pi_y(B^{cs}_2)$
is contained in the union of $\cW^{cu}_{\alpha_{box}}(\bar P_{t_2}(\bar u))$
and $\cW^{cu}_{\alpha_{box}}(\bar P_{t_2}(\bar v))$ and by the coherence stated in Section~\ref{ss.moreconstants}
the ${\cal F}$-boundary of $\pi_xP_{t_1}\pi_y(B^{cs}_2)$ is also contained in the union of the projection by $\pi_x$ of the plaques $\cW^{cu}(\bar P_{t_1}(\bar u))$
and $\cW^{cu}(\bar P_{t_1}(\bar v))$.
Moreover by applying Lemma~\ref{l.otherinclusion} to $\varphi_{t_1}(y)$, $\pi_x\circ P_{t_1}\circ \pi_y(B^{cs}_2)$ and to the box containing $B^{cu}_1$,
the iterate $\pi_xP_{t_1}\pi_y(B^{cs}_2)$ can not intersect the ${\cal E}$-boundary of
$B^{cu}_1$. This implies that $\pi_xP_{t_1}\pi_y(B^{cs}_2)$ is contained in $B^{cu}_1$,
hence $B^{cs}_2\subset B^{cs}_1$.
\smallskip

The proof of Theorem~\ref{t.existence-box} is now complete. \qed

\section{Uniform hyperbolicity}\label{s.uniform}
In this section we prove Theorem~\ref{Thm:1Dcontracting} (see section~\ref{ss.Dcontracting}).
\medskip

\noindent
{\bf Standing assumptions.}
In the whole section, $(\cN,P)$ is a $C^2$ local fibered flow over a topological flow $(K,\varphi)$
which is not a periodic orbit
and $\pi$ is a $C^2$-identification compatible with $(P_t)$ on an open set $U$ {as in Definition~\ref{d.compatible}}
such that:
\begin{enumerate}
\item[(B1)] There is a dominated splitting $\cN=\cE\oplus \cF$ and the fibers of $\cE,\cF$ are one-dimensional.
\item[(B2)] $\cE$ is uniformly contracted on an open set $V$ containing $K\setminus U$.
\item[(B3)] $\cE$ is uniformly contracted on any compact invariant proper subset of $K$.
\item[(B4)] $\cE$ is topologically contracted.
\end{enumerate}

The main result of this section is
Proposition~\ref{Pro:measure-contracted}, which is proved in the next two sections.

\begin{Proposition}\label{Pro:measure-contracted}

Under the standing assumptions above, for any ergodic invariant measure $\mu$ whose support is $K$, if the Lyapunov exponent of $\mu$ along $\cF$ is positive, then the Lyapunov exponent exponent of $\mu$ along $\cE$ is negative.

\end{Proposition}

Consider a measure $\mu$ as in the statement of the proposition.
We recall that $K=\supp(\mu)$ is not a periodic orbit.
In particular the assumptions of Theorem~\ref{t.existence-box} are satisfied.
One will choose $x\in K\setminus \overline V$, some $\beta_x$ small (to be precised later)
and consider a Markovian box $R\subset B(0_x,\beta_x)$.
\smallskip

The proof is divided into two cases: the non-minimal case and the minimal case.
\subsection{The non-minimal case}

In this section, we will prove Proposition~\ref{Pro:measure-contracted} when {\bf the dynamics on $K$ is not minimal}.
Since $K=\supp(\mu)$, the dynamics on $K$ is transitive.
One can thus fix a non-periodic point $x\in K\setminus \overline V$ whose orbit is dense in $K$ and reduce $r_0$
so that:
\begin{itemize}
\item[--] the ball $U(x,r_0)\subset K$ centered at $x$ with radius $r_0$ is contained in $U$,
\item[--] the maximal invariant set in $K\setminus U(x,r_0)$ is non-empty.
\end{itemize}
We still denote $\mu_x:=({\pi_x})_*(\mu|_{U(x,r_0)})$.

\subsubsection{Notations, choices of constants}\label{ss.nota}

\paragraph{a -- The constant $\beta_x$, the box $R$, the sets $\widehat R$ and $W$.}
One chooses some constants $\beta_x>0$ and $T_x\geq 1$ which satisfy Lemma~\ref{l.continuity-Pliss}.
One will also assume that $\beta_x$ smaller than $\beta_S$ in Lemma~\ref{Lem:schwartz}.
Theorem~\ref{t.existence-box} associates to $\beta_x$ a box $R\subset B(0_x,\beta_x)\subset\cN_x$ whose interior has positive $\mu_x$-measure by Theorem~\ref{t.existence-box}.
\medskip

\noindent
{\it Notation.}
For any box $B\subset \cN_x$, we denote by $\widehat B$ the following open subset of $K$:
$$\widehat B:=\{y\in K,\; d(x,y)<r_0/2 \text{ and } \pi_x(y)\in\Int(B)\}.$$
Let $W$ be the set of points $z$ such that $\varphi_s(z)\notin \widehat R$ for any $s\in [0,1]$.

\paragraph{b -- The constants $T_\cF,C_{\cE},\lambda_{\cE}, C_{\cF},\lambda_{\cF}$.}
Note that $\cE$ is uniformly contracted on the set $W':=\bigcup_{s\in [0,1]}\varphi_s(W)$: indeed this set is disjoint from the open set $\widehat R$
and by our choice of $r_0$ it contains a non-empty compact invariant proper set $K'\subset K$;
if $\cE$ is not uniformly contracted on the set $W'$, one gets a contradiction with our assumption (B3).
Proposition~\ref{l.summability} can thus be applied to $W$ and gives some $C_{\cE}^0,\lambda_{\cE}$ and $T_\cF\geq T_x$.

One sets $C_{\cE}=C^0_\cE\max_{-1\leq s\leq 1}\|DP_t\|^2$. Consequently a piece of orbit $(y,\varphi_t(y))$ is $(C_{\cE},\lambda_{\cE})$-hyperbolic
for $\cE$, 
once there exists $s_0,s_1\in [-1,1]$ and a piece of orbit $(\varphi_{-k}(z),z)$ satisfying the assumptions of Proposition~\ref{l.summability}
and $y=\varphi_{-k+s_1}(z)$, $\varphi_t(y)=\varphi_{s_2}(z)$.
\smallskip

\paragraph{c -- Hyperbolicity for $\cE$ in $K\setminus \widehat R$.} By (B3), the bundle $\cE$ is uniformly contracted outside $\widehat R$; one can thus relax again $C_{\cE},\lambda_{\cE}$
so that any $y\in K$ and $t>0$ satisfy:
$$\forall s\in [1/2,t-1/2],\; \varphi_s(y)\notin \widehat R \;\; \Rightarrow\;\; \|DP_{t}|\cE(y)\|\le C_{\cE} \lambda_{\cE}^{-t}.$$

\paragraph{d -- Hyperbolicity for $\cF$.} 
As in Section~\ref{ss.assumptions}, $\lambda$ is associated to the $2$-domination $\cE\oplus \cF$.

Let $\lambda_{\cF}:=\lambda^{1/2}$. There is $C_{\cF}>0$ with the following property:\hspace{-1cm}\mbox{}
\smallskip

If $(y,\varphi_{t-s}(y))$ is a $(2T_\cF,\lambda_{\cF})$-Pliss string for some $s\in [0,1]$,
then $(y,\varphi_{t}(y))$ is $(C_{\cF}^{1/2},\lambda_{\cF})$-hyperbolic for $\cF$.

\paragraph{e -- The constants $\lambda'_{\cE},\beta_{box}$.}
We need weaker constants $C'_{\cE},\lambda'_{\cE}$ for the hyperbolicity along $\cE$. We first set $\lambda_{\cE}'=\lambda_{\cE}^{1/2}$ and
then apply the following lemma: fixing a transition $(x_0,t_0)$, it defines for points $y\in \widehat B^{cs}$ a piece of orbit $(y,\varphi_{t(y)}(y))$
shadowing $(x_0,\varphi_{t_0}(x_0))$.

\begin{Lemma}\label{l.shadow-box}
There exists $\beta_{box}>0$ such that
if $B_1,\dots,B_k$, $t_{box}$ are boxes and the constant associated to $R,C_{\cF},\lambda_{\cF},\beta_{box}$
by Theorem~\ref{t.existence-box}, then the following holds for some $C'_{\cE}>0$.

Let $(x_0,t_0)$ be a transition with $t_0>2t_{box}$ such that $(x_0,\varphi_n(x_0))$ is
$T_\cF$-Pliss, where $n\in [t_0-1,t_0]$;
let $B^{cs},B^{cu}$ be the associated sub-boxes, and $y\in \widehat {B}^{cs}$.
Then there exist $\theta\in\lip_2$ and $t(y)$ such that:
\begin{enumerate}
\item $\varphi_{t(y)}(y)\in\widehat B^{cu}$,
\item $|\theta(0)|\leq 1/4$, $|\theta(t_0)-t(y)|\leq 1/4$,
\item $d(\varphi_t(x_0), \varphi_{\theta(t)}(y))<r_0/2$ for $t\in [-1,t_0+1]$,
\item $(y,\varphi_{t(y)}(y))$ is a $(2T_\cF,\lambda_{\cF})$-Pliss string.
\end{enumerate}
If $(x_0,\varphi_{t_0}(x_0))$ is $(C_{\cE},\lambda_{\cE})$-hyperbolic for $\cE$, {then}
$(y,\varphi_{t(y)}(y))$ is $(C'_{\cE},\lambda'_{\cE})$-hyperbolic for $\cE$.
\end{Lemma}

From the items 1 and 4 of this Lemma and the choice of $C_{\cF}$, $(y,t(y))$ is a transition. The associated boxes are $B^{cs}$ and $B^{cu}$.
Indeed, the items 2 and 3 together with the item~\ref{box4} of
Theorem~\ref{t.existence-box} imply that the associated center-unstable box coincides with $B^{cu}$.
By the Global invariance, the center-stable box $\pi_x\circ P_{-t(y)}\pi_{\varphi_{t_0(y)}(y)}(B^{cu})$
coincides with $B^{cs}$.

\begin{proof}
For $\lambda'\in (1,\lambda_{\cE}^{1/4})$, Lemma~\ref{l.shadowingandhyperbolicity} gives $C',\delta,\rho$. We may take $\rho \in (0,1/3)$.
The Global invariance then associates to $\delta,\rho$ some constants $\beta,r$. 
By the Local injectivity, we can take $\beta_{box}\in (0,\beta)$ smaller such that
for any $x,y\in U$ that are $r_0$-close to $x$ and satisfy $d(\pi_x(y),\pi_x(x_0))\leq \beta_{box}$, there exists $s\in [-1/4,1/4]$
such that $d(x,\varphi_s(y))<r$.

By Lemma~\ref{l.continuity-Pliss} (with the constants $\beta_x$ and $T_\cF\geq T_x$ introduced in paragraphs (a) and (b) above), we obtain $\theta\in \lip_2$
such that $|\theta(0)|\leq 1/4$ and item 3 is satisfied; moreover $(y,\varphi_{\theta(y)+a}(y))$ is a $(2T_\cF,\lambda^{1/2})$-Pliss string
for any $a\in [-1,1]$.
By the Local injectivity, there exists $t(y)\in [\theta(t_0)-1/4, \theta(t_0)+1/4]$
such that $d(\varphi_{t(y)}(y),x)<r_0/2$ and $\pi_x\circ \varphi_{t(y)}(y)=\pi_x\circ \varphi_{\theta(t_0)}(y)$.
In particular item 2 holds. Moreover 
$\pi_{\varphi_{t_0}(x_0)}\circ \varphi_{\theta(t_0)}(y)=P_{t_0}\circ \pi_{x_0}(y)\in P_{t_0}\circ \pi_{x_0}(B^{cs})$. Its projection by $\pi_{x}$ belongs to $B^{cu}=\pi_{x}\circ P_{t_0}\circ \pi_{x_0}(B^{cs})$,
so $\pi_{x} \circ \varphi_{\theta(t_0)}(y)\in B^{cu}$ giving the first item.

Let us assume now that $(x_0,\varphi_{t_0}(x_0))$ is $(C_{\cE},\lambda_{\cE})$-hyperbolic for $\cE$
and consider $\sigma\in [t_{box},t_0+1/4]$ such that
\begin{itemize}
\item[--] $\varphi_{\theta(\sigma)}(y)\in \widehat R$,
\item[--] $\varphi_{\theta(s)}(y)\notin \widehat R$ for $s\in [t_{box},t_0]$ satisfying $\theta(s)\leq \theta(\sigma)-1/2$.
\end{itemize}
From the property stated at paragraph (c) above, we have for any $s\in [t_{box},\sigma]$
$$\|DP_{\theta(s)-\theta(t_{box})}|\cE(\varphi_{\theta(t_{box})}(y))\|\le C_{\cE}{\lambda _\cE}^{-(\theta(s)-\theta(t_{box}))}.$$
Since $\theta(t_{box})<2t_{box}$, there exists $C_1>0$ independent from $x_0,t_0,y$ such that for any $s\in [0,\sigma]$,
\begin{equation}\label{e.hyp1}
\|DP_{\theta(s)}|\cE(y)\|\le C_1{\lambda_{\cE}}^{-\theta(s)}.
\end{equation}

From Local injectivity (recalled above) and since $\diam(P_{s}\circ \pi_{x_0}(B^{cu}))<\beta_{box}$
for $s\in [t_{box},t_0]$ (by item~\ref{box3} of Theorem~\ref{t.existence-box}), there exists $\varepsilon\in[-1/4,1/4]$ such that
$d(\varphi_{\theta(\sigma)+\varepsilon}(y),\varphi_{\sigma}(x_0))<r$.
The Global invariance gives $\theta'\in\lip_{1+\rho}$, such that for each $s\in [t_{box}-1,t_0+1]$ one has
$d(\varphi_{\theta'(s)}(y),\varphi_s(x_0))<\delta$ and ${\theta'}(\sigma)=\theta(\sigma)+\varepsilon$.

Lemma~\ref{l.shadowingandhyperbolicity} now gives for each $s\in [t_{box}-1,t_0+1]$,
$$\|DP_{\theta'(s)-\theta'(t_{box})}|\cE(\varphi_{\theta'(t_{box})}(y))\|\leq C'{\lambda'}^{s-t_{box}}\|DP_{s-t_{box}}|\cE(\varphi_{t_{box}}(x_0))\|.$$
Since $(x_0,\varphi_{t_0}(x_0))$ is $(C_{\cE},\lambda_{\cE})$-hyperbolic for $\cE$,  since $\theta'$ is $3/2$-bi-Lipschitz,
and since $\lambda'<\lambda_{\cE}^{1/4}$, one gets $C_2>0$ (depending on $t_{box}$, not on $x_0,t_0,y$) such that for any $s\in [t_{box}-1,t_0+1]$,
\begin{equation}\label{e.hyp2}
\|DP_{\theta'(s)-\theta'(t_{box})}|\cE(\varphi_{\theta'(t_{box})}(y))\|\leq
C_2{\lambda'}^{s-t_{box}}C_{\cE}\lambda_{\cE}^{-(s-t_{box})}\leq C_2C_{\cE}\lambda_{\cE}^{-(\theta'(s)-\theta'(t_{box}))/2}.
\end{equation}
Combining~\eqref{e.hyp1} and~\eqref{e.hyp2},
one deduces that $(y,\varphi_{\theta(t_0)}(y))$ is $(C'_{\cE},\lambda'_{\cE})$-hyperbolic for $\cE$ for some constant $C'_{\cE}$, provided
$\theta'(t_{box}-1)\leq \theta(\sigma)$ and $\theta'(t_0+1)\geq t(y)$.

Since $\theta'$ is $2$-bi-Lipschitz, one gets $\theta'(t_{box}-1)\leq \theta'(\sigma)-1/2=\theta(\sigma)+\varepsilon-1/2<\theta(\sigma)$.
One can apply Proposition~\ref{p.no-shear} to $\varphi_{\theta(\sigma)}(y)$, the reparametrization $\theta'\circ \theta^{-1}$ and the
interval $[\theta(\sigma),\theta(t_0)]$. Since $|\theta'(\sigma)-\theta(\sigma)|<2$, one gets $\theta'(t_0)+1/2\geq \theta(t_0)$.
Since $\theta'$ is $4/3$-bi-Lipschitz, this gives $\theta'(t_0+1)\geq \theta'(t_0)+3/4\geq \theta(t_0)+1/4\geq t(y)$ and concludes the proof.
\end{proof}

\paragraph{f -- The sub-boxes $B_1,\dots,B_k$ and the constants $t_{box}$, $\Delta_{box},C'_{\cE}$.}
Finally we apply Theorem~\ref{t.existence-box} to $R,C_{\cF},\lambda_{\cF},\beta_{box}$
and obtain the sub-boxes $B_1,\dots,B_k$ and the constants $t_{box}$, $\Delta_{box}$ that we fix now.
Lemma~\ref{l.shadow-box} gives $C'_{\cE}$.

\subsubsection{Existence of large hyperbolic returns}

Since we have to prove that the Lyapunov exponent of $\mu$ along $\cE$ is negative,
one can reduce to the case where the following condition holds:
\begin{enumerate}
\item[(B5)] The Lyapunov exponent of $\mu$ along $\cE$ is larger than $-\log(\lambda_{\cE})$.
\end{enumerate}
In the following we will say that a point $z\in K$ is called \emph{regular} if :

\begin{itemize}

\item[--] the orbit of $z$ equidistributes towards $\mu$, i.e. $\frac{1}{t}\int_0^t\delta_{\varphi_{s}(z)}ds\to\mu$ as $t\to\pm\infty$.

\item[--] For any iterate $\varphi_t(z)$, if $d(\varphi_t(z),x)<r_0$, then $\pi_x(\varphi_t(z))$ is not contained in the boundary of $R$,
nor of any box $B_i$, $1\le i\le k$.

\end{itemize}
By Birkhoff ergodic theorem and since the boundaries of the boxes $R$ and $B_i$ have zero measure (for $\mu_x=(\pi_x)_*(\mu)$),
the set of regular points has full measure for $\mu$.

\begin{Lemma}\label{l.transition}
For any $T_0>0$, there exists a regular point $x_0$ and $t_0>T_0$ such that
\begin{itemize}
\item[--] both $x_0$ and $\varphi_{t_0}(x_0)$ are in $\widehat R$,
\item[--] $(x_0,\varphi_{n}(x_0))$ is $T_\cF$-Pliss for some $n\in [t_0-1,t_0]$.
\end{itemize}
\end{Lemma}
\begin{proof}
Let us take a regular point $y$. We have $\omega(x)=\alpha(x)=K$.
By assumption the maximal invariant set outside $\widehat R$ is a non-empty compact invariant proper set $K_0$ of $K$.
One can assume that $y$ is very close to $K_0$. Thus, backward iterates $\varphi_{t_1}(y)$
and forward iterate $\varphi_{t_2}(y)$ in $\widehat R$ occur for $t_1$ and $t_2$ large:
clearly, one can choose $y$ such that $t_2-t_1>T_0+1$.
Choosing $t_1,t_2$ close to their infimum values, one furthermore gets that
$\varphi_{t}(y)\not\in \widehat R$ for $t\in (t_1+1/4,t_2-1/4)$.

Let $x_0:=\varphi_{t_1}(y)$.
Note that $x_0$ is not $(C_{\cE},\lambda_{\cE})$-hyperbolic for $\cE$: since $x_0$ is regular,
its forward orbit equidistributes on the measure $\mu$ and this would imply that the Lyapunov exponent of $\mu$
along $\cE$ is less than or equal to $-\log(\lambda_{\cE})$, contradicting the assumption (B5).
By Proposition~\ref{l.summability} and the choice of constants in~\ref{ss.nota}(b), there exists a forward iterate $\varphi_n(x_0)\not\in W$, $n\geq 1$,
such that $(x_0,\varphi_n(x_0))$ is a $T_\cF$-Pliss string for $\cF$.
By definition of $W$, there is $t_0\in [n,n+1]$ such that $\varphi_{t_0}(x_0)\in \widehat R$.
By our choice of $t_1,t_2$, one has $t_0\geq t_2-1/2-t_1>T_0$.
\end{proof}

\subsubsection{Contraction at returns}
Let $(x_0,t_0)$ be given by Lemma~\ref{l.transition} for $T_0>2 t_{box}$. {Consider sub-boxes $B_{i_0}$, $B_{j_0}$ such that}
$$\pi_x(x_0)\in \Interior(B_{i_0}),\quad \pi_x(\varphi_{t_0}(x_0))\in \Interior(B_{j_0}).$$
We get a transition $(x_0,t_0)$ between $B_{i_0}$ and $B_{j_0}$.
By Theorem~\ref{t.existence-box}, one thus gets a center-stable sub-box $B^{cs}\subset B_{i_0}$ and a center-unstable
sub-box $B^{cu}\subset B_{j_0}$ as in the statement of this theorem.

\begin{Lemma}\label{l.contraction}
If $t_0$ is large enough, there exists $\lambda_*>1$ (depending on $t_0$) such that
for any regular $y\in \widehat B^{cs}$,
there is $\tau>t_{box}$ satisfying $\varphi_{\tau}(y)\in \widehat B^{cs}$ and
$$\|DP_{\tau}|\cE(y)\|\le \lambda_*^{-\tau}.$$
\end{Lemma}

The proof of this lemma breaks into 5 steps.

\paragraph{Step 1. Definition of the times $\sigma<t(y)\leq \tau$.}
For any regular $y\in \widehat B^{cs}$, Lemma~\ref{l.shadow-box} gives a time $t(y)$ such that $\varphi_{t(y)}(y)\in \widehat B^{cu}$ and $(y,\varphi_{t(y)}(y))$ shadows the piece of orbit $(x_0,\varphi_{t_0}(x_0))$.

The forward orbit of $y$ is dense in $K$, hence there is $\tau\geq t(y)$ such that
{$$\varphi_{\tau}(y)\in \widehat B^{cs},~\textrm{but}~\varphi_{s}(y)\not\in \widehat B^{cs}~\textrm{for any}~s\in (t(y),\tau-1).$$}

We also introduce a return time $\sigma\in[0,t(y)-1]$ (possibly equal to $0$) such that
{$$\varphi_{\sigma}(y)\in \widehat B^{cs},~\textrm{but}~\varphi_{s}(y)\not\in \widehat B^{cs}~\textrm{for any}~s\in (\sigma+1,t(y)-1).$$}

\paragraph{Step 2. Definition of the times $t_1<t_2<\dots<t_\ell$.} We now introduce intermediate times between $\sigma+1$
and $\tau-1$. We first set 
$$t_1=t(\varphi_{\sigma}(y)).$$
By applying Lemma~\ref{l.shadow-box} twice,
the orbit segment $(y,\varphi_{t_1}(y))$ is $({C'_{\cE}}^2,\lambda_{\cE}')$-hyperbolic for $\cE$.
Let $C_2=\max_{0\leq t\leq 2} \|DP_t\|{\lambda'}_\cE^2$.
From Lemma~\ref{l.shadow-box},
we get
$$\tau\geq t(y)\geq \frac 1 2 t_0-\frac 1 4.$$
If $t_1+2\ge \tau$, then provided $t_0$ has been chosen large enough one gets
$$\|{DP_{\tau}}{|\cE}(y)\|\leq C_2\;{C'_{\cE}}^2{\lambda_{\cE}'}^{-\tau}\leq
C_2\;{C'_{\cE}}^2{\lambda_{\cE}'}^{-\frac 1 2 (\frac {t_0}{2}-1/4)}{\lambda'_{\cE}}^{-\tau/2}\leq {\lambda'_{\cE}}^{-\tau/2}.$$
Hence the conclusion of Lemma~\ref{l.contraction} holds in this case with $\lambda_*={\lambda_{\cE}'}^{1/2}$.
A similar discussion holds when $\tau\leq t(y)+2$.
Thus, without loss of generality, we can assume that:
$$t_1+2< \tau \text{ and } t(y)+2<\tau.$$

\begin{Sublemma}\label{l.extistence-tm}
There exists a sequence of times $\{t_m\}_{m=2}^\ell$ in $[t_1+1,\tau-1]$ such that:

\begin{itemize}

\item[--] $\varphi_{t_m}(y)\in \widehat R$ and $(\varphi_{\sigma}(y),\varphi_{t_m}(y))$ is $(C_{\cF},\lambda_{\cF})$-hyperbolic for $\cF$.

(Equivalently $(\varphi_{\sigma}(y), t_m-\sigma)$ is a transition.)

\item[--] $t_m\geq t_{m-1}+1$ and 
$(\varphi_{t_{m-1}}(y),\varphi_{t_{m}}(y))$ is $(C_{\cE},\lambda_{\cE})$-hyperbolic for $\cE$.

\item[--] $(\varphi_{t_\ell}(y),\varphi_{\tau}(y))$ is $(C_{\cE},\lambda_{\cE})$-hyperbolic for $\cE$.

\end{itemize}

\end{Sublemma}
\begin{proof}
We define inductively the increasing sequence of integers $\{n_m\}_{m=1}^\ell$ such that:

\begin{itemize}

\item[--] $n_1=0$ and for any $2\le m\le \ell$, the piece of orbit $(\varphi_{t_1}(y),\varphi_{t_1+n_m}(y))$ is a $T_\cF$-Pliss string, $n_{m}-n_{m-1}\geq 2$
and $\varphi_{t_1+n_m}(y)\notin W$;

\item[--] for any integer $0\leq n\leq \tau(y)-t_1(y)-2$ such that neither $n$, nor $n-1$ belong to
$\{n_1,\dots,n_\ell\}$, then either $\varphi_{t_{1}+n}(y)\in W$ or
$(\varphi_{t_1}(y),\varphi_{t_{1}+n}(y))$ is not a $T_\cF$-Pliss string.

\end{itemize}
By definition of $W$, there exists $t_m\in [t_1+n_m,t_1+n_m+1]$
such that $\varphi_{t_m}(y)\in \widehat R$.
Note that we have $t_m\geq t_{m-1}+1$ and $t_\ell+1\leq \tau$.

By our choice of $C_{\cF},\lambda_{\cF}$, the piece of orbit
$(\varphi_{t_1}(y),\varphi_{t_{m}}(y))$ is $(C_{\cF}^{1/2},\lambda_{\cF})$-hyperbolic for $\cF$.
By Lemma~\ref{l.shadow-box},
$(\varphi_{\sigma}(y),\varphi_{t_1}(y))$ is a $(2T_\cF,\lambda_{\cF})$-Pliss string,
hence is also $(C_{\cF}^{1/2},\lambda_{\cF})$-hyperbolic for $\cF$.
Consequently $(\varphi_{\sigma}(y),\varphi_{t_{m}}(y))$  is $(C_{\cF},\lambda_{\cF})$-hyperbolic for $\cF$.

By our choice of $n_m$, any integer $n$ with $n_{m-1}+2\leq n<n_{m}$ either belongs to $W$
or satisfies that $(\varphi_{t_1}(y),\varphi_{t_{1}+n}(y))$  is not $T_\cF$-Pliss.
Proposition~\ref{l.summability} and the choice of $(C_{\cE},\lambda_{\cE})$ implies that
$(\varphi_{t_1+n_{m-1}}(y),\varphi_{t_1+n_{m}}(y))$
and $(\varphi_{t_{m-1}}(y),\varphi_{t_{m}}(y))$
are $(C_{\cE},\lambda_{\cE})$-hyperbolic for $\cE$. This gives the second item.
The third item is obtained similarly.
\end{proof}

\paragraph{Step 3. Construction of center-unstable boxes associated to the times $t_m$.}
By Sublemma~\ref{l.extistence-tm},
$(\varphi_{\sigma}(y), t_m-\sigma)$ is a Markovian transition
between boxes in $\{B_1,B_2,\cdots,B_k\}$ for any $1\leq m\leq \ell$.
By Theorem~\ref{t.existence-box} it defines a center-stable sub-box $B_m^{cs}$ and a center-unstable sub-box $B_m^{cu}$. Moreover the distortion of $B_{m}^{cu}$ is bounded by the constant $\Delta_{box}$.

\begin{Sublemma}\label{l.box-disjoint}

The interiors of the boxes $B_m^{cu}$, for $t_{box}<m\leq \ell$, are mutually disjoint.

\end{Sublemma}

\begin{proof}
Let us first notice that if $m>t_{box}$, the center-stable sub-box $B^{cs}_m$ is contained in $B^{cs}$:
indeed, let us consider the transitions $(\varphi_{\sigma}(y),\varphi_{t_1}(y))$
and $(\varphi_{\sigma}(y),\varphi_{t_m}(y))$. We have $t_1-\sigma>t_{box}$ and $t_m-t_1>t_{box}$.
Moreover, the boxes associated to the first transition are $B^{cs},B^{cu}$ (as explained after the Lemma~\ref{l.shadow-box}).
Theorem~\ref{t.existence-box}, item~\ref{box5}, implies that $B^{cs}_m$ is contained in $B^{cs}=B^{cs}_1$.

Assume by contradiction that the interiors of $B_i^{cu}$ and $B_j^{cu}$, for $i\neq j$ larger than $t_{box}$, intersect.
Up to exchange $i$ and $j$,
the item~\ref{box4} of Theorem~\ref{t.existence-box} gives $\theta\in \lip_2$ such that
\begin{itemize}
\item[--] $d(\varphi_{s}(y),\varphi_{\theta(s)}(y))<r_0/2$,
for any $s\in [\sigma,t_i]\cap \theta^{-1}([\sigma,t_j])$,
\item[--]$|\theta(t_i)-t_j|\leq 1/2$ and $\theta(\sigma)\geq \sigma-1$.
\end{itemize}
\begin{Claim}
$\theta(\sigma)>\sigma+2$.
\end{Claim}
\begin{proof}
By Proposition~\ref{p.no-shear}, $\theta(\sigma)\in[\sigma-1,\sigma+2]$
implies $|\theta(t_i)-t_i|<1/2$. This gives
$|t_i-t_j|<1$ and this contradicts the definition of the sequence $(t_m)$ since
$t_{m+1}-t_m\geq 1$ for any $m$.
\end{proof}

Since $B^{cs}_i\subset B^{cs}$, the image $\pi_{\varphi_{\sigma}(y)}(B_i^{cs})=P_{-(t_i-\sigma)}\circ\pi_{\varphi_{t_i}(y)}(B^{cu}_i)$ by $\pi_x$ is contained in $B^{cs}$.
Since $\pi_x\circ \varphi_{t_j}(y)$ belongs to $B^{cu}_j\subset B^{cu}_i$, one gets
$$\pi_x\circ P_{-(t_i-\sigma)}\circ\pi_{\varphi_{t_i}(y)}(\varphi_{t_j}(y))\in B^{cs}.$$
Using the Global invariance  this gives
$$\pi_x\circ P_{-(\theta(t_i)-\theta(\sigma)(y))}\circ \pi_{\varphi_{\theta(t_i)}(y)}(\varphi_{t_j}(y))\in B^{cs}.$$
Since $|t_j-\theta(t_i)|\leq 1/2$, the Local invariance gives 
$0_{\varphi_{\theta(t_i)}(y)}= \pi_{\varphi_{\theta(t_i)}(y)}(\varphi_{t_j}(y))$,
hence $\pi_{x}\circ \varphi_{\theta(\sigma)}(y)\in B^{cs}$.
The local injectivity gives $s$ with $|\theta(\sigma)-s|\leq 1/4$
such that $\varphi_s(y)\in \widehat B^{cs}$.

\begin{Claim}
We have $t(y)+1<s<\tau-1$.
\end{Claim}
\begin{proof}
Since $\theta(\sigma)>\sigma+2$, we have $\sigma+1<s$.
By definition of $\sigma$, one gets $s\geq t(y)-1$.

Note that $\pi_x(\varphi_s(y))=\pi_x(\varphi_{t(y)}(t))$ when $s\in[t(y)-1,t(y)+1]$, 
hence $\varphi_{t(y)}(y)\in \widehat B^{cs}$; this gives $\tau=t(y)$ by definition and contradicts the assumption $\tau>t(y)+1$.

Since $\theta(\sigma)\leq \theta(t_i)\leq t_j+1/2$,
one gets $s\leq t_j+3/4$. Since $t_\ell+2<\tau$, we have $s<\tau-1$.
\end{proof}

We have thus obtained a time $s$ which contradicts the definition of $\tau$.
This concludes the proof of Sublemma~\ref{l.box-disjoint}.
\end{proof}

\paragraph{Step 4. Summability.}
Let $J^{cs}(y)=\cW^{cs}(y)\cap \pi_y(B^{cs})$.
In each sub-box $B_j$ of $R$, we choose a $C^1$-curve $\gamma_j$ tangent to $\cC^\cE$ with endpoints in $\partial^{cu}B_j$. Set
$$L_\cB=\sum_{1\leq j\leq k}|\gamma_j|.$$
It only depends on $R$ and $B_1,\dots,B_k$, but not on the points $x_0$, $y$.

\begin{Sublemma}
We have
$$\sum_{i=0}^{[\tau]}|P_i(J^{cs}(y))|\le C_{sum}:=\frac{2{C'_{\cE}}^2\lambda'_{\cE}}{\lambda'_{\cE}-1}\Delta_{box}L_{\cB}(1+t_{box}).$$
\end{Sublemma}
\begin{proof}
For each $1\leq m\leq \tau$ we have $|P_{t_m}(J^{cs}(y))|\leq \Delta_{box}L_{\cB}$. Moreover
$P_{t_m}(J^{cs}(y))$ is a curve tangent to $\cC^\cE$ and crosses $B^{cu}_m$.
Since 
the interiors of 
$\{B_m^{cu},\; t_{box}<m\leq \ell\}$ are mutually disjoint (Lemma~\ref{l.box-disjoint}) and are center-unstable sub-boxes of $B_1,\dots,B_k$
which have distortion bounded by $\Delta_{box}$, we have that
$$\sum_{1\leq m\leq \ell}|P_{t_m}(J^{cs}(y))|\le \Delta_{box}L_{\cB}(1+t_{box}).$$
By Sublemma~\ref{l.extistence-tm}, $(\varphi_{t_m}(y),\varphi_{t_{m+1}}(y))$ is $(C_{\cE},\lambda_{\cE})$-hyperbolic for $\cE$. Thus, we have
$$\sum_{t_m\leq i\leq t_{m+1}}|P_{i}(J^{cs}(y))|\le \frac{C_{\cE}\lambda_{\cE}}{\lambda_{\cE}-1}|P_{t_m}(J^{cs}(y))|.$$
A similar estimate holds for integers $i$ in $[t_\ell,\tau]$. Hence
$$\sum_{t_1\leq i\leq \tau}|P_i(J^{cs}(y))|\le \frac{C_{\cE}\lambda_{\cE}}{\lambda_{\cE}-1}\Delta_{box}L_{\cB}(1+t_{box}).$$
We have shown previously that $(y,\varphi_{t_1}(y))$ is $({C'_{\cE}}^2,\lambda_{\cE}')$-hyperbolic for $\cE$,
hence $$\sum_{0\le i<t_1}|P_i(J^{cs}(y))|\le \frac{{C'_{\cE}}^2\lambda'_{\cE}}{\lambda'_{\cE}-1}\Delta_{box}L_{\cB}.$$
The estimate of the sublemma follows from these two last inequalities.
\end{proof}

\paragraph{Step 5. End of the proof of Lemma~\ref{l.contraction}.}
By item~\ref{box3} of Theorem~\ref{t.existence-box}, for any $0\leq s \leq \tau$ one has
$P_s(J^{cs}(y))\subset B(0_{\varphi_s(y)},\beta_S)$. Lemma~\ref{Lem:schwartz} associates to $C_{Sum}$ a constant $C_S>1$, independent from $x_0,t_0$
and gives
$$\|DP_{\tau}|{\cE(y)}\|\le C_{S}\frac{|P_{\tau}J^{cs}(y)|}{|J^{cs}(y)|}.$$
By construction $\tau\geq \frac {t_0}{2}-1/4$.
Moreover, $|J^{cs}(y)|$ is bounded away from zero independently from $x_0,t_0$.
The topological hyperbolicity of $\cE$ ensures that $|P_{\tau}J^{cs}(y)|$ is arbitrarily small
if $\tau$ is large. As a consequence, if $t_0$ is large enough, for any regular $y\in \widehat B^{cs}$, one gets
$$\|DP_{\tau}|{\cE(y)}\|\leq \frac 1 2.$$
By assumption (B3), $\cE$ is uniformly contracted on the maximal invariant set in $K\setminus \widehat B^{cs}$.
The time $t(y)$ is bounded uniformly in $y$ and $\varphi_s(y)$ does not meet $\widehat B^{cs}$ for $s\in (t(y),\tau-1)$.
Hence there exists $C_B,\lambda_B>1$ (depending on $x_0,t_0$) such that for any regular $y\in \widehat B^{cs}$,
$$\|DP_{\tau}|{\cE(y)}\|\leq C_B\lambda_B^{-\tau}.$$
Choosing $\lambda_*>1$ close to $1$ one gets for any $t>0$,
$$\min(1/2, C_B\lambda_B^{-t})\leq \lambda_*^{-t},$$
which gives the estimate of Lemma~\ref{l.contraction}.
\qed

\subsubsection{Proof of Proposition~\ref{Pro:measure-contracted} in the non-minimal case}

We can now conclude the proof of Proposition~\ref{Pro:measure-contracted} when the dynamics on $K$ is non-minimal.
For any regular point $y\in \widehat B^{cs}$, we have obtained a contracting return $\tau(y)$. This allows to define an increasing sequence of times $\{\tau_n\}_{n\in\NN}$ such that
$\tau_0=\tau(y)$ and $\tau_{n+1}=\tau_n+\tau(\varphi_{\tau_n}(y))$.
By Lemma~\ref{l.contraction},
we have for any $n\geq 0$
$$\|{DP_{\tau_n}}{|\cE}(y)\|\le \lambda_*^{-\tau_n}.$$
Since $y$ is regular, $\frac 1 {\tau_n} \log({\|DP_{\tau_n}}{|\cE}(y)\|)$ converges as $n\to +\infty$ to the Lyapunov exponent of $\mu$ along $\cE$.
Consequently this Lyapunov exponent is smaller or equal to $-\log(\lambda_*)$. Hence it is negative as announced.
\qed

\subsection{The minimal case}

In this section, we will continue to prove Proposition~\ref{Pro:measure-contracted}, now assuming that {\bf the dynamics on $K$ is minimal}. We will apply a local version of the result of Pujals and Sambarino \cite{PS1}.

\begin{Theorem}\label{Thm:localized-pujals-sambarino}
Assume that $f:~W_1\to W_2$ is a $C^2$ diffeomorphism, where $W_1,W_2\subset \RR^2$ are open sets, and that there is a compact invariant set $\Lambda\subset W_1\cap W_2$ of $f$ such that:
\begin{itemize}
\item[--] every periodic point in $\Lambda$ is a hyperbolic saddle,
\item[--] $\Lambda$ admits a dominated splitting $T_\Lambda\RR^2=E\oplus F$,
\item[--] $\Lambda$ does not contain a circle tangent to $E$ or $F$ which is invariant by an iterate of $f$,
\end{itemize}
then $\Lambda $ is hyperbolic.
\end{Theorem}

Note that Pujals-Sambarino stated their theorem for global diffeomorphisms
of a compact surface, but their proof gives also the local result above.
It is also obtained in \cite{CPS}.
\medskip

Our goal now is to reduce the minimal case to Theorem~\ref{Thm:localized-pujals-sambarino}
by introducing a local surface diffeomorphism and an invariant compact set $\Lambda$.

Let us consider some $r$ small such that the ``No small period" assumption holds for some $\varkappa<1/2$.
As before one chooses a point $x\in K\setminus \overline V$, some $\beta_x$, and a box $R\subset B(0_x,\beta_x)$
given by Theorem~\ref{t.existence-box}. We introduce the set
$$\widehat R=\{y\in K,\; d(y,x)<r_0/2 \text{ and } \pi_x(y)\in R\}.$$
Assuming that $\beta_x$ has been chosen small enough, the Local injectivity associates to any $y\in \widehat R$,
a point $y'\in \widehat R$ such that $d(y',x)<r/2$ and $\pi_x(y)=\pi_x(y')$. Moreover $y'=\varphi_t(y)$ for some $t\in [-1/4,1/4]$.

\paragraph{The set $\Lambda$.}
We introduce the following set:
$$\Lambda:=\pi_x(\widehat R)=\{u\in R,\; \exists y\in K,\; d(x,y)<r_0/2 \text{ and } u=\pi_x(y)\}.$$
Note that, one can choose the points $y$ in the definition
of $\Lambda$ to be $r/2$-close to $x$ and in particular to satisfy $d(x,y)\leq r_0/3$.

\begin{Lemma}
$\Lambda$ is compact and contained in the interior of $R$.
\end{Lemma}
\begin{proof}
Indeed, let us consider $\{u_n\}_{n\in\NN}$ in $\Lambda$ such that $\lim_{n\to\infty}u_n=u$.
We take $y_n\in K$ that is $r_0/3$-close to $x$ such that $\pi_x(y_n)=u_n$.
Taking a subsequence if necessary, we assume that $y=\lim_{n\to\infty}y_n$. This point is $r_0/3$-close to $x$ and by continuity of the identification $\pi_x(y)=u$.
Hence $u\in \Lambda$, proving that $\Lambda$ is compact.

Since $K$ is not   periodic and  is minimal,
$K$ does not contain periodic orbits.
By property~\ref{i.markov-boundary} of Theorem~\ref{t.existence-box},
one deduces that for any point $y$ which is $r_0$-close to $x$,
the projection $\pi_x(y)$ is disjoint from the boundary $\partial R$.
Hence $\Lambda=\pi_x(\widehat R)$ is contained in the interior of $R$.
\end{proof}

\paragraph{The return map $f$ on $\Lambda$.}
For any $u\in \Lambda$, one defines $f(u)$ as follows:
consider $y\in \widehat R$ such that $\pi_x(y)=u$ and choose the smallest $t\ge 1$ such that
$d(\varphi_t(y),x)\leq r_0/3$ and $\pi_x(\varphi_t(y))\in \Lambda$
(such a $t$ exists because $K$ is minimal).
We then define $f(u)=\pi_x(\varphi_t(y))$.

\begin{Lemma}\label{l.well-def}
$f$ is well defined.
\end{Lemma}
\begin{proof}
We have to check that the definition of $f(u)$ does not depend on the choice of $y$.
So we consider $y,y'\in \widehat R$ such that $\pi_x(y)=\pi_x(y')$ and the minimal times $t,t'$ as in the previous definition.
By the Local injectivity, there exists $s_0\in [-1/4,1/4]$ such that $y_0:=\varphi_{s_0}(y)$ is $r/2$-close to $x$ and $\pi_x(y)=\pi_x(y_0)$.
One builds similarly $s'_0$ and $y'_0$. Then $y_0,y_0'$ are $r_0$ close to each other and satisfy $\pi_x(y_0)=\pi_x(y'_0)$
so that by the Local injectivity, $y'_0=\varphi_{s}(y_0)$ for some $s\in [-1/4,1/4]$.
In particular $y'=\varphi_{s+s_0-s'_0}(y)$ with $|s+s_0-s'_0|\leq 3/4$.

Using the Local injectivity, there exists $\tau\in [-1/4,1/4]$ such that
$\varphi_{t+\tau}(y)$ is $r/2$-close to $x$ and satisfies $\pi_x(\varphi_{t+\tau}(y))=\pi_x(\varphi_t(y))$.
Since $y_0:=\varphi_{s_0}(y)$ and $\varphi_{t+\tau}(y)$ are both $r/2$-close to $x$, the ``No small period" assumption implies that
$|t+\tau-s_0|$ is either larger or equal to $2$, or smaller than $1/2$. Since by definition $t\geq 1$
one has $t+\tau-s_0\geq 2$.

Thus, $\varphi_t(y)$ is the image of $y'$ at the time $t-(s_0+s+s'_0)=(t+\tau-s_0)-(s'_0+\tau+s_0)$,
which is larger than $2-3/4>1$. By minimality in the definition of $t$, $\varphi_t(y)$ is a forward iterate of $\varphi_{t'}(y')$.
In a similar way $\varphi_{t'}(y')$ is a forward iterate of $\varphi_{t}(y)$.
Hence these two points coincide.
\end{proof}

The next lemma shows that the orbits under $f$ correspond to (the projection by $\pi_x$ of)
orbits under $\varphi$ restricted to $\widehat R$.
\begin{Lemma}\label{l.orbit-f}
For any $y\in \widehat R$, let $t=\min\{s\geq 1,\; d(\varphi_s(y),x)\leq r_0/3 \text{ and }\varphi_s(y)\in \widehat R\}$.
Then for any $s\in [0,t]$ such that $\varphi_s(y)\in \widehat R$, we have
$s\notin (3/4,3/2)$. Moreover:
\begin{itemize}
\item[--] if $s\leq 3/4$, $\pi_x(\varphi_s(y))=\pi_x(y)$,
\item[--] if $s\geq 3/2$, $\pi_x(\varphi_s(y))=\pi_x(\varphi_t(y))$ (which coincides with $f(\pi_x(y))$)
and $s\geq t-1/4$.

\end{itemize}
\end{Lemma}
\begin{proof}
The proof of Lemma~\ref{l.well-def} showed that if $y,y'\in \widehat R$ have the same projection by $\pi_x$,
then $y'=\varphi_s(y)$ for some $|s|\leq 3/4$.

On the other hand if $y,y'\in \widehat R$ belong to the same orbit (i.e. $y'=\varphi_s(y)$) but have different projection by $\pi_x$,
then $|s|>3/2$. Indeed, by the Local injectivity, there exists $y_0=\varphi_{s_0}(y)$ and $y'_0=\varphi_{s'_0}(y')$
which are $r/2$-close to $x$ such that $|s_0|+|s'_0|\leq 1/2$. Since $y_0,y'_0$ have different projections by $\pi_x$,
the Local invariance implies that $|s-s_0+s'_0|\ge2$. This gives $|s|> 3/2$.

These two properties imply that for any $y,y'\in \widehat R$ with $y'=\varphi_s(y)$, then $s\notin (3/4,3/2)$
and these points have the same projection by $\pi_x$ if $|s|\leq 3/4$.

Let us assume that $3/2\leq s\leq t$. One considers by the Local injectivity
$s'$ such that $|s-s'|\leq 1/4$, $\varphi_{s'}(y)$ is $r/2$-close to $x$
and $\pi_x(\varphi_{s'}(y))=\pi_x(\varphi_{s}(y))$.
Then $s'\geq 1$ and by definition of $t$, one gets $s'\geq t$.
Consequently $s\geq t-1/4$ and $\pi_x(\varphi_s(y))=\pi_x(\varphi_t(y))$.
\end{proof}

\begin{Lemma}\label{l.continuous}
The map $f$ is continuous.
\end{Lemma}
\begin{proof}
Fix $u\in\Lambda$. There is $y\in \widehat R$ such that $d(x,y)< r/2$ and $u=\pi_x(y)$.
For any $u'\in \Lambda$ close to $u$,
there exists $y'\in \widehat R$ with the same properties and such that $\pi_y(y')$ is arbitrarily close to $0_y$.
So by the Local injectivity and the ``No small period" assumption, one can choose $y'$ arbitrarily close to $y$.
Let $t,t'$ be the times associated to $y,y'$ as in Lemma~\ref{l.orbit-f}.

By continuity of the flow,
$\varphi_t(y')$ is $r_0$-close to $x$ and has a projection by $\pi_x$
close to $\pi_x(\varphi_t(y))\in \Lambda$.
Since $\Lambda$ is compact and contained in the interior of $R$,
$\pi_x(\varphi_t(y'))$ belongs to $R$ (hence to $\Lambda)$.
We claim that it coincides with $f(u')$ which will conclude the proof.

Let us assume by contradiction that
$\pi_x(\varphi_t(y'))\neq \pi_x(\varphi_{t'}(y'))$.
Lemma~\ref{l.orbit-f} implies that $t\geq t'+3/2$.
Then $\varphi_{t'}(y)$ is $r_0/2$-close to $x$ and projects by $\pi_x$ to $R$.
We get $\varphi_{t'}(y)\in \widehat R$ with $1\leq t'\leq t-3/2$, contradicting Lemma~\ref{l.orbit-f}.
\end{proof}

By repeating the above construction for negative times, we will obtain another map.
Then Lemma~\ref{l.orbit-f} shows that it is the inverse of $f$. Since $\varphi$ is minimal, this gives:

\begin{Corollary}
$f$ is a homeomorphism and induces a minimal dynamics on $K$.
\end{Corollary}
\noindent
\paragraph{Extension of $f$ as a local diffeomorphism.}
For any $u\in \Lambda$ we choose $y$ and $t$ as in the definition of $f$.
One gets a local $C^2$-diffeomorphism $f_u$
from a (uniform) neighborhood of $u$
to a (uniform) neighborhood of $f(u)$
defined by $\pi_xP_t\pi_y$.
(The uniformity comes from the fact that $t$ is uniformly bounded.
The Local invariance shows it does not depend on $y$.)

\begin{Lemma}
There exists a local diffeomorphism on a neighborhood of $\Lambda$
which extends $f$ and each $f_u$.
\end{Lemma}
\begin{proof}
For $u',u\in \Lambda$ that are close, we have to show that the diffeomorphisms $f_u,f_{u'}$
matches on uniform neighborhood of $u$ and $u'$.
Let us consider $y,y'$ and $t,t'$ defining the local diffeomorphisms,
such that $y,y'$ are $r/2$-close to $x$. Note that
from the proof of Lemma~\ref{l.continuous},
$y,y'$ (resp. $t,t'$) can be chosen arbitrarily close
if $u,u'$ are close.

Take $z \in \cN_x$ in the intersection of the domains close to $\pi_x(u)$ and $\pi_x(u')$.
Its projection by $\pi_y$ and $\pi_{y'}$
gives $v\in\cN_y$ and $v'\in \cN_{y'}$ close to $0_y$ and $0_{y'}$
whose orbits under $P$ remain close to the zero section during the time $t$ (resp. $t'$).
By Global invariance, there exists an increasing homeomorphism $\theta$ of $\RR$ close to the identity
such that
$$f_u(z )=\pi_xP_t(v)=\pi_xP_{\theta(t)}(v')=\pi_xP_{t'}(v')=f_{u'}(z ).$$

One deduces that the maps $f_u$ define a $C^2$ map on a neighborhood of $\Lambda$
which extends $f$. Since the same construction can be applied with the local diffeomorphisms $f_u^{-1}$,
one concludes that $f$ is a diffeomorphism.
\end{proof}

\noindent
\paragraph{Extensions of the bundles $E,F$.}
\begin{Lemma}
The tangent bundle over $\Lambda$ admits a splitting $E\oplus F$ which is invariant and dominated by $f$.
Moreover $E$ is uniformly contracted by $f$ on $\Lambda$ if and only if $\cE$ is uniformly contracted by the flow $(P_t)$ on $K$.
\end{Lemma}
\begin{proof}
At each point $u\in \Lambda$ we define the spaces $E(u),F(u)$ as the image by $D\pi_x(0_y)$
of $\cE(y),\cF(y)$ where $y\in \widehat R$ and $\pi_x(y)=u$.
These spaces are well defined: if $y'\in\widehat R$ also satisfies $\pi_x(y')=u$, then $y'=\varphi_t(y)$
for some $t\in [-1,1]$; the Local injectivity and the invariance of the bundles $\cE$ implies that
$D\pi_x(0_y).\cE(y)=D\pi_x(0_{y'}).\cE(y')$. The same holds for $\cF$.
The continuity of the families $\pi_{y,x}$ and of the bundles $\cE,\cF$ over the
$0$-section of $\cN$ implies that $E,F$ are continuous over $\Lambda$.

Let us consider $u'=f(u)$ and two points $y,y'\in \widehat R$ that are $r_0/2$-close to $x$ such that
$\pi_x(y)=u$ and $\pi_{x}(y')=u'$. Then, there exits $t>0$ such that $\varphi_t(y)=y'$ so that
$DP_t(0_y).\cE(y)=\cE(y')$. Consequently, we obtain the invariance of $E$ by $Df$:
$$Df(u).E(u)=D\pi_{x}(0_{y'})\circ DP_t(0_y)\circ D\pi_y(0_x)\circ D\pi_x(0_y).\cE(y)=D\pi_{x}(0_{y'}).\cE(y')=E(u').$$

Note that the splitting $E\oplus F$ on $\Lambda$ is dominated for the dynamics of $Df$
since $Df^N(u)$ coincides for $N$ large with $D\pi_x\circ DP_t\circ D\pi_y$ for some large $t>0$
and some $y\in \widehat R$ satisfying $u=\pi_x(y)$ and since $\cE\oplus \cF$ is dominated for the dynamics of $DP_t$.
Since all orbits of $f$ correspond to the orbit under $\varphi$ (by minimality of $K$),
the argument proves that $E$ is uniformly contracted by $f$ on $\Lambda$ if and only if $\cE$ is uniformly contracted by the flow $(P_t)$ on $K$.
\end{proof}

\noindent
\paragraph{End of the proof of Proposition~\ref{Pro:measure-contracted} in the minimal case.}
Since the set $K$ is minimal and not a periodic orbit, the set $\Lambda$ does not
contain any periodic orbit. Note that $\Lambda$ cannot contain a closed curve
tangent to $E$ nor a closed curve tangent to $F$ since $\Lambda$
is contained in $R$ which has arbitrarily small diameter.
So Pujals-Sambarino's theorem applies and $\Lambda$ is a hyperbolic set for $f$.
This implies that $E$ and $\cE$ are uniformly contracted by $f$ and $(P_t)$ respectively as required.
\medskip

The proof of Proposition~\ref{Pro:measure-contracted} is now complete. \qed

\subsection{Fibered version of Ma\~n\'e-Pujals-Sambarino's theorem}\label{ss.Dcontracting}

\begin{proof}[Proof of Theorem~\ref{Thm:1Dcontracting}]
Let us assume that a local fibered flow $(\cN,P)$ satisfies the assumptions of Theorem~\ref{Thm:1Dcontracting}.
We suppose furthermore that $K$ does not contain a normally expanded irrational torus,
and that $\cE$ is uniformly contracted over each periodic orbit.

Assume by contradiction that the bundle $\cE$ is not uniformly contracted.
Then, there exists a non-empty invariant compact subset $\widetilde K\subset K$ such that
\begin{itemize}
\item[--] $\cE$ is is not uniformly contracted over $\widetilde K$,
\item[--] but $\cE$ is uniformly contracted over any invariant compact proper subset $\widetilde K'\subset \widetilde K$.
\end{itemize}
The assumptions (A1), (A2), (A3) are satisfied and
the Theorem~\ref{Thm:topologicalcontracting} can be applied to $\widetilde K$.
By our assumptions, the two first conclusions are not satisfied,
hence the bundle $\cE$ over $\widetilde K$ is topologically contracted. Note also that since $\cE$ is contracted over periodic orbits of $K$,
the set $\widetilde K$ is not reduced to a periodic orbit.
The properties (B1), (B2), (B3), (B4) are satisfied on $\widetilde K$.

Since $\cE$ is not uniformly contracted over $\widetilde K$ and is one-dimensional,
there exists an ergodic measure $\mu$ with support contained in $\widetilde K$ whose Lyapunov exponent along $\cE$
non-negative. By domination, the Lyapunov exponent along $\cF$ is positive.
Since $\cE$ is uniformly contracted over any invariant proper compact subset,
the support of $\mu$ coincides with $\widetilde K$.
Proposition~\ref{Pro:measure-contracted} applies to $\widetilde K$ and $\mu$
and contradicts the fact that the Lyapunov exponent of $\mu$ along $\cE$ is non-negative.
Hence $\cE$ is uniformly contracted over $K$.
\end{proof}

\section{Generalized Ma\~n\'e-Pujals-Sambarino theorem for singular flows}\label{s.MPS-theo}

In this section, we will prove Theorem~A' by using Theorem~\ref{t.compactification} and Theorem~\ref{Thm:1Dcontracting}.
We consider a manifold $M$ and an invariant compact set $\Lambda$ for a $C^2$ vector field $X$ on $M$ whose singularities are hyperbolic
and have simple real eigenvalues (in particular their number is finite).
The results trivially holds for isolated singularities (since by assumption they admit a negative Lyapunov exponent).
Hence, it is enough to assume that the set of regular orbits is dense in $\Lambda$.

In the last subsection we will prove the easy side of Theorem A'.
In all the other subsections, 
we assume that
the linear Poincar\'e flow on $\Lambda\setminus {\rm Sing}(X)$ has a dominated splitting $\cN=\cE\oplus \cF$
and prove the existence of a dominated splitting for the tangent flow.

\subsection{Compactification}
One first applies Theorem~\ref{t.compactification} and gets maps
$$i\colon \Lambda \setminus \sing(X)\to K:=\widehat \Lambda
\text{ and } I\colon \cN M|_{\Lambda \setminus \sing(X)}\to \cN:=\widehat {\cN M}.$$

The set $K$ is the closure of $\Lambda\setminus \sing(X)$
in the blowup $\widehat M$ of $M$ at each singularity $\sing(X)\cap \Lambda$,
so that the map $i$ is the canonical injection of $\cN M|_{\Lambda \setminus \sing(X)}$ in $\cN$, and $I$ is the canonical injection of $\cN M|_{\Lambda \setminus \sing(X)}$
inside the compactification $\cN$. In the following we drop the injections $i$ and $I$.

The set $K$ is endowed with a flow $\widehat \varphi$ which extends the flow $\varphi$ on $\Lambda\setminus \sing(X)$.
The rescaled sectional Poincar\'e flow extends as a $C^2$ local fibered flow $\widehat P^*$ in a neighborhood of the $0$-section of
$\cN$ over $K$. The fibers of $\cN$ have dimension $\dim(M)-1$.

\subsection{Identifications}
We now choose an open set $U\subset K$
such that $K\setminus U$ is an arbitrarily small neighborhood of
the compact set $K\setminus (\Lambda\setminus \sing(X))$.
For any singularity $\sigma\in \Lambda$, we denote by $d^s,d^u$ its stable and unstable dimensions. {Since it is hyperbolic, $d^s+d^u=\dim M$.}
Let us choose a $C^1$ chart at $\sigma$ which identifies $\sigma$ with $0\in \RR^{d^s+d^u}$,
and the local stable and unstable manifolds with $\RR^{d^s}\times \{0\}$ and $\{0\}\times \RR^{d^u}$.
There exists two (closed) differentiable balls $B^s\subset \RR^{d^s}\times \{0\}$ and $B^u\subset \{0\}\times \RR^{d^u}$
that are transverse to the linear vector field $u\mapsto DX(0).u$ { on $\RR^{d^s+d^u}\setminus \{0\}$}.
For $\varepsilon>0$ small, the vector field $X$ is transverse to the boundary
of the $\varepsilon$-scaled neighborhood $B_\sigma=\varepsilon.(B^s\times B^u)$
at any point of $\partial B^s\times \Int(B^u)$ and of
$\Int(B^s)\times \partial B^u$. Note that for $x\in \partial B^s\times \partial B^u$,
the image $\varphi_t(x)$ does not belong to $B_\sigma$ for any small $t\neq 0$.
The union of the $B_\sigma${s for all singularities} is a neighborhood of $K\setminus (\Lambda\setminus \sing(X))$ and
its complement defines the open set $U$.
Thus, the ``Transverse boundary" property holds.
\medskip

Since there is no fixed point of $\widehat \varphi$ in $\overline U$,
one can rescale the time (i.e. consider the new flow $t\mapsto {\widehat\varphi}_{t/C}$ for some large $C>1$)
so that any periodic orbit which meets $\overline U$ has period larger than $10$.
Then the ``No small period" property follows from the continuity of the flow.
\medskip

If $\overline \beta>0$ is small enough, for any point $x$ in a neighborhood of $\overline U$
the image of $B(0,\overline \beta)\subset \cN_x$ by the exponential map $\exp_x$ is transverse to the vector field $X$.
Consequently, for $\varepsilon>0$ small,
if $\beta_0\in (0,\overline\beta)$ and $r_0>0$ are much smaller than $\varepsilon$,
for any point $y\in M$ such that $d(x,y)<r_0$ and any $u\in B(0,\beta_0)\subset \cN_y$,
there exists a unique $s\in (-\varepsilon,\varepsilon)$ satisfying
$\varphi_s(\exp_y(u))$ belongs to $\exp_x(B(0,\overline \beta))$. After rescaling, we thus define the identification
$$\pi_{y,x}(u):=\|X(x)\|^{-1}.\exp_x^{-1}\circ \varphi_s\circ \exp_y(\|X(y)\|.u).$$
Since $X$ is $C^2$, the map $\pi_{y,x}$ is $C^2$ also.
By the uniqueness of the parameter $s$, we obtain the relation $\pi_{z,x}\circ \pi_{y,z}=\pi_{y,x}$.
\medskip

The Local injectivity now follows immediately by choosing $t=s$ as in the definition of the identification $\pi_{y,x}$.
Let us consider $x,y\in U$, $t\in [-2,2]$ and $u\in B(0_y, \beta_0)$ such that $y$ and $\varphi_t(y)$ are $r_0$-close to $x$.
If $r_0$ has been chosen small enough the ``No small period" property  implies that $t$ is small.
Then the uniqueness of $s$ in the definition of the identification $\pi$ implies
$\pi_x\circ \widehat P^*_t(u)=\pi_x(u)$. This gives the Local invariance.

\subsection{Global invariance}

The following two lemmas follow from the fact that 
the vector field $X$ is almost constant in the {$\varkappa\|X(y)\|$-neighborhood of $y$} for $\varkappa>0$ small enough,

\begin{Lemma}\label{l.flow}
For any $\rho>0$, there exists $\delta>0$ with the following property.

If $y\in \Lambda\setminus \sing(X)$ and
if $z=\exp_y(u)$ for some $u\in B(0_y,\delta\|X(y)\|)\subset \cN_y$,
then for any $s\in (0,1)$, there exists a unique $s'\in (0,2)$ such that
$\varphi_{s'}(z)=\exp_{\varphi_s(y)}\circ P_s(u)$.
Moreover $\max(s/s',s'/s)<1+\rho/3$.
\end{Lemma}

\begin{Lemma}\label{l.project}
For any $\delta,t_0>0$, there exists $\beta>0$ with the following property.

For any $y\in \Lambda\setminus \sing(X)$ and $z\in M$ such that
$d(z,y)\leq 10\beta\|X(y)\|$, there exists a unique $t\in (-t_0,t_0)$
such that $\varphi_t(z)$ belongs to the image by $\exp_y$ of $B(0_y,\delta{\|X(y)\|})\subset \cN_y$.
\end{Lemma}

We can now check the last item of the Definition~\ref{d.compatible} for the local fibered flow $\widehat P^*$.
Let us fix $\delta,\rho>0$ small: by reducing $\delta$, one can assume that Lemma~\ref{l.flow} above holds.
One then chooses $t_0>0$ small such that $d(x,\varphi_t(x))<\delta/3$ for any $x\in M$ and $t\in [-t_0,t_0]$.
One fixes $r>0$ and $\beta\in (0,\delta)$ small such that:
\begin{itemize}
\item[a--]
for any $y,y'$ and $u\in \cN_y$, $u'\in \cN_{y'}$ as in the statement of Global invariance, then
$\varphi_t\exp_y({\|X(y)\|}u)=\exp_{y'}({\|X(y')\|}u')$ for some $t\in [-t_0,t_0]$ (arguing as in the proof of Local injectivity), 
\item[b--] $\delta,\beta$ satisfy the Lemma~\ref{l.project},
\item[c--] for any $x\in M$, {$d(x,\exp_x(w))<\delta/3$ for any $w\in T_xM$ satisfying $\|w\|\le 10\beta\|X(x)\|$.}
\end{itemize}

Consider any $y,y'$ , $u\in \cN_y$, $u'\in \cN_{y'}$ and $I,I'$ as in the statement of the Global invariance.
Lemma~\ref{l.flow} can be applied
to $y$ and $z=\exp_y(\|X(y)\|.u)$: for each $s\in (0,1)$, one defines $\theta_0(s)\in (0,2)$ to be equal to the $s'$ given by Lemma~\ref{l.flow}.
The map $\theta_0$ is $(1+\rho/3)$-bi-Lipschitz and increasing. Moreover $\theta_0(0)=0$.
Since $\| \widehat P^*_s(u)\|\leq \beta<\delta$ for any $s\in I$,
one has $P_s(\|X(y)\|.u)\in B(0,\delta\|X(\varphi_s(x))\|)$ and
one can apply inductively Lemma~\ref{l.flow} to the points $\varphi_s(y)$ and $P_s(\|X(y)\|.u)$,
which defines $\theta_0$ on $I$. One gets:
$$\forall s\in I,~~~\exp_{\varphi_s(y)}\circ P_s(\|X(y)\|.u)=\varphi_{\theta_0(s)}(z).$$
The same argument for $y'$ and $z'=\exp_{y'}(\|X(y')\|.u')$ defines a map $\theta_0'\colon I'\to \RR$.
\medskip

Let us now consider $s\in I\cap \theta_0^{-1}\circ \theta_0'(I')$.
{By (a),} since $\pi_y(u')=u$, there exists $t\in [-t_0,t_0]$ such that $\varphi_t(z)=z'$.
By the definition of $\theta_0$, the points $\varphi_s(y)$ and $\varphi_{\theta_0(s)}(z)$ are $\delta/3$-close.
Since $|t|\leq t_0$, the points $\varphi_{\theta_0(s)}(z)$  and $\varphi_{\theta_0(s)+t}(z)= \varphi_{\theta_0(s)}(z')$ are $\delta/3$-close.
Since $\theta_0(s)\in \theta'_0(I')$ and using (c) above, the points $\varphi_{\theta_0(s)}(z')$ and $\varphi_{(\theta_0')^{-1}\circ\theta_0(s)}(y')$ are $\delta/3$-close.
Consequently, the points $\varphi_s(y)$ and $\varphi_{\theta(s)}(y')$ are $\delta$-close, where $\theta=(\theta_0')^{-1}\circ\theta_0$.
Note that $\theta$ is bi-Lipschitz for the constant $(1+\rho/3)^2<1+\rho$ and satisfies $\theta(0)=0$. This proves the first part of the Global invariance.
\medskip

Finally we take $v\in \cN_y$, $v'=\pi_{y'}(v)$ in $\cN_{y'}$ such that
$\|\widehat P^*_{s}(v)\|<\beta$ for each $s\in I\cap \theta^{-1}(I')$.
Set $\zeta=\exp_y(\|X(y)\|.v)$.
{By Lemma~\ref{l.project},} there exists a unique $t'\in (-t_0,t_0)$
such that {$\varphi_{t'}(\zeta)=\exp_{y'}(w')$ for some $w'\in B(0,\delta)\subset \cN_{y'}$.}
By definition of $\pi_{y,y'}$, it coincides with
$\exp_{y'}(\|X(y')\|.v')$.

Arguing as above, there exists $\theta_1$ such that
$\exp_{\varphi_s(y)}\circ P_s(\|X(y)\|.v)=\varphi_{\theta_1(s)}(\zeta)$ and $\theta_1(0)=0$.
$d(\varphi_{\theta_1(s)}(\zeta), \varphi_s(y))$ is smaller
than $\|P_s(\|X(y)\|.v)\|$ and $\beta\|X(\varphi_s(y))\|$.
Similarly,
$$d(\varphi_{\theta_0(s)}(z),\varphi_s(y))\leq 
 \beta\|X(\varphi_s(y))\|.$$
 $$d(\varphi_{\theta'_0(s')}(z'),\varphi_{s'}(y'))\leq 
 \beta\|X(\varphi_{s'}(y'))\|.$$
 Since $z'=\varphi_t(z)$, to each $s\in I\cap \theta^{-1}(I')$
 one associates $s'$ such that $\theta_0(s)=\theta'_0(s')+t$, one gets
 $$d(\varphi_{\theta_0(s)}(z),\varphi_{s'}(y'))\leq 
 \beta\|X(\varphi_{s'}(y'))\|.$$
 If $\beta$ has been chosen small enough, one deduces that $\beta\|X(\varphi_{s}(y))\|$ and $\beta\|X(\varphi_{\theta_0(s)}(z))\|$
 are smaller than $2\beta \|X(\varphi_{s'}(y'))\|$.
Hence,
\begin{equation}\label{e.bound-projection}
d(\varphi_{\theta_1(s)}(\zeta),\varphi_{s'}(y'))\leq 
 5\beta\|X(\varphi_{s'}(y'))\|.
 \end{equation}
By Lemma~\ref{l.project},
 one can find $\sigma(s)\in (-t_0,t_0)$
 such that $\varphi_{\theta_1(s)+\sigma(s)}(\zeta)$ belongs to the image by $\exp_{\varphi_{s'}(y')}$ of
 $B(0,\delta)\subset \cN_{\varphi_{s'}(y')}$. In the case $s'=0$, since $t$ and $\sigma$ are small,
 {the definition of $\pi_y$ gives} $\varphi_{\theta_1(s)+\sigma(s)}(\zeta)=\exp_{y'}(\|X(y')\|.v')$. Applying Lemma~\ref{l.flow} {inductively}, one has that $\varphi_{\theta_1(s)+\sigma(s)}(\zeta)=\exp_{\varphi_{s'}(y')}(P_{s'}(\|X(y')\|.v'))$ for any $s\in I\cap \theta^{-1}(I')$.
By~\eqref{e.bound-projection} and Lemma~\ref{l.project}, one deduces $\|P_{s'}(\|X(y')\|.v')\|\leq \delta\|X(\varphi_{s'}(y'))\|$,
that is $\|\widehat P^*(v')\|\leq \delta$ as wanted.

We have obtained
$$\varphi_{\sigma(s)}\circ \exp_{\varphi_s(y)}\circ P_s(\|X(y)\|.v)=\exp_{\varphi_{s'}(y')}(P_{s'}(\|X(y')\|.v')).$$
When $\varphi_s(y)$ and $\varphi_{s'}(y')$ are in a neighborhood of $\overline U$,
one deduces by definition of identification,
$$\pi_{\varphi_{s}(y)}\circ P_{s'}(\|X(y')\|.v')= P_s(\|X(y)\|.v).$$
By definition of $\theta$ and $s'$, one notices that $\theta(s)$ and $s'$ are close.
Hence $$\pi_{\varphi_{s'}(y')}\circ P_{\theta(s)}(\|X(y')\|.v')=P_{s'}(\|X(y')\|.v').$$
This gives $\pi_{\varphi_s(y)}\circ \widehat P^*_{\theta(s)}(v')=\widehat P^*_s(v)$ and completes the proof of the Global invariance.

\subsection{Dominated splitting}
We have assumed that the linear Poincar\'e flow $\psi$ on $\Lambda\setminus \sing(X)$
admits a dominated splitting, denoted by  $\cN M|_{\Lambda\setminus \sing(X)}= \cE\oplus \cF$.
It extends as a dominated splitting $\cN=\widehat \cE\oplus \widehat \cF$ over $K$ for the extended rescaled linear Poincar\'e flow $\widehat \psi^*$ (hence for $\widehat P^*$).
Indeed:
\begin{itemize}
\item[--] dominated splittings are invariant under rescaling, hence $\cE\oplus \cF$ is a dominated splitting for
$\psi^*$ (and $\widehat \psi^*$) over $\Lambda\setminus \sing(X)$;
\item[--] for continuous linear cocycles, dominated splittings extend to the closure.
\end{itemize}

The existence of a dominated splitting for the tangent flow on $\Lambda$ can be restated
as the uniform contraction of the bundle $\widehat \cE$.

\begin{Proposition}\label{p.mixed-domination}
Under the previous assumptions, these two properties are equivalent:
\begin{itemize}
\item[I-- ] There exists a dominated splitting $T_\Lambda M=E\oplus F$ for the tangent flow $D\varphi$ with $\dim(E)=\dim(\widehat \cE)$
and $X\subset F$;
\item[II-- ] $\widehat \cE$ is uniformly contracted by $\widehat P^*$ (and $\widehat \psi^*$) over $K$.
\end{itemize}
\end{Proposition}
\begin{proof}
Let us prove $I\Rightarrow II$.
From Property I, we have a dominated splitting between $E$ and $\RR X$, hence there exists
$C>0$ and $\lambda\in (0,1)$ such that for any $t>0$ and any $x\in \Lambda$,
$$\|D{\varphi_t}|{E}(x)\|\leq C\lambda^t \|D{\varphi_t}|{\RR X}(x)\|=
C\lambda^t\frac{\|X(\varphi_t(x))\|}{\|X(x)\|}.$$
Since the angle between $E$ and $X$ is uniformly bounded away from zero,
the projection of {$E(z)$ on $X(z)^\perp$} and its inverse are uniformly bounded, hence the ratio between
$\|D{\varphi_t}|{E}(x)\|$ and $\|{\psi_t}|{\widehat \cE}(x)\|$ is bounded.
This implies that there exists $C'$ such that:
$$\|\widehat {\psi_t^*}|{\widehat \cE}(x)\|=\frac{\|X(x)\|}{\|X(\varphi_t(x))\|}\|{\psi_t}|{\widehat \cE}(x)\|\leq C'\lambda^t.$$
The bundle $\widehat \cE$ is thus uniformly contracted over $\Lambda\setminus \sing(X)$, hence over {$K$} also. This gives Property {II}.
The implication $II\Rightarrow I$ is a restatement of \cite[Lemma 2.13]{GY}.
\end{proof}

\subsection{Uniform contraction of $\widehat \cE$ near the singular set}
{From the assumptions of Theorem~A',} any singularity $\sigma\in \Lambda$ has a dominated splitting
$T_\sigma M=E^{ss}\oplus F$, where $E^{ss}$ is uniformly contracted, has the same dimension as $\cE$,
and the associated invariant manifold $W^{ss}(\sigma)$ intersects $\Lambda$ only at $\sigma$.
Let $V$ be a small open neighborhood of the {compact set $K\setminus (\Lambda\setminus \sing(X))$}.

\begin{Lemma}\label{l.contraction-V}
The bundle $\widehat \cE$ is uniformly contracted on $V$ by $\widehat P^*$.
\end{Lemma}
\begin{proof}
We use the notations and the discussions of section~\ref{ss.blow-up}.
For each singularity $\sigma\in \Lambda$, let $\Delta_\sigma$ be the set of unit vectors $u\in {E^{cu}(\sigma)}\subset T_\sigma M$.
It is compact and $\widehat \varphi$-invariant. The splitting at $\sigma$ induces a dominated splitting
$E^{ss}\oplus {E^{cu}}$ of the extended bundle $\widehat {TM}$ over $\Delta_\sigma$.

For regular orbits near $\Delta_\sigma$, the lines $\RR X$ are close to ${E^{cu}(\sigma)}$, hence have a uniform angle with $E^{ss}$.
Consequently, the dominated splitting $E^{ss}\oplus { E^{cu}}$ over $\Delta_\sigma$ projects on the extended normal bundle
$\widehat \cN$ over $\Delta_\sigma$ as a splitting $\cE'\oplus \cF'$ where $\dim(\cE')=\dim(E^{ss})=1$, which is dominated for the
linear Poincar\'e flow $\psi$ (and hence for $\psi^*$ also).

The dominated splitting $E^{ss}\oplus { E^{cu}}$ induces a dominated splitting between $E^{ss}$ and the extended line field
$\RR \widehat{X_1}$. The proof of Proposition~\ref{p.mixed-domination} above shows that the extended rescaled linear Poincar\'e flow
$\widehat \psi^*$ contracts $\cE'$ (above $\Delta_\sigma$).

On $K\cap\Delta_\sigma$, the dominated splittings $\widehat \cE\oplus \widehat \cF$ and $\cE'\oplus \cF'$ have the same dimensions,
hence coincide. Moreover since $W^{ss}(\sigma)\cap \Lambda=\{\sigma\}$, we have $p^{-1}(\sigma)\cap K=\Delta_\sigma\cap K$,
{where $p$ denotes the projection $K\to \Lambda$}.
This shows that $\widehat \cE$ is uniformly contracted by $\widehat P^*$ over $p^{-1}(\sing(X)\cap \Lambda)\cap K$,
hence on any small neighborhood $V$.
\end{proof}

\subsection{Periodic orbits and normally expanded invariant tori}
We now check that the two first conclusions of Theorem~\ref{Thm:1Dcontracting} do not hold.

\begin{Lemma}
For any periodic orbit $\cO$ in $K$, the bundle $\widehat \cE|_{\cO}$ is uniformly contracted.
\end{Lemma}
\begin{proof}
By Lemma~\ref{l.contraction-V}, the bundle $\widehat \cE$ is contracted over periodic orbits contained in $V$.
The other periodic orbits are lifts (by the projection $p\colon \widehat M\to M$) of orbits in $\Lambda\setminus \sing(X)$.
{ From the assumptions of Theorem~A',}
the Lyapunov exponents along $\cE$ are all negative for periodic orbits
in  $\Lambda\setminus \sing(X)$. This concludes.
\end{proof}

\begin{Lemma}
There do not exist any normally expanded irrational torus $\cT$ for $(K,\widehat \varphi)$.
\end{Lemma}
\begin{proof}
We now { use the fact that} $M$ is three-dimensional.
Let us assume by contradiction that there exists a normally expanded irrational torus $\cT$ for $(K,\widehat \varphi)$.
By Lemma~\ref{l.torus}, the bundle $\widehat \cF$ is uniformly expanded over $\cT$.
Since it does not contain fixed point of $\widehat \varphi$, it projects by $p$ in $\Lambda\setminus \sing(X)$.
By construction, the dynamics of $\widehat \psi^*$ over $\cT$ and $\psi^*$ over $p(\cT)$ are the same,
hence $\cF$ is uniformly expanded over $p(\cT)$ by $\psi^*$.
Since $p(\cT)\cap \sing(X)=\emptyset$, one deduces that $\cF$ is uniformly expanded over $p(\cT)$ by $\psi$.
Proposition~\ref{p.mixed-domination} implies that the tangent flow over $p(\cT)$ has a dominated
splitting $TM|_{p(\cT)}=E^c\oplus E^{uu}$, with $\dim(E^{uu})=1$.

As a partially hyperbolic set each $x\in p(\cT)$ has a strong unstable manifold $W^{uu}(x)$.
Note that $W^{uu}(x)\cap p(\cT)=\{x\}$ (since the dynamics is topologically equivalent to
an irrational flow). Then~\cite{BC-whitney} implies that $p(\cT)$ is contained in a two-dimensional submanifold $\Sigma$
transverse to $E^{uu}$ and locally invariant by $\varphi_1$: there exists a neighborhood $U$ of $p(\cT)$ in $\Sigma$ such that
$\varphi_1(U)\subset \Sigma$. Since $p(\cT)$ is homeomorphic to $\TT^2$, it is open and closed in $\Sigma$, hence coincides with $\Sigma$.
This shows that $p(\cT)$ is $C^1$-diffeomorphic to $\TT^2$, normally expanded
and carries a dynamics topologically equivalent to an irrational flow.
This contradicts the assumptions of Theorem~A'.
Consequently there do not exist any normally expanded irrational torus $\cT$ for $(K,\widehat \varphi)$.
\end{proof}

\subsection{Proof of the domination of the tangent flow}
Under the assumptions of Theorem~A',
the fibers {of $\cE$ and $\cF$ are one-dimensional.}
Note that one can choose $U$ after $V$ such that $U\cup V=K$.
We have thus shown that $\widehat P^*$ over $K$ satisfies the setting of Theorem~\ref{Thm:1Dcontracting},
that $\widehat \cE$ is uniformly contracted over the periodic orbits and that there is no
normally expanded irrational torus.
One deduces that $\widehat \cE$ is uniformly contracted by $\widehat P^*$ above $K$.
Proposition~\ref{p.mixed-domination} then implies that there exists a dominated splitting $T_\Lambda M=E\oplus F$
such that $E$ is one-dimensional (as for $\widehat \cE$) and $X(x)\subset F(x)$ for any $x\in \Lambda$.

This shows one side of Theorem A': a domination of the linear Poincar\'e flow
implies a domination $TM|_\Lambda=E\oplus F$ of the tangent flow with $\dim(E)=1$.

\subsection{Proof of the domination of the linear Poincar\'e flow}
The other direction of Theorem A' is easier.

\begin{Proposition}\label{p.oneside}
Under the assumptions of Theorem A', if there exists a dominated splitting
$TM|_{\Lambda}=E\oplus F$ with $\dim(E)=1$ for the tangent flow $D\varphi$,
then:
\begin{itemize}
\item[--] $X(x)\subset F(x)$ for any $x\in \Lambda$,
\item[--] the linear Poincar\'e flow on $\Lambda\setminus {\rm Sing}(X)$ is dominated,
\item[--] $E$ is uniformly contracted.
\end{itemize}
\end{Proposition}
\begin{proof}
We first prove that $X(x)\in F(x)$ for any $x\in \Lambda$. Otherwise, using the domination,
there exists a non-empty
invariant compact subset $\Lambda'$ such that $X(x)\in E(x)$ for any $x\in \Lambda'$
and there exists a regular orbit in $\Lambda$ which accumulates in the past on $\Lambda'$.

Let $\mu$ be an ergodic measure on $\Lambda'$. If ${\rm supp}(\mu)$ is not a singularity,
the domination implies that all the Lyapunov exponents along $\cF$ are positive,
hence the measure is hyperbolic and is supported on a periodic orbit that is a source.
This contradicts the assumptions of Theorem A'.

If $\mu$ is supported on a singularity $\sigma$, it is by construction
limit of regular points $x_n\in \Lambda$ such that $\RR X(x_n)$
converges towards $E(\sigma)$. This implies that one of the separatrices of $W^{ss}(\sigma)$
is contained in $\Lambda$, a contradiction. The first item follows.
\smallskip

Since $X\subset F$ and the angle between $E$ and $F$ is bounded away from zero,
the projection of $T_xM=E\oplus F$ to $\cN_x$ along $\RR X(x)$ defines a
splitting $\cE(x)\oplus \cF(x)$ into one-dimensional subspaces, for each
$x\in \Lambda\setminus {\rm Sing}(X)$.
This splitting is continuous and invariant under $\psi$.

Let us consider two non-zero vectors $u\in \cE(x)$ and $v\in \cF(x)$
and $t>0$. We have
$$\|\psi_{-t}.v\|\leq \|D\varphi_{-t}.v\|.$$
Since $E$ and $\cE$ are uniformly transverse to $\RR X$, there is a constant
$C>0$ such that
$$\|\psi_{-t}.u\|\geq C^{-1}\|D\varphi_{-t}.u\|.$$
The domination $E\oplus F$ thus implies the domination $\cE\oplus \cF$.
\smallskip

One can then apply Lemma~\ref{l.domina-contraction}, whose proof is contained in the proof of~\cite[Lemma 3.6]{BGY}.
\begin{Lemma}\label{l.domina-contraction}
Consider a $C^1$ vector field and
an invariant compact set $\Lambda$ endowed with a dominated splitting $T_\Lambda M=E\oplus F$
such that $E$ is one-dimensional, $X(x)\in F(x)$ for any $x\in \Lambda$ and $E(\sigma)$ is contracted for each $\sigma\in \Lambda\cap \sing(X)$.
Then $E$ is uniformly contracted.
\end{Lemma}
The bundle $E$ is thus uniformly contracted. This ends the proof of Proposition~\ref{p.oneside}
\end{proof}
The proof of Theorem A' is now complete.\qed

\section{$C^1$-generic three-dimensional vector fields}\label{s.generic}

In this section we prove Theorem~\ref{Thm-domination}, Corollary~\ref{c.main} and the Main Theorem.

\subsection{Chain-recurrence and genericity}\label{ss.chain}

For $\varepsilon>0$, we84 say that a sequence
$x_0,\dots,x_n$ is an $\varepsilon$-\emph{pseudo-orbit}
if for each $i=0,\dots,n-1$ there exists $t\geq 1$
such that $d(\varphi_t(x_i),x_{i+1})<\varepsilon$.
A non-empty invariant compact set $K$ for a flow $\varphi$ is \emph{chain-transitive} if for any $x,y\in K$ (possibly equal) and any $\varepsilon>0$,
there exists a $\varepsilon$-pseudo-orbit
$x_0=x,\dots,x_n=y$ with $n\geq 1$.

A \emph{chain-recurrence class} of $\varphi$ is a chain-transitive set
which is maximal for the inclusion.
The \emph{chain-recurrent set} of $\varphi$
is the union of the chain-recurrence classes.
See~\cite{Con}.
\medskip

We then recall known results on generic vector fields.
We say that a property is satisfied by generic vector fields in $\cX^r(M)$ if it holds
on a dense G$_\delta$ subset of of $\cX^r(M)$.

\begin{Theorem}\label{Lem:generic}
For any manifold $M$, any $r\geq 1$ and any generic vector field $X$ in $\cX^r(M)$,
\begin{itemize}
\item[--] each {periodic orbit or singularity}
is hyperbolic and has simple (maybe complex) eigenvalues,

\item[--] there do not exist any invariant subset $\cT$ which is diffeomorphic to $\TT^2$, normally expanded and supports
a dynamics topologically equivalent to an irrational flow.
\end{itemize}
\end{Theorem}
\begin{proof}
The first part is similar to the proof of the Kupka-Smale property~\cite{kupka,smale}.

The second part can be obtained from a Baire argument by showing that if $X\in \cX^r(M)$ preserves an invariant subset $\cT$ which is diffeomorphic to $\TT^2$, normally expanded,
and which supports a dynamics topologically equivalent to an irrational flow, then there exists a neighborhood $U$ of $\cT$
and an open set of vector fields $X'$ $C^r$-close to $X$ whose the maximal invariant set in $U$ is $\cT$
and whose dynamics on $\cT$ has a hyperbolic periodic orbit.
In order to prove this perturbative statement, one first notice that all the Lyapunov exponents along $\cT$ for $X$ vanish and
$\cT$ is $r$-normally hyperbolic. By~\cite[Theorem 4.1]{hirsch-pugh-shub}, the set $\cT$ is $C^r$-diffeomorphic to $\TT^2$ and any $C^r$-perturbation of $X|_{\cT}$
extends to $M$. Since Morse-Smale vector fields are dense in $\cX^r(\TT^2)$ by~\cite{Pe}, the result follows.
\end{proof}

Here are some consequence of the connecting lemma for pseudo-orbits.
\begin{Theorem}\label{t.connecting}
If $X$ is generic in $\cX^1(M)$, then for any non-trivial chain-recurrence class $C$ containing a hyperbolic singularity $\sigma$ whose unstable space is one-dimensional,
$C$ is Lyapunov stable and every separatrix of $W^u(\sigma)$ is dense in $C$.

In particular if $\dim(M)=3$, a chain-transitive set which strictly contains a singularity is a chain-recurrence class.
\end{Theorem}
\begin{proof}
The first part has been shown~\cite[Lemmas 3.14 and 3.19]{GY}: it is
a consequence of the version for flows of the connecting lemma proved in~\cite{BC}.

If $\dim(M)=3$, and if $\Lambda$ is a non-trivial chain-transitive set
containing a singularity $\sigma$, then $\sigma$ cannot be a sink, nor a source.
Let us assume that $\sigma$ has a one-dimensional unstable space (otherwise it has one-dimensional
stable space and the proof is similar).
From the first part, $\Lambda$ should contain one of the separatrix of
$W^u(\sigma)$. {Since every separatrix of $W^u(\sigma)$ is dense in $C(\sigma)$,} $\Lambda$ coincides with the chain-recurrence class of $\sigma$.
\end{proof}

In order to obtain the singular hyperbolicity on a chain-transitive set, 
{it suffices to} check that the tangent flow has a dominated splitting.
In the non-singular case, this is proved in~\cite[Lemma 3.1]{BGY} from~\cite{ARH}\footnote{Note that it could also be obtained from Theorem~A' with a Baire argument.}.

\begin{Theorem}\label{t.hyperbolic}
If $\dim(M)=3$ and if $X$ is generic in $\cX^1(M)$, then
for any chain-transitive set $\Lambda$ such that $\Lambda\cap \sing(X)=\emptyset$,
if the linear Poincar\'e flow on $\Lambda$ has a dominated splitting, then $\Lambda$ is hyperbolic.
\end{Theorem}

In the singular case, this is \cite[Theorem C]{GY}.

\begin{Theorem}\label{t.GY}
If $\dim(M)=3$ and if $X$ is generic in $\cX^1(M)$, then
any non-trivial chain-recurrence class whose tangent flow has a dominated splitting and which contains a singularity
is singular hyperbolic.
\end{Theorem}

Let us summarize some properties satisfied by $C^1$ generic vector fields that are away from homoclinic tangencies.

\begin{Theorem}\label{Thm:Lorenz-like}
Consider a generic $X\in \cX^1(M)$, a non-trivial chain-recurrent class $C$ and neighborhoods $\cU$, $U$
of $X$, $C$ such that for any $Y\in \cU$, the maximal invariant set of $Y$ in $U$ does not contain
a homoclinic tangency of a hyperbolic periodic (regular) orbit.
Then, there exists a dominated splitting on $C\setminus{\rm Sing}(X)$ for the linear Poincar\'e flow.
\end{Theorem}
\begin{proof}
This is a variation of the arguments of~\cite{GY}.
We first state a general genericity result.

\begin{Lemma}
For any generic $X\in \cX^1(M)$,  any non-trivial chain-transitive set $\Lambda$ is the limit for the Hausdorff topology of a sequence
of hyperbolic periodic saddles.
\end{Lemma}
\begin{proof}
Any non-trivial chain-transitive set is the limit for the Hausdorff topology of a sequence of hyperbolic periodic orbits $\gamma_n$.
This has been shown in~\cite{crovisier-approximation} for diffeomorphisms, but the proof is the same for vector fields.
If there exists infinitely many $\gamma_n$ that are saddles, the lemma is proved. One can thus deal with the case
all the $\gamma_n$ are sinks (the case of sources is similar).

By~\cite[Lemma 2.23]{GY}, the sinks $\gamma_n$ are not uniformly contracting at the period (see the precise definition there).
By~\cite[Lemma 2.6]{GY}, by an arbitrarily small $C^1$-perturbation, one can turn the $\gamma_n$, $n$ large, to saddles.
Then by a Baire argument, one concludes that $\Lambda$ is the limit of a sequence of hyperbolic periodic saddles.
\end{proof}

\cite[Corollary 2.10]{GY} asserts that if $\Lambda$ is the limit of a sequence of hyperbolic saddles for the Hausdorff topology,
then there exists a dominated splitting for the linear Poincar\'e flow on $\Lambda\setminus \sing(X)$,
assuming $X$ is not accumulated by vector fields in $\cX^1(M)$ with a homoclinic tangency.
The same proof can be localized, assuming that $\Lambda$ is a chain-recurrence class
and that there is no homoclinic tangency in a neighborhood of $\Lambda$ for vector fields $C^1$-close to $X$.
\end{proof}

We also state the result proved in~\cite{CY2} which asserts that
singular hyperbolicity implies robust transitivity
for generic vector fields in dimension $3$
(and improve a previous result by Morales and Pacifico~\cite{MP}).

\begin{Theorem}\label{t.robust-transitivity}
If $\dim(M)=3$ and if $X\in \cX^1(M)$ is generic,
any singular hyperbolic chain-recurrence class is robustly transitive.
\end{Theorem}

\subsection{Singularities of Lyapunov stable chain-recurrence classes}

The domination of the linear Poincar\'e flow constrains the local dynamics at singularities.

\begin{Proposition}\label{p:Lorenz-like2}
Assume $\dim(M)=3$.
Consider a generic $X\in \cX^1(M)$ and a non-trivial chain-recurrence class $C$ containing a singularity
{ with stable dimension equal to $2$} such that
there exists a dominated splitting on $C\setminus{\rm Sing}(X)$ for the linear Poincar\'e flow.

Then, any singularity $\sigma$ in $C$ has { stable dimension equal to $2$}, real simple eigenvalues
and satisfies $W^{ss}(\sigma)\cap C=\{\sigma\}$.
\end{Proposition}

Note that for any singularity $\sigma$ { with stable dimension equal to $2$} and real simple eigenvalues, any point $x\in W^u(\sigma)$ has a well defined two-dimensional
center unstable plane $E^{cu}(x)$ (it is the unique plane at $x$ which converge to the center-unstable plane of $\sigma$ by backward iterations).
We will use the next lemma.

\begin{Lemma}\label{l.non-domination}
Assume $\dim(M)=3$.
Consider any $X\in \cX^1(M)$, any singularity { with stable dimension equal to $2$} and real simple eigenvalues,
any $x\in W^u_{loc}(\sigma)\setminus \{\sigma\}$ {satisfying $\omega(x)\cap W^{ss}(\sigma)\setminus\{\sigma\}\neq\emptyset$.}
There exists $\alpha>0$ small such that for any neighborhood $\cU$ of $X$ in $\cX^1(M)$
and any $\varepsilon>0$,
there is $Y\in \cU$ satisfying:
\begin{itemize}
\item[--] $X=Y$ on $\{\varphi_{-t}(x), t\geq 0\}\cup \{\sigma\}$,
\item[--] there exists $s>0$ such that the flow $\varphi^Y$ associated to $Y$ satisfies
$$d(\varphi^Y_s(x),x)<\varepsilon \text{ and } d(D\varphi^Y_s(x).E^{cu}(x), E^{cu}(x))>\alpha.$$
\end{itemize}
\end{Lemma}
\begin{proof}
Up to replace $X$ by a vector field close, one can assume that:
\begin{itemize}
\item[--] $x\in (W^u(\sigma)\cap W^{ss}(\sigma))\setminus \{\sigma\}$ (using the connecting lemma~\cite{hayashi}).
\item[--] There exists a chart on a neighborhood of $\sigma$ which linearizes $X$: in particular, at any point $z$ in the chart
one defines the planes $H_z, V_z\subset T_zM$ which are parallel to $E^{ss}(\sigma)\oplus E^u(\sigma)$
and to $E^{cu}({\sigma})$ respectively;
the flow along pieces of orbits in the chart preserves the bundle $H$.
Moreover $E^{cu}(z)=V_z$ at points $z$ in the local unstable manifold of $\sigma$.
\item[--] $E^{cu}(x)$ is not tangent
to $T_xW^s(\sigma)$.
\end{itemize}

Let $\alpha>0$ smaller than $d(H_\sigma,V_\sigma)$.
{ Note that if $E\subset TM$ is a plane at a point $z$ in the orbit of $x$
such that $E\neq E^{cu}(z)$ and $X(z)\in E$, then
$D\varphi_{t}. E^{cu}(z)$ converge to the bundle $H$ when $t$ goes to $+\infty$.
This is a direct consequence of the dominated splitting between
$E^s$ and $E^u$.}

After a small perturbation which preserves $\{\sigma\}\cup \{\varphi_t(x),t\in \RR\}$
and  $\{D\varphi_{-t}.E^{cu}(x),t>0\}$ and whose support is disjoint from a small neighborhood of $\sigma$,
one can thus assume that there exist $t_1>0$ arbitrarily large such that for any $t>t_1$
$$ D\varphi_{t}. E^{cu}(x)=H_{\varphi_{t}(x)}.$$
{
Indeed after a small perturbation in a small neighborhood of $\varphi_1(x)$,
and which does not change the orbit of $x$, one can assume that
the new center-unstable space $E$ at $\varphi_2(x)$ does not coincide with
the initial one $E^{cu}_{\varphi_2(x)}$.
The property mentioned above then implies that
$ D\varphi_{t}. E^{cu}(x)$ converges $H_{\sigma}$ as $t$ goes to $+\infty$.
A new perturbation at a large iterate of $x$ then guaranties that
$ D\varphi_{t_1}. E^{cu}(x)=H_{\varphi_{t_1}(x)}.$
}

After a small perturbation near $\varphi_{t_1}(x)$,
one gets a forward iterate $\varphi_{t_2}(x)$, $t_2>t_1$, arbitrarily close to $x$
such that $D\varphi_{t_2}.E^{cu}(x)=H_{\varphi_{t_2}(x)}$.
This implies that for the new vector field $E^{cu}(x)\subset T_xM$ has a large forward iterate arbitrarily close
to $H_x\subset T_xM$ as required.
\end{proof}

It has the following consequence.

\begin{Corollary}\label{c.non-domination}
Assume $\dim(M)=3$.
For any generic $X\in \cX^1(M)$ and any chain-recurrence class $C$ containing a singularity $\sigma$
{ with stable dimension equal to $2$ and}
real simple eigenvalues such that $W^{ss}(\sigma)\cap C\neq \{\sigma\}$,
there exists $x\in W^u_{{loc}}(\sigma)\cap C$ and $t_n\to +\infty$ such that
\begin{equation}\label{e.non-domination}
\varphi_{t_n}(x)\rightarrow x \text{ and }
D\varphi_{t_n}(x).E^{cu}(x)\not \rightarrow E^{cu}(x).
\end{equation}
\end{Corollary}
\begin{proof}
For $\varepsilon, \delta, \alpha >0$ let us consider the open property:
\begin{description}
\item[$P(\varepsilon,\delta,\alpha)$:]
$\exists x\in W^u_\delta(\sigma),\exists\;s>0,\; d(\varphi_s(x),x)<\varepsilon \text{ and } d(D\varphi_s(x).E^{cu}(x), E^{cu}(x))>\alpha$.
\end{description}
Consider $X$ is generic and a singularity $\sigma$ { with stable dimension equal to $2$ and} real simple eigenvalues.
Lemma~\ref{l.non-domination} gives $\alpha>0$.
By genericity, for any integers $N_1,N_2$,
one can require that if $P(1/N_1, 1/N_2,\alpha)$ holds for an arbitrarily small perturbation of $X$ and the continuation of $\sigma$, then it holds also for $X, \sigma$.

If $W^{ss}(\sigma)\cap C\neq \{\sigma\}$, then Lemma~\ref{l.non-domination} shows that $P(1/N_1,1/N_2,\alpha)$ holds for any $N_1,N_2$ by small $C^1$-perturbation.
Hence $X,\sigma$ satisfies $P(1/N_1,1/N_2,\alpha)$ for any $N_1,N_2$. This gives 
$x\in W^u(\sigma)\cap C$ and $t_n\to +\infty$ such that~\eqref{e.non-domination} holds.
\end{proof}

\begin{proof}[Proof of Proposition~\ref{p:Lorenz-like2}]
Since $X$ is generic, one can assume that any singularity is hyperbolic.
Since $C$ contains a singularity { with stable dimension equal to $2$}, by Theorem~\ref{t.connecting} it is Lyapunov stable,
in particular it contains the unstable manifolds of its singularities. Let $\cN=\cE\oplus \cF$ be the dominated splitting for the linear Poincar\'e flow
on $C\setminus \sing(X)$.

Any singularity $\sigma$ { with stable dimension equal to $2$} has real eigenvalues: indeed, by iterating backwards the dominated splitting of the linear Poincar\'e
flow along an unstable orbit of $\sigma$, one deduces that the two stable eigenvalues have different moduli.
If one assumes by contradiction that $W^{ss}(\sigma)\cap C(\sigma)\neq \{\sigma\}$,
then there exists $x\in W^u(\sigma)\cap C$ and $t_n\to +\infty$ such that~\eqref{e.non-domination} holds.
We have $\cF(x)=E^{cu}(x)\cap \cN_x$: indeed, for any two planes $E_1,E_2\in T_xM$ containing $X(x)$
and different from $E^{cu}(x)$, the backward iterates converge to $E^{ss}(\sigma)\oplus E^u(\sigma)$,
and get arbitrarily close; hence for any two lines
$\cE_1,\cE_2\subset \cN_x$ different from $E^{cu}(x)\cap \cN_x$, the backward iterates under the linear Poincar\'e flow $\psi$
get arbitrarily close
and this property characterizes the space $\cF(x)$. Moreover $\cF$ is continuous at non-singular points of $C$
(by uniqueness of the dominated splitting). But this contradicts~\eqref{e.non-domination} {as in Corollary~\ref{c.non-domination}} which can be restated as:
$$\varphi_{t_n}(x)\rightarrow x \text{ and }
\psi_{t_n}(x).\cF(x)\not \rightarrow \cF(x).$$
We thus have $W^{ss}(\sigma')\cap C(\sigma)= \{\sigma\}$.

Let us assume now by contradiction that $C$ contains a singularity { with stable dimension equal to $1$}.
One can apply the previous discussions to $-X$: the class contains the stable manifold of its singularities.
This contradicts the fact that $W^{ss}(\sigma)\cap C= \{\sigma\}$.
Consequently any singularity in $C$ has { stable dimension equal to $2$}. This ends the proof.
\end{proof}

\subsection{Dominated splitting on singular classes}\label{Sec:dom-spl-sing-class}
Now we can prove Theorem~\ref{Thm-domination}.
Let us consider $X$ in the residual set of vector fields
satisfying Theorems~\ref{Lem:generic},  \ref{t.connecting}, \ref{t.hyperbolic}, \ref{Thm:Lorenz-like},
Proposition~\ref{p:Lorenz-like2} and the following property:

\emph{If the chain-recurrence class $C(\sigma)$ of a hyperbolic singularity has no
dominated splitting for the tangent flow, then for the vector fields $Y$ that are $C^1$-close to $X$,
the tangent flow on the chain-recurrence class $C(\sigma_Y)$ of the continuation of $X$
has no dominated splitting.}
\smallskip

This property can be deduced from the semi-continuity of $Y\mapsto C(\sigma_Y)$
and the fact that if the tangent flow on an invariant compact set $\Lambda_0$ for $Y_0$ is dominated,
it is still the case for any compact set $\Lambda$ Hausdorff close to $\Lambda_0$
and vector fields $Y$ $C^1$-close to $Y_0$.
\medskip

Let $\Lambda$ be a chain-transitive set of $X$ such that the tangent flow on $\Lambda$ has a dominated splitting
$E\oplus F$.
If $\Lambda$ is non-singular, let us consider the two disjoint invariant compact
sets $K_E:=\{x, X(x)\in E\}$ and $K_F:=\{x, X(x)\in F\}$.
For $\varepsilon>0$ let $A:=\{x, d(X(x),E)\geq \varepsilon\}$.
By domination, the set $A$ is sent into its interior by large forward iterates.
The chain-transitivity implies that $\Lambda=A$ or $A=\emptyset$.
Since the orbit of any point $x\notin K_E\cup K_F$ accumulates to $K_F$ in the future and
to $K_E$ in the past, this gives $\Lambda=K_E$ or $\Lambda=K_F$.
Without loss of generality we assume the first case.
Then $\dim(E)=2$, since otherwise $F$ is uniformly expanded by the domination and
$\Lambda$ is a source. 
Thus for any $x\in \Lambda$,
$\cE(x):=E(x)\cap \cN_x$ and $\cF(x):=(F(x)\oplus \RR X(x))\cap \cN_x$
are two one-dimensional lines which define two bundles invariant and dominated under $\psi$.

In the remaining case, $\Lambda$ contains a singularity $\sigma$.
By Theorem~\ref{t.connecting}, $\Lambda$ is a chain-recurrence class.
The existence of a dominated splitting is an open property:  there are neighborhoods $\cU$, $U$
of $X$, $\Lambda$ such that for any $Y\in \cU$, the maximal invariant set of $Y$ in $U$
has a dominated splitting; in particular, it does not contain a homoclinic tangency of a periodic orbit.
Hence the assumptions of Theorem~\ref{Thm:Lorenz-like} hold and  the linear Poincar\'e flow on $\Lambda\setminus \sing(X)$
is dominated. This proves one half of Theorem~\ref{Thm-domination}.
\medskip

We now assume $\Lambda$ is a chain-transitive set of $X$ such that the linear Poincar\'e flow on $\Lambda\setminus \sing(X)$ has a dominated splitting. 
If $\Lambda$ does not contain any singularity, it is a hyperbolic set by Theorem~\ref{t.hyperbolic}:
hence it has a dominated splitting also.
Thus we assume that $\Lambda$ contains a singularity $\sigma$ { with stable dimension equal to $2$}
(the case of dimension $1$ is similar).
By Theorem~\ref{t.connecting}, $\Lambda$ is a chain-recurrence class and is Lyapunov stable.

By Proposition~\ref{p:Lorenz-like2}, any singularity $\widetilde \sigma\in \Lambda$ has { stable dimension equal to $2$},
real simple eigenvalues and satisfies $W^{ss}(\widetilde \sigma)\cap \Lambda=\{\widetilde \sigma\}$.
Note that these properties still hold and the linear Poincar\'e flow is still dominated for the vector fields $C^1$-close and the chain-recurrence class of
the continuation of $\sigma$: one uses that the chain-recurrence class of the continuation of $\sigma$ vary semi-continuously
with the vector field $X$, that the set of singularities is finite and Proposition~\ref{p.robustness-DS}.
We consider a vector field $Y$ that is $C^1$-close to $X$, and that is $C^2$,whose regular periodic orbits are hyperbolic and with no normally expanded invariant torus whose dynamics is topologically equivalent
to an irrational flow (see Theorem~\ref{Lem:generic} in $\cX^3(M)$).
In particular the periodic orbits in the chain-recurrence class $C(\sigma_Y)$ of the continuation of $\sigma$ for $Y$
are neither sources nor sinks and have a negative Lyapunov exponent.
One can thus apply Theorem~A':
there exists a dominated splitting for the tangent flow on $C(\sigma_Y)$.
By our choice of the generic vector field $X$, this is also the case for the chain-recurrence class of $\sigma$ for $X$, which is the set $\Lambda$.

The Theorem~\ref{Thm-domination} is proved. \qed

\subsection{Dichotomy for three-dimensional vector fields}

We now complete the proofs of the Main Theorem and of
Corollary~\ref{c.main}.

\begin{proof}[Proof of the Main Theorem]
Let us consider a vector field $X$ in the intersection of the residual sets provided by Theorems~\ref{Thm-domination}, \ref{Lem:generic},  \ref{t.hyperbolic}, \ref{t.GY} and \ref{Thm:Lorenz-like}.
Let us  assume that it can not be accumulated in $\cX^1(M)$ by vector fields with homoclinic tangencies.

Let $C$ be a chain-recurrence class of $X$. By Theorem~\ref{Lem:generic}, if $C$ is an isolated singularity or a regular periodic orbit, it is hyperbolic.
If $C$ is non-trivial, by Theorem~\ref{Thm:Lorenz-like} the linear Poincar\'e flow on $C\setminus \sing(X)$ is dominated.
Using Theorem~\ref{t.hyperbolic} if $C\cap \sing(X)=\emptyset$, the class $C$ is hyperbolic.
In the remaining case, $C$ is non-trivial, contains a singularity,  the tangent flow on $C$ is dominated (by Theorem~\ref{Thm-domination}).
Hence $C$ is singular hyperbolic (by Theorem~\ref{t.GY} ).

Since any chain-recurrence class of $X$ is singular hyperbolic,  $X$ is singular hyperbolic.
\end{proof}

\begin{proof}[Proof of Corollary~\ref{c.main}]
Let $\cO$ be the set of $C^1$ vector fields on $M$
whose chain-recurrence classes are robustly transitive.
We then introduce the dense set $\cU=\cO\cup(\cX^{1}(M)\setminus \overline{\cO})$.

We claim that $\cO$ (and thus $\cU$) is open.
Indeed for $X\in \cO$, each chain-recurrence class is isolated
in the chain-recurrent set: let us consider a class $C$;
by semi-continuity of the chain-recurrence
classes for the Hausdorff topology, if $C'$ is another class having
a point close to $C$, it is contained in a small neighborhood of $C$,
hence coincide with $C$ by definition of the robust transitivity.
This implies that $X$ has only finitely many chain-recurrence
classes $C_1,\dots,C_k$. By robust transitivity,
each of them admits a neighborhood $U_1,\dots,U_k$
so that for any $Y$ close to $X$ in $\cX^1(M)$,
the maximal invariant set in each $U_i$ is robustly transitive.
By semi-continuity of the chain-recurrence classes,
each class of $Y$ has to be contained in one of the $U_i$,
hence is robustly transitive, as required.

Let $\cG$ be the dense G$_\delta$ set of vector fields
in $\cX^1(M)$ such that Theorem~\ref{t.robust-transitivity} holds.
Let us consider any $X\in \cU$ that can not be approximated by
vector fields exhibiting a homoclinic tangency.
By the Main Theorem, there exists $X'$ arbitrarily close
to $X$ in $\cX^1(M)$ which is singular hyperbolic.
Since singular hyperbolicity is an open property and $\cG$
is dense, one can also require that $X'\in \cG$,
hence each chain-recurrence class of $X$ is robustly transitive.
We have thus shown that $X\in \overline \cO$.
By definition of $\cU$ this gives $X\in \cO$
and the Corollary follows.
\end{proof}

\vskip 20pt

\begin{tabular}{l l l}
\emph{\normalsize Sylvain Crovisier}
& \quad\quad \quad &
\emph{\normalsize Dawei Yang}
\medskip\\

Laboratoire de Math\'ematiques d'Orsay
&& School of Mathematical Sciences\\
CNRS - Universit\'e Paris-Sud
&& Soochow University\\
Orsay 91405, France
&& Suzhou, 215006, P.R. China\\
\texttt{Sylvain.Crovisier@math.u-psud.fr}
&& \texttt{yangdw1981@gmail.com, yangdw@suda.edu.cn}
\end{tabular}


\small
\begin{thebibliography}{XXW}

\bibitem[ABC]{abc-measure} F. Abdenur, C. Bonatti and S. Crovisier,
Non-uniform hyperbolicity for $C^1$-generic diffeomorphisms.
\emph{Israel J. Math.} {\bf 183} (2008), 1--60.

\bibitem[ABS]{ABS}
V. Afra\u\i movi\v c, V. Bykov, and L. Silnikov, The origin and structure of the Lorenz attractor.
 {\it Dokl.
Akad. Nauk SSSR} {\bf 234} (1977), 336--339.

\bibitem[ARH]{ARH}
A. Arroyo and F. Rodriguez Hertz, Homoclinic bifurcations and uniform hyperbolicity for three-dimensional flows.
{\it Ann. Inst. H. Poincar\'e Anal. Non Lin\'eaire} {\bf20} (2003), 805--841.

\bibitem[AP]{AP}
V. Ara\'ujo and M.J. Pacifico, Three-dimensional flows.
\emph{Ergebnisse der Mathematik und ihrer Grenzgebiete} \textbf{53}. Springer (2010). 

\bibitem[BC$_1$]{BC} C. Bonatti and S. Crovisier, R\'ecurrence et g\'en\'ericit\'e.
\textit{Invent. Math.} {\bf 158} (2004), 33--104.

\bibitem[BC$_2$]{BC-whitney} C. Bonatti and S. Crovisier,
Center manifolds for partially hyperbolic set without strong unstable connections.
\emph{Journal of the IMJ} \textbf{15} (2016), 785--828 .

\bibitem[BDV]{BDV}C. Bonatti L. D\'iaz and M. Viana, Dynamics beyond uniform
hyperbolicity. A global geometric and probabilistic perspective. \emph{Encyclopaedia of Mathematical Sciences} {\bf 102}.
Mathematical Physics, III. Springer-Verlag (2005).

\bibitem[BGY]{BGY}C. Bonatti, S. Gan and D. Yang, Dominated chain recurrent classes with singularities.
{\it Ann. Sc. Norm. Super. Pisa Cl. Sci.} {\bf 14} (2015), 83--99.

\bibitem[Co]{Con}C. Conley, {\it Isolated invariant sets and Morse
index}. CBMS Regional Conference Series in Mathematics {\bf38}. AMS (1978).

\bibitem[Cr$_1$]{crovisier-approximation} S. Crovisier,
Periodic orbits and chain-transitive sets of $C^1$-diffeomorphisms.
\textit{Publ. Math. Inst. Hautes \'Etudes Sci.} \textbf{104} (2006), 87--141.

\bibitem[Cr$_2$]{C-asterisque} S. Crovisier,  Perturbation de la dynamique de diff\'eomorphismes en topologie $C^1$.  {\it Ast\'erisque} {\bf354} (2013).

\bibitem[CP]{CP} S. Crovisier and E. Pujals,
Essential hyperbolicity and homoclinic bifurcations: a dichotomy phenomenon/mechanism for
diffeomorphisms. \textit{Invent. Math.} {\bf 201} (2015), 385--517.

\bibitem[CPS]{CPS} S. Crovisier, E. Pujals and M. Sambarino, \emph{Hyperbolicity of the extremal bundles.} {In preparation}.

\bibitem[CY$_1$]{CY} S. Crovisier and D.Yang, On the density of singular hyperbolic three-dimensional vector fields: a conjecture of Palis.
{\it C. R. Math. Acad. Sci. Paris} {\bf353} (2015), 85--88.

\bibitem[CY$_2$]{CY2} S. Crovisier and D.Yang, \emph{Robust transitivity of singular hyperbolic attractors}.
In preparation.

\bibitem[dMvS]{dMvS} W. de Melo and S. van Strien, {\it One-dimensional dynamics.}
Ergebnisse der Mathematik und ihrer Grenzgebiete \textbf{25}.
Springer-Verlag (1993). 

\bibitem[GY]{GY} S. Gan and D. Yang, \emph{Morse-Smale systems and horseshoes for three dimensional flows.}
To appear at \emph{Ann. Sci. \'Ecole Norm. Sup.} Preprint (2013) ArXiv:1302.0946.

\bibitem[G]{Gu}
J. Guckenheimer, A strange, strange attractor. {\it The Hopf bifurcation  theorems and its applications. Applied
Mathematical Series} {\bf 19}. Springer-Verlag (1976), 368--381.

\bibitem[GW]{GW}
J. Guckenheimer and R. Williams, Structural stability of Lorenz attractors. {\it  Publ. Math. Inst. Hautes \'Etudes Sci.} {\bf 50} (1979), 59--72.

\bibitem[H]{hayashi} S. Hayashi,
Connecting invariant manifolds and the  solution of the $C^1$-stability and $\Omega$-stability conjectures for flows. \textit{Ann. of Math.} {\bf 145} (1997), 81--137 and \textit{Ann. of Math.} {\bf 150} (1999), 353--356.

\bibitem[HH]{hector-hirsch-foliation} G. Hector and U. Hirsch,
\emph{Introduction to the geometry of foliations}. Vieweg (1981).

\bibitem[HPS]{hirsch-pugh-shub} M. Hirsch, C. Pugh and M. Shub,
\textit{Invariant manifolds.} Lecture Notes in Mathematics \textbf{583} (1977).

\bibitem[K]{kupka}
I. Kupka. Contributions \`a la th\'eorie des champs g\'en\'eriques.
\textit{Contrib. Differential Equations} \textbf{2} (1963), 457--484.

\bibitem[LGW]{lgw-extended} M. Li, S. Gan and L. Wen, Robustly transitive singular sets via
approach of extended linear Poincar\'e flow. {\it Discrete Contin.
Dyn. Syst. }{\bf 13 } (2005), 239--269.

\bibitem[Li$_1$]{Lia63} S. Liao, Certain ergodic properties of a differential system on a compact differentiable manifold. {\it Acta Sci. Natur. Univ. Pekinensis} {\bf9} (1963), 241--265 and 309--326. Translated in {\it Front. Math. China} {\bf1} (2006), 1--52. 

\bibitem[Li$_2$]{Lia89} S. Liao, On $(\eta,d)$-contractible orbits of vector
fields. {\it Systems Science and Mathematical Sciences}
{\bf2} (1989), 193--227.

\bibitem[Lo]{Lo}
E. N. Lorenz, Deterministic nonperiodic flow. {\it J. Atmosph. Sci.} {\bf 20} (1963), 130--141.

\bibitem[M$_1$]{Man85} R. Ma\~n\'e, Hyperbolicity, sinks and measure in one-dimensional dynamics. {\it Comm. Math. Phys.},
{\bf 100} (1985), 495--524.

\bibitem[M$_2$]{Man87} R. Ma\~n\'e, {\it Ergodic Theory and Differential Dynamics}. Springer-Verlag (1987).

\bibitem[MCY]{MCY} Z. Mi, Y. Cao and D. Yang, \emph{SRB measures for diffeomorphisms with continuous invariant splittings.} Preprint (2015) ArXiv:1508.02747.

\bibitem[MP]{MP} C. Morales and M. Pacifico,  A dichotomy for
three-dimensional vector fields. {\it Ergodic Theory and Dynamical
Systems} {\bf23} (2003), 1575--1600.

\bibitem[MPP]{MPP} C. Morales, M. Pacifico, and E. Pujals, Robust transitive singular sets for 3-flows are partially hyperbolic
attractors or repellers. {\it  Ann. of Math.} {\bf160} (2004), 375--432.

\bibitem[Ne]{Ne1} S. Newhouse, Nondensity of axiom A(a) on $S^2$. {\it Global analysis I, Proc. Symp. Pure Math. AMS}
{\bf14} (1970), 191-202.

\bibitem[No]{No} A. Nobile, Some properties of the Nash blowing-up.
\emph{Pacific J. Math.} \textbf{60} (1975), 297--305.

\bibitem[Pa$_1$]{palis} J. Palis,
Homoclinic bifurcations, sensitive-chaotic dynamics and strange attractors.
\emph{in} Dynamical systems and related topics (Nagoya, 1990).
\emph{Adv. Ser. Dynam. Systems} \textbf{9}. World Sci. Publishing (1991), 466--472.

\bibitem[Pa$_2$]{Pal00} J. Palis, A global view of dynamics and a conjecture of
the denseness of finitude of attractors. {\it Ast\'erisque}, {\bf
261} (2000), 335--347.

\bibitem[Pa$_3$]{Pal05} J. Palis, A global perspective for non-conservative dynamics.
{\it Ann. Inst. H. Poincar\'e Anal. Non Lin\'eaire}, {\bf22} (2005),
485--507.

\bibitem[Pa$_4$]{Pal08} J. Palis, Open questions leading to a global perspective in
dynamics. {\it Nonliearity}, {\bf21} (2008), 37--43.

\bibitem[Pe]{Pe} M. Peixoto, Structural stability on two-dimensional manifolds. {\it Topology} {\bf 1} (1962), 101--120.

\bibitem[PS$_1$]{PS1} E. Pujals and M. Sambarino,
Homoclinic tangencies and hyperbolicity for surface diffeomorphisms.
\textit{Ann. of Math.} \textbf{151} (2000), 961--1023.

\bibitem[PS$_2$]{PS2} E. Pujals and M. Sambarino, Integrability on codimension one dominated splitting. {\it Bull. Braz. Math. Soc.} {\bf 38} (2007), 1--19.

\bibitem[Pu]{Pu} C. Pugh, Structural stability on $M$.
\emph{An. Acad. Brasil. Ci.} \textbf{39} (1967), 45--48.

\bibitem[S]{smale} S. Smale,
Stable manifolds for differential equations and diffeomorphisms.
\textit{Ann. Scuola Norm. Sup. Pisa} \textbf{17} (1963), 97--116.

\bibitem[T]{takens} F. Takens, Singularities of vector fields. \emph{Publ. Math. Inst. Hautes \'Etudes Sci.} \textbf{43} (1974), 47--100.

\end{thebibliography}
\end{document}